\documentclass[11pt]{article}
\usepackage{etex}

\usepackage[utf8x]{inputenc}
\usepackage{microtype}
\usepackage[T1]{fontenc}
\usepackage{times}
\usepackage{float}
\usepackage{mathrsfs}
\usepackage{verbatim,amsmath,amsthm,amsfonts,amssymb,latexsym,graphicx,mathtools,extpfeil,color,mathabx}
\usepackage{stmaryrd}
\usepackage{epstopdf,pinlabel}
\usepackage[all]{xy}
\usepackage{graphicx}
\usepackage{tabularx,caption}
\usepackage{subcaption}
\usepackage{hyperref}
\usepackage{tikz-cd}
\usepackage{relsize}
\hypersetup{
  colorlinks = true,
  linkcolor  = black
}
\usepackage{multicol, color}
\definecolor{green}{rgb}{.1,.6,0}

\DeclareMathAlphabet{\mathpzc}{OT1}{pzc}{m}{it}
\usepackage{accents}
\newlength{\dhatheight}
\DeclareMathOperator{\Sym}{Sym}
\DeclareMathOperator{\rank}{rank}
\DeclareMathOperator{\tr}{tr}
\DeclareMathOperator{\ad}{ad}
\DeclareMathOperator{\ind}{index}
\DeclareMathOperator{\Diff}{Diff}
\DeclareMathOperator{\im}{im}

\newcommand{\N}{{\bf N}}
\newcommand{\Z}{{\bf Z}}
\newcommand{\Q}{{\bf Q}}
\newcommand{\R}{{\bf R}}
\newcommand{\C}{{\bf C}}
\newcommand{\rD}{{\rm D}}
\newcommand{\rI}{{\rm I}}
\newcommand\e{{\rm e}}
\newcommand{\re}{{\rm e}}
\newcommand{\bi}{{\bf i}}
\newcommand{\U}{{\rm U}}
\newcommand{\SU}{{\rm SU}}
\newcommand{\Sp}{{\rm Sp}}
\newcommand{\SO}{{\rm SO}}
\newcommand{\Or}{{\rm O}}
\newcommand{\su}{\mathfrak{su}}
\newcommand\A{{\mathbb A}}
\newcommand{\id}{{\rm id}}
\newcommand\DXw{{\rm D}_{X,w}}
\newcommand\hDXw{\widehat {\rm D}_{X,w}}

\newcommand\IN{{\rm I}^N_*}
\newcommand\IIN{{\mathbb I}^N_*}
\newcommand\CS{{\rm CS}}
\newcommand\Vgd{{\mathbb V_{g,d}}}
\newcommand\VgdN{{\mathbb V_{g,d}^N}}

\newcommand\bIgd{{{\mathbb I}_{g,d}}}
\newcommand\bIgdN{{{\mathbb I}_{g,d}^N}}

\newcommand\SHI{{\rm SHI}}
\newcommand\vab{v_{(\alpha,\beta)}}

\newcommand\tbIgd{{\widetilde {\mathbb I}_{g,d}}}

\newcommand\BS{\Sigma(a_1,a_2,a_3)}
\newcommand\cp{\overline{\bf{CP}}^2}
\newcommand\Nnd{{\mathcal N_{N,d}}}

\def\({\mathopen{}\left(}
\def\){\right)\mathclose{}}

\newcommand{\Addresses}{{
  \bigskip
  \footnotesize

  Aliakbar Daemi, \textsc{Simons Center for Geometry and Physics, State University of New York,
  Stony Brook, NY 11794}\par\nopagebreak
  \textit{E-mail address}: \texttt{adaemi@scgp.stonybrook.edu}

  \medskip

   Yi Xie , \textsc{Simons Center for Geometry and Physics, State University of New York,
  Stony Brook, NY 11794}\par\nopagebreak
  \textit{E-mail address}: \texttt{yxie@scgp.stonybrook.edu}
}}

\theoremstyle{plain}
\newtheorem{theorem-intro}{Theorem}

\DeclareMathOperator{\gker}{ker_{gen}}
\DeclareMathOperator{\I}{I}
\DeclareMathOperator{\KHI}{KHI}

\usepackage{aliascnt}
\numberwithin{equation}{section}

\def\makeautorefname#1#2{\AtBeginDocument{\expandafter\def\csname#1autorefname\endcsname{#2}}}

\newcommand{\mynewtheorem}[2]{
  \newaliascnt{#1}{equation}
  \newtheorem{#1}[#1]{#2}
  \aliascntresetthe{#1}
  \makeautorefname{#1}{#2}
}
\mynewtheorem{theorem}{Theorem}
\mynewtheorem{prop}{Proposition}
\mynewtheorem{cor}{Corollary}
\mynewtheorem{construction}{Construction}
\mynewtheorem{lemma}{Lemma}
\mynewtheorem{conjecture}{Conjecture}
\mynewtheorem{question}{Question}
\mynewtheorem{problem}{Problem}

\numberwithin{substep}{step}
\makeautorefname{step}{Step}
\makeautorefname{substep}{Step}

\numberwithin{subcase}{case}
\makeautorefname{case}{Case}
\makeautorefname{subcase}{case}

\theoremstyle{remark}
\mynewtheorem{remark}{Remark}

\theoremstyle{definition}
\mynewtheorem{definition}{Definition}
\mynewtheorem{example}{Example}
\mynewtheorem{exercise}{Exercise}
\mynewtheorem{convention}{Convention}
\newtheorem*{convention*}{Convention}
\newtheorem*{conventions*}{Conventions}

\makeautorefname{chapter}{Chapter}
\makeautorefname{section}{Section}
\makeautorefname{subsection}{Section}
\makeautorefname{subsubsection}{Section}

\title{\large \bf Sutured Manifolds and Polynomial Invariants from Higher Rank Bundles}
\author{\bf \sc \large Aliakbar Daemi, Yi Xie}
\date{}
\begin{document}
\maketitle
\begin{abstract}
	For each integer $N\geq 2$, Mari\~no and Moore defined generalized Donaldson invariants by the methods of quantum field theory, and made predictions about the values of these invariants.
	Subsequently, Kronheimer gave a rigorous definition of generalized Donaldson invariants using the moduli spaces of anti-self-dual connections on hermitian vector bundles of rank $N$.
	In this paper, Mari\~no and Moore's predictions are confirmed for simply connected elliptic surfaces without multiple fibers and certain surfaces of general type in the case that $N=3$. The primary motivation
	is to study 3-manifold instanton Floer homologies which are defined by higher rank bundles. In particular, the computation of the generalized Donaldson invariants are exploited to define a Floer homology
	theory for sutured 3-manifolds.
\end{abstract}
\newpage
\tableofcontents
\newpage
\section{Introduction}
Sutured manifolds were introduced by Gabai \cite{G:Sut-1} to study foliations and the Thurston norm of 3-manifolds \cite{Th:norm}. A sutured manifold is a pair of a 3-manifold $M$ and an oriented 1-manifold $\alpha\subset M$ which decomposes the boundary of $M$ in an appropriate way. In \cite{G:Sut-1}, Gabai also defines an operation on sutured manifolds, which is called {\it surface decomposition}. Surface decompositions can be used to simplify sutured manifolds. Foliations of sutured manifolds are also well-behaved with respect to surface decompositions. As a result, Gabai was able to construct {\it taut} foliations for certain families of 3-manifolds in an inductive way.

Floer homological invariants serve as another set of tools for studying  topology and geometry of 3-dimensional manifolds. Such invariants were initially constructed for closed and oriented 3-manifolds: $\U(N)$-instanton Floer homology \cite{Fl:I, F:sur-rel, KM:YAFT}, Heegaard Floer homology \cite{OzSz:HF}, monopole Floer homology \cite{KM:monopoles-3-man}, and embedded contact homology \cite{ECH:notes-I,ECH:notes-II}. Later, Juh\'asz defined {\it sutured Floer homology}, a generalization of Heegaard Floer homology to {\it balanced} sutured 3-manifolds \cite{juhasz}.\footnote{For the definition of balanced sutured 3-manifolds, see Definition \ref{balance-suture-manifolds}.} Subsequently, sutured version of $\U(2)$-instanton Floer homology \cite{KM:suture}, monopole Floer homology \cite{KM:suture}, and embedded contact homology were constructed \cite{CGHH:suture,CGH:suture,KST:ECH-natural}. In particular, Kronheimer and Mrowka used sutured $\U(2)$-instanton homology as the main ingredient to establish that Khovanov homology detects the unknot \cite{KM:Kh-unknot}. This invariant was also used to reprove Property P for knots \cite{KM:suture}, and it lies in the core of a program in the hope of finding a computer-free proof of the famous four color theorem \cite{KM:J-sharp}. The primary motivation for this article is to extend $\U(N)$-instanton Floer homology to sutured manifolds for higher values of $N$.

\subsection{Motivation}
Fix an integer $N\geq 2$, and let $K$ be a knot in an integral homology sphere $Y$. Let also $\mu$ denote an element of the knot group, $\pi_1(Y\backslash K)$, represented by a meridian of $K$:
\begin{question}\label{rank-N-rep}
	Does there exist a representation
	$\varphi: \pi_1(Y\backslash K) \to \SU(N)$ with non-abelian image such that:
	\begin{equation}\label{mon-cond}
		\varphi(\mu)=
			c\left[
			\begin{array}{cccc}
				1&0&\dots&0\\
				0&\zeta&\dots&0\\
				\vdots&\vdots&\ddots&\vdots\\
				0&0&\dots&\zeta^{N-1}\\
			\end{array}
			\right]
	\end{equation}
	where $\zeta=\e^{2\pi \bi/N}$, and $c=\e^{\pi \bi/N}$ or $1$ depending on whether $N$ is even or odd?
\end{question}
\noindent
In the case that $K$ is the unknot, the answer to the above question is clearly negative. Note also that if for a knot $K$, there is a representation to $\SU(N)$ with the mentioned properties, then there is also a desired representation from $\pi_1(Y\backslash K)$ to $\SU(lN)$ for any positive integer $l$.

Suppose $Y$ is a homotopy sphere\footnote{By the Poincar\'e Conjecture, this is equivalent to say that $Y=S^3$. However, we are making this assumption to show that our proposed approach does not require the Poincar\'e Conjecture.} and the answer to Question \ref{rank-N-rep} for any non-trivial knot $K$ in $Y$ is positive. A non-abelian representation $\varphi$ satisfying \eqref{mon-cond} determines a non-trivial representation of $\pi_1(\Sigma_N(K))$ with $\Sigma_N(K)$ being the $N$-fold cyclic branched cover of $Y$, branched along $K$. This verifies the {\it Covering Conjecture}, which asserts that $\Sigma_N(K)$, for a non-trivial knot $K$, is not a homotopy sphere \cite[Problem 3.38]{Ki:problem-list}. A consequence of the Covering Conjecture is the {\it Smith Conjecture}, stating that a non-trivial knot is not the fixed point set of an orientation preserving homeomorphism $f:S^3\to S^3$ of order $N$ \cite[Problem 3.38]{Ki:problem-list}. The Covering Conjecture and the Smith Conjecture are both theorems, proved by {\it geometrization} techniques \cite{Smith-conj}.

Kronheimer and Mrowka's sutured $\U(2)$-instanton homology group, $\SHI_*^2$, can be employed to answer Question \ref{rank-N-rep} affirmatively for $N=2$ (and hence for any even $N$) and any non-trivial knot $K$ \cite{KM:suture}.\footnote{The original notation for sutured $\U(2)$-instanton homology is $\SHI_*$. Here we use the superscript $2$ to indicate that this invariant is the sutured version of $\U(2)$-instanton homology.} Associated to any knot $K$, there is a sutured manifold $(M(K),\alpha(K))$ where $M(K)$ is the knot complement and $\alpha(K)$ is the union of two oppositely oriented meridional curves. Kronheimer and Mrowka proved that if the dimension of $\SHI_*^2(M(K),\alpha(K))$ is greater than $1$, then there is a non-abelian representation of the knot group of $K$ that satisfies \eqref{mon-cond}. Similar to foliations, $\SHI_*^2$ also behaves well with respect to surface decomposition, and one can inductively construct non-trivial elements of $\SHI_*^2(M(K),\alpha(K))$ after simplifying $(M(K),\alpha(K))$ by a series of sutured decomposition. In particular, the dimension of $\SHI_*^2(M(K),\alpha(K))$ is at least two for a non-trivial knot $K$. It is also shown in \cite{Fri:Eta-Slice,BoFri:met-rep-I} that if $K$ is a knot with non-trivial Alexander polynomial, then the answer to Question \ref{rank-N-rep} is positive for infinitely many values of $N$. In the light of the success of $\SHI_*^2$ in addressing Question \ref{rank-N-rep}, it is natural to look for the generalization of $\SHI_*^2$ for higher values of $N$.

The essential device in the definition of sutured Floer homology group $\SHI_*^2$ is an excision theorem for $\U(2)$-instanton Floer homology \cite{F:sur-rel,DB:sur-rel,KM:suture}. The proof of the excision theorem is in turn based on Mu\~noz's characterization of the structure of a $\U(2)$-instanton Floer homology group associated to the 3-manifold $S^1 \times \Sigma$ where $\Sigma$ is a Riemann surface \cite{Mun:ring}. Mu\~noz's work borrows some results about the cohomology ring of the moduli space of rank 2 stable bundles \cite{Zag:stable-ring,KN:stable-ring,ST:stable-ring,Bar:stable-ring}, which are not available for higher values of the rank.

In the present paper, we establish an excision theorem for $N=3$ using the relationship between instanton Floer homology and generalizations of Donaldson invariants from \cite{K:higher}. Roughly speaking, there is a $(3+1)$-dimensional topological quantum field theory which associates $\U(N)$-instanton Floer homology to 3-manifolds, and its values for closed 4-manifolds is given by $\U(N)$ analogues of Donaldson's polynomial invariants. This relationship between $\U(2)$-instanton Floer homology and polynomial invariants have been extensively used to compute the invariants of 4-manifolds. In this paper, we firstly use the TQFT structure to compute the $\U(3)$-polynomial invariants of some families of smooth 4-manifolds. Next, we work in the other direction, and use our knowledge of $\U(3)$-polynomial invariants to obtain a better understanding of certain $\U(3)$-Floer homologies. This allows us to prove the excision theorem and define a Floer homology group $\SHI_*^3$ for sutured manifolds in the case that $N=3$. Computations of generalized polynomial invariants in the physics literature \cite{MM:higher} suggest that our approach can be also exploited for higher values of $N$.

\subsection{Statement of Results}
In his groundbreaking work \cite{Don:inv}, Donaldson defined polynomial invariants for a smooth manifold $X$ using the moduli space of {\it Anti-Self-Dual} connections on $X$. In his work, $X$ is simply connected, $b^+(X)$ is an integer greater than $1$, and the ASD connections are assumed to be defined on an $\SU(2)$-bundle $E$ over $X$. Although the assumption on $b^+(X)$ is essential, the definition of polynomial invariants was subsequently generalized to the case that $X$ is not simply connected \cite{KM:str-thm} and $E$ is a $\U(N)$-bundle \cite{K:higher,Luc:Thesis}. Polynomial invariants have been extensively studied in the case that $N=2$. However, there is not much known about these invariants for higher values of $N$.

For a smooth and connected 4-manifold $X$, suppose the algebra $\A(X)$ is defined as:
\begin{equation*}
	\mathbb{A}(X):=\Sym^* (H_{0}(X)\oplus H_{2}(X)) \otimes \Lambda^* (H_{1}(X)).
\end{equation*}
where $H_i(X)$ is computed with coefficients in $\C$.
Form the tensor product algebra $\A(X)^{\otimes (N-1)}$, and for $\alpha\in H_i(X)$ and $2\leq r \leq N$, let $\alpha_{(r)}$ be the corresponding element in the $(r-1)^{\rm st}$ factor of $\mathbb A(X)^{\otimes (N-1)}$. In the case that $\alpha$ is the generator of $H_0(X)$, this element of $\mathbb A(X)^{\otimes (N-1)}$ is denoted by $a_r$. We also define a grading on $\mathbb A(X)^{\otimes (N-1)}$ such that for $\alpha\in H_i(X)$, the degree of $\alpha_{(r)}$ is equal to $2r-i$. A Hermitian vector bundle $E$ of rank $N$ on $X$ is determined by its first and second Chern classes. Suppose $c_1(E)$ is represented by an embedded surface $w$ in $X$ and $c_2(E)[X]=k$. Then the {$\U(N)$-polynomial invariants} associated to the bundle $E$ is a linear map\footnote{Our definition of $\U(N)$-polynomial invariants slightly differs from the Culler's definition. See Subsection \ref{pol-invts} for more details.}:
\begin{equation*}
	\rD^N_{X,w,k}:\A(X)^{\otimes (N-1)} \to \C.
\end{equation*}
For $z\in \A(X)^{\otimes (N-1)}$, the complex number $\rD^N_{X,w,k}(z)$ is non-zero only if:
\begin{equation} \label{dimension-fomrula}
	\deg(z)=4Nk-2(N-1)w\cdot w-(N^2-1)\frac{\chi(X)+\sigma(X)}{2}
\end{equation}
Therefore, we will not lose any information, if we combine these invariants as:
\begin{equation*} \label{DNXw}
	\rD^N_{X,w}:=\sum_{k}\rD^N_{X,w,k}.
\end{equation*}

A substantial part of the present paper is devoted to computing $\U(3)$-polynomial invariants of some families of of algebraic surfaces. Our first result in this direction is the following:
\begin{theorem-intro} \label{K3-poly-invts}
	Suppose $X$ is a $K3$ surface. Then for any embedded oriented surface $w$ in $X$ and any element $z\in \A(X)^{\otimes 2}$:
	\begin{equation}\label{K3-simple-type}
		\rD^3_{X,w}(a_2^3z)=27\rD^3_{X,w}(z) \hspace{1cm}\rD^3_{X,w}(a_3z)=0.
	\end{equation}
	Moreover, if $\Gamma$ and $\Lambda$ are two elements of $H_2(X)$, then:
	\begin{equation}\label{K3-series}
		\rD^3_{X,w}((1+\frac{a_2}{3}+\frac{a_2^2}{9})\cdot \e^{\Gamma_{(2)}+\Lambda_{(3)}})=\e^{\frac{Q(\Gamma)}{2}-Q(\Lambda)}
	\end{equation}
\end{theorem-intro}
In order to clarify the statement of the above theorem, the following remarks are in order. The left hand side of \eqref{K3-series} is defined as:
\begin{equation*}
	\rD^3_{X,w}((1+\frac{a_2}{3}+\frac{a_2^2}{9})\cdot \e^{\Gamma_{(2)}+\Lambda_{(3)}}):=
	\sum_{i=0}^\infty\sum_{j=0}^\infty\frac{\rD^3_{X,w}((1+\frac{a_2}{3}+\frac{a_2^2}{9})\Gamma_{(2)}^i\Lambda_{(3)}^j)}{i!j!}
\end{equation*}
Theorem \ref{K3-poly-invts} asserts that the above series for a $K3$ surface is convergent, and the resulting number is equal to $\e^{Q(\Gamma)/2-Q(\Lambda)}$. Here $Q$ denotes the intersection form of $X$. That is to say, $Q(\Gamma)$ is the algebraic intersection number of $\Gamma$ with itself. In general, the intersection number of two homology classes $\Gamma$ and $\Gamma'$ is denoted by $\Gamma \cdot \Gamma'$. Since \eqref{K3-series} holds for all choices of $\Gamma$ and $\Lambda$, Formula \eqref{dimension-fomrula} allows us to compute the following polynomial invariants for all choices of non-negative integers $i$, $j$, an integer $k\in\{0,1,2\}$, and homology classes $\Gamma$ and $\Lambda$:
\begin{equation*}
	\rD^3_{X,w}(a_2^k\Gamma_{(2)}^i\Lambda_{(3)}^j)
\end{equation*}
These invariants determine $\rD^3_{X,w}(z)$ for all $z\in \A(X)^{\otimes 2}$, because the $K3$ surface satisfies \eqref{K3-simple-type} and $b_1(X)=0$.

Our computation of the invariants of $K3$ surfaces motivates the following definition: a smooth 4-manifold $X$ with $b^+(X)\geq 2$ and $b^1(X)=0$ has $w$-simple type with respect to an embedded surface $w$, if :
\begin{equation} \label{simple-type-w}
	\rD^3_{X,w}(a_2^3 z)=27 \rD^3_{X,w}(z) \hspace{1cm} \rD^3_{X,w}(a_3z)=0
\end{equation}
for all $z \in \A(X)^{\otimes 2}$. The 4-manifold $X$ has {\it simple type} if it has $w$-simple type with respect to any $w$ in $X$. As in the case of the $K3$ surfaces, if $X$ has simple type and the series:
\begin{equation*}
	\widehat \rD_{X,w}({\e}^{\Gamma_{(2)}+\Lambda_{(3)}}):=
	\rD^3_{X,w}((1+\frac{a_2}{3}+\frac{a_2^2}{9})\cdot \e^{\Gamma_{(2)}+\Lambda_{(3)}})
\end{equation*}
is convergent for all choices of $w$ and $\Gamma, \Lambda \in H_2(X)$, then these series determine all polynomial invariants of $X$.

We can extend our calculation for the $K3$ surfaces to a larger family of complex surfaces. Suppose $W(m,n)$ is the blowup of ${\bf CP}^1\times {\bf CP}^1$ at the $4mn$ singular points of the following (complex) curve:
	\begin{equation*}
		B:=\{p_1,\cdots , p_{2m}\}\times {\bf CP}^1 \cup {\bf CP}^1\times \{q_1,\cdots, q_{2n}\}.
	\end{equation*}
Let $\widetilde B$ be the proper transform of $B$, and define $X(m,n)$ to be the branched double cover of $W(m,n)$, branched along the smooth curve $\widetilde B$. The horizontal and vertical fibrations of $W(m,n)$ by projective lines lift to two fibrations of $X(m,n)$ whose generic fibers are denoted by $f_{m-1}$ and $f_{n-1}$. The Riemann surface $f_i$, for $i \in \{m-1, n-1\}$, has genus $i$. The complex surface $X(2,2)$ is a $K3$ surface. More generally, $X(m,2)$ is an elliptic surface without multiple fibers, which is usually denoted by $E(m)$ \cite{GS:Kir-cal}.

\begin{theorem-intro}\label{En}
	The elliptic surface $E(n)$ has simple type. Moreover, there are rational numbers $\hbar_1$ and $\hbar_2$ independent of $n$ such that for any embedded surfaces $w$ in $E(n)$ and
	$\Gamma,\Lambda \in H_2(E(n))$, the series
	$\widehat \rD_{E(n),w}({\e}^{\Gamma_{(2)}+\Lambda_{(3)}})$ is equal to:
	\begin{equation*}
		{\rm e}^{\frac{Q(\Gamma)}{2}-Q(\Lambda)}[\hbar_1 \cosh(\sqrt{3}f\cdot \Gamma)-
		2\hbar_2\cos(-\frac{2\pi}{3}w\cdot f+\sqrt 3 f \cdot \Lambda)]^{n-2}.
	\end{equation*}	
	where $f=f_1$ represents an elliptic fiber of $E(n)$.
	Furthermore, $\hbar_1+\hbar_2=\pm1$ for an appropriate choice of the sign.
\end{theorem-intro}
The constant numbers $\hbar_1$ and $\hbar_2$ In Theorem \ref{En} are respectively equal to $\frac{2}{3}$ and $\frac{1}{3}$ \cite{DX:in-prep}. The set of surfaces $X(m,n)$, as smooth 4-manifolds, are closed with respect to taking {\it fiber sums}.\footnote{See section \ref{FFH} for a review of the definition of fiber sum} For example, we can take the fiber sum of $X(m,n_1)$ and $X(m,n_2)$ along the fiber $f_{m-1}$, and the resulting 4-manifold is diffeomorphic to $X(m,n_1+n_2)$. Given two embedded surfaces $\Sigma_1\subset X(m,n_1)$ and $\Sigma_2\subset X(m,n_2)$ which intersect a fiber in the same number of points, we can form a surface $\Sigma_1 \# \Sigma_2 \subset X(m,n_1+n_2)$. Suppose $\mathcal H(m,n_1,n_2) \subset H_2(X(m,n_1+n_2))$ is the space of homology classes generated by homology classes of the surfaces of the form $\Sigma \# \Sigma'$. The following theorem about $X(m,4)$ is a consequence of Theorem \ref{D-inv-con-sum} about the polynomial invariants of fiber sums. In fact, Theorem \ref{D-inv-con-sum} can be used to obtain similar results about other surfaces in the family $X(m,n)$.

\begin{theorem-intro}\label{X(m,4)}
	For $m\geq 3$, let $w\subset X(m,4)$ be an embedded surface which has the form $w_1 \# w_2$ for $w_i \subset X(m,2)$
	and $w\cdot f_{m-1}\neq 0$ mod 3. 	Let $K$ denote the canonical class of $X(m,4)$.
	Then there are rational numbers $\hbar_3$ and $\hbar_4$, independent of $m$, such that for
	$\Gamma,\Lambda \in \mathcal{H}(m,2,2)$ the series
	$\widehat \rD_{X(m,4),w}({\e}^{\Gamma_{(2)}+\Lambda_{(3)}})$ is convergent and is equal to :
	\begin{equation*}
		{\rm e}^{\frac{Q(\Gamma)}{2}-Q(\Lambda)}[\frac{1}{2}\hbar_1^2 \hbar_3^{m-2}\cosh(\sqrt{3}K\cdot \Gamma)+
		2\hbar_2^2\hbar_4^{m-2}\cos(-\frac{2\pi}{3}w\cdot K+\sqrt 3 K \cdot \Lambda)]
	\end{equation*}	
	where $\hbar_1, \hbar_2$ are the constants of Theorem \ref{En}.
\end{theorem-intro}
\noindent
We do not attempt to find the undetermined constants $\hbar_3$ and $\hbar_4$ here. We also believe that $X(m,4)$ has simple type, and the above theorem holds for any choice of $w\subset X(m,4)$ and homology classes $\Gamma$ and $\Lambda$. But the current version of the theorem is sufficient for our 3-dimensional applications.

The algebraic surfaces in Theorems \ref{K3-poly-invts}, \ref{En} and \ref{X(m,4)} are representatives of surfaces with different possible finite Kodaira dimensions. $K3$ surfaces, elliptic surface $E(n)$ with $n\geq 3$ and $X(m,4)$ for $m\geq 3$ have Kodaira dimensions $0$, $1$ and $2$, respectively. Theorem \ref{K3-poly-invts} shows that the $\U(3)$-polynomial invariants of a $K3$ surface associated to homology classes $\Gamma$ and $\Lambda$  are determined by the self-intersection of these homology classes. On the other hand, for the $\U(3)$-polynomial invariants of $E(n)$ and $X(m,4)$ we also need the pairing of $\Gamma$ and $\Lambda$ with the canonical class. Recall that the the first Chern class of the canonical classes of $E(n)$ and $X(m,4)$ are represented by $(n-2)f$ and $(m-2)f_3+2f_{m-1}$, respectively.

In Section \ref{FH}, we introduce various Floer homology groups associated to the 3-manifold $S^1 \times \Sigma$, and explain how these vector spaces admit ring structure. We also characterize the vector space structure on these Floer homology groups. Theorems \ref{En} and \ref{X(m,4)} allow us to obtain further information about the ring structure of these rings. We use this information to obtain an excision theorem for $\U(3)$-instanton Floer homology. With the aid of this excision theorem, we construct the promised sutured Floer homology $\SHI_*^3$, following Kronheimer and Mrowka's approach in \cite{KM:suture}. This sutured Floer homology group has the following property:
\begin{theorem-intro}\label{su(3)-good-rep}
	For a knot $K$ in a homology sphere $Y$, suppose the dimension of  $\SHI_*^3(M(K),\alpha(K))$ is greater than $1$.
	Then there is a non-abelian representation of $\pi_1(Y\backslash K)$ into $\SU(3)$
	that satisfies the holonomy condition \eqref{mon-cond} for $N=3$.
\end{theorem-intro}
\noindent
The proof of this theorem is given in Corollary \ref{nonabrep}. We conjecture that $\dim(\SHI_*^3(M(K),\alpha(K))>1$ for any non-trivial knot $K$ in a homology sphere $Y$ such that $Y\backslash K$ is irreducible. This answers Question \ref{rank-N-rep} affirmatively for $N=3$ and any non-trivial knot $K$ in an integral homology sphere $Y$ (without the irreducibility assumption on $Y\backslash K$). We hope to come back to this conjecture elsewhere.

\subsection{Outline of Contents}
Section 2 gives a review of the moduli spaces of anti-self-dual connections on 4-manifolds (possibly with boundary) and $\U(N)$-polynomial invariants. This section also contains a non-vanishing theorem for $\U(N)$-polynomial invariants of algebraic surfaces. The second half of Section 2 discusses how the $\U(3)$-polynomial invariants behave in the presence of negative embedded spheres. In particular, we recall the results of Culler's thesis \cite{Luc:Thesis} about the blowup formula for $\U(3)$-polynomial invariants and discuss how this formula can be simplified for smooth 4-manifolds with simple type. Section 3 deals with various Floer homology groups, which appear in this paper. After giving an exposition of $\U(N)$-instanton Floer homology, we study various Floer homologies of $\Sigma\times S^1$ where $\Sigma$ is an oriented surface. We also discuss a generalization of $\U(N)$-instanton Floer homology, which is known as Fukaya-Floer homology in the case that $N=2$.

The Floer homology groups of Section 3 are our main tools in computing $\U(3)$-polynomial invariants of several complex surfaces in Section 4. In particular, the proofs of Theorems \ref{K3-poly-invts}, \ref{En} and \ref{X(m,4)} are given in this section. In Section 4, we also study the behavior of $\U(3)$-polynomial invariants with respect to fiber sum. In Section 5, we prove our excision theorem and define the sutured Floer homology group $\SHI_*^3$. To make the exposition of the paper more comprehensible, we postpone providing proofs for technical results in Sections 2 and 3 until Section 6. These results are proved by gluing theory of the moduli spaces of anti-self-dual connections. Section 7 concerns various questions and conjectures which naturally arise from our work on this paper.

All manifolds in this paper, are smooth and oriented. Given such a manifold $X$, we will write $H_i(X)$ and $H^i(X)$ for the homology and cohomology groups of $X$ with complex coefficients. If we need to work with another coefficient ring $R$, then we use the notations $H_i(X,R)$ and $H^i(X,R)$. Our main results for this paper concern $\U(3)$-polynomial invariants and $\U(3)$-instanton Floer homologies. However, we believe that our method for the construction of $\SHI_*^3$ should work for arbitrary $N$. Therefore, we try to state our results for general $N$, when it is possible.

{\it Acknowledgements.}
We thank Lucas Culler, Simon Donaldson, Peter Kronheimer, Claude LeBrun and Tomasz Mrowka for helpful conversations. We also thank Victor Mikhaylov for verifying that our definition of the simple type condition matches the predictions from the physics literature. We are very grateful to the {\it Simons Center for Geometry and Physics} for providing a stimulating environment while we were working on this project.

\section{Higher Rank Bundles and Polynomial Invariants}
\subsection{$\U(N)$-polynomial Invariants} \label{pol-invts}
In this section, we review the definition of $\U(N)$-polynomial invariants of 4-manifolds based on \cite{K:higher,Luc:Thesis}. For $N=2$, there is an extensive literature on the subject (see, for example, \cite{Don:inv,DK,MM:Don-inv,KM:str-thm}). For higher values of $N$, these invariants were firstly defined in \cite{MM:higher} by the methods of quantum field theory. A rigorous definition of polynomial invariants for higher rank bundles are given in \cite{K:higher}. As we mentioned earlier, the polynomial invariants of a 4-manifold $X$ are homomorphisms defined on the algebra $\A(X)^{\otimes (N-1)}$. In \cite{K:higher}, the polynomial invariants are defined only on the sub-algebra:
\begin{equation*}
	\A(X)\otimes 1 \otimes \dots \otimes 1
\end{equation*}
Kronheimer's definition was subsequently generalized to the algebra $\A(X)^{\otimes (N-1)}$ in \cite{Luc:Thesis}. The construction of Fukaya-Floer homology in Subsection \ref{IIN} is based on Culler's modification of $\U(N)$-polynomial invariants in \cite{Luc:Thesis}. Therefore, we attempt to give enough background on his treatment to motivate the construction of $\U(N)$-Fukaya-Floer homology.

Suppose $X$ is a smooth, closed, oriented and connected 4-manifold, $w$ is an oriented embedded surface in $X$, and $k$ is an integer. Then there is a $\U(N)$-bundle $P$, unique up to isomorphism, over $X$ such that $c_1(P)={\rm P.D.}[w]$ and $c_2(P)[X]=k$. An explicit construction of this $\U(N)$-bundle can be given as follows. Suppose $D(w)$ is a regular neighborhood of $w$ in $X$ whose boundary is denoted by $S(w)$. Then we can consider a Hermitian line bundle on $D(w)$ which is trivialized on $S(W)$ and its relative first Chern class is given by the Thom class of the disc bundle $D(w)$. By extending the trivialization to the complement of $D(w)$, we obtain a Hermitian line bundle $L_w$ where $c_1(L_w)={\rm P.D.}[w]$. The direct sum of $L_w$ and the trivial bundle $\underline{\C}^{N-1}$ defines a $\U(N)$-bundle $P_0$ on $X$ with $c_2(E_0)[X]=0$ and $c_1(E_0)=c_1(L_w)$. Next, fix a $\U(N)$-bundle on the 4-dimensional ball $D^4$ which is trivialized on the boundary and its relative second Chern class is given by $k {\rm P.D.}[{\rm pt}]$. Removing a ball from $X\backslash D(w)$ and gluing the above ball gives rise to the same 4-manifold. We can also use the trivializations to glue the $\U(N)$-bundle on $D^4$ to $P_0$ and produce a $\U(N)$-bundle $P$ such that $c_1(P)={\rm P.D.}[w]$, $c_2(P)[X]=k$ and the determinant bundle of $P$ is identified with $L_w$.

A {\it 2-cycle} $w$ in a closed 4-manifold is a union of embedded closed surfaces in $X$. We can apply the above construction of the previous paragraph to obtain a Hermitian line bundle $L_{w_i}$ for each connected component $w_i$ of $w$. Then we can replace $L_w$ in the previous paragraph with the tensor product of the line bundles $L_{w_i}$ and produce a $\U(N)$-bundle $P$ with $c_1(P)={\rm P.D.}[w]$ and $c_2(P)[X]=k$. The {\it topological energy} of $P$ is defined to be:
\begin{equation*}
	\kappa:=k-\frac{N-1}{2N}w\cdot w
\end{equation*}
Thus the bundle $P$ is determined by the pair $(\kappa,w)$ up to a canonical isomorphism. We say a closed 2-cycle $w$ in $X$ is coprime to $N$, if there is an embedded oriented surface $\Sigma \subset X$ such that the intersection number $w\cdot \Sigma$ is coprime to $N$.

Suppose $P$ is a $\U(N)$-bundle on a closed 4-manifold determined by a pair $(\kappa,w)$. Fix an integer $l \geq 3$ and an arbitrary smooth connection $B_0$ on $L_w$. Let $\mathcal A_\kappa(X,w)$ be the space of $L^2_l$ connections on $P$ whose induced connections on $\det(P)=L_w$ is equal to $B_0$. If $\su(P)$ is the bundle associated to the conjugation action of $\U(N)$ on the Lie algebra $\su(N)$ of $\SU(N)$, then $\mathcal A_\kappa(X,w)$ is an affine space modeled on the Banach space $L^2_l(X,\su(P)\otimes \Lambda^1)$. We will also write $\mathcal G_\kappa(X,w)$ for the space of $L^2_{l+1}$ automorphisms of $P$ whose fiber-wise determinant is equal to 1. Then $\mathcal G_\kappa(X,w)$ forms a Banach Lie group with Lie algebra $L^2_{l+1}(X,\su(P))$. This Lie group acts on $\mathcal A_\kappa(X,w)$, and the quotient space is denoted by $\mathcal B_\kappa(X,w)$. We will write $[A]$ for an element of $\mathcal B_\kappa(X,w)$, represented by a connection $A$. The center of the Lie group $\U(N)$ induces a finite subgroup of $\mathcal G_\kappa(X,w)$. If this subgroup is the stabilizer of a connection $A$, then $A$ is an {\it irreducible} connection. Otherwise, the connection $A$ is called {\it reducible}. The space of irreducible connections on $P$ are denoted by $\mathcal A^*_\kappa(X,w)$, and we will write $\mathcal B^*_\kappa(X,w)$ for the quotient space.

Fix a Riemannian metric on $X$ and let $*$ denote the associated Hodge operator on differential forms of $X$. Then $*$ defines an involution on the space of 2-forms on $X$, and 2-forms in the $1$-eigenspace (respectively, $(-1)$-eigenspace) are called {\it self-dual} (respectively,  {\it anti-self-dual}). A connection $A\in \mathcal A_\kappa(X,w)$ is {\it anti-self-dual} if it satisfies the following equation:
\begin{equation} \label{ASD}
	F_0^+(A)=0
\end{equation}
where $F_0(A)$ denotes the projection of the curvature of $A$ to the space $\su(P)$, and $F_0^+(A)$ is the self-dual part of $F_0(A)$. In another word, the connection induced by $A$ on the associated ${\rm PU}(N)$-bundle to $P$ has anti-self-dual curvature. The equation \eqref{ASD} is invariant with respect to the action of $\mathcal G_\kappa(X,w)$ and the quotient space of ASD connections is denoted by $\mathcal M_\kappa(X,w)$.

The local behavior of the moduli space $\mathcal M_\kappa(X,w)$ around an element $[A]$ is governed by the following elliptic complex, denoted by $\mathcal D_A$:
\begin{equation}\label{ASD-com-closed}
	L^2_{l+1}(X, \mathfrak {su}(P))  \xrightarrow{\hspace{1mm}d_A\hspace{1mm}}
	L^2_{l}(X,\mathfrak {su}(P)\otimes \Lambda^1)
	\xrightarrow{\hspace{1mm}d_A^+\hspace{1mm}} L^2_{l-1}(X, \mathfrak {su}(P) \otimes \Lambda^+)
\end{equation}
where $\Lambda^+$ denotes the bundle of self-dual forms on $X$. The $i^{\rm th}$ cohomology group of this complex is denoted by $H^i_A$. The connection $A$ is irreducible if and only if $H^0_A$ is trivial. We say $A$ is {\it regular}, if $H^2(A)$ is trivial. If $A$ is an irreducible and regular ASD connection, then in a neighborhood of $[A]$, the moduli space is a smooth manifold of the same dimension as $H^1_A$. In this case, the dimension of $H^1_A$ is given explicitly by the following formula:
\begin{equation} \label{ind-form-closed}
	4N\kappa-(N^2-1)\frac{\chi(X)+\sigma(X)}{2}.
\end{equation}
In general, the index of the elliptic complex $\mathcal D_A$ is given by \eqref{ind-form-closed}.

The ASD equation \eqref{ASD} can be perturbed by changing the metric on $X$. {\it Holonomy perturbations} determine another useful family of perturbations of the ASD equations \cite{Don:ori,Cliff:Cas-invt,Fl:I,Don:YM-Floer,K:higher,KM:YAFT}. By abuse of notation, a solution of the perturbation of the ASD equation by a holonomy perturbation is still called an ASD connection, and the moduli space of the solutions of the perturbed equation is still denoted by $\mathcal M_\kappa(X,w)$. Suppose $w$ is coprime to $N$ and $b^+(X)\geq 1$. Then for a generic choice of the metric on $X$ and a small holonomy perturbation the moduli space $\mathcal M_\kappa(X,w)$ consists of only irreducible and regular connections \cite{K:higher}. Therefore, the moduli space is a smooth manifold whose dimension is given in \eqref{ind-form-closed}. This manifold is also orientable and in the case that $N$ is odd \cite{K:higher}, a canonical choice of an orientation can be fixed. If $N$ is even, to fix an orientation of the moduli space, we need an orientation of the determinant line of the following elliptic operator:
\begin{equation*} \label{ASD-op-nontw}
	d^+ \oplus d^*: L^2_{l}(X,\Lambda^1) \to L^2_{l-1}(X,\Lambda^+)\oplus L^2_{l-1}(X)
\end{equation*}
Such an orientation of the determinant line is called a {\it homology orientation} of $X$. The $\U(N)$-polynomial invariants of $X$ are given by integrating appropriate cohomology classes on $\mathcal M_\kappa(X,w)$.

The pull-back of $P$ to the product space $\mathcal A_\kappa(X,w) \times X$ admits an action of $\mathcal G_\kappa(X,w)$ which lifts the obvious action on the base. The quotient space defines a ${\rm PU}(N)$-bundle $\mathbb{P}$ over $ \mathcal B^*_\kappa(X,w)\times X$, called the {\it universal bundle} associated to $P$. In general, $\mathbb{P}$ cannot be lifted to an $\SU(N)$-bundle. However, we can define Chern classes of $\mathbb{P}$ as rational cohomology classes of $ \mathcal B^*_\kappa(X,w)\times X$, because the rational cohomology groups of the classifying spaces ${\rm BPU}(N)$ and ${\rm BSU}(N)$ are isomorphic. With a slight abuse of notation, we will denote the $i^{\rm th}$ Chern class of $\mathbb{P}$ with $c_i(\mathbb{P})$ for $2\leq i \leq N$.

The slant product of the Chern classes of the universal bundle and the homology classes of $X$ gives rise to cohomology classes of $\mathcal B^*_\kappa(X,w)$.
This construction can be used to define an algebra homomorphism:
\begin{equation} \label{mu-map}
	\mu: \mathbb A(X)^{\otimes (N-1)}\to H^\ast(\mathcal B^\ast_\kappa(X,w)).
\end{equation}	
where $\mu$ is the unique algebra homomorphism that satisfies the following property:
\begin{equation} \label{mu-r}
	\mu(\alpha_{(r)})=(-1)^rc_r(\mathbb{P})/\alpha.
\end{equation}
Our convention for the definition of $\mu$ slightly differs from that of \cite{Luc:Thesis} where $(-1)^r$ does not appear in the definition of $\mu(\alpha_{(r)})$.

Let the 2-cycle $w$ be coprime to $N$, and arrange a metric and a small holonomy perturbation such that the resulting moduli space consists of irreducible and regular points. We will temporarily write $\mathcal M_\kappa(X,w,\pi_0)$ for the moduli space to emphasize the dependence on $\pi_0$, denoting the metric and the holonomy perturbation. If $N$ is even, fix a homology orientation for $X$. Then $\mathcal M_\kappa(X,w,\pi_0)$ can be canonically oriented. Let $d=\dim(\mathcal M_\kappa(X,w,\pi_0))$, and fix $z\in \A(X)^{\otimes(N-1)}$ such that $\deg(z)=d$. If the moduli space $\mathcal M_\kappa(X,w,\pi_0)$ is compact, then we can evaluate $\mu(z)$ with respect to the fundamental class of $\mathcal M_\kappa(X,w,\pi_0)$ and obtain a number $\rD_{X,w,k}^N(z)$. (Recall that $k$ is the second Chern number of $P$.) We wish to show that this number does not depend on $\pi_0$, the metric and the holonomy perturbation. Suppose $\pi_1$ is another choice of a metric and a small holonomy perturbation avoiding reducible and irregular points. If $b^+(X)\geq 2$, then we can find a path $\{\pi_t\}_{0 \leq t \leq 1}$ of metrics and small holonomy perturbations such that the 1-parameter family of moduli spaces:
\begin{equation} \label{1-par-ASD}
	\bigcup_{t} \mathcal M_\kappa(X,w,\pi_t)
\end{equation}
is a smooth manifold of dimension $d+1$. Since the class $\mu(z)$ can be pulled back to \eqref{1-par-ASD}, Stokes theorem implies that $\pi_0$ and $\pi_1$ give rise to the same number $\rD_{X,w,k}^N(z)$, assuming \eqref{1-par-ASD} is also compact. However, $\mathcal M_\kappa(X,w,\pi_0)$ and the 1-parameter family of moduli spaces are not compact in general and we need to pursue a geometric approach to define the evaluation of $\mu(z)$ on $\mathcal M_\kappa(X,w,\pi_0)$.

Uhlenbeck compactification of the moduli space $\mathcal M_{\kappa}(X,w)$ compensates for the non-compactness of this space. Suppose $\{[A_n]\}$ is a sequence of the elements of $\mathcal M_{\kappa}(X,w)$. Then there is a multi-set ${\bf x}$ of $m$ points in $X$, and a connection $A_\infty \in \mathcal M_{\kappa-m}(X,w)$ such that, after passing to a subsequence, $(h_n)_*(A_n)$ is $L^p_1$-convergent to $A_\infty$ on $X\backslash {\bf x}$ for any given real number $p$ \cite[Proposition 11]{K:higher}. Here $h_n$ is an isomorphism from the $\U(N)$-bundle carrying $A_n$ to the $\U(N)$-bundle carrying $A_\infty$ which is defined only on $X\backslash {\bf x}$ and its determinant does not depend on $n$.

Another key input is that $\mathbb P$ can be replaced with a Hermitian vector bundle \cite{Luc:Thesis}. Consider the standard representation of $\SU(N)$ on $\C^N$. The tensor product $\SU(N)$-space $(\C^N)^{\otimes N}$ induces a representation of the group ${\rm PU}(N)$. Therefore, we can associate a vector bundle $\mathbb E$ of rank $N^N$ to $\mathbb P$. We call $\mathbb E$ the {\it universal complex vector bundle}. The Chern class $c_i(\mathbb P)$ can be written as a polynomial in terms of Chern classes $c_j(\mathbb E)$ for $2\leq j \leq i$. For example, $c_2(\mathbb P)$ is equal to $\frac{1}{N^N} c_2(\mathbb E)$. Therefore, it suffices to define $\rD^N_{X,w,k}(z)$ for the elements $z \in \A(X)^{\otimes (N-1)}$ that:
\begin{equation}\label{sample-z}
	\hspace{3cm} \mu(z)=c_{i_1}(\mathbb E)/\alpha_1 \cdot  \dots \ \cdot c_{i_m}(\mathbb E)/\alpha_m \hspace{1cm}2\leq i_j\leq N.
\end{equation}
In order to define this polynomial invariant, let $\Sigma_i$ be a submanifold of codimension at least 2 which represents the homology class $\alpha_i$.

The Chern classes of a vector bundle $V$ of rank $r$ over a manifold $M$ can be represented by stratified subspaces of $M$. The vector bundle ${\rm Hom}(\C^{r-i+1},V)$ is stratified by rank. If $s$ is a generic section of $V$, then:
\begin{equation} \label{rep-Chern}
	N_i:=\{x\in M \mid \rank(s(x))\leq r-i\}
\end{equation}
is a stratified subspace of $M$ with strata of even codimension. Therefore, the fundamental class of $N_i$ determines a well-defined homology class, which is the Poincar\'e dual of $c_i(V)$.

In order to define $\rD^N_{X,w,k}(z)$ for $z$ in \eqref{sample-z}, fix an open neighborhood $\nu(\Sigma_j)$ of the submanifold $\Sigma_j$ such that the inclusion of this open set in $X$ induces a surjective map of fundamental groups. The unique continuation theorem implies that if the holonomy perturbation in the definition of $\mathcal M_\kappa(X,w)$ is small enough, then the restriction of any element of this moduli space to $\nu(\Sigma_j)$ is irreducible \cite{K:higher}. We also assume that $\nu(\Sigma_j)$ and $\nu(\Sigma_k)$ intersect only if $\Sigma_j$ and $\Sigma_k$ are embedded surfaces and each point of $X$ lie on at most two such open neighborhoods. Analogous to $\mathbb E$, we can form a vector bundle $\mathbb E_j$ of rank $N^N$ on $\mathcal B^*(\nu(\Sigma_j))\times \nu(\Sigma_j)$. In order to produce a representative for $c_{i_j}(\mathbb E_j)$, we form the bundle ${\rm Hom}(\C^{N^N-i_j+1},\mathbb E_j)$, and form a subspace $V_{i_j}(\Sigma_j)$ of $\mathcal B^*(\nu(\Sigma_j))\times \nu(\Sigma_j)$ as in \eqref{rep-Chern}.

The vector bundles $\mathbb E$ and $\mathbb E_j$ are related to each other. The pull back of $\mathbb E_j$ with respect to the restriction map $r_j:\mathcal M_{\kappa}(X,w)\times \Sigma_j \to \mathcal B^*(\nu(\Sigma_j))\times \nu(\Sigma_j)$ is the restriction of the bundle $\mathbb E$. This suggests that $\rD^N_{X,w,k}(z)$ can be defined as the signed count of the points in the following {\it cut-down moduli space}:
\begin{equation*}
	\mathcal N_\kappa(X,w;z):=\{([A],x_1,\dots,x_m) \in \mathcal M_\kappa(X,w)\times \Sigma_1\times \dots \times \Sigma_m \mid r_j([A],x_j)\in V_{i_j}(\Sigma_j)\}
\end{equation*}
For a generic choice of $V_{i_j}(\Sigma_j)$, $\mathcal N_\kappa(X,w;z)$ is a compact 0-dimensional space, and the orientation of $\mathcal M_\kappa(X,w)$ fixes a sign on each point on this space \cite{K:higher, Luc:Thesis}. The compactness of the cut-down moduli space is a consequence of Uhlenbeck compactification using a standard counting argument \cite{Don:inv,DK,K:higher, Luc:Thesis}. Counting the points of the cut-down with respect to the associated signs defines a number which only depends on the homology class of $\Sigma_1$, \dots $\Sigma_m$, and this dependence for each homology class is linear. If $b^+(X)\geq 2$, we can adapt the geometric counting argument to the moduli space \eqref{1-par-ASD}, and show that the invariant $\rD_{X,w,k}^N$ does not depend on the choice of the metric on $X$ and the holonomy perturbation of the ASD equation.

There is also a standard trick which allows us to define $\rD_{X,w,k}^N$ in the case that $w$ is not coprime to $N$. Suppose $\hat X$ denotes the blowup of $X$ at an arbitrary point. Suppose also $E$ is the exceptional sphere of $\hat X$. Then the cycle $w+E$ is coprime to $N$. In the non-coprime case, define:
\begin{equation*}
	\rD^N_{X,w,k}(z):=\rD^N_{\hat X,w+E,k}(E_{(2)}^{N-1}z).
\end{equation*}
In the case that $w$ is already coprime, the above definition turns into an identity which is proved in \cite{K:higher}.

The invariant $\rD^N_{X,w,k}(z)$ is defined to be zero if $\deg(z)$ is not equal to the dimension of $\mathcal M_\kappa(X,w)$. For a fixed 2-cycle $w$ in $X$, the moduli spaces of ASD connections appear in different dimensions whose values mod $4N$ are constant. We use \eqref{DNXw} to define $\rD_{X,w}^N(z)$, which is called the {\it $\U(N)$-polynomial invariant of $(X,w)$ evaluated at $z$}. This number is non-zero only if:
\begin{equation*}
	\deg(z) \equiv 2(N+1)w\cdot w-(N^2-1)\frac{\chi(X)+\sigma(X)}{2} \mod \,4N
\end{equation*}
In particular, if $N$ is an odd integer, the invariant $\rD_{X,w}^N(z)$ vanishes for the classes $z$ that $\deg(z)$ is not divisible by $4$.

There are relationships among the polynomial invariants of $X$ associated to different 2-cycles. If $w$ and $w'$ are two 2-cycles in $X$, then:
\begin{equation} \label{w+Nw'}
	\rD_{X,w+Nw'}^N(z)=(-1)^{cw'\cdot w'}\rD_{X,w}^N(z).
\end{equation}
where $c$ is zero if $N$ is odd or divisible by 4, and is equal to $1$ if $N$ is 2 mod 4. The invariants associated to the 2-cycles $w$ and the same 2-cycle with the reverse orientation, denoted by $-w$, are also related. As it is explained in \cite{K:higher}, there is a diffeomorphism from $\mathcal M_\kappa(X,w)$ to $\mathcal M_\kappa(X,-w)$. This diffeomorphism is orientation preserving if $N$ is odd, and change the orientation by the factor of $(-1)^{w\cdot w}$ if $N$ is even. The diffeomorphism lifts to an anti-linear isomorphism of the universal complex vector bundles. Therefore, we have:
\begin{equation}\label{wto-w}
	\rD_{X,-w}^N(z)=(-1)^{(N-1)w\cdot w}\rD_{X,w}^N(\tau(z)).
\end{equation}
where $\tau:\A(X)^{\otimes (N-1)} \to \A(X)^{\otimes (N-1)}$ is the algebra homomorphism that maps $\alpha_{(r)}$ to $(-1)^r\alpha_{(r)}$.

Suppose $\Gamma^2$, $\dots$, $\Gamma^N$ are elements of $H_2(X)$ and $z\in\A(X)^{\otimes (N-1)}$. To avoid the convergence issue, we define $\rD^N_{X,w}(z\e^{\Gamma^2_{(2)}+\dots+\Gamma^N_{(N)}})$ in a slightly different way in compare to the introduction:
\begin{equation*}
	\rD^N_{X,w}(z\e^{\Gamma^2_{(2)}+\dots+\Gamma^N_{(N)}}):=\sum_{0 \leq i_2,\dots,i_N< \infty}
	\frac{\rD_{X,w}(z(\Gamma^2_{(2)})^{i_2}\dots (\Gamma^N_{(N)})^{i_N})}{i_2!\dots i_N!}t_2^{i_2} \dots t_N^{i_N}
\end{equation*}
where $t_i$ is a formal variable. Therefore, the {\it $\U(N)$-series} $\rD^N_{X,w}(z\e^{\Gamma^2_{(2)}+\dots+\Gamma^N_{(N)}})$ is an element of the ring of formal power series $\C[\![t_2,\dots,t_N]\!]$. We can also define an element of $\C[\![t_2,\dots,t_N]\!]$ for linear functionals $f_2,\dots,f_N:H_2(X) \to \C$ and a a power series $g(x)=\sum_{i=0}^{\infty}b_i x^i$:
\begin{equation} \label{conv-exp-cos-sin}
	g(f_2(\Gamma^2)+\dots+f_N(\Gamma^N)):=\sum_{i=0}^{\infty}b_i(f_2(\Gamma^2)t_2+\dots+f_N(\Gamma^N)t_N)^i
\end{equation}
We use a similar convention in the case that $f_i$ are homogeneous polynomials of higher degree (eg. the intersection form $Q$). These conventions allow us to rephrase Theorems \ref{K3-poly-invts}, \ref{En} and \ref{X(m,4)} in terms of the identities of the elements of $\C[\![t_2,t_3]\!]$.

\begin{remark}
	Suppose $X$ is a 4-manifold with $b^+(X)=1$. For a generic metric and for small holonomy perturbations the moduli spaces $\mathcal M_\kappa(X,w)$ contains only irreducible connections.
	For each such choice of metric, we can apply the construction of this section to define a $\U(N)$-polynomial invariant $\rD_{X,w}^N$. However, this polynomial invariant depends on the choice of the metric,
	because a 1-parameter family of moduli spaces as in \eqref{1-par-ASD} might have reducible connections. As a result, the space of metrics on $X$ can be divided into chambers, such that
	$\rD_{X,w}^N$ is constant only inside the interior of each chamber. Such polynomial invariants have been studied for $N=2$ in \cite{K:b+=1,KM:b+=1}.
\end{remark}

\subsection{Cylindrical Ends and Moduli Spaces}\label{cyl-mod}
One of the important themes of this article is the interplay between gauge theory on 3-manifolds and 4-manifolds. One can see the interaction by considering the analogues of the geometrical objects from the previous subsection on 4-manifolds with boundary. Suppose $W$ is a 4-manifold with boundary $Y$, and fix a metric which is a product metric in a collar neighborhood of $Y$. A smooth 2-cycle in $W$ is a properly embedded 2-dimension submanifold of $W$. A 2-cycle $w$ in $W$ is a union of smooth 2-cycles in $W$ whose boundary determines a smooth 1-manifold $\gamma\subset Y$. We will say that the boundary of the pair $(W,w)$ is the pair $(Y,\gamma)$. We also form a pair of non-compact manifolds $(W^+,w^+)$ by adding the cylindrical ends $[0,\infty) \times Y$ and $[0,\infty) \times \gamma$ to $W$ and $w$.

As in the previous part, we can associate a $\U(N)$-bundle $Q$ to the pair $(Y,\gamma)$. This $\U(N)$-bundle is trivialized on the complement of a regular neighborhood $D(\gamma)$ of $\gamma$, and its relative first Chern class on $D(\gamma)$ is given by the Thom class. We will also write $L_\gamma$ for the determinant bundle of $Q$. Similar to the 4-dimensional case, let $\mathcal B(Y,\gamma)$ be the space of equivalence classes of connections on $Q$ whose determinants are equal to a fixed connection on $L_\gamma$. Given $\alpha \in \mathcal B(Y,\gamma)$, the stabilizer of a connection representing $\alpha$ is denoted by $\Gamma_\alpha$. The element $\alpha$ is irreducible if $\Gamma_\alpha$ is equal to the center of $\SU(N)$. The bundle $Q$ can be extended to a $\U(N)$-bundle $P$ on $W$ using the 2-cycle $w$. The bundle $P$ also determines a $\U(N)$-bundle on $W^+$ in an obvious way which will be also denoted by $P$.

Fix an element $\alpha\in \mathcal B(Y,\gamma)$, and let $A_0$ and $A_1$ be two connections on $P$ whose restrictions to the end $[0,\infty)\times Y$ are the pull-back of representatives of $\alpha$. Then we say $A_0$ and $A_1$ represent the same {\it path}, if there is an automorphism $g$ of $P$ with determinant 1 such that $g^*(A_1)-A_0$ is a compactly supported 1-form. The equivalence class of a connection under this equivalence relation is called a {\it path along $(W,w)$ based at $\alpha$}. The  {\it topological energy} of a path $p$ represented by a connection $A$ is defined by the following Chern-Weil integral:
\begin{equation*}
	\kappa(p):=\frac{1}{16N\pi^2}\int_{W^+} \tr(\ad(F(A)) \wedge \ad(F(A)))
\end{equation*}
Here $\ad(F(A))$ is regarded as a 2-form with coefficients in ${\rm End}(\su(P))$. The product $\wedge$ in the integrand is induced by the wedge product of differential forms and composition of the elements of ${\rm End}(\su(P))$. The above integral is independent of the chosen connection $A$ and only depends on $p$. The constant in front of the integral is chosen such that if $p$ is replaced by another path $p'$ based at $\alpha$, then the energy changes by an integer.

An important special case for us is the cylinder manifold $W=[0,1]\times Y$. Fix connections $\alpha, \beta \in \mathcal B(Y,\gamma)$ and let $p$ be a path along $([0,1]\times Y,[0,1]\times \gamma)$ based at $\alpha$ and $\beta$ on $\{0\}\times Y$ on $\{1\}\times Y$. Then $p$ induces a path, in the ordinary sense, from $\alpha$ to $\beta$ in $\mathcal B(Y,\gamma)$. The topological energy of the path $p$ defines a number which its value, up to an integer, depends only on $\alpha$ and $\beta$. Therefore, we can fix $\beta_0 \in \mathcal B(Y,\gamma)$ and define a functional $\CS:\mathcal B(Y,\gamma) \to \R/\Z$, called the {\it Chern-Simon functional}, where $\CS(\alpha)$ is equal to the topological energy of any path from $\alpha$ to $\beta_0$. Since $\beta_0$ is chosen arbitrarily, $\CS$ is well-defined only up to a constant. But in the case that $\gamma$ is empty, the trivial connection $\Theta$ gives a canonical choice of $\beta_0$. Critical points of the Chern-Simons functional are represented by connections $A$ on $Q$ such that:
\begin{equation*}	
	F_0(A)=0.
\end{equation*}	
A critical point $\alpha\in \mathcal B(Y,\gamma)$ is called {\it non-degenerate} if the Hessian of the Chern-Simons functional is non-degenerate at $\alpha$.

Suppose $\alpha$ is a non-degenerate critical point of the Chern-Simons function. Suppose also $A_0$ is a connection on $W^+$ that represents a path $p$ based at $\alpha$. Then $\mathcal A_{p}(W,w;\alpha)$ is the following space of connections:
\begin{equation*}
	\mathcal A_{p}(W,w;\alpha):=\{A_0+a \mid a \in L^2_{l,\delta}(W, \su(P)\otimes \Lambda^1)\}
\end{equation*}
where $l \geq 3$, $\delta$ is a small positive number, and the weighted Sobolev space $L^2_{l,\delta}(W, \su(P))$ is defined as follows. Let $t$ be a function on $W^+$ that agrees with the cylindrical coordinate on the end of $W^+$. For a vector bundle $E$ on $W$ and a positive constant $\delta$, the Banach space $L^2_{l,\delta}(W, E)$ is defined as $e^{-\delta t}L^2_{k}(W^+, E)$. Suppose $\mathcal G_{p}(W,w;\alpha)$ is also defined as:
\begin{equation*}
	\mathcal G_{p}(W,w;\alpha):=\{g\in {\rm Aut}(P) \mid \det(g)=1,\, \nabla_{A_0} g \in L^2_{l,\delta}(W,\su(P)\otimes \Lambda^1)\}.
\end{equation*}
Then $\mathcal G_{p}(W,w;\alpha)$ is a Banach Lie group. This group acts on $\mathcal A_{p}(W,w;\alpha)$ and the quotient space is denoted by $\mathcal B_{p}(W,w;\alpha)$. The space of the elements $[A]\in \mathcal B_p(W,w;\alpha)$, that $F^+_0(A)=0$, forms the {\it moduli space of ASD connections associated to the path $p$}. This moduli space is denoted by $\mathcal M_p(W,w;\alpha)$.

It is useful to form a framed version of the moduli space of the ASD equations. Any gauge transformation in $\mathcal G_p(W,w;\alpha)$ is asymptotic to an element of $\Gamma_\alpha$. Define the {\it framed gauge group} $\mathcal G^0_p(W,w;\alpha)$ to be the subspace of the elements of $\mathcal G^0_p(W,w;\alpha)$ which are asymptotic to the trivial element of $\Gamma_\alpha$. The framed gauge group is also a Banach Lie group and its Lie algebra can be identified with $L^2_{l+1,\delta}(W,\su(P))$. We write $\widetilde {\mathcal B}_p(W,w;\alpha)$ for the quotient of $\mathcal A_p(W,w;\alpha)$ with respect to the action of $\mathcal G^0_p(W,w;\alpha)$. The set of the elements of $\widetilde {\mathcal B}_p(W,w;\alpha)$ which satisfy the ASD equation is called the {\it framed (or based) moduli space of ASD connections associated to $p$} and is denoted by $\widetilde {\mathcal M}_p(W,w;\alpha)$.

There is an important relationship between the ASD equation and the Chern-Simons functional. We can define an inner product on $\mathcal B(Y,\gamma)$ using the following expression
\begin{equation*}
	\langle a,b \rangle:=-\frac{1}{16N\pi^2}\int_Y \tr(\ad(a) \wedge *\ad(b))\hspace{1cm}a,b\in \Omega^1(Y,\su(P))
\end{equation*}
where $\tr$ is defined using the trace on ${\rm End}(\su(N))$. Suppose $\{\alpha(t)\}_{t \in [0,1]}$ is a path in $\mathcal B(Y,\gamma)$. This path defines a trajectory of the downward gradient flow of $\CS$ with respect to the above metric if it satisfies the following equation:
\begin{equation}  \label{traj-CS}
	\frac{d\alpha(t)}{dt}=-*F_0(\alpha(t)).
\end{equation}
The path $\{\alpha(t)\}$ determines a connection $A$, in temporal gauge, on $[0,1]\times Y$, and \eqref{traj-CS} is equivalent to the ASD equation $F^+_0(A)=0$. This relationship between the ASD equation and the Chern-Simons functional allows us to conclude from non-degeneracy of the critical points of $\CS$ that the moduli spaces $\widetilde{\mathcal M}_p(W,w;\alpha)$ are analytically well-behaved.

The local behavior of the framed moduli space $\widetilde{\mathcal M}_p(W,w;\alpha)$ around an element $[A]$ is modeled by the following elliptic complex:
\begin{equation}\label{ASD-com}
	L^2_{l+1,\delta}(W, \mathfrak {su}(P))  \xrightarrow{\hspace{1mm}d_A\hspace{1mm}}L^2_{l,\delta}(W,  \Lambda^1\otimes \mathfrak {su}(P))
	\xrightarrow{\hspace{1mm}d^+_A\hspace{1mm}} L^2_{l-1,\delta}(W,  \Lambda^+\otimes \mathfrak {su}(P))
\end{equation}
This complex is Fredholm and its homology groups are denoted by $H^0_A$, $H^1_A$ and $H^2_A$. Then $H^0_A=0$, and the element $[A]$ is called {\it regular}, if $H^2_A$ is also trivial. In this case,  $\widetilde{\mathcal M}_p(W,w;\alpha)$ is smooth in a neighborhood of $A$, and $H^1_A$ gives a model for the tangent space of the framed moduli space at $[A]$. Therefore, the index of the above complex for a (not necessarily regular) ASD connection $A$ is called the {\it expected dimension} of $\widetilde{\mathcal M}_p(W,w;\alpha)$ and is denoted by $\dim_e(\widetilde{\mathcal M}_p(W,w;\alpha))$.

We slightly generalize above discussion  to include the case of 4-dimensional cobordisms. A cobordism $W:Y_0 \to Y_1$ is a 4-manifold with boundary $\overline{Y_0}\coprod Y_1$. We also assume that a 2-cycle $w:\gamma_0 \to \gamma_1$ in $W$ is given, and $P$ is the associated $\U(N)$-bundle. Suppose also $\alpha_0$ and $\alpha_1$ are flat connections on $Y_0$ and $Y_1$, and $p$ is a path along $(W,w)$ from $\alpha_0$ to $\alpha_1$. As before, we assume that $\alpha_0$ and $\alpha_1$ are non-degenerate. In this case, $W^+$, $\mathcal B_p(W,w;\alpha_0,\alpha_1)$ and $\mathcal M_p(W,w;\alpha_0,\alpha_1)$ are defined as in the previous case by regarding $W$ as a 4-manifold with boundary. However, there is an alternative elliptic complex that one can associate to the elements of $\mathcal M_p(W,w;\alpha_0,\alpha_1)$. Suppose $W^+$ is the Riemannian 4-manifold given by adding cylindrical ends to $W$. In this case, we identify the cylindrical end corresponding to the incoming end with $(-\infty , 0] \times Y_0$. As in the previous case, suppose $t$ is a function on $W^+$ that agrees with the cylindrical coordinates on the ends and define $L^2_{l,\delta}(W, E)$. Therefore, an element of $L^2_{l,\delta}(W, E)$ is allowed to have exponential growth on the incoming end and it is forced to have an exponential decay on the outgoing end. For $[A]\in \mathcal B_p(W,w;\alpha_0,\alpha_1)$, the ASD operator $\mathcal D_A$ is defined as follows:
\begin{equation} \label{ASD-operator}
	d_A^*\oplus d_A^+:L^2_{l,\delta}(W, \Lambda^1\otimes \mathfrak {su}(P)) \to L^2_{l-1,\delta}(W, (\Lambda^0\oplus \Lambda^+)\otimes \mathfrak {su}(P))
\end{equation}

The operator $\mathcal D_A$ is elliptic, and the excision property of elliptic operators shows that the index of $\mathcal D_A$ is additive with respect to composition of cobordisms.  This index can be computed explicitly using Aiyah-Patodi-Singer index theorem \cite{Taubes:L2-mod-space, MMR:L2-mod-space}:
\begin{equation} \label{ind-DA}
	\ind(\mathcal D_A)=4N\kappa(p)-(N^2-1)\frac{\chi(W)+\sigma(W)}{2}+\sum_{i=0,1} (-1)^i \frac{h^0(Y_i;ad_{\alpha_i})-\rho_{ad_{\alpha_i}}(Y_i)}{2}
\end{equation}
where $h^0(Y_i;\ad_{\alpha_i})$ denotes the dimension of $H^0(Y_i;\ad_{\alpha_i})$. Moreover, for a flat connection $a$ on a vector bundle $V$ over $Y$, $\rho_a$ is Atiyah-Patodi-Singer $\rho$-invariant of $a$ \cite{APS:II}. As an example, let $\chi_j$ be a $\U(1)$-connection on $L(a,b)$ whose holonomy around the standard generator of $\pi_1(L(a,b))$ is equal to $\e^{\frac{2\pi\bi j}{a}}$. Then it is shown in \cite{APS:II} that:
\begin{equation}\label{lens-rho}
	\rho_{\chi_j}(L(a,b))=-\frac{4}{a}\sum_{k=1}^{a-1} \cot(\frac{\pi k}{a})\cot(\frac{\pi kb}{a})\sin^2(\frac{\pi kj}{a})
\end{equation}
Suppose a 4-manifold $W$ is regarded as a cobordism from the empty 3-manifold to the boundary of $W$. Then $\dim_e(\widetilde{\mathcal M}_p(W,w;\alpha))$ is equal to $\ind(\mathcal D_A)+h^0(\alpha)$ where $A$ is a connection that represents the path $p$.

In the case of a cylinder $[0,1] \times Y$, the ASD operator can be used to define a relative $\Z/4N\Z$-grading on critical points of the Chern-Simons functional associated to a pair $(Y,\gamma)$. Fix an arbitrary critical point $\beta_0$ of $\CS$, and let $\alpha$ be another critical point of $\CS$. Let the connection $A$ represent an arbitrary path $p$ from $\alpha$ to $\beta_0$. Then the {\it Floer grading} of $\alpha$, denoted by $\deg(\alpha)$, is defined to be $\ind(\mathcal D_A)$ mod $4N$. Since the choice of $\beta_0$ is arbitrary, this grading only gives rise to a relative $\Z/4N\Z$-grading. In the case that $\gamma$ is empty, we can make this grading absolute by requiring $\beta_0=\Theta$. In this case, we can invoke the index formula \eqref{ind-DA} to compute $\deg(\alpha)$ as follows:
\begin{equation} \label{degree}
	\deg(\alpha)\equiv 4N \cdot \CS(\alpha)-\frac{N^2-1}{2}+\frac{h^0(Y;ad_{\alpha})-\rho_{ad_\alpha}(Y)}{2} \mod 4N
\end{equation}

Analogous to the case of closed 4-manifolds, we can avoid non-regular points in $\mathcal M_p(W,w;\alpha)$ by perturbing the ASD equation. Suppose $\alpha$ is irreducible and non-degenerate. Then there are small holonomy perturbations of the ASD equation, supported in $[0,1] \times Y \subset W^+$, such that the resulting moduli spaces consist of regular points \cite{KM:YAFT}. Alternatively, if $b^+(X)\geq 2$ and $w$ is coprime to $N$, then we can arrange for a small perturbation of the ASD equation such that the moduli spaces $\mathcal M_p(W,w;\alpha)$, for arbitrary $\alpha$ and $p$ with $\kappa(p)$ bounded, are regular \cite{K:higher}.
In the case that $\mathcal M_p(W,w;\alpha)$ consists of only regular points, it is an orientable smooth manifold.

In general, the critical points of the Chern-Simons functional might not be non-degenerate. Suppose that all critical points of the Chern-Simons functional associate to a pair $(Y,\gamma)$ are irreducible. Then, $\CS$ can be also perturbed by appropriate perturbations such that the critical points of the resulting functional are irreducible and non-degenerate \cite[Proposition 3.10]{KM:YAFT}. The family of perturbations that is used in \cite{KM:YAFT} are also defined in terms of the holonomies of connections on $Y$. Suppose $\CS_\pi$ is such a perturbation of $\CS$ by a small holonomy perturbation and $\alpha$ is a critical point of $\CS_\pi$. Suppose also $(W,w)$ is a pair whose boundary is equal to $(Y,\gamma)$. The negative-gradient flow line of $\CS_{\pi}$ determines a perturbation of the ASD equation on the end $[0,\infty)\times Y$ of $W^+$. This perturbation can be extended to $W^+$ such that the corresponding moduli space contains only regular points \cite{KM:YAFT}. As in the previous case, the elliptic operator \eqref{ASD-operator}, can be used to study the local behavior of the moduli spaces. In particular, in the case of a cylinder, index of $\mathcal D_A$ can be used to define a relative $\Z/4N\Z$-grading on the critical points of $\CS_\pi$.

\subsection{Non-vanishing Theorem for Algebraic Surfaces} \label{non-van}
For a complex projective surface, the moduli spaces of ASD connections can be identified with the moduli spaces of \emph{stable bundles} with the fixed determinant \cite{Don:stable}.
Stable bundles have been studied extensively using various techniques in algebraic geometry. Thus one can use algebro-geometric methods to extract information about the polynomial invariants of a complex projective surface. For example, the following theorem about non-vanishing of $\U(N)$-polynomial invariants of algebraic surfaces is a generalization of Donaldson's celebrated theorem about $\U(2)$-invariants \cite{Don:inv,DK}:
\begin{theorem}\label{nonvan}
	Suppose $X$ is a complex projective surface with positive geometric genus, $h$ is a hyperplane class (or equivalently a very ample class),
	and $w$ is a 2-cycle representing $c_1$ of a holomorphic line bundle $L$. Suppose also $w$ is coprime to $N$. Then:
	\begin{equation*}
		\rD_{X,w}^N(h_{(2)}^d) > 0
	\end{equation*}
	when $d$ is large enough and
	\begin{equation*}
	d \equiv (N+1)w\cdot w-(N^2-1)\frac{\chi(X)+\sigma(X)}{4} \mod 2N
	\end{equation*}
\end{theorem}

\begin{proof}
	Let $M^N_\kappa(X,L)$ be the moduli space of stable bundles with rank $N$, determinant $L$, and energy $\kappa$. We firstly review the proof of the non-vanishing theorem in the rank 2 case.
	In this case, the key idea is to find a projective embedding of $M^2_\kappa(X,L)$ and then interpret the polynomial invariant $\rD_{X,w}^2(h_{(2)}^d)$ as a multiple of the degree of the moduli space.
	The main steps of the proof can be summarized as follows:
	\begin{enumerate}
		\item  Suppose $C$ is an algebraic curve and $\mathcal N(C,d)$ is the moduli space of stable bundles of rank 2 and degree $d$ on $C$. Then there is a projective embedding of $\mathcal N(C,d)$
		into a projective space ${\bf P}(W)$ \cite{Gi1}. Moreover, this projective embedding is given by the sections of a large power $\mathcal L^d$ of the {\it determinant line bundle} over $\mathcal N(C,d)$ \cite{DK}.
		\item  For any stable bundle  $\mathcal E\in M^2_\kappa(X,\mathcal L)$, there is $p_0$ such that if $C\subset X$  is a generic curve
		in the linear system $|\mathcal{O}(p)|$, for $p\geq p_0$, then the restriction of $\mathcal E$ to $C$ is also stable \cite{MR}. Using this result together with the fact that moduli space
		$M^2_\kappa(X,L)$ has finite type, we can find a projective embedding $J:M^2_\kappa(X,L) \to {\bf P}(V)$.
		This embedding is given by restricting the elements of $M^2_\kappa(X,L)$ to finitely many curves in the linear
		system $|\mathcal{O}(p)|$ and applying Gieseker embedding from the previous part.
		\item For $\kappa$ large enough, the moduli space $M^2_\kappa(X,\mathcal L)$ is not empty \cite{taubes1984, Gi2}.
		\item The dimension of the irregular part of $M^2_\kappa(X,\mathcal L)$ is strictly smaller than the virtual dimension of the moduli space, when $\kappa$ is large enough \cite{Don:inv}.
	\end{enumerate}
	By combining these facts, it can be shown that the map $J$ can be chosen such that the degree of the quasi-projective variety $J(M_\kappa(X,I))$ is equal to:
	\begin{equation*}
		K^d{\rD}_{X,w}^N(h_{(2)}^d)
	\end{equation*}
	for an appropriate integer $K$, and for a large enough $d$. Therefore, the invariant ${\rD}_{X,w}^N(h_{(2)}^d)$ is positive.
	This proof can be generalized to the case that $N>2$.  Steps 1 and 2 are proved for an arbitrary rank in the references mentioned above.
	The generalization of the the third step to the higher rank case is given in \cite{LQ}. The higher rank version of the fourth fact is also proved in \cite{GL}.
	Then the same arguments as in \cite{Don:inv,DK} can be used to realize
	$\rD_{X,w}^N(h_{(2)}^d)$ as a multiple of the degree of a quasi-projective variety, which verifies the claim in Theorem \ref{nonvan}.
\end{proof}

\begin{remark}
	In Theorem \ref{nonvan}, we assume that $X$ is a projective surface. But a similar non-vanishing theorem can be formulated for any K\"ahler surface, because K\"ahler surfaces can be deformed into
	projective surfaces \cite{Kod:CXK} and the $\U(N)$-polynomial invariants only depend on the smooth structure.
\end{remark}

\subsection{Negative Embedded Spheres}\label{neg-spheres}
Motivated by \cite{ruberman1996relations}, Fintushel and Stern used embedded 2-spheres with negative self-intersection to study polynomial invariants of a 4-manifold $X$ \cite{fintushel1996blowup,FS:str-thm} (See also \cite{Wie:neg-sph}). The same idea is exploited in \cite{Luc:Thesis} to obtain information about the properties of the polynomial invariants associated to higher rank bundles. Suppose $\tau$ is an embedded sphere in a 4-manifold $X$ which has self-intersection $-2$. Fix a 2-cycle $w$ with $w\cdot \tau=0,$ and $z \in \A(\langle\tau\rangle^{\perp})^{\otimes 2}$. In general, $\A(V)$ for a vector subspace $V\subseteq H_2(X)$ denotes the sub-algebra $\Sym^*(H_0(X)\oplus V)\otimes \Lambda^*(H_1(X))$, and $\langle\tau\rangle^{\perp}$ is the subspace of $H_2(X)$ consisting of the homology classes orthogonal to $\tau$. The following formulas about the $\U(3)$-polynomial invariants of $X$ are proved in \cite{Luc:Thesis}:
\begin{itemize} \label{Lucas-identites}
	\item[$(C_1)$] $\DXw^3(\tau_{(2)}^2 z)=-2\rD_{X,w+\tau}^3(z)$
	\item[($C_2$)] $\DXw^3(\tau_{(2)}^4z)=-4\DXw^3(a_2\tau_{(2)}^2 z)-3\DXw^3(\tau_{(3)}^2 z)$
	\item[($C_3$)] $\DXw^3(\tau_{(2)}^3 \tau_{(3)} z)=-3\DXw^3(a_3\tau_{(2)}^2 z)-\DXw^3(a_2 \tau_{(2)} \tau_{(3)} z)$
\end{itemize}	

By a similar approach, we shall prove the following proposition in Subsection \ref{neg-spheres-gluing}:

\begin{prop} \label{-3-identities}
	Suppose $X$ is a smooth 4-manifold with $b^+(X)\geq 2$ and $w$ is a 2-cycle. Suppose also $\sigma$
	is an embedded sphere with self-intersection $-3$ and
	$z \in \A(\langle\sigma\rangle^{\perp})^{\otimes 2}$.
	\begin{enumerate}
	\item[(i)] If $w \cdot \sigma \equiv 1$ mod 3, then there is a constant number $c$ such that:
		\begin{equation*}
			\DXw^3((-\frac{3}{2} \sigma_{(3)}-\frac{3}{2} \sigma_{(2)}^2-a_2)z)=c\rD^3_{X,w-\sigma}(z).
		\end{equation*}
	\item[(ii)] If $w \cdot \sigma \equiv 2$ mod 3, then for the same constant $c$ as above:
		\begin{equation*}
			\DXw^3((\frac{3}{2} \sigma_{(3)}-\frac{3}{2} \sigma_{(2)}^2-a_2)z)=c\rD^3_{X,w+\sigma}(z).
		\end{equation*}
	\item[(iii)] If $w \cdot \sigma \equiv 0$ mod 3, then the following two formulas hold:
	\begin{equation*}
		\DXw^3((\sigma_{(2)}^4+4a_2\sigma_{(2)}^2+3\sigma_{(3)}^2)z)=0\hspace{1cm}
		\DXw^3((\sigma_{(2)}^3 \sigma_{(3)}+3a_3\sigma_{(2)}^2+a_2 \sigma_{(2)} \sigma_{(3)}) z)=0
	\end{equation*}
	\end{enumerate}
\end{prop}
In Subsection \ref{E(2)}, we will show that the constant $c$ in the statement of this proposition is equal to $-3$.

\subsection{Blowup Formula for 4-manifolds with Simple Type} \label{blowup-sub}
In this subsection, we review the properties of $\U(3)$-polynomial invariants of blown up 4-manifolds. This part has the same theme as the previous subsection, because the exceptional divisor of a blown up 4-manifold gives rise to a $(-1)$-sphere. We start with an exposition of the main result in Culler's thesis \cite{Luc:Thesis}. Given a 4-manifold $X$, let $\hat X$ denote the blow up of $X$ at one point, which is diffeomorphic to the 4-manifold $X\#\cp$. We will also denote the exceptional sphere in $\hat X$ with $E$. If $w$ is a 2-cycle in $X$, then it induces a 2-cycle in $\hat X$ which will be also denoted by $w$. Similarly, we can regard $\A(X)$ as a sub algebra of $\A(\hat X)$.

\begin{prop} \label{blowup-inital}
	Suppose $w$ is a 2-cycle in $X$ and $z\in \A(X)^{\otimes 2}$. For $0\leq i \leq 2$ and $0\leq j \leq 1$, the invariant
	$\rD_{\hat X,w}^3(E_{(2)}^iE_{(3)}^jz)$ is equal to $\rD^3_{X,w}(z)$ if $i=j=0$ and it is zero otherwise. The invariant
	$\rD^3_{\hat X,w+E}(E_{(3)}z)$ is equal to $\rD^3_{X,w}(z)$.
	For $0\leq i \leq 5$, the invariant $\rD^3_{\hat X,w+E}(E_{(2)}^iz)$ is equal to:
	\begin{equation*}
	\left\{
		\begin{array}{ll}
			0,                         & \hbox{\text{if}~i=0,1,3;} \\
			\rD^3_{X,w}(z),              & \hbox{\text{if}~i=2;} \\
			- \rD^3_{X,w}(a_2 z),      & \hbox{\text{if}~i=4;} \\
			\rD^3_{X,w}(a_3 z),     & \hbox{\text{if}~i=5.}
		\end{array}
	\right.
	\end{equation*}
\end{prop}

\begin{proof}
	As we mentioned in Subsection \ref{pol-invts}, the identity $\rD^3_{\hat X,w+E}(E_{(2)}^2z)=\rD^3_{X,w}(z)$
	is proved in \cite{K:higher}. The other identities can be proved by a similar method.
	(See Proposition 62 in \cite{Luc:Thesis} and the succeeding discussion.)
\end{proof}

According to this proposition, some of the polynomial invariants of $\hat X$ are determined by the invariants of $X$. The following theorem claims that a similar pattern holds for all invariants of $\hat X$. This theorem is essentially proved in \cite{Luc:Thesis}. We just slightly expand the results of this thesis using the same methods:

\begin{theorem} \label{blowup-gen-mafld}
	Suppose $w$ is a 2-cycle in $X$ and $z\in \A(X)^{\otimes 2}$.
	For non-negative integers $i$ and $j$, there are polynomials $B_{i,j}, S_{i,j} \in \Q[a_2,a_3]$ which are independent of $X$,
	$w$, $z$, and they satisfy the following identities:
	\begin{equation*}
		\rD^3_{\hat X,w}(\e^{E_{(2)}+E_{(3)}}z)=\rD^3_{X,w}(z\sum_{i,j}B_{i,j}(a_2,a_3)t_2^it_3^j)
	\end{equation*}
	and
	\begin{equation*}
		\rD^3_{\hat X,w+E}(\e^{E_{(2)}+E_{(3)}}z)=\rD^3_{X,w}(z\sum_{i,j}S_{i,j}(a_2,a_3)t_2^it_3^j).
	\end{equation*}	
	Moreover, $B:=\sum_{i,j}B_{i,j}(a_2,a_3)t_2^it_3^j$ and $S:=\sum_{i,j}S_{i,j}(a_2,a_3)t_2^it_3^j$,
	as power series in the variables $t_2$ and $t_3$ with coefficients in the polynomial ring of $a_2$, $a_3$, are uniquely determined
	by the initial values given in Proposition \ref{blowup-inital} and the following four PDEs:
	\begin{equation} \label{PDE1-2}
		B_{22}B-B_2B_2=-S\circ \tau\cdot S, \hspace{1cm}S_{22}S-S_2S_2=-S\circ \tau\cdot B,
	\end{equation}
	\begin{equation}\label{PDE3}
		B_{2222}B-4B_{222}B_2 + 3B_{22}B_{22}=-4a_2(B_{22}B-B_2B_2)-3(B_{33}B-B_3B_3)
	\end{equation}
	and
	\begin{equation*}\label{PDE4}
		B_{2223}B-3B_{223}B_2 -B_{222}B_3 + 3B_{23}B_{23} = \hspace{4cm}
	\end{equation*}
	\begin{equation} \label{PDE-2}
		\hspace{4cm}=-3a_3(B_{22}B-B_2B_2) -a_2 (B_{23}B-B_2 B_3).
	\end{equation}
	Here $\tau$ maps $(t_2,t_3)$ to $(-t_2,-t_3)$.
	The subscript $2$ means taking partial derivative with respect to $2$ and the subscript $3$ should be interpreted similarly.
\end{theorem}

Using \eqref{wto-w}, the power series $S$ can be used to compute the invariant of $\rD^3_{X,w-E}$.

\begin{proof}[Sketch of the proof]
	The main tool is a trick that due to Fintushel and Stern \cite{fintushel1996blowup}. Suppose $E$ and $E'$ are the two exceptional spheres in $X\#2\cp$. Then the homology class $E-E'$ is represented by a
	$(-2)$-sphere. Suppose $\widehat C_{k,l}$ (respectively, $\widehat C_{k,l}'$) is the identity that is given by applying $(C_1)$ from Subsection \ref{neg-spheres} to $w$
	(respectively, $w+E+E'$) and $\tilde z:=(E+E')_{(2)}^k(E+E')_{(3)}^lz$. Similarly, we can derive identities $\overline C_{k,l}$ and $\widecheck C_{k,l}$ by applying $(C_2)$ and $(C_3)$
	to $w$ and $\tilde z$ as above. These identities can be used to prove the existence of $B_{i,j}$ and $S_{i,j}$ inductively. To be a bit more detailed, firstly
	one shows the existence of $B_{k,0}$ and $S_{k,0}$ inductively using the initial values in Proposition \ref{blowup-inital},
	$\widehat C_{k,0}$ and $\widehat C_{k,0}'$ \cite[Proposition 69]{Luc:Thesis}. Then another inductive argument with the aid of identities
	$\overline C_{k,l}$, $\widecheck C_{k,l}$ and Proposition \ref{blowup-inital} shows the existence of $B_{k,l}$ \cite[Proposition 73]{Luc:Thesis}.
	Finally, $\widehat C_{k,l}$, $\widehat C_{k,l}'$, the initial values of Proposition \ref{blowup-inital} and the fact that $S_{0,1}$ is non-zero imply the existence of
	$S_{k,l}$ in an inductive way. After proving the existence of the polynomials $B_{k,l}$ and $S_{k,l}$, the identities $\widehat C_{k,l}$, $\widehat C_{k,l}'$,
	$\overline C_{k,l}$ and $\widecheck C_{k,l}$ can be used to write four differential equations for $B$ and $S$ which are given in \eqref{PDE1-2}, \eqref{PDE3} and \eqref{PDE4}.
	The existence proof shows that these PDEs and the initial values are enough to uniquely determine $B$ and $S$.
\end{proof}

In general, computing the exact form of the power series $B$ and $S$ (equivalently, solving the PDEs in the statement of Theorem \ref{blowup-gen-mafld}) is not straightforward. In the following corollary, we show that for 4-manifolds with simple type, the blow up formula has a simple form:
\begin{theorem}\label{blowup}
	Suppose $(X,w)$ is a pair of a 4-manifold and a 2-cycle such that $X$ has $w$-simple type. Suppose also $\hat{X}$ is the blowup of $X$ at one point and $E$ is the exceptional
	class. Then there are power series $b(t_2,t_3), \, s(t_2,t_3) \in \Q[\![t_2,t_3]\!]$ such that:
	\begin{equation} \label{blowup-b}
		\widehat \rD_{\hat{X},w}(\e^{E_{(2)}+E_{(3)}}z)=\hDXw(z)b(t_2,t_3)
	\end{equation}
	and
	\begin{equation}\label{blowup-s}
		\widehat \rD_{\hat{X},w+E}(\e^{E_{(2)}+E_{(3)}}z)=\hDXw(z)s(t_2,t_3)
	\end{equation}
	for $z \in \A(X)^{\otimes 2}$. The power series $b$ and $s$ are given by the following formulas:
	\begin{equation} \label{b-formula}
		b(t_2,t_3)=\frac{1}{3}\e^{-\frac{t_2^2}{2}+t_3^2}[\cosh (\sqrt{3}t_2) +2\cos (\sqrt{3}t_3)],
	\end{equation}
	\begin{equation} \label{s-formula}
		s(t_2,t_3)=\frac{1}{3}\e^{-\frac{t_2^2}{2}+t_3^2}[\cosh (\sqrt{3}t_2)-\cos (\sqrt{3}t_3)+\sqrt {3} \sin(\sqrt{3}t_3)].
	\end{equation}
\end{theorem}
Formula \eqref{b-formula} for the power series $b(t_2,t_3)$ was previously found by means of quantum field theory arguments in \cite[Formula 6.22]{EGM:blowup}. When comparing the formulas here and in \cite{EGM:blowup}, the reader should note that $t_3$ needs to be replaced with $\bi t_3$ because of different conventions in the definition of the $\mu$ map.
\begin{proof}
	Evaluating $B$ and $S$  of Theorem \ref{blowup-gen-mafld} at $a_2=3$ and $a_3=0$ produces $b$ and $s$ with the required property.
	These power series satisfy equations \eqref{PDE1-2}, \eqref{PDE3} and \eqref{PDE4} where $a_2$ and $a_3$ in the latter two equations
	are replaced with $3$ and $0$. In fact, the same proof as the existence proof in Theorem \ref{blowup-gen-mafld} shows that $b$ and $s$
	are uniquely determined by these equations and the initial vales in Proposition \ref{blowup-inital} (with $a_2$ and $a_3$ replaced by $3$ and $0$).
	The power series \eqref{b-formula} and \eqref{s-formula} satisfy the required conditions. Therefore, they are equal to $b$ and $s$.
\end{proof}

\begin{remark} \label{S--E}
	Identity \eqref{wto-w} implies that:
 \begin{equation}\label{blowup-s}
		\widehat \rD_{\hat{X},w-E}(\e^{E_{(2)}+E_{(3)}}z)=\hDXw(z)s(t_2,-t_3)
 \end{equation}
\end{remark}

\section{Floer Homologies for Closed 3-manifolds} \label{FH}
\subsection{Admissible Pairs} \label{IN}
A pair $(Y,\gamma)$ of a connected closed 3-manifold and an embedded oriented 1-manifold is {\it $N$-admissible} if there is an embedded oriented surface $\Sigma$ in $Y$ such that the integer $\Sigma \cdot \gamma$ is coprime to $N$. Suppose $Q$ is the $\U(N)$-bundle associated to the pair $(Y,\gamma)$. The admissibility of $(Y,\gamma)$ is what is called the {\it non-integral} condition for the $\U(N)$-bundle $Q$ in \cite{KM:YAFT}. In particular, we can use the construction of  \cite{KM:YAFT} and associate the {\it instanton Floer homology group} $\IN(Y,\gamma)$ to an $N$-admissible pair $(Y,\gamma)$. Instanton Floer homology can be lifted to a functor from a cobordism category $\text{\sc cob}_{\mathbb A}$ to a certain category of vector spaces.

An object of the category $\text{\sc cob}_{\mathbb A}$ is an $N$-admissible pair. A morphism from a pair $(Y_0,\gamma_0)$ to another pair $(Y_1,\gamma_1)$ is a triple $(W,w,z)$ where $W:Y_0 \to Y_1$ is a cobordism, $w:\gamma_0 \to \gamma_1$ is a 2-cycle, and $z \in \mathbb A(W)^{\otimes (N-1)}$. The composition of two morphisms:
\begin{equation*}
	(W,w,z):(Y_0,\gamma_0) \to (Y_1,\gamma_1) \hspace{1cm} (W',w',z'):(Y_1,\gamma_1) \to (Y_2,\gamma_2)
\end{equation*}
is equal to $(W'\circ W,w'\circ w,z'\cdot z)$.

Suppose {$\text{\sc vector}_n$} is the category of relatively $\Z/n\Z$-graded vector spaces over $\C$. An object of this category is a vector space $V$ with a direct sum decomposition:
\begin{equation*}
	V=\bigoplus_{j\in J} V_j
\end{equation*}
where $\Z/n\Z$ acts transitively and freely on $J$. A morphism in this category from $V=\bigoplus_{j\in J} V_j$ to $V'=\bigoplus_{j'\in J'} V'_{j'}$ is a complex linear map $f:V \to V'$ such that $f$ maps each $V_j$ to a summand $V'_{h(j)}$ of $V'$ such that $h(j+k)=h(j)+k$. Let $\text{\sc p-vector}_n$ be the category that has the same objects as {$\text{\sc vector}_n$}. A morphism in {$\text{\sc p-vector}_n$} is a vector space homomorphism as above which is well-defined only up to a sign. Instanton Floer homology gives a functor $\IN: \text{\sc cob}_{\mathbb A} \to \text{\sc p-vector}_{4N}$.

\begin{remark}
	The invariant constructed in \cite{KM:YAFT} is more general than the one we described here.
	In \cite{KM:YAFT}, a version of
	instanton Floer homology is constructed for a triple $(Y,\gamma,K)$ where $K$ is a link in $Y$ and $\gamma$ determines a
	$\U(N)$-bundle on $Y$ that satisfies a certain non-integral condition. We need to consider only the case that $K$ is
	the empty link. On the other hand, in \cite{KM:YAFT}, the cobordism maps $\IN(W,w,z)$ are defined only in the case that $z=1$.
	The more general case, is a straightforward generalization and is reviewed below.
\end{remark}

For an $N$-admissible pair $(Y,\gamma)$, the vector space $\IN(Y,\gamma)$ is defined by applying Morse homological methods to the Chern-Simons functional $\CS:\mathcal B(Y,\gamma) \to \R$. The admissibility condition implies that all critical points of $\CS$ are irreducible \cite[Proposition 3.1]{KM:YAFT}. Therefore, we can arrange for $\CS_\pi$, a perturbation of the Chern-Simons functional with a small holonomy perturbation, such that all of its critical points are irreducible and non-degenerate \cite[Proposition 3.10]{KM:YAFT}. Then $\CS_\pi$ has finitely many critical points. Suppose $\alpha$ and $\beta$ are two critical points of $\CS_\pi$ and $p$ is a path along $([0,1]\times Y,[0,1]\times \gamma)$ based at $\alpha$ and $\beta$ on $\{0\}\times Y$ and $\{1\}\times Y$. We will write $\mathcal M_p(\alpha,\beta)$ for the moduli space of the solutions to the perturbed ASD equation on $\R \times Y$ associated to the path $p$. Here the perturbation of the ASD equation is induced by the perturbation of the Chern-Simons functional, i.e., the elements of $\mathcal M_p(\alpha,\beta)$ can be regarded as the downward gradient flow lines of $\CS_\pi$. We can also assume that $\CS_\pi$ is chosen such that the elements of $\mathcal M_p(\alpha,\beta)$ for all choices of $\alpha$, $\beta$ and $p$ are regular \cite[Proposition 3.18]{KM:YAFT}. There is an $\R$-action on $\mathcal M_p(\alpha,\beta)$ given by translation along the $\R$-factor. The quotient space is denoted by $\breve{\mathcal M}_p(\alpha,\beta)$. The dimension formula of the previous section implies that:
\begin{equation*}
	\dim(\breve{\mathcal M}_p(\alpha,\beta))=\deg(\alpha)-\deg(\beta)-1 \mod 4N
\end{equation*}

Instanton Floer homology of the pair $(Y,\gamma)$ is given by the homology of a chain complex $(\mathfrak C^\pi_*,d)$ associated to the functional $\CS_\pi$. The vector space $\mathfrak C^\pi_*$ is freely generated by the critical points of $\CS_\pi$. The differential of a generator $\alpha$ of $\mathfrak C^\pi_*$ is also defined as:
\begin{equation}\label{differential}
	d(\alpha):=\sum_{p:\alpha\to \beta} \#\breve{\mathcal M}_p(\alpha,\beta)\beta
\end{equation}
where the sum is over all paths $p$ that $\breve{\mathcal M}_p(\alpha,\beta)$ is 0-dimensional. These 0-dimensional moduli spaces are compact and we orient them as in \cite[Subsection 3.6]{KM:YAFT}. Then $\#\breve{\mathcal M}_p(\alpha,\beta)$ denotes the signed count of the points in $\breve{\mathcal M}_p(\alpha,\beta)$. We use the Floer grading $\deg$ to define a relative $\Z/4N\Z$-grading on $\mathfrak C^\pi_*$. Then the differential $d$ decreases this grading by 1. The relatively $\Z/4N\Z$-graded vector space $\IN(Y,\gamma)$ is defined to be the homology of the chain complex $(\mathfrak C^\pi_*,d)$. This homology group is independent of the choice of the Riemannian metric on $Y $ and the perturbation of the Chern-Simons functional.

Suppose the Chern-Simons functional associated to the pair $(Y,\gamma)$ is {\it Morse-Bott}. That is to say, the set of critical points of $\CS$ is a smooth manifold, and the Hessian of $\CS$ is invertible only in the normal direction to the critical manifold. Following the above definition, we need to work with perturbations of the Chern-Simons functional to define $\IN(Y,\gamma)$. However, one can still derive some information about $\IN(Y,\gamma)$ using the unperturbed Chern-Simons functional:
\begin{prop}
	If the Chern-Simons functional of the pair $(Y,\gamma)$ is Morse-Bott, then $\dim(\IN(Y,\gamma))\leq \dim(H_*({\rm crit}(\CS)))$.
\end{prop}
\begin{proof}
	This claim can be verified using the standard spectral sequence that starts from the homology of the critical manifold of $\CS$
	and abuts to the instanton Floer homology of $(Y,\gamma)$.
\end{proof}

Suppose $(W,w):(Y_0,\gamma_0) \to (Y_1,\gamma_1)$ is a cobordism of $N$-admissible pairs. After fixing Riemannian metrics on $Y_0$ and $Y_1$, we form the non-compact manifold $W^+$ by adding cylindrical ends to $W$. Suppose also a perturbation of the Chern-Simons functional for $(Y_i,\gamma_i)$ is fixed such that we can form the chain complex $(\mathfrak C_*^{\pi_i},d)$ as above. We use a perturbation of the ASD equation on $W^+$ which is compatible with the chosen perturbations of the Chern-Simons functionals. Given a generator $\alpha_i$ of $\mathfrak C_*^{\pi_i}$ and a path $p:\alpha_0 \to \alpha_1$ along $(W,w)$, let $\mathcal M_p(W,w; \alpha_0,\alpha_1)$ be the moduli space of the corresponding equation. We can pick a perturbation of the ASD equation on $W$ such that this moduli space is a smooth manifold. By fixing a homology orientation of $W$, we can also assume that all such moduli spaces are oriented \cite{KM:YAFT}. Moreover, if $\mathcal M_p(W,w; \alpha_0,\alpha_1)$ is 0-dimensional, then it is compact.

Define a map $\mathfrak C(W,w):\mathfrak C_*^{\pi_0} \to \mathfrak C_*^{\pi_1}$ by:
\begin{equation} \label{cob-map}
	\mathfrak C(W,w)(\alpha_0):=\sum_{p:\alpha_0 \to \alpha_1} \#\mathcal M_p(W,w; \alpha_0,\alpha_1)\alpha_1
\end{equation}
where the sum is over all paths that $\mathcal M_p(W,w; \alpha_0,\alpha_1)$ is 0-dimensional. The term $\#\mathcal M_p(W,w; \alpha_0,\alpha_1)$ is equal to the signed count of the points in $\mathcal M_p(W,w; \alpha_0,\alpha_1)$. In fact, \eqref{cob-map} defines a chain map and the induced map at the level of homology, $\IN(W,w,1):\IN(Y_0,\gamma_0) \to \IN(Y_1,\gamma_1)$, determines a morphism of the category $\text{\sc p-vector}_{4N}$. More generally, given $z \in \mathbb A(W)^{\otimes (N-1)}$, we can cut down the moduli space $\mathcal M_p(W,w; \alpha_0,\alpha_1)$ using $z$ as in subsection \eqref{pol-invts}, and construct a smooth submanifold $\mathcal M_p(W,w; \alpha_0,\alpha_1,z)$ such that:
\begin{equation} \label{dim-form}
	\dim(\mathcal M_p(W,w; \alpha_0,\alpha_1,z))=\dim(\mathcal M_p(W,w; \alpha_0,\alpha_1))-\deg(z).
\end{equation}
To be more precise, $\mathcal M_p(W,w; \alpha_0,\alpha_1,z)$ is a linear combination of the spaces whose dimensions are given by \eqref{dim-form}. Replacing $\mathcal M_p(W,w; \alpha_0,\alpha_1)$ in \eqref{cob-map} with $\mathcal M_p(W,w; \alpha_0,\alpha_1,z)$ determines a new chain map and the associated map at the level of homology is $\IN(W,w,z): \IN(Y_0,\gamma_0) \to \IN(Y_1,\gamma_1)$. This map is well-defined, namely, it does not depend on the metric on $W$, the perturbation of the ASD equation and the choices involved in the construction of $\mathcal M_p(W,w; \alpha_0,\alpha_1,z)$. Because we fix the central part of connections associated to $(W,w)$, the cobordism maps only depend on the induced ${\rm PU}(N)$-bundles. In particular, if $w$ is replaced with $w+nw'$, for a closed 2-cycle $w'$, then the induced ${\rm PU}(N)$-bundles are the same and they determine the same cobordism maps. This property is the analogue of Identity \eqref{w+Nw'} for closed 4-manifolds.

\begin{remark}
	A priori it might seem that the cycles in 3-manifolds and 4-dimensional cobordisms only keep track of the first
	Chern classes of $\U(N)$-bundles. However, they have strictly more information than these cohomology classes.
	For example, an element of $H^2(Y,\Z)$, for a 3-manifold $Y$, determines a $\U(N)$-bundle up to an isomorphism of
	$\U(N)$-bundles. But an embedded 1-manifold $\gamma$ in $Y$ determines a $\U(N)$-bundle up to a canonical isomorphism.
	(See \cite[Section 4]{KM:Kh-unknot} for more details.) As a manifestation of this issue,
	suppose cohomology classes $\alpha$ and $\alpha'$ are given on cobordisms
	$W:Y_0 \to Y_1$ and $W':Y_1 \to Y_2$ such that the restriction of these classes on $Y_1$ agree with each other. Then there might
	be an ambiguity to glue these cohomology classes and construct an element of $H^2(W'\circ W;\Z)$. On the other hand, there
	is not such an ambiguity if $\alpha$  and $\alpha'$ are represented by embedded surfaces $w\subset W$ and $w' \subset W'$ that
	$w|_{Y_1}=w'|_{Y_1}$.
\end{remark}

\begin{remark}
	There is not much difficulty in extending the definition of instanton Floer homology to the case of disconnected 3-manifolds. Suppose $Y$ is a disconnected 3-manifold and $\gamma \subset Y$ is a 1-cycle.
	Then we say $(Y,\gamma)$ is $N$-admissible if for each connected component $Y_0$ of $Y$, the pair $(Y_0,\gamma \cap Y_0)$ is $N$-admissible. Then we can repeat the definition analogous to the connected case and
	construct instanton Floer homology for the pair $(Y,\gamma)$. This instanton Floer homology can be computed in terms of the invariants for the connected components:
	\begin{equation*}
		\IN(Y,\gamma):=\IN(Y_1,\gamma_1) \otimes \dots \otimes \IN(Y_n,\gamma_n)
	\end{equation*}
	Here $Y_i$'s are connected components of $Y$ and $\gamma_i=\gamma \cap Y_i$. It is also possible to define cobordism maps for a cobordism of pairs $(W,w)$ between two (not necessarily connected) $N$-admissible pairs
	and $z \in \mathbb A(W)^{\otimes (N-1)}$. However, these maps are not as well-behaved with respect to composition as in the previous case. Consider two triples:
	\begin{equation*}
		(W,w,z):(Y_0,\gamma_0) \to (Y_1,\gamma_1) \hspace{1cm} (W',w',z'):(Y_1,\gamma_1) \to (Y_2,\gamma_2).
	\end{equation*}
	If $Y_1$ is not connected, what we can say about the cobordism maps is:
	\begin{equation*}
		\IN(W'\circ W,w'\circ w,z\cdot z')=c\cdot \IN(W',w',z') \circ \IN(W,w,z)
	\end{equation*}
	for some non-zero constant $c$. The simplest way to fix this issue about functoriality is to work with a variation of the category $\text{\sc p-vector}_{4N}$ where the morphisms are well-defined only up to a non-zero scalar.
	We follow this approach when it is necessary to work with disconnected 3-manifolds.
\end{remark}

\begin{remark}
	A slightly unsatisfying point about $\IN$ is the sign ambiguity in the definition of the cobordism maps. This issue can be avoided
	in a strightforward way. In the case that $N$ is an even number, we need to change the definition of the category
	${\text {\sc cob}_{\mathbb A}}$ slightly. Let ${\widetilde{\text {\sc cob}}_{\mathbb A}}$ have the same objects as
	${\text {\sc cob}_{\mathbb A}}$. But a morphism of this new category is a quadruple $(W,w,z,o_W)$ where $W$, $w$ and
	$z$ are as before and $o_W$ is a homology orientation for $W$. Then $\IN$ can be lifted to a functor
	$\widetilde {\rm I}_*^N:{\widetilde{\text {\sc cob}}_{\mathbb A}} \to \text{\sc vector}_{4N}$.
	The main point is that initially there is an ambiguity in the orientation of the moduli spaces  $\mathcal M_p(W,w; \alpha_0,\alpha_1)$
	that appear in the definition of the cobordism maps, and a homology orientation of $W$ fixes this ambiguity. In the case that $N$ is odd, there is not
	such ambiguity and one can readily lift $\IN$ to $\widetilde {\rm I}_*^N:{\text {\sc cob}}_{\mathbb A} \to \text{\sc vector}_{4N}$.
\end{remark}

The definition of cobordism maps can be extended to the case that one of the ends is the empty pair. Suppose $X$ is a 4-manifold with boundary $Y$ and $w$ is a properly embedded surface in $X$ such that $\gamma:=\partial w=w \cap Y$. Assume that $(Y,\gamma)$ is an $N$-admissible pair. Given any element $z\in \mathbb A(X)^{\otimes (N-1)}$, we can form an element ${\rm D}_{X,w}^{N}(z)$ of $\IN(Y,\gamma)$. This construction is the extension of $\U(N)$-polynomial invariants for closed 4-manifolds in the previous section. Alternatively, $(X,w,z)$ can be regarded as a cobordism from the empty pair to the $N$-admissible pair $(Y,\gamma)$. Although the empty pair is not $N$-admissible, the formula \eqref{cob-map} can be used to define ${\rm D}_{X,w}^{N}(z)$.

\begin{remark} \label{Z/4N-grading-II}
	Given $z_i\in \mathbb A(X_i)^{\otimes (N-1)}$, we can consider the relative elements ${\rm D}_{X_i,w_i}^N(z_i)\in \IN(Y,\gamma)$. Each of these relative elements lies in a graded summand of $\IN(Y,\gamma)$.
	Therefore, the difference $\deg({\rm D}_{X_2,w_2}^N(z_2))-\deg({\rm D}_{X_1,w_1}^N(z_1))$ of the relative $\Z/4N\Z$-gradings is a well-defined number in $\Z/4N\Z$ and is equal to:
	\begin{equation*}
		2(N+1)(w_2^2-w_1^2)-(N^2-1)(\frac{\chi(X_2)+\sigma(X_2)}{2}-\frac{\chi(X_1)+\sigma(X_1)}{2})-(\deg(z_2)-\deg(z_1))
	\end{equation*}
	Note that the term $w_i^2$ is not well-defined and depends on a framing of the 1-cycle $\gamma$. However, the difference $w_2^2-w_1^2$ is independent of the framing and the above expression is well-defined.
	A similar formula can be written for the difference between the gradings of two cobordisms with the same ends.
\end{remark}

Suppose $(X,w)$ is a cobordism from an $N$-admissible pair $(Y,\gamma)$ to the empty pair, namely, $X$ is a 4-manifold whose boundary is identified with $\overline Y$, the 3-manifold $Y$ with the reverse orientation. The boundary of the embedded surface $w$ is also identified with $\overline \gamma$. Suppose also $z \in \mathbb A(X)^{\otimes (N-1)}$. Similarly, we can construct a functional ${\rm D}^{X,w}_{N}(z):\IN(Y,\gamma) \to \C$.

Similar to cobordism maps, ${\rm D}_{X,w}^{N}(z)$ and ${\rm D}^{X,w}_{N}(z)$ satisfy some functorial properties. For example, if $(X,w,z)$ and $(W,w',z')$ are chosen such that:
\begin{equation*}
	\partial (X,w)=(Y_0,\gamma_0) \hspace{1cm} (W,w',z'):(Y_0,\gamma_0) \to (Y_1,\gamma_1),
\end{equation*}
then:
\begin{equation*}
	{\rm D}_{W\circ X,w'\circ w}^{N}(z\cdot z')=\IN(W',w',z') \circ {\rm D}_{X,w}^{N}(z).
\end{equation*}	
A similar property holds for ${\rm D}^{X,w}_{N}(z)$. There is also an important relation among these invariants and invariants of closed manifolds form the previous section:

\begin{prop} \label{gluing}
	Suppose $Y$ is a closed and connected 3-manifold, $(Y,\gamma)$ is an $N$-admissible pair,
	and $X_1$ and $X_2$ are two smooth 4-manifolds with $\partial X_1=Y$ and $\partial X_2=\overline Y$. Suppose also
	oriented properly embedded surfaces $w_i \subset X_i$ are given such that $\partial w_1=\gamma$ and $\partial w_2=\overline \gamma$. If
	$b^+(X_2 \circ X_1)\geq 2$, then for $z_i \in \mathbb A(X_i)^{\otimes (N-1)}$:
	\begin{equation} \label{gluing-form}
		{\rm D}_{X_2 \circ X_1,w_2\circ w_1}^{N}(z_1\cdot z_2)={\rm D}^{X_2,w_2}_{N}(z_2) \circ {\rm D}_{X_1,w_1}^{N}(z_1)
	\end{equation}
	If $b^+(X_2 \circ X_1)=1$, then a similar formula holds where the left hand side of \eqref{gluing-form} is interpreted as the invariant of the
	chamber associated to metrics with a long neck along $Y$.
\end{prop}

There is another way that the invariants of a closed 4-manifold can be related to the cobordism maps:
\begin{prop} \label{trace}
	Let $(Y,\gamma)$ be an $N$-admissible pair, and $(W,w):(Y,\gamma) \to (Y,\gamma)$ be a cobordism of pairs. Let $W$ and $Y$ be connected.
	Let $\widetilde W$ be the closed 4-manifold, given by gluing the incoming end of $W$ to its outgoing end.
	Suppose also $\widetilde w\subset \widetilde W$ is defined similarly. If $b^+(\widetilde W)\geq 2$, then for $z \in \mathbb A(W)^{\otimes (N-1)}$:
	\begin{equation} \label{gluing-form-2}
		{\rm D}_{\widetilde W,\widetilde w}^{N}(z) =N\cdot {\rm str}(\IN(W,w,z))
	\end{equation}
	where ${\rm str}$ denotes the super-trace of $\IN(W,w,z)$.\footnote{Since the $\Z/4N\Z$-grading  on instanton Floer homology is defined only relatively,
	there is a sign
	ambiguity on the right hand side, even after fixing a homology orientation for $W$. There is also a sign ambiguity on the left hand side because we did not fix
	any homology orientation for $\widetilde W$. Therefore, Equation \eqref{gluing-form-2} should be interpreted as an equality up to sign. Although we
	do not need it here, both sign ambiguities can be fixed after fixing a homology orientation for $W$.}
	If $b^+(\widetilde W)=1$, then a similar formula holds where the left hand side of \eqref{gluing-form-2} is interpreted as the invariant of the
	chamber associated to metrics with a long neck along $Y$.
\end{prop}

\subsection{Floer Homology of $S^1 \times \Sigma$} \label{HF-S1-S}
Suppose $\Sigma$ is the connected oriented Riemann surface of genus $g$. The 3-manifold $Y_g:= \Sigma \times S^1$ and the embedded 1-manifold $\gamma_{g,d}:=\{x_1 ,\,\dots,\ x_d\} \times S^1$ determine an $N$-admissible pair if $(d,N)=1$. The $\U(N)$-bundle associated to the pair $(Y_g,\gamma_{g,d})$ is the pull-back of a $\U(N)$-bundle $Q_d$ of degree $d$ on $\Sigma$. Let also $L_d$ denote the determinant of $Q_d$. Recall that the space $\mathcal A(Y_g,\gamma_{g,d})$ is constructed using an auxiliary connection on $L_{\gamma_{g,d}}$. We assume that this connection is the pull-back of a connection $B_0$ on $L_d$. Similar to the 3-dimensional case, $\mathcal A(\Sigma, Q_d)$ is  defined to be the space of $\U(N)$-connections on $Q_d$ whose determinants are equal to $B_0$. The space $\mathcal G_d$ is also defined to be the group of determinant 1 automorphisms of $Q_d$.

Let $\VgdN$ be the vector space $\IN(Y_g,\gamma_{g,d})$. The critical points of the Chern-Simons functional for the pair $(Y_g,\gamma_{g,d})$ can be identified with $N$ copies of the following space:
\begin{equation*}
	\Nnd:=\{A \in \mathcal A(\Sigma, Q_d) \mid F_0(A)=0 \}/\mathcal G_d.
\end{equation*}
In fact, we can pull back any element of $\Nnd$ to $\Sigma \times [0,1]$ and then identify the connections on $\Sigma \times \{0\}$ and $\Sigma \times \{1\}$ using an element of $\mathcal G_d$ which is induced by a central element of $\SU(N)$. The space $\Nnd$ is a smooth manifold of dimension $(N^2-1)(2g-2)$ \cite{AB:YM-RS}. Moreover, the Chern-Simons functional in this case is Morse-Bott, hence $\dim(\VgdN)\leq N\dim(H^*(\Nnd))$.

The space $\Nnd$ has been extensively studied in the literature. This space is a K\"ahler manifold and can be identified with the moduli space of {\it stable bundles} of rank $N$ and degree $d$ on a Riemann surface of genus $g$ with a fixed determinant. The Poincar\'e polynomial of this manifold can be computed inductively \cite{HN:Coh-Vec-RS,DR:Coh-Vec-RS,AB:YM-RS}. Furthermore, a set of generators for the cohomology ring of this space is given \cite{AB:YM-RS}. We review a slightly reformulated description of these generators which are more suitable for our purposes here.

Consider the 4-manifold $X_g:=\Sigma \times S^2$  and the surface $\chi_g:=\{x_1,\,\dots,\,x_d\} \times S^2$. The pull back of the elements of $\Nnd$ to $X_g$ are ASD connections associated to the pair $(X_g,\chi_g)$ with  $\kappa=0$. In particular, $\Nnd$ can be regarded as a subset of $\mathcal B_0(X_g,\chi_g)$. Consider the subalgebra  $\mathbb A^{N}_{g}:=\mathbb A(\Sigma)^{\otimes (N-1)}$ of $\mathbb A(X_g)^{\otimes (N-1)}$. The $\mu$-map in \eqref{mu-map} determines a graded algebra homomorphism $\Psi: \mathbb A^{N}_{g} \to H^*(\Nnd)$. Note that this map is equivariant with respect to ${\rm Diff}(\Sigma)$, the group of diffeomorphisms of $\Sigma$.

\begin{prop} \label{cohomology-gen}
	The map $\Psi$ is surjective. In particular, the cohomology ring of $\Nnd$ is generated by the following elements:
	\begin{equation}\label{coh-class}
		p_r:=\Psi(a_r) \hspace{1cm} q^j_r:=\Psi(l^j_{(r)})  \hspace{1cm} s_r:=\Psi(\Sigma_{(r)})
	\end{equation}
	where $2\leq r \leq N$ and $\{l^j\}_{1\leq j \leq 2g}$ forms a set of generators for $H_1(\Sigma,\Z)$.
\end{prop}
The action of ${\rm Diff}(\Sigma)$ on $H_1(\Sigma)$ factors through the action of the symplectic group $\Sp(2g)$. Therefore, this proposition implies that the same holds for $H^*(\Nnd)$.
\begin{proof}
	In \cite{AB:YM}, a universal $\U(N)$-bundle $\mathbb F$ is constructed over the product manifold $\Nnd \times \Sigma$ and it is shown that the cohomology ring of
	$\Nnd$ is generated by the following classes:
	\begin{equation} \label{tilde-coh-class}
		\tilde p_r:=c_r(\mathbb F)\slash [a] \hspace{1cm} \tilde q^j_r:=c_r(\mathbb F)\slash [l^j]  \hspace{1cm} \tilde s_r:=c_r(\mathbb F)\slash [\Sigma]
	\end{equation}
	The map $\Psi$ is defined similarly using the universal ${\rm PU}(N)$-bundle $\mathbb P$ over the space $ \mathcal B_\kappa(X_g,\chi_g) \times X_g$.
	The restriction of $\mathbb P$ to $\Nnd \times \Sigma \subset  \mathcal B_\kappa(X_g,\chi_g) \times X_g$ is isomorphic to the ${\rm PU}(N)$-bundle associated to
	$\mathbb F$. Thus the cohomology classes in \eqref{tilde-coh-class} can be identified with the corresponding ones in \eqref{coh-class} and this verifies the claim.
\end{proof}

The vector space $\VgdN$ admits a ring structure which is the analogue of the cup product on $H^*(\Nnd)$. Suppose $P$ is the pair of pants cobordism from two copies of $S^1$ to one copy of $S^1$. Then the triple $(\Sigma \times P,\{x_1,\dots,x_d\}\times P,1)$ defines a map $m:\VgdN \otimes \VgdN \to \VgdN$. To be more precise, we need to fix a homology orientation in the case that $N$ is even. Suppose $\Delta_g:= \Sigma \times D^2 $ and $\delta_{g,d}:=\{x_1,\dots,x_d\}\times D^2$ where $D^2$ is the 2-dimensional disc. We fix an arbitrary homology orientation on $\Delta_g$ and let $e:={\rm D}^N_{\Delta_g,\delta_{g,d}}(1) \in \VgdN$. We shall see in the proof of Proposition \ref{non-deg-pairing} that $e$ is non-zero. We choose a homology orientation on $\Sigma\times P$ such that $m(e,e)=e$. Then functoriality of instanton Floer homology can be used to show that $m$ defines a ring structure on $\VgdN$ with the unit $e$. We turn the relative $\Z/4N\Z$-grading on $\VgdN$ into an absolute grading by requiring that the unit element has degree 0. With this convention, the multiplication map is $\Z/4N\Z$-graded, namely, the product of two elements of degree $i_1$ and $i_2$ has degree $i_1+i_2$.

Suppose $B$ is a cylinder, regarded as a cobordism with two circles as the incoming end and the empty outgoing end. Then the pair:
\begin{equation} \label{Omega}
	\Omega_g:=\Sigma \times B \hspace{1cm}\omega_g:=\{x_1,\dots,x_d\}\times B
\end{equation}
determines a pairing $\langle\,,\,\rangle:\VgdN\otimes \VgdN \to \C$ which is defined in the following way after we choose an arbitrary homology orientation on $\Omega_g$:
\begin{equation*}
	{\rm D}^{\Omega_g,\omega_g}_N(1).
\end{equation*}

\begin{prop} \label{VgdN-vector-space}
	The space $\VgdN$ as a complex vector space with an action of $\Diff(\Sigma)$, is isomorphic to
	$H^*(\Nnd (\Sigma))[u]/(u^N-1)$. In particular, the action of $\Diff(\Sigma)$ on $\VgdN$ factors through an action of $\Sp(2g)$
\end{prop}

In the case that $N=2$, one can show that $\VgdN$ is isomorphic to $H^*(\Nnd (\Sigma))[u]/(u^N-1)$ using the results of \cite{DS:At-Fl}. Mu\~noz gave an alternative proof of this proposition for $N=2$ \cite{Mun:Gl}, and the proof of the general case is based on his approach.

\begin{proof}
	 We can define a $\Diff(\Sigma)$-equivariant algebra homomorphism $\Phi:\mathbb A^{N}_{g}[u]/(u^N-1) \to \VgdN$
	in the following way:
	\begin{equation*}
		\Phi(u^iz):= {\rm D}^N_{\Delta_g,\delta_{g,d}+i \Sigma }(z)
	\end{equation*}
	where on the right hand side $z$ is regarded as an element of $\mathbb A(\Delta_g)^{\otimes (N-1)}$.
	Let $S: H^*(\Nnd) \to  \mathbb A^{N}_{g}$ be a graded $\Sp(2g)$-equivariant right inverse
	for the map $\Psi$. Extend $\Psi$ to an algebra homomorphism from $\mathbb A^{N}_{g}[u]/(u^N-1)$ to $H^*(\Nnd (\Sigma))[u]/(u^N-1)$ by
	requiring that $\Psi(u)=u$. Similarly, we can assume $S$ is defined on $H^*(\Nnd (\Sigma))[u]/(u^N-1)$.
	We claim that the $\Diff(\Sigma)$-equivariant map $\Phi \circ S:H^*(\Nnd (\Sigma))[u]/(u^N-1)\to \VgdN$ is injective.
	If the claim does not hold, then there is:
	\begin{equation*}
		\hspace{2cm} p=\sum_{m=1}^M z_mu^{i_m}\in \mathbb A^{N}_{g}[u]/(u^N-1) \hspace{1.5cm} z_m \in \mathbb A^{N}_{g},\,0\leq i_m <N
	\end{equation*}
	such that $\Psi(p)\neq 0$ and $\Phi(p)=0$. We assume that each $z_m$ is non-zero and lies in one of the graded summands of $\mathbb A^{N}_{g}$.
	Furthermore, if $m<n$, then $\deg(z_m)\geq \deg(z_n)$ and equality holds only if $i_m \neq i_n$.
	Let $z' \in \mathbb A_{g}^{N}$ be such that:
	\begin{equation} \label{degree-cup}
		\deg(z')+\deg(z_1)=(N^2-1)(2g-2)
	\end{equation}
	and the cup product of $\Psi(z')$ and $\Psi(z_1)$ is non-zero.
	By Proposition \ref{gluing}, the pairing $\langle \Phi(p), \Phi(u^{N-i_1}z') \rangle$ is equal to:
	\begin{equation} \label{the-pairing}
		\sum_{m=1}^M{\rm D}_{\Sigma \times S^2,w_{g,d}+(N+i_m-i_1)\Sigma}(z'z_m)
	\end{equation}
	where the polynomial invariants are computed in the chamber that the fiber $\Sigma$ is small.
	The dimension formula in \eqref{ind-form-closed} shows that the dimension of each component of the moduli space associated to the pair
	$(\Sigma \times S^2,w_{g,d}+k\Sigma)$ is at least $(N^2-1)(2g-2)$.
	Moreover, if $k$ is not divisible by $N$, this dimension is strictly greater than $(N^2-1)(2g-2)$. In the case that $k=0$ (or divisible by $N$), the moduli space of dimension
	$(N^2-1)(2g-2)$ is given by the pull-back of the elements of $\Nnd$ to $X_g$.
	Therefore, the only non-zero term in \eqref{the-pairing} is ${\rm D}_{\Sigma \times S^2,w_{g,d}}(z'z_1)$ which is given by evaluating the cohomology class
	$\mu(z'z_1)$ on the pull-back of the elements of $\Nnd$ to $X_g$. Note that the moduli space is compact in this case and we do not need
	to use the geometric representatives to 	
	evaluate this invariant. Therefore, ${\rm D}_{\Sigma \times S^2,w_{g,d}}(z'z_1)$ is equal to $ \Psi(z')\cup \Psi(z)[\Nnd]$ which is non-zero by assumption.
	This contradicts the assumption that $\Phi(p)=0$. Therefore, the map $\Phi \circ S$ is injective. We already know that the dimension of $\VgdN$ is not greater than
	$N\dim(H^*(\Nnd))$. Therefore, $\Phi \circ S$ is a bijection. In particular, the action of $\Diff(\Sigma)$ on $\VgdN$ factors through an action of $\Sp(2g)$.
\end{proof}

\begin{cor} \label{VgdN-generators}
	The ring $\VgdN$ is generated by the following elements:
	\begin{equation*}
		\epsilon={\rm D}_{\Delta_g,\delta_{g,d}+\Sigma}(1),\hspace{.1 cm} \aleph_r={\rm D}_{\Delta_g,\delta_{g,d}}(a_r),\hspace{.1cm} o_r^j={\rm D}_{\Delta_g,\delta_{g,d}}(l^j_{(r)}),
		\hspace{.1cm} \rho_{r}={\rm D}_{\Delta_g,\delta_{g,d}}(\Sigma_{(r)})
	\end{equation*}
	where $2\leq r\leq N$ and $1\leq j\leq 2g$. Furthermore, the $\Z/4N\Z$-grading of these elements are given by:
	\begin{equation*}
		\deg(\epsilon)=4d\hspace{1 cm} \deg(\aleph_r)=-2r\hspace{1cm} \deg(o_r^j)=-2r+1
		\hspace{1cm} \deg(\rho_{r})=-2r+2
	\end{equation*}
\end{cor}
\begin{proof}
	The first part is an immediate consequence of Proposition \ref{VgdN-vector-space}. The second part can be also verified easily using Remark \ref{Z/4N-grading-II}.
\end{proof}
Fix $S: H^*(\Nnd) \to  \mathbb A^{N}_{g}$ as in the proof of Proposition \ref{VgdN-vector-space}. Then we can use the isomorphism $\Phi \circ S$ to pull back the ring structure, the paring and the $\Z/4N\Z$-grading of $\VgdN$ into $H^*(\Nnd)[u]/(u^N-1)$. Suppose the new multiplication map and the paring on $H^*(\Nnd)[u]/(u^N-1)$ are also denoted by $m$ and $\langle\,,\,\rangle$. We also fix the cohomological $\Z$-grading on $H^*(\Nnd)[u]/(u^N-1)$ where we set $\deg(u)=0$. The cohomological and the $\Z/4N\Z$-gradings differ by a sign after collapsing into $\Z/4\Z$-gradings. In the proof of Proposition \ref{VgdN-vector-space}, we show that for any element $p$ of degree $i$ in $H^*(\Nnd)[u]/(u^N-1)$, there is an element $q$ of degree $(N^2-1)(2g-2)-i$ with $\langle p,q\rangle \neq 0$. In particular, $\langle\,,\,\rangle$ defines a non-degenerate pairing.

Suppose $p_1$ and $p_2$ are two elements of degree $i_1$ and $i_2$ in $H^*(\Nnd)[u]/(u^N-1)$. Then the product $m(p_1,p_2)$ consists of terms in various gradings. However, there are some constraints on the degrees of these terms. Firstly, they all have the same $\Z/4\Z$-grading because the multiplication map is graded with respect to the $\Z/4N\Z$-grading on $\VgdN$. Moreover, an argument similar to that of Proposition \ref{VgdN-vector-space} shows that the pairing of $m(p_1,p_2)-p_1 \cup p_2$ and any element of $H^*(\Nnd)[u]/(u^N-1)$ with cohomological degree less than or equal to $(N^2-1)(2g-2)-i_1-i_2$ is zero. Therefore, the degree of the terms in $m(p_1,p_2)$ are at most $i_1+i_2$ and the term with the maximal degree is $p_1\cup p_2$. Therefore, the product $m$ is a {\it deformation} of the cup product. We summarize these properties of $\VgdN$ in the following proposition:

\begin{prop} \label{non-deg-pairing}
	The pairing $\langle\,,\,\rangle$ on $\VgdN$ is non-degenerate and the product $m:\VgdN\times \VgdN \to \VgdN$ is a deformation of the cup product,
	preserving the $\Z/4\Z$-grading.
\end{prop}

Multiplication with the elements of $\VgdN$, constructed in Corollary \ref{VgdN-generators}, defines a series of operators on $\VgdN$. We will use the same notation to denote these operators. The operator $\epsilon$ can be alternatively described as the cobordism map associated to the triple $([0,1]\times Y_g, [0,1]\times \gamma_{g,d}\cup \Sigma,1)$. Similarly the remaining operators are cobordism maps associated to triples $([0,1]\times Y_g, [0,1]\times \gamma_{g,d},z)$ for appropriate choices of $z$. The operators $\epsilon$, $\aleph_r$ and $\rho_r$ commute with each other and $o_r^j$. However, $o_{r}^{j}$ and $o_{r'}^{j'}$ anti-commute with each other.

In the special case that $g=1$, the moduli space $\Nnd(\Sigma)$ consists of only one point. Therefore, $\mathbb V_{1,d}^N$ has $N$ generators with exactly one generator $\alpha_i$ in degree $4i$ with respect to $\Z/4N\Z$-grading. In fact, the (non-perturbed) Chern-Simons functional associated to the $N$-admissible pair $(Y_1,\gamma_{1,d})$ has irreducible and non-degenerate critical points (cf. \cite{K:higher}). The operator $\epsilon$ maps $\alpha_i$ to $\alpha_{i+d}$.  The following proposition characterizes the action of some of the point classes in the case that $d=1$:
\begin{prop} \label{relations-T^3}
	The operators $\aleph_i: \mathbb V_{1,1}^N \to \mathbb V_{1,1}^N$ satisfy the following identities:
	\begin{equation}
		\aleph_2=N\epsilon^{-1} \hspace{2cm} \aleph_{2i-1}=0
	\end{equation}
\end{prop}

\begin{proof}
	The second identity can be verified easily, because $\deg( \aleph_{2i-1})$ is not divisible by 4. The first claim is proved in \cite{X:Thesis}.
	Since $\aleph_2$ and $\epsilon$ commute with each other and $\deg(\aleph_2)=-4$, we can conclude that $\aleph_2=c\epsilon^{-1}$. Therefore, we just need to show that $\tr(\aleph_2\circ \epsilon)=N^2$. Using Proposition \ref{trace},
	this can be reduced to show that:
	\begin{equation*}
		\rD^N_{T^4,T^2\times\{pt\}\cup \{pt\}\times T^2}(\aleph_2)=N^3
	\end{equation*}
	which is established in \cite{X:Thesis} using properties of stable bundles on abelian varieties.
\end{proof}
In \cite{DX:in-prep}, we will give another proof of this proposition which is independent of the results of \cite{X:Thesis}.

\subsection{Fukaya-Floer Homology} \label{FFH}
Suppose $X_1$ and $X_2$ are 4-manifolds with $\partial X_1=Y$, $\partial X_2=\overline Y$, and $X$ is given by gluing these manifolds along their boundaries. Suppose also a 1-cycle $\gamma\subset Y$ and 2-cycles $w_i\subset X_i$ are chosen such that $\partial w_1=\gamma$ and $\partial w_2=\overline \gamma$. The cycles $w_1$ and $w_2$ can be glued to each other to form a 2-cycle $w\subset X$. We also assume that $(Y,\gamma)$ is an $N$-admissible pair. Then Floer homology for this $N$-admissible pair provides a useful device to relate $\U(N)$-polynomial invariants of the following form:
\begin{equation} \label{decomposable}
	\hspace{3cm}\rD_{X,w}^N(z_1\cdot z_2)\hspace{1cm} z_i\in \A(X_i)^{\otimes (N-1)}
\end{equation}
to the relative invariants associated to $(X_1,w_1,z_1)$ and $(X_2,w_2,z_2)$ (cf. Proposition \ref{gluing}). As we shall see in the next section, this decomposition theorem for polynomial invariants is a useful tool for computational purposes. However, all polynomial invariants of $(X,w)$ do not have the form in \eqref{decomposable}. There are homology classes $\Gamma \in H_2(X)$ such that $\Gamma$ is not the sum of the elements in $H_2(X_1)$ and $H_2(X_2)$. Then, for example, $\rD_{X,w}^N(\Gamma_{(2)}^i\Gamma_{(3)}^j)$ cannot be expressed in terms of the relative invariants. In this section, we introduce an extension of Floer homology which admits relative invariants for such polynomial invariants. This extension of Floer homology was already constructed in \cite{Fuk:FF,DB:FF} for $N=2$ and is known as {\it Fukaya-Floer homology}.

Our extension of Floer homology is a module over a ring $R_N$. Let $R_{N,j}$ be the polynomial ring over the variables $t_{k,i}$, for $2\leq k \leq N$ and $1\leq i \leq j$, modulo the relations $t_{k,i}^2=0$:
\begin{equation*}
	R_{N,j}:=\C[t_{k,i}; 2\leq k \leq N, 1\leq i\leq j]/(t_{k,i}^2)
\end{equation*}
For $j\geq l\geq 0$, there is an obvious map from $R_{N,j}$ to $R_{N,l}$ which maps $t_{k,i}$ to $t_{k,i}$ when $i \leq l$ and maps $t_{k,i}$ to $0$ when $i > l$. The ring $R_N$ is defined to be the inverse limit of this system of rings. For example, for each $2\leq k \leq N$, we have an element of $R_N$ as follows:
\begin{equation*}
	s_k:= \sum_{i=1}^{\infty} t_{k,i}
\end{equation*}

The ring of polynomials $\C[\![t_2,\cdots, t_N]\!]$ can be regarded as a subring of $R_N$ by mapping $t_k$ to $s_k \in R_N$. Under this inclusion we have:
\begin{equation*}
	\frac{t_k^l}{l!} \to \sum_{\substack{S \subset \N\\|S|=l}} \prod_{i \in S} t_{k,i}
\end{equation*}
The full-version of Fukaya-Floer homology for $N=2$, whose construction is sketched in \cite{DB:FF}, is expected to be a module over $\C[\![t_2]\!]$. However, our construction is slightly different and we obtain a module over the ring $R_N$ for general $N$. This is partly because the definition of polynomial invariants for higher rank bundles slightly differs from the classical definition of $\U(2)$-polynomial invariants. Another reason is that even for $N=2$, the authors were not able to avoid some analytical difficulties related to the non-compactness of the moduli space of ASD connections and construct a ring over $\C[\![t_2]\!]$.

Consider the $N$-admissible pair $(Y,\gamma)$, and let $L=(l_2, \dots, l_N)$ be an $(N-1)$-tuple of the elements of $H_1(Y)$. Fukaya-Floer homology associates to $(Y,\gamma, L)$ an $R_{N,j}$-module $\mathbb I^{N,j}_*(Y,\gamma,L)$, for each non-negative integer $j$, and an $R_{N,j}$-module homomorphism $f_j^k:\mathbb I^{N,j}_*(Y,\gamma,L) \to \mathbb I^{N,k}_*(Y,\gamma,L)$, for each pair of non-negative integers $j\geq k\geq 0$. If $j\geq k\geq l\geq 0$, then we require that $f_k^l \circ f_j^k=f_j^l$. Fukaya-Floer homology of $(Y,\gamma, L)$ is the inverse limit of the inverse system $(\{\mathbb I^{N,j}_*(Y,\gamma,L)\}_j,\{f_j^k\}_{j\geq k})$. The $R_{N,j}$-module $\mathbb I^{N,j}_*(Y,\gamma,L)$ is the homology of a chain complex $(\mathfrak C_*^{N,j}(Y,\gamma,L), d_{N,j})$. In fact, we can arrange for a perturbation $\CS_{\pi_j}$ of the Chern-Simons functional of the $N$-admissible pair $(Y,\gamma)$ such that $\mathfrak C_*^{N,j}(Y,\gamma,L)=\mathfrak C_*^{\pi_j}(Y,\gamma)\otimes_\C R_{N,j}$.

The differential $d_{N,j}$ of the Fukaya-Floer chain complex has the following form:
\begin{equation} \label{widetilde d}
	d_{N,j}(\alpha)=
	\sum_{\substack{\overline S=(S_2,\cdots, S_N)\\p:\alpha\to \beta}}  h_{\overline S} (\alpha,\beta) (\prod_{ i\in S_k} t_{k,i} ) \beta
\end{equation}
where $S_k\subset \{1,\dots,j \}$ and the path $p$ is chosen such that the dimension of the moduli space $\mathcal M_p(\alpha,\beta)$ is equal to:
\begin{equation*}
	2|S_2|+4|S_3|+\dots+2(N-1)|S_N|+1.
\end{equation*}	
The constant term of the differential is equal to the differential $d$ of the Floer chain complex. That is to say, if we evaluate all variables $t_{k,i}$ at zero, then we recover $d$. The definition of the other terms in \eqref{widetilde d} are discussed in Subsection \ref{IIN}. We extend the Floer grading to $\widetilde {\mathfrak C}_*^{N,j}(Y,\gamma,L)$ by requiring that $\deg(t_{k,i})=2(k-1)$. Then the differential $d_{N,j}$ has degree $-1$.

Suppose $(X_1,w_1)$ is a pair of a 4-manifold and a 2-cycle which fills the $N$-admissible pair $(Y,\gamma)$. Suppose also $z_1\in \A(X_1)^{\otimes (N-1)}$ and $\Gamma^2$, $\dots$ $\Gamma^N$ are properly embedded surfaces in $X_1$ where $[\partial \Gamma^i]=l_i\subset H_1(Y)$. Then one can associate an element of $\IIN(Y,\gamma,L)$ to $(X_1,w_1,z_1,\Gamma^2,\dots, \Gamma^N)$ which is denoted by:
\begin{equation}\label{rel-element-IIN}
	{\rm D}^N_{X_1,w_1}(z_1\cdot \e^{\Gamma^2_{(2)}+\dots+\Gamma^N_{(N)}})
\end{equation}
The element in \eqref{rel-element-IIN} is given by a system of cycles ${\rm D}^{N,j}_{X_1,w_1}(z_1\cdot \e^{\Gamma^2_{(2)}+\dots+\Gamma^N_{(N)}})\in \mathfrak C_*^{N,j}(Y,\gamma,L)$. We have:
\[
  {\rm D}^{N,j}_{X_1,w_1}(z_1\cdot \e^{\Gamma^2_{(2)}+\dots+\Gamma^N_{(N)}})=
  \sum_{\substack{\overline S=(S_2,\cdots, S_N),\,\alpha}}  m_{\overline S} (\alpha)
  (\prod_{ i\in S_k} t_{k,i} ) \alpha
\]
for appropriate choices of complex numbers $m_{\overline S} (\alpha) $. Evaluating all the variables $t_{k,i}$ at zero produces a cycle in $\mathfrak C_*^{\pi_j}(Y,\gamma)$ which represents the relative invariant ${\rm D}^{N}_{X_1,w_1}(z_1)$.

Next, let $(X_2,w_2)$ be a cobordism from $(Y,\gamma)$ to the empty pair. Suppose $z_2\in \A(X_2)^{\otimes (N-1)}$ and $\Lambda_2$, $\dots$, $\Lambda_N$ are properly embedded surfaces in $X_2$ where $[\partial \Lambda^j]=-l_j$. In this case, there is an $R_N$-linear map from $\IIN(Y,\gamma,L)$ to $R_N$ associated to $(X_2,w_2,z_2, \Lambda^2,\dots, \Lambda^N)$, which is denoted by:
\begin{equation}\label{rel-funct-IIN}
	{\rm D}_N^{X_2,w_2}(z_2 \cdot \e^{\Lambda^2_{(2)}+\dots+\Lambda^N_{(N)}})
\end{equation}
The construction of the element \eqref{rel-element-IIN} and the functional \eqref{rel-funct-IIN} is given in Subsection \ref{IIN}. We can glue the 4-manifolds $(X_1,w_1)$ and $(X_2,w_2)$ to form a closed pair $(X_2\circ X_1,w_2\circ w_1)$. The embedded surfaces $\Gamma^j$ and $\Lambda^j$ can be also glued to each other to form a closed embedded surface $\Gamma^j\#\Lambda^j$. Then we have:
\begin{equation*}
	{\rm D}^N_{X_2\circ X_1,w_2\circ w_1}(z_1\cdot z_2\cdot \e^{(\Gamma^2\#\Lambda^2)_{(2)}+\dots+(\Gamma^N\#\Lambda^N)_{(N)}})=\hspace{4cm}
\end{equation*}
\begin{equation} \label{FF-gluing-thm}
	\hspace{4cm}={\rm D}_N^{X_2,w_2}(z_2\cdot \e^{\Lambda^2_{(2)}+\dots+\Lambda^N_{(N)}}) \circ
	{\rm D}^N_{X_1,w_1}(z_1\cdot \e^{\Gamma^2_{(2)}+\dots+\Gamma^N_{(N)}}).	
\end{equation}
This claim shall be proved as Proposition \ref{FF-gluing-thm-2} in Subsection \ref{IIN}. A priori, the right hand side of the above equality is an element of $R_N$. Part of the claim is that the right hand side belongs to $\C[\![t_2,,\dots,t_N]\!]\subset R_N$ and is equal to the given power series on the left hand side.

Fukaya-Floer homology is also functorial with respect to cobordisms of $N$-admissible pairs. Suppose $(W,w):(Y_0,\gamma_0) \to (Y_1,\gamma_1)$ is such a cobordism. For $2 \leq i \leq N$, suppose also $\Gamma^i$ is a properly embedded surface in $W$ such that $\Gamma^i\cap Y_j$ represents the homology class $l_i^j\in H_1(Y_j)$. For any $z\in \A(W)$, there is a homomorphism:
\begin{equation*}
	\IIN(W,w,z\e^{\Gamma^2_{(2)}+\dots+\Gamma^N_{(N)}}):\IIN(Y_0,\gamma_0,L_0) \to \IIN(Y_1,\gamma_1,L_1)
\end{equation*}
where $L_j=(l_2^j,\dots,l_N^j)$. This construction is functorial with respect to composition of cobordisms.

Suppose $(Y_g,\gamma_{g,d})$ is the $N$-admissible pair from Subsection \ref{HF-S1-S} and $L_g=(l_2,\dots,l_N)$ is the $(N-1)$-tuple of the elements of $H_1(Y_g)$ where $l_i$ is an $S^1$ fiber. In this article, the main example of Fukaya-Floer homology for us is $\IIN(Y_g,\gamma_{g,d}, L_g)$, which is denoted by $\bIgdN$. Analogous to the previous subsection, we can define a ring homomorphism $\widetilde \Phi: \A_{g}^N[u]/(u^N-1)  \to \bIgdN$ as follows:
\begin{equation*}
	\widetilde{\Phi}(u^iz):= {\rm D}^N_{\Delta_g,\delta_{g,d}+i \Sigma }(z\e^{D^2_{(2)}+\dots+D^2_{(N)}})
\end{equation*}
Recall that $D^2$ is a 2-dimensional disc and $\Delta_g= \Sigma \times D^2 $ and $\delta_{g,d}=\{x_1,\dots,x_d\}\times D^2$. Similarly, we can define a multiplication and a pairing on $\bIgdN$ by repeating the construction of Subsection \ref{HF-S1-S}.

\begin{prop} \label{bIgdN-prop}
	For any non-zero element $q \in \bIgdN$, there are $z \in \A_g^N$ and $1\leq i \leq N$ such that:
	\begin{equation*}
		\langle q,\rD_{\Delta_g,\delta_{g,d}+i \Sigma}^N(z\e^{D^2_{(2)}+\dots+D^2_{(N)}}) \rangle \neq 0.
	\end{equation*}
\end{prop}
\begin{proof}
	Suppose $q$ is given by the sequence $\{q_j\}_{j\geq 0}$ where
	$q_j\in \mathbb I_*^{N,j}(Y_g,\gamma_{g,d}, L_g)$. Since $q$ is non-zero, there is $j$ such that $q_j$ is non-zero. To verify
	the claim, it suffices to show that there are $z \in \A_g^N$ and $1\leq i \leq N$ such that:
	\[
	  \langle q_j,\rD_{\Delta_g,\delta_{g,d}+i \Sigma}^{N,j}(z\e^{D^2_{(2)}+\dots+D^2_{(N)}}) \rangle \neq 0.
	\]
	Suppose $(\mathfrak C_*^{\pi_j}(Y_g,\gamma_{g,d}),d)$ is a Floer chain complex for the pair $(Y_g,\gamma_{g,d})$ such that
	$\mathfrak C_*^{\pi_j}(Y_g,\gamma_{g,d})\otimes R_{N,j}$ can be used to define
	$\mathbb I_*^{N,j}(Y_g,\gamma_{g,d}, L_g)$. Suppose $q_j$ is represented by the following element of this complex:
	\[
	  \sum_{\overline S=(S_2,\dots,S_N)\in \mathcal I} b_{\overline S}(\prod_{i\in S_k}t_{k,i})\alpha_{\overline S}
	\]
	where $\mathcal I$ is a set consisting $(N-1)$-tuples of finite subsets of $\{1,\dots,j\}$,
	$b_{\overline S}$ is a non-zero complex number
	and $\alpha_{\overline S} \in \mathfrak C_*^{\pi_j}(Y_g,\gamma_{g,d})$.
	Suppose ${\overline S}_0=(S_2^0,\dots, S_N^0)\in \mathcal I$ is a minimal element with respect to the partial order on
	$\mathcal I$ induced by inclusion. Then $d\alpha_{{\overline S}_0}=0$, and by changing the representative if necessary,
	we can assume that $\alpha_{{\overline S}_0}$ is not a boundary in $\mathfrak C_*^{\pi_j}(Y_g,\gamma_{g,d})$.
	Therefore, by Corollary \ref{VgdN-generators} and Proposition \ref{non-deg-pairing},
	there are $z \in \A_g^N$ and $1\leq i \leq N$ such that:
	\begin{equation*}
		\langle\alpha_{\overline S_0} ,\rD_{\Delta_g,\delta_{g,d}+i\Sigma}^N(z) \rangle \neq 0.
	\end{equation*}
	This implies that the coefficient of $\prod_{2\leq k\leq N} \prod_{i\in S^0_k}t_{k,i}$ in the following pairing is non-zero:
	\[
	  \langle q_j,\rD_{\Delta_g,\delta_{g,d}+i\Sigma}^{N,j}(z\e^{D^2_{(2)}+\dots+D^2_{(N)}}) \rangle.
	\]
\end{proof}

\begin{example}\label{FFH-T}
	In the case that $g=1$, the Fukaya-Floer homology group of the triple $(Y_g,\gamma_{g,d}, L_g)$
	is equal to $R_N^{\oplus N}$. This can be easily seen from the construction of Fukaya Floer homology in Subsection \ref{IIN}.
	In fact, this $R_N$-module is freely generated by the critical points $\alpha_1$, $\dots$, $\alpha_N$ of the
	Chern-Simons functional of $(Y_1=T\times S^1,\gamma_{1,d})$.
	Consider the operator:
	\begin{equation} \label{tilde-epsilon}
	  \widetilde \epsilon:=\IIN([0,1] \times Y_1,[0,1] \times \gamma_{1,d}+T,
	  \e^{([0,1] \times \gamma)_{(2)}+\dots+([0,1] \times \gamma)_{(N)}})
	\end{equation}
	where $\gamma$ is an $S^1$-fiber of $Y_1$. This operator is of order $N$ and has degree $4d$. Moreover, $\alpha_i$, $\widetilde \epsilon(\alpha_i)$,
	$\dots$, $\widetilde \epsilon^{N-1}(\alpha_i)$ form a basis for $\IIN((Y_g,\gamma_{g,d}, L_g))$ for any $i$. In particular,
	the kernel of $\widetilde\epsilon-1$ is equal to
	$R_N\cdot(1+\widetilde \epsilon+\dots+\widetilde \epsilon^{N-1})(\alpha_i)$.
\end{example}

Next, we discuss a prototype for 4-manifolds with boundary $Y_g$ which are of interest to us. Suppose $(X_1,\Sigma)$ is a pair of a 4-manifold and an embedded surface of genus $g$ with self-intersection 0 such that a regular neighborhood of $\Sigma$ in $X_1$ is identified with $\Delta_g$. Suppose $(X_2,\Sigma)$ is another such pair. As the notation suggests, the embedded surfaces in $X_1$ and $X_2$ are identified with each other. Remove regular neighborhoods of $\Sigma$ in $X_1$ and $X_2$ to produce 4-manifolds whose boundaries are $Y_g$, and then glue the resulting two 4-manifolds along their common boundaries by the orientation-reversing diffeomorphism that maps $(z,x)\in S^1\times \Sigma$ to $(\bar z,x)$. This 4-manifold is denoted by $X_1\#_\Sigma X_2$, and is called the {\it fiber sum} of $X_1$ and $X_2$ along $\Sigma$. We will also write $X_i^\circ$ for the complement of a neighborhood of $\Sigma$ in $X_i$. Then $X_i^\circ$  can be also regarded as a subspace of $X_1\#_\Sigma X_2$.

Elements of $H_2(X_1)$ and $H_2(X_2)$ can be glued to each other to construct elements of $H_2(X_1 \#_\Sigma X_2)$. Suppose $\iota_i:H_2(X_i) \to \C$ denotes the map that computes the intersection number of an element of $H_2(X_i)$ with $\Sigma$. Suppose also $\mathcal K$ is the subspace of the elements $(\Gamma,\Lambda) \in H_2(X_1) \oplus H_2(X_2)$ such that $\iota_1(\Gamma)=\iota_2(\Lambda)$. Then there is a homomorphism $\#:\mathcal K \to H_2(X_1 \#_\Sigma X_2)$ with the property that:
\begin{equation} \label{sharp-property}
	j_1^\#(\Gamma\#\Lambda)=j_{1}^\circ(\Gamma) \hspace{2cm} j_2^\#(\Gamma\#\Lambda)=j_{2}^\circ(\Lambda)
\end{equation}
Here $j_i^\circ:H_2(X_1) \to H_2(X_i^\circ,\partial X_i^\circ)$ is the composition of the map from $H_2(X_i)$ to the relative homology $H_2(X_i,\Delta_g)$ and the excision isomorphism. To abbreviate our notation, from now on, we will write $\Gamma^\circ$ and $\Lambda^\circ$ for $j_{1}^\circ(\Gamma)$ and $j_{2}^\circ(\Lambda)$. The maps $j_i^\#:H_2(X_1 \#_\Sigma X_2) \to H_2(X_i^\circ, \partial X_i^\circ)$ are also defined similarly.

The homomorphism $\#$ is not uniquely defined and we proceed as follows to fix one such homomorphism.  Suppose $\Gamma$ and $\Lambda$ are integral homology classes. Then these homology classes can be represented by oriented embedded surfaces, which we denote with the same notation. We can assume that these surfaces are transversal to $\Sigma$, and intersect $\Sigma$ in the same set of points with the same signs. Then there is an obvious way to glue $\Gamma$ and $\Lambda$ and to produce an oriented embedded surface in $X_1 \#_\Sigma X_2$. The homology class $\Gamma\#\Lambda$ is defined to be the homology of the glued up surface. We apply this construction to an integral basis of $\mathcal K$ and extend it linearly.

We use Poincar\'e duality to define $\mathcal L \subseteq H^2(X_1)\oplus H^2(X_2)$, the counterpart of $\mathcal K$, and the gluing map $\#:\mathcal L\to H^2(X_1\#X_2)$. Suppose $(K,L)\in \mathcal L$ and $(\Gamma,\Lambda) \in \mathcal K$. Then we have the following equalities of the pairing of cohomology classes with homology classes:
\begin{equation*}
	(K\#L)[\Gamma\#\Lambda]=K[\Gamma]+L[\Lambda]
\end{equation*}
Similarly, we can glue two cycles $w_1\subset X_1$ and $w_2\subset X_2$ that intersect $\Sigma$ transversely in the same set of points with the same signs. The resulting 2-cycle in $X_1 \#_\Sigma X_2$ is denoted by $w_1\#w_2$. We will also write $w_i^\circ$ for the intersection $w_i\cap X_i^\circ$.

\begin{prop} \label{fiber-sum-universal-formula}
	For $1\leq i \leq 4$, suppose $X_i$ is a 4-manifold and $T$ is an embedded surface of genus one in $X_i$. Suppose also $w_i\subset X_i$ is a 2-cycle such that $w_i \cdot T$ is coprime to $N$. For each $2\leq l \leq N$,
	suppose also $\Gamma_i^l$ is an element of $H_2(X_i)$ such that $\Gamma_i^j \cdot T=1$. For $1\leq i,k \leq 4$, let $D_{i,k}$ be the following element of $\C[\![t_2,\dots,t_N]\!]$:
	\begin{equation*}
		\sum_{j=1}^N \rD^N_{X_i\#_TX_k,w_i\# w_k+jT}({\re}^{(\Gamma_i^2\# \Gamma_k^2)_{(2)}+\dots+(\Gamma_i^N\# \Gamma_k^N)_{(N)}}).
	\end{equation*}
 	Then:
	\begin{equation*}
		D_{1,2}D_{3,4}=D_{1,4}D_{3,2}.
	\end{equation*}
\end{prop}
\begin{proof}
	The following elements lie in the kernel of the operator $\widetilde \epsilon$ in \eqref{tilde-epsilon}:
	\begin{equation*}
		D_i:=\sum_{j=1}^N \rD^N_{X_i^\circ ,w_i^\circ +jT}({\re}^{(\Gamma_i^2)^\circ_{(2)}+\dots+(\Gamma_i^N)^\circ_{(N)}}).
	\end{equation*}
	Moreover, Identity \eqref{FF-gluing-thm} implies that:
	\begin{equation} \label{D-iD-k}
		D_{i,k}=\frac{1}{N}\langle D_i, D_k \rangle.
	\end{equation}	
	Because $D_i\in \ker(\widetilde \epsilon-1)$, the claim is a consequence of the description of
	$\ker(\widetilde \epsilon-1)$ in Example \ref{FFH-T}.
\end{proof}

\subsection{An $\SU(3)$--instanton Floer Homology for $\Sigma(2,3,23)$}\label{sigma-23}

In Subsection \ref{IN}, Floer homology is defined for an $N$-admissible pair $(Y,\gamma)$. Then the computation of the polynomial invariants for a closed pair $(X,w)$ which can be decomposed along a copy of $(Y,\gamma)$ can be reduced to computing relative invariants for each component of $X\backslash Y$ (Proposition \ref{gluing}). One wishes to extend the definition of Floer homology so that it can be used in studying polynomial invariants of a pair $(X,w)$ which is decomposed along a non-admissible pair. However, there is little known in this direction even when $N=2$. For $N=2$, the most satisfactory answer is provided in the case that  $Y$ is an integral homology sphere and $\gamma$ is empty \cite{Fl:I} which will be denoted by ${\rm I}^2_*(Y)$. The main reason that one can define ${\rm I}^2_*(Y)$ for an integral homology sphere is that the only reducible flat connection on $Y$ is the non-degenerate trivial connection. In order to extend ${\rm I}^N_*(Y)$ to higher values of $N$, one would face more complicated reducible connections. Due to this complication, there are some difficulties in extending the definition of Floer homology of integral homology spheres to higher values of $N$. In this section, we make a modest progress in this direction and define ${\rm I}^3_*(Y)$ for the Brieskorn homology sphere $\Sigma(2,3,23)$. Meanwhile, we compute some of the gauge theoretical invariants for flat connections on $\Sigma(2,3,23)$.

Suppose the positive integers $a_1$, $a_2$ and $a_3$ are pairwise coprime, and $\Sigma(a_1,a_2,a_3)$ is the associated Brieskorn sphere:
\begin{equation*}
	\Sigma(a_1,a_2,a_3):=\{(z_1,z_2,z_3)\in \C^3\mid z_1^{a_1}+z_2^{a_2}+z_3^{a_3}=0,\,|z_1|^2+|z_2|^2+|z_3|^2=1\}
\end{equation*}
This 3-manifold is an integral homology sphere. There is an $S^1$-action on this 3-manifold where:
\begin{equation*}
	\e^{2\pi \bi \theta} \cdot (z_1,z_2,z_3):=(\e^{2\pi \bi a_2 a_3\theta}z_1,\e^{2\pi \bi a_1 a_3\theta}z_2,\e^{2\pi \bi a_1 a_2\theta}z_3)
\end{equation*}
This action turns $\Sigma(a_1,a_2,a_3)$ into a Seifert fiber space over $S^2$ with 3 exceptional orbits. Complex conjugation on $\C^3$ induces a diffeomorphism of $\Sigma(2,3,23)$ which will be also called complex conjugation. There is also a standard presentation of the fundamental group of this 3-manifold given as:
\begin{equation} \label{fund-gp}
	\pi_1(\BS)=\langle x_1,\,x_2,\,x_3,\,h \mid [h,x_i]=1,\,x_i^{a_i}h^{\beta_i}=1,\, x_1x_2x_3=1\rangle.
\end{equation}
where $\beta_i$ is given by the following identity:
\begin{equation*}
	\frac{\beta_1}{a_1}+\frac{\beta_2}{a_2}+\frac{\beta_3}{a_3}=\frac{1}{a}
\end{equation*}
with $a=a_1a_2a_3$. The central element $h$ in \eqref{fund-gp} is represented by a generic fiber of the Seifert fibration.

Suppose $W$ is the space $(\BS \times D^2)/{S^1} $ where the $S^1$-action is the product of the Seifert action on $\BS$ and the standard action on $D^2$. Alternatively, $W$ is the mapping cylinder of the fibration of $\BS$ over $S^2$. This space is an orbifold and has three singular points. A neighborhood of these singular points are diffeomorphic to cones on the lens spaces $L(a_i,\beta_i)$.
Thus removing neighborhoods of the orbifold points produces a cobordism $W_0$ from the union of three lens spaces $L(a_i,\beta_i)$ to $\BS$. We will denote the union of the lens spaces with $Y$.

The fundamental group of $W_0$ is equal to $\pi_1(\BS)/\langle h\rangle$ and the inclusion of $\BS$ in $W_0$ induces the quotient map at the level of fundamental groups. Moreover, the induced map from the fundamental group of $L(a_i,\beta_i)$ to $\pi_1(\BS)/\langle h\rangle$ maps the standard generator of $\pi_1(L(a_i,\beta_i))$ to $x_i$. The description of the fundamental group implies that the first homology of $W_0$ is trivial. We also  have the following short exact sequence:
\[\xymatrix{
  0\ar[r]
        & H^2(W_0,\partial W_0,\Z)\cong \Z  \ar[r]^{\hspace{2mm}\iota}
        &H^2(W_0,\Z)\cong \Z  \ar[r]
        & H^2(\partial W_0,\Z) \cong \Z/a\Z\ar[r] & 0\\
}\]
where the map $\iota$ is multiplication by $a$. The self-intersection pairing, defined on the image of $\iota$, maps a generator of $\im(\iota)$ to $-a$. In particular, $b^+(W_0)$ is equal to 0.

The space $L:= \Sigma(a_1,a_2,a_3)\times D^2$ defines an orbifold $S^1$-bundle on $W$. In particular, the restriction of $L$ to $W_0$, denoted by $L_0$, is a smooth $S^1$-bundle.  The first Chern class of $L_0$ is a generator of $H^2(W_0,\Z)$. The restriction of this Chern class to $L(a_i,\beta_i)$ is equal to $\beta_i$ times the standard generator of $H^2(L(a_i,\beta_i),\Z)$. In particular, for any complex line bundle on $\partial W_0$, there is $k$ such that the restriction of $L_0^k$ to the boundary is isomorphic to the given line bundle.

The rational cohomology class induced by $c_1(L_0)$ can be lifted to $H^2(W_0,\partial W_0,\Q)$. In particular, $c_1(L_0)^2$ is well-defined and is equal to $-\frac{1}{a}$. For our purposes, we also fix a connection $B_0$ on $L_0$ whose restrictions to a neighborhood of $\partial W_0$ is the pull-back of a flat connection on $\partial W_0$. In particular, we can assume that the restriction of this connection in a regular neighborhood of $\BS$ is the trivial connection.

Suppose $\alpha$ is a flat $\SU(N)$-connection on $\BS$ whose holonomy around the fiber is central. Therefore, $\alpha$ can be extended as a flat ${\rm PU}(N)$--connection $A$ to $W_0$. The holonomy of $\alpha$ induces a conjugacy class $r_i$ in ${\rm PU}(N)$ corresponding to the loop $x_i$. The class $r_i$ has order $a_i$ and determines a flat ${\rm PU}(N)$-connection on $\pi_1(L(a_i,\beta_i))$ which matches the restriction of $A$ to $L(a_i,\beta_i)$. This connection will be also denoted by $r_i$. As it was pointed in \cite{FS:HFSF}, the connection $A$ can be used to compute some of the gauge theoretical invariants of $\alpha$:

\begin{prop} \label{ind-W0}
	Let $\alpha$ and $A$ be given as above. Then $\rho_{\ad(\alpha)}$ is equal to:
	\begin{equation} \label{rho-inv}
			(T_\alpha+1-N^2)+\sum_{i=1}^3\rho_{\ad(r_i)}(L(a_i,\beta_i))
	\end{equation}
	where $T_\alpha$ is the number of trivial summands in the irreducible decomposition of the representation associated to $\ad_{\alpha}$.
\end{prop}
Using the calculations of \cite{APS:II} for lens spaces, Formula \eqref{rho-inv} allows us to compute the $\rho$-invariant of any connection as above.
\begin{proof}
	According to \cite{APS:II}:
	\begin{equation*}
		\rho_{\ad(\alpha)}-\sum_{i=1}^3\rho_{\ad(r_i)}(L(a_i,\beta_i))=(N^2-1)\sigma(W_0)-\sigma_A(W_0)
	\end{equation*}	
	where $\sigma_A(W_0)$ denotes the signature of the twisted cohomology group $H^2(W_0;\ad(A))$ determined by the flat
	${\rm PU}(N)$-connection $\ad(A)$.
	This twisted cohomology group can be decomposed according to the irreducible decomposition of $A$ (or equivalently $\alpha$).
	The argument of \cite[Lemma 2.6]{FS:HFSF} shows that the contribution of the non-trivial summands is equal to 0. On the
	other hand, each trivial summand contributes $-1$ to the sum, because $\sigma(W_0)=-1$.
\end{proof}

The underlying ${\rm PU}(N)$-bundle of the connection $A$ on $W_0$ can be lifted to a $\U(N)$-bundle $E$. There are also non-negative integers $k_1$, $\dots$, $k_N$ such that the restriction of the connection:
\begin{equation} \label{comp-red}
	B_0^{k_1} \oplus \dots \oplus B_0^{k_N}.
\end{equation}
to $Y$, the union of the three lens spaces, is equal to the restriction of $A$ to $Y$. In particular, $L_0^{k_1} \oplus \dots \oplus L_0^{k_N}$, the underlying $\U(N)$-bundle of \eqref{comp-red}, has the same restriction as $E$ on $Y$. Therefore, the determinant of $L_0^{k_1} \oplus \dots \oplus L_0^{k_N}$ is equal to $\det(E) \otimes L_0^{ak}$ for an appropriate integer $k$. After replacing $k_1$ with $k_1+ak$, the new bundle $L_0^{k_1} \oplus \dots \oplus L_0^{k_N}$ has the same determinant as $E$ and the connection in \eqref{comp-red} and $A$ give rise to the same connection after restriction to $Y$.

Since $A$ and the connection in \eqref{comp-red} agree on boundary components except on $\BS$, where $A$ gives the connection $\alpha$ and $B_0$ restricts to the trivial connection, we can conclude that:
\begin{align} \label{CS-W0}
	\CS(\alpha)&=\mathcal E(B_0^{k_1} \oplus \dots \oplus B_0^{k_N})-\mathcal E(A)\nonumber\\
			&=-\frac{1}{2N}\sum_{i<j}(k_i-k_j)^2c_1(L)^2\nonumber\\
			&=\frac{1}{2N}\sum_{i<j}\frac{(k_i-k_j)^2}{a}
\end{align}
Therefore, this gives us a strategy to compute the Chern-Simons functional of the flat connection $\alpha$.

Now we focus on $N=3$ and the Brieskorn homology sphere $\Sigma(2,3,23)$. Suppose $\alpha$ is an irreducible connection on $\Sigma(2,3,23)$. The irreducibility assumption implies that the holonomy of $\alpha$ along the generic fiber of the Seifert fibration is central. In fact, this central element is equal to the identity \cite{Bod:SU(3)-Bris}. It is also shown in \cite{Bod:SU(3)-Bris} that there are 44 such irreducible representations which are non-degenerate. One can use the method of \cite{Bod:SU(3)-Bris} to find the conjugacy classes of holonomies corresponding to the elements $x_1$, $x_2$ and $x_3$ of $\pi_1(\Sigma(2,3,23))$. Any irreducible flat connection on $\Sigma(2,3,23)$ is characterized by its holonomies along $x_3$. The possible conjugacy classes for $x_3$ are listed in Table~\ref{table-irr-hol}. On the other hand, the holonomies along $x_1$ and $x_2$ are conjugate to the diagonal matrices ${\rm Diag}(1,-1,-1)$ and ${\rm Diag}(1,\zeta,\zeta^2)$ where $\zeta=\e^{2\pi \bi/3}$.
Knowledge of these conjugacy classes allows us to apply Proposition \ref{ind-W0} and Identity \ref{CS-W0} to compute the $\rho$-invariants, the Chern-Simons functional and hence the degrees of irreducible flat connections on $\Sigma(2,3,23)$:

\begin{prop}
	There are ten irreducible flat connections of degree 0, five irreducible flat connections of degree 2, nine irreducible flat connections of degree 4,
	five irreducible flat connections of degree 6, nine irreducible flat connections of degree 8 and six irreducible flat connections of degree 10. There is not any
	irreducible flat connections of odd degree. The Chern-Simons functional and the $\rho$-invariants of these irreducible flat connections can be found in
	Tables~\ref{inv-irr-23-1} and \ref{inv-irr-23-2}.
\end{prop}

Any reducible flat connection on $\Sigma(2,3,23)$ is either trivial or $\SU(2)$-reducible, because this 3-manifold is an integral homology sphere. In particular, non-trivial reducible $\SU(3)$-connections on $\Sigma(2,3,23)$ can be regarded as irreducible $\SU(2)$-connections. The results and the methods of \cite{FS:HFSF} can be utilized to study such connections. The holonomy of any non-trivial flat $\SU(2)$-connection along the fiber of the Seifert fibration is the central element $-id$.  The holonomies of this connection along $x_1$ and $x_2$ are respectively conjugate to ${\rm Diag}(\bi,-\bi)$ and ${\rm Diag}(\e^{\pi \bi/3},\e^{-\pi \bi/3})$. As in the $\SU(3)$-case, an irreducible flat $\SU(2)$-connection on $\Sigma(2,3,23)$ is determined by the conjugacy class of its holonomy along $x_3$. The eigenvalues of holonomy term along $x_3$ are equal to $\e^{2\pi \bi k/23}$ and $\e^{-2\pi \bi k/23}$ for $2\leq k \leq 9$:

\begin{prop}[\cite{FS:HFSF}]\label{Sigma-red}
	There are 8 irreducible $\SU(2)$-connection on $\Sigma(2,3,23)$. For each degree $2i+1 \in \Z/8\Z$,
	there are exactly two such irreducible connections and there is no irreducible connection of even degree.
\end{prop}
Proposition \ref{ind-W0} and Identity \eqref{CS-W0} give a strategy to compute the Chern-Simons functional and the $\rho$-invariants of irreducible $\SU(2)$-connections. These computations can be used to verify the second part of the above proposition.

We also need to compute the degrees of flat $\SU(2)$-connections when they are regarded as $\SU(3)$-connections. To distinguish between these connections, we will write $\widetilde \alpha$ for the $\SU(3)$-connection associated to an $\SU(2)$-connection $\alpha$. The values of the Chern-Simons functional of $\alpha$ and $\widetilde \alpha$ are equal to each other. However, the $\rho$-invariants of these two connections are different because $\ad_{\alpha}$ and $\ad_{\widetilde \alpha}$ define two different representation of the fundamental group. We cannot use Proposition \ref{ind-W0} to compute the $\rho$-invariant of $\widetilde \alpha$ because this connection does not extend to the cobordism $W_0$. In  \cite{BHKK:rho-C}, cut and paste methods have been utilized to compute the difference $\rho_{\ad_{\widetilde \alpha}}-\rho_{\ad_{\alpha}}$ for $\SU(2)$-flat connections on a family of homology spheres which include $\Sigma(2,3,23)$. In particular, the following proposition can be extracted from \cite{BHKK:rho-C}. The claim about the non-degeneracy of reducible flat connections on $\Sigma(2,3,23)$ in the following proposition is also proved in \cite{BHKK:rho-C}. For more details about reducible flat $\SU(3)$-connections on $\Sigma(2,3,23)$ see Table \ref{inv-red-23}.

\begin{prop}
	The eight non-trivial reducible flat connections on $\Sigma(2,3,23)$ are non-degenerate. The $\SU(3)$-degrees of these connections are
	given as follows: there are one connection of degree 1, one connection of degree 3, one connection of degree 5,
	 two connections of degree 7,  two connections of degree 9 and one connection of degree 11.
\end{prop}

Define $\I_*^3(\Sigma(2,3,23))$, the Floer homology of $\Sigma(2,3,23)$, to be the complex vector space generated by the irreducible flat connections on $\Sigma(2,3,23)$. The significance of this vector space for us is a gluing theorem for the 4-manifolds which split along a copy of $\Sigma(2,3,23)$. The proof of the following theorem will be given in Subsection \ref{long-neck}:

\begin{prop} \label{gluing-23}
	Let $X_1$ and $X_2$ be two 4-manifolds such that
	$b^+(X_1),b^+(X_2)\geq 1$, $\partial X_1=\Sigma(2,3,23)$ and $\partial X_2=\overline{\Sigma(2,3,23)}$.
	Let $w_i$ be a closed 2-cycle in $X_i$ and $z_i \in \A(X_i)^{\otimes 2}$.
	Assume that:
	\[
	  \deg(z_1) \equiv -4w_1^2-4(\chi(x_1)+\sigma(X_1))+4 \hspace{1cm} ({\rm mod} 12).
	\]
	Then there are:
	\begin{equation*}
		\rD_{X_1,w_1}^3(z_1)\in \I_4^3(\Sigma(2,3,23))\hspace{1cm} \rD^{X_2,w_2}_3(z_2):\I_*^3(\Sigma(2,3,23)) \to \C
	\end{equation*}
	such that:
	\begin{equation*}
		\rD^{X_2,w_2}_3(z_2) \circ \rD_{X_1,w_1}^3(z_1)=\rD_{X_1\#_{\Sigma(2,3,23)}X_2,w_1\cup w_2}^3(z_1\cdot z_2)
	\end{equation*}
	Moreover, $\rD^{X_2,w_2}_3(z_2)$ is non-zero only on the terms of the following degree:
	\begin{equation*}
		4w_2^2+4(\chi(X_2)+\sigma(X_2))-4+\deg(z_2).
	\end{equation*}
\end{prop}

\section{Computing Polynomial Invariants}
In this section, which is the heart of this paper, we firstly compute the $\U(3)$-polynomial invariants of $E(n)$. The rank $2$ invariants of elliptic surfaces were partially computed in \cite{F:el-SU(2),MF:com-sur,Fr:el-SO(3)} using algebro-geometric techniques. A complete calculation of the $\U(2)$-polynomial invariants of elliptic surfaces are given in \cite{KM:Rec,FS:str-thm,Li:elliptic-surf}. In \cite{FS:str-thm} and \cite{Li:elliptic-surf}, {\it Gompf decomposition} of elliptic surfaces play a key role in computing the invariants. In the introduction, we also recalled a construction of elliptic surfaces which give rise to a decomposition of $E(n)$ into fiber sum of $n$ copies of $E(1)$. This decomposition of elliptic surfaces can be also exploited to compute some of the $\U(2)$-polynomial invariants \cite{MM:el,D:F-AG}. Our method for computing the $\U(3)$-invariants of elliptic surfaces uses the Gompf decomposition, the fiber-sum decomposition, and the rich group of the symmetries of elliptic surfaces. In the last subsection of this section, we also give a general gluing theorem for fiber-sums. Similar results for $\U(2)$-invariants are proved in \cite{Mun:Gl}. The proof of the rank 3 case follows the same strategy as in \cite{Mun:Gl}.

In the next two sections, we mainly focus on polynomial invariants and instanton Floer homology in the case $N=3$. Therefore, we shall drop 3 from our notation ${\rm I}_*^3$, ${\rm D}_{X,w}^3$, et cetera, when it does not make any confusion. Because we are working with an odd value of $N$, there is not any sign ambiguity in the definition of ${\rm I}_*^3(W,w,z)$, ${\rm D}_{X,w}^3(z)$, and we do not need to fix a homology orientations for the underlying 4-manifold.

\subsection{Structure of the Invariants of $E(n)$}\label{E(n)-str}

A construction of the elliptic surface $E(n)$ was reviewed in the introduction. The simplest 4-manifold in this family, $E(1)$, can be also constructed by blowing up the projective plane $\C{\bf P}^{2}$ at the nine intersection points of two generic cubics. The pencil of cubics generated by the two cubics determines an elliptic fibration of $E(1)$. The nine exceptional divisors give rise to sections of this elliptic fibrations, which are embedded sphere with self-intersection $-1$. The manifold $E(n)$ is fiber sums of $n$ copies of $E(1)$ along the fibers of the elliptic fibration. The fibration of $E(1)$ induces an elliptic fibration for $E(n)$, and we will write $f$ for a regular fiber of this fibration. By taking the connected sums of the exceptional sections of $E(1)$, we can form nine disjoint embedded spheres in $E(n)$. These are sections of the elliptic fibration of $E(n)$ and have self-intersection $-n$. We fix one of these sections and we will denote it by $\sigma$. When it does not make any confusion, we will use the same notation to denote the homology and the cohomology classes associated to $f$ and $\sigma$.

We can assume that there is a cusp fiber in the elliptic fibration of $E(1)$ by choosing appropriate cubics. This gives a cusp fiber $f_0$ in the fibration of $E(n)$. A regular neighborhood of $\sigma \cup f_0$ is a 4-manifold with boundary $\Sigma(2,3,6n-1)$ which is called the {\it Gompf nucleus} and is denoted by $G(n)$ \cite{G:nuc-elliptic}. The intersection form of $G(n)$ is given as follows:
\begin{equation*}
	\left[
	\begin{array}{cc}
		0&1\\
		1&-n
	\end{array}	
	\right ]
\end{equation*}
The complement of $G(n)$ in $E(n)$, denoted by $B(2,3,6n-1)$, is a {\it Milnor fiber} and its intersection form is given by:
\begin{equation}\label{int-form-Mil-fib}
	n(-E_8) \oplus 2(n-1)	\left[
	\begin{array}{cc}
		0&1\\
		1&-2
	\end{array}	
	\right ].
\end{equation}
Here each $-E_8$ summand has a basis of embedded spheres with self-intersection -2 which intersect each other according to $-E_8$. The $i^{\rm th}$ summand of the second type in \eqref{int-form-Mil-fib} has a basis of an embedded torus $g_i$ with self-intersection 0 and an embedded $(-2)$-sphere $\tau_i$ where $g_i$ and $\tau_i$ intersect each other positively at one point \cite{GS:Kir-cal}.

The 4-manifold $E(2)$, which is a $K3$ surface, plays a special role in this family. For example, $E(2)$ enjoys a rich group of symmetries. As a manifestation of this fact, we have the following proposition:
\begin{lemma} \label{orbits-line-bundle-E(2)}
	Suppose $e$ and $e'$ are two non-zero elements of $H^2(E(2),\Z)$
	with $e\cup e \equiv e'\cup e'$ mod 3.
	Then there is an orientation preserving diffeomorphism $\Phi$ of $E(2)$
	such that $\Phi^*(e)\equiv e'$ mod 3.
	In particular, the action of ${\rm Diff}(E(2))$ on $H^2(E(2), \Z/3\Z)$ has four orbits.
\end{lemma}
\begin{proof}
	Let $e_1, e_2 \in H^2(E(2),\Z)$ be chosen such that $e_i \cup e_j $ is zero
	when $i=j$, and is equal to 1 when $i \neq j$.
	Because the action of ${\rm Diff}(E(2))$ on the primitive cohomology
	classes of a fixed self-intersection is transitive, there is an element $\Psi$ of ${\rm Diff}(E(2))$ such that $\Psi^*(e)=m(e_1+ne_2)$ for $m,n \in \Z$. Therefore
	$\Psi^*(e)$, mod 3,
	 is equal to one of the following elements:
	\begin{equation} \label{E(2)-Z/3Z}
		0 \hspace{1cm} e_1 \hspace{1cm} e_1-e_2 \hspace{1cm} e_1+e_2.
	\end{equation}
	The non-zero classes in \eqref{E(2)-Z/3Z}  can be distinguished from each other by their self-intersection. Therefore, there is $\Psi' \in {\rm Diff}(E(2))$ such that
	$\Psi^*(e)\equiv \Psi'^*(e')$ mod 3.
\end{proof}

For $1\leq i \leq 3$, suppose $w_i$ is the 2-cycle in $E(2)$ given by $\sigma-(i-1)f$. Then $w_i^2 \equiv i$ mod 3. Therefore, these 2-cycles and the empty cycle give representative for the orbits of the action of ${\rm Diff}(E(2))$ on $H^2(E(2), \Z/3\Z)$. Alternatively, let $w_i'$ be the union of $i$ elements of the nine disjoint spheres of self-intersection $-2$ in $E(2)$ that were constructed above. If $i$ is a positive integer and $i \equiv j$ mod 3 with  $1\leq j \leq 3$, then define $w_i:=w_j$ and $w_i':=w_j'$. We also define $w_0$ and $w_0'$ to be the empty cycles.

The group of diffeomorphisms of $E(n)$, for $n\geq 3$, is more constrained than that of $E(2)$. For example, any diffeomorphism of $E(n)$ maps the homology class $f$ to $\pm f$. Therefore, we cannot expect that the analogue of Lemma \ref{orbits-line-bundle-E(2)} holds for an arbitrary $n$. However, $E(n)$ still has a big diffeomorphism group and we can prove the following weakened version of Lemma \ref{orbits-line-bundle-E(2)}:

\begin{lemma} \label{orbits-line-bundle-E(n)}
	Suppose the integers $n\geq 3$ and $1\leq i \leq 3$ are given and $u \in H^2(B(2,3,6n-1),\Z)$, satisfying $u\cup u \equiv i$ mod 3, is fixed.
	Suppose also $e\in H^2(E(n),\Z)$ is such that $e\cup f \equiv 0$ mod 3 and $e\cup e \equiv i$ mod 3. Then there is an element $\Phi$ of ${\rm Diff}(E(n))$ such that:
	 \begin{equation} \label{Phi-e}
	 	\Phi^*(e) \equiv kf+u \hspace{.3cm} \text{or}\hspace{.3cm} kf \hspace{1cm}\text{ mod 3}
	\end{equation}
	where $k=0$ or $1$. Moreover, the map induced by $\Phi$ on $H^2(G(n),\Z)$ is $\pm id$.
	In particular, the action of ${\rm Diff}(E(n))$ on $\langle f\rangle^\perp$ in $H^2(E(n), \Z/3\Z)$ has eight orbits.
\end{lemma}

In the statement of Lemma, we regard $u$ as an element of $H^2(E(n),\Z)$ using the inclusion of $B(2,3,6n-1)$ in $E(n)$. Note also that the second case in \eqref{Phi-e} holds only if $i=3$.

\begin{proof}
	The element $e$ can be written as the sum:
	\begin{equation*}
		rf+s \sigma+ v
	\end{equation*}
	where $v \in H^2(B(2,3,6n-1),\Z)\subset H^2(E(n),\Z)$. Because $e \cup f \equiv 0$ mod 3, $s$ is divisible by 3. There is also a diffeomorphism of $E(n)$ that maps $f$
	to $-f$ and $\sigma$ to $-\sigma$ \cite[Lemma 3.7]{G:nuc-elliptic}.
	After applying this diffeomorphism if it is necessary, we can assume $e \equiv kf+ v$ mod 3 where $k=0$ or $1$.
	Suppose $\SO(H^2(E(n),\Z))$ denotes an element of the special orthogonal group of the lattice
	$H^2(E(n),\Z)$ with respect to the intersection bi-linear form. According to \cite[Proposition 3.3]{GHS:ab-or-gp}, there is a spinor norm one element of
	$\SO(H^2(E(n),\Z))$ that fixes $f$, $\sigma$, and maps $v$ to an element of the form $m(\tau_1+ng_1)$.
	This element can be realized by a diffeomorphism of $E(n)$\cite{MF:com-sur}. This can be used to verify the claim as in
	Lemma \ref{orbits-line-bundle-E(2)}.
\end{proof}
The eight orbits in Lemma \ref{orbits-line-bundle-E(n)} can be represented by $w_{k,0}=kf$ and $w_{k,l}=kf+\tau_1-(l-1)g_1$ for $k=0,1$ and $l=1,2,3$. Alternatively, we can use $w_{k,l}'=kf+\tau_1+\dots +\tau_l$. If $l$ is a positive integer and $l \equiv j$ mod 3 with  $1\leq j \leq 3$, then let $w_{k,l}=w_{k,j}$ and $w_{k,l}'=w_{k,j}'$.

Consider the $\U(3)$-polynomial invariant $\rD^3_{E(n),w}(\Gamma_{(2)}^i\Lambda_{(3)}^{j})$ for the homology classes $\Gamma$ and $\Lambda$. This polynomial is invariant with respect to the action of the diffeomorphism group of $E(n)$ on $(w,\Gamma,\Lambda)$. Therefore, we can use Lemmas \ref{orbits-line-bundle-E(2)} and \ref{orbits-line-bundle-E(n)} to focus on a smaller subset of possible values for $w$. Since changing $w$ mod 3 would not change the polynomial invariant and $H^2(E(n),\Z/3\Z)$ is finite, the polynomial $\rD^3_{E(n),w}(\Gamma_{(2)}^i\Lambda_{(3)}^{j})$ is invariant with respect to the action of a finite index subgroup of ${\rm Diff}(E(n))$ on $(\Gamma,\Lambda)$. This action of the diffeomorphism group factors through the action of the algebraic group $\Or(H_2(E(n)))$. Here the orthogonal group is defined using the intersection form on the complex vector space $H_2(E(n))$. Suppose also $\SO(H_2(E(n));f)$ is the subgroup of $\Or(H_2(E(n)))$ consisting of the orthogonal transformations that map $f\in H_2(E(n))$ to itself and have determinant $1$.  As another manifestation of the big diffeomorphism group of $E(n)$, it is shown in \cite{MF:com-sur} that the image of any finite index subgroup of ${\rm Diff}(E(n))$ in $\Or(H_2(E(n)))$ contains an algebraically dense subgroup of $\SO(H_2(E(n));f)$. Therefore, the polynomial $\rD^3_{E(n),w}(\Gamma_{(2)}^i\Lambda_{(3)}^{j})$ is invariant with respect to the action of $\SO(H_2(E(n));f)$ on $(\Gamma,\Lambda)$. In the case that $n=2$, one can even replace $\SO(H_2(E(n));f)$ with $\SO(H_2(E(n)))$.

\begin{lemma} \label{SO-inv-pol}
	Suppose $(V,Q)$ is a pair of a complex vector space of dimension greater than 2 and a quadratic form.
	Suppose also $\tau_1$ and $\tau_2$ are two vectors orthogonal to each other such that
	$Q(\tau_1)$ and $Q(\tau_2)$ are non-zero.
	Suppose $P:V \oplus V \to \C$ is a bi-homogeneous polynomial of bi-degree $(d_1,d_2)$ that is invariant with respect to the diagonal action of $\SO(V)$ on
	$V \oplus V$. Then $P$ is determined by its value on $W_0\oplus W_1$ where $W_0$ is the span of $\tau_1$, and $W_1$ is the span of $\tau_1$ and $\tau_2$.
	Moreover, $P$ has the following form:
	\begin{equation*}
		P(x,y)=\sum_{\substack{i,j,k\geq 0 \\2i+j=d_1\\j+2k=d_2}} c_{i,j,k} Q(x)^i Q(x,y)^j Q(y)^k
	\end{equation*}
	for appropriate constants $c_{i,j,k}\in \C$.
\end{lemma}
In our application, $V$ will be $H_2(E(2))$, and $\tau_1$, $\tau_2$ will be two disjoint embedded spheres in $E(2)$.
\begin{proof}
	Suppose $(\alpha,\beta) \in V \oplus V$ are such that $Q(\alpha)\neq 0$. Then the vector $\beta$ can be written as the sum $\beta_0+\beta_1$ where $\beta_0$
	is a multiple of $\alpha$, and $\beta_1$ is orthogonal to $\alpha$. We also assume that $Q(\beta_1) \neq 0$.
	It is straightforward to find an element of $\SO(V)$ which maps $(\alpha,\beta)$ to an element of the form $W_0\oplus W_1$.
	The set of vectors $(\alpha,\beta)$ as above is also dense in $V \oplus V$. Therefore, $P$ is determined by its values on $W_0\oplus W_1$.
	
	By evaluating $P$ on elements of $W_0\oplus W_1$, we have:
	\begin{equation*}
		P(\lambda \tau_1, \mu_1 \tau_1+\mu_2 \tau_2)=\sum m_{a,b,c} \lambda^{a}\mu_1^{b}\mu_2^{c}.
	\end{equation*}
	There is an element of $\SO(V)$ which maps $\tau_1$ to itself (respectively, $-\tau_1$) and $\tau_2$ to $-\tau_2$ (respectively, itself). Therefore, $m_{a,b,c}$
	is non-zero only if $a+b$ and $c$ are even. This implies that there are constants $c_{i,j,k}$ such that:
	\begin{equation*}
		P(x,y)=\sum_{\substack{2i+j=d_1\\j+2k=d_2}} c_{i,j,k} Q(x)^i Q(x,y)^j Q(y)^k
	\end{equation*}
	for $(x,y)\in W_0\oplus W_1$, where $j$ and $k$ are non-negative integers.
	This implies the second part of the lemma because both sides of the above equality is invariant with respect to the action of $\SO(V)$.
	Arguing as in  \cite[Chapter 6, Lemma 2.2]{MF:com-sur}, we can also assume that $i$ only takes non-negative integers.
\end{proof}

\begin{lemma} \label{SO-k-inv}
	Suppose $(V,Q)$ is a pair of a complex vector space of dimension greater than 3 and a quadratic form. Fix a vector $f\in V$ with $Q(f)=0$ and let
	the vectors $k$, $\tau$ be chosen such that $Q(k,f)$ and $Q(\tau)$ are non-zero, and $Q(\tau,k)=Q(\tau,f)=0$.
	Suppose also $P:V \oplus V \to \C$ is a polynomial that is invariant with respect to the diagonal action of $\SO(V;f)$.
	Then $P(x,y)$ is determined by its values on $W_0 \oplus W_1$ where $W_0$ is the span of the vectors $f$ and $k$, and $W_1$ is the span of $f$, $k$ and $\tau$.
 \end{lemma}
In our application, $V$, $f$, $k$ and $\tau$ will be $H_2(E(n))$, $f$, $\sigma$ and an embedded sphere in $B(2,3,6n-1)$, respectively.
\begin{proof}
	The proof is similar to that of the first part of Lemma \ref{SO-inv-pol}.
	For a given $(\alpha,\beta)\in V \oplus V$, assume $Q(\alpha,f)\neq 0$.
	Then the vector $\beta$ can be uniquely written as $\beta_0+\beta_1$ where $\beta_0\in Span(f,\alpha)$, and $\beta_1$
	is orthogonal to $Span(f,\alpha)$. As another assumption, we require that $Q(\beta_1) \neq 0$. These assumptions hold for a dense subset of $V \oplus V$.
	It can be easily checked that there is an element of $\SO(V;k)$ which maps $(\alpha,\beta)$ to
	$W_0 \oplus W_1$. Therefore, $P|_{W_0 \oplus W_1}$ determines $P$ on $V \oplus V$.
\end{proof}

\subsection{Invariants of $E(2)$} \label{E(2)}
In this section, we study the $\U(3)$- polynomial invariants of $E(2)$ and a 2-cycle $w$. Lemma \ref{orbits-line-bundle-E(n)} shows that we can assume that the 2-cycle $w$ is either empty or $w_i$, for $1\leq i \leq 3$. Equivalently, we can replace $w_i$ with $w_i'$. Throughout this subsection, we will follow the same notation as in the previous part to denote the surfaces $\sigma$, $\tau_i$ and $g_i$ embedded in $E(2)$.

\begin{prop}\label{a3=0,a2/3=shift}
	The invariant ${\rm D}_{E(2),w_i}$ satisfies the following identities:
	\begin{equation} \label{simple-type-w-i}
		{\rm D}_{E(2),w_i}((\frac{a_2}{3})^jz)={\rm D}_{E(2),w_{i+j}}(z) \hspace{1cm} {\rm D}_{E(2),w_i}(a_3z)=0
	\end{equation}
	for $z \in \mathbb A(E(2))^{\otimes 2}$.
	In particular, for $1\leq i \leq 3$, $E(2)$ has $w_i$-simple type and $\widehat {\rm D}_{E(2),w_i}(\e^{\Gamma_{(2)}+\Lambda_{(3)}})$ is independent of the choice of $i$.
\end{prop}

\begin{proof}
	Suppose $N(f)$ is a neighborhood of a regular fiber and $X$ is the complement of $N(f)$. We also identify the boundary of $X$ with $Y_1= S^1 \times f$.
	The 4-manifold $X$ contains a copy of $B(2,3,11)$. Therefore, we can find two disjoint embedded spheres $\tau_1$ and $\tau_2$ of
	self-intersection $-2$ in $X$.
	Let $S$ be the subspace of $H_2(X)$ spanned by the vectors $\tau_1$ and $\tau_2$, and $\A(S)$ be the sub-algebra $\Sym^*(H_0(X) \oplus V)$ of $\mathbb A(X)$.
	Then $w_i$ can be decomposed as $w\#w'$ where $w$ (respectively, $w'$) is a
	2-cycle in $X$ (respectively, $N(f)$), and $w$, $w'$ intersect $Y_1$ in $\gamma:=S^1 \times \{{\rm pt}\}$. For $z \in \A(S)^{\otimes 2}$,
	Proposition \ref{relations-T^3} with the aid of
	the functoriality properties discussed in Subsection \ref{IN} implies that:
	\begin{align}
		\rD_{E(2),w_i}(a_2^ja_3^kz)=& \rD^{N(f),w'}(1) \circ \rI_*(Y_1 \times [0,1], \gamma \times [0,1],a_2^ja_3^k)\circ \rD_{X,w}(z) \nonumber\\
		=&3^j0^k\rD^{N(f),w'}(1) \circ  \rI_*(Y_1 \times [0,1], \gamma \times [0,1]-jf,1)\circ \rD_{X,w}(z) \nonumber\\
		=&3^j0^k\rD_{E(2),w_i-jf}(z). \nonumber
	\end{align}
	This verifies \eqref{simple-type-w-i} for the case $z\ \in \A(S)^{\otimes 2}$. Using Lemma \ref{SO-inv-pol}, the same claim holds for general $z$,
	and as a result $\widehat {\rm D}_{E(2),w_i}(\e^{\Gamma_{(2)}+\Lambda_{(3)}})$ is equal to:
	\begin{equation*}
		{\rm D}_{E(2),w_{1}}(\e^{\Gamma_{(2)}+\Lambda_{(3)}})+{\rm D}_{E(2),w_{2}}(\e^{\Gamma_{(2)}+\Lambda_{(3)}})+
		{\rm D}_{E(2),w_{3}}(\e^{\Gamma_{(2)}+\Lambda_{(3)}}).
	\end{equation*}
	In particular, $\widehat {\rm D}_{E(2),w_i}(\e^{\Gamma_{(2)}+\Lambda_{(3)}})$ does not depend on $i$.
	\end{proof}
A similar argument, using Lemma \ref{SO-k-inv} instead of Lemma \ref{SO-inv-pol}, proves the following analogous statement for $E(n)$:
\begin{prop}\label{wk,l-ind}
	The polynomial invariants ${\rm D}_{E(n),w_{k,l}}$, for $1\leq l \leq 3$, satisfies the following identities:
	\begin{equation} \label{simple-type-E(n)}
		{\rm D}_{E(n),w_{k,l}}((\frac{a_2}{3})^jz)={\rm D}_{E(2),w_{k,l+j}}(z) \hspace{1cm} {\rm D}_{E(n),w_{k,l}}(a_3z)=0
	\end{equation}
	for $z \in \A(E(n))^{\otimes 2}$. In particular, for $1\leq l \leq 3$, $E(n)$ has $w_{k,l}$-simple type and $\widehat {\rm D}_{E(n),w_{k,l}}(\e^{\Gamma_{(2)}+\Lambda_{(3)}})$ is independent of the choice of $l$.
\end{prop}
\begin{prop} \label{K3-inv-w-i}
	For $1\leq i \leq 3$, we have:
	\begin{equation} \label{E(2)-inv-w-i}
		\widehat \rD_{E(2),{w_i}}({\rm e}^{\Gamma_{(2)}+\Lambda_{(3)}})={\rm e}^{\frac{Q(\Gamma)}{2}-Q(\Lambda)}
	\end{equation}
\end{prop}
The two sides of \eqref{E(2)-inv-w-i} are power series in $t_2$ and $t_3$ where the coefficients of $t_2^it_3^j$ are bi-homogeneous polynomials on $H_2(E(2))\oplus H_2(E(2))$ of degree $(i,j)$. Identity \eqref{E(2)-inv} means for each choice of $(i,j)$ these coefficients are equal to each other.
\begin{proof}

	Using Lemma \ref{SO-inv-pol}, we can find a power series $g(r,s,t)\in \C[\![r,s,t]\!]$ such that:
	\begin{equation*}
		\widehat \rD_{E(2),w_i}({\rm e}^{\Gamma_{(2)}+\Lambda_{(3)}})\cdot {\rm e}^{-\frac{Q(\Gamma)}{2}+Q(\Lambda)}= g(Q(\Gamma),Q(\Gamma,\Lambda),Q(\Lambda))
	\end{equation*}
	The constant term of $g$ is equal to 1 \cite{K:higher}.
	Suppose $\tau$ is an embedded sphere in $E(2)$ of self-intersection -2 such that $w_i\cdot \tau=0$.
	Identities $(C_1)$, $(C_2)$ and $(C_3)$ of Subsection \ref{neg-spheres} for $\tau$ imply that:
	\begin{equation*}
		\frac{\partial g}{\partial r}(r,s,t)=0\hspace{.5cm}4\frac{\partial^2 g}{\partial r^2}(r,s,t)-\frac{\partial g}{\partial t}(r,s,t)=0
		\hspace{.5cm}2\frac{\partial^2 g}{\partial s\partial r}(r,s,t)+\frac{\partial g}{\partial s}(r,s,t)=0
	\end{equation*}
	Therefore, $g$ is equal to the constant power series $1$.
\end{proof}

Suppose $X$ is the blowup of $E(2)$ and $w$ is the 2-cycle $f$ in $X$. If $E$ is the exceptional sphere in $X$, then Corollary \ref{blowup} can be used to compute the invariants of $(X,w)$:
\begin{equation*}
	\hDXw({\rm e}^{\Gamma_{(2)}+\Lambda_{(3)}})=\frac{1}{3}{\rm e}^{\frac{Q(\Gamma)}{2}-Q(\Lambda)} (\cosh(\sqrt{3}E\cdot \Gamma)+ 2\cos(\sqrt{3}E\cdot\Lambda)).
\end{equation*}

The homology class $\sigma+E$ can be represented by an embedded (-3)-sphere $\sigma'$ in $X$.  Fix $\Gamma, \Lambda \in H_2(X)$  which are orthogonal to $\sigma'$. Then the above formula can be used to show:\begin{equation*}
	\hDXw((-\frac{3}{2}\sigma_{(3)}'-\frac{3}{2}\sigma_{(2)}'^2-a_2){\rm e}^{\Gamma_{(2)}+\Lambda_{(3)}})= \hspace{6cm}
\end{equation*}
\begin{equation}\label{aa}	
	 \hspace{1.5cm}{\rm e}^{\frac{Q(\Gamma)}{2}-Q(\Lambda)}
	(-\sqrt{3} \sin (\sqrt{3}E\cdot \Lambda) + \cos (\sqrt{3}E\cdot \Lambda) -\cosh(\sqrt{3}E\cdot \Gamma)).
\end{equation}
 By another application of Theorem \ref{blowup} and Remark \ref{S--E}, $\widehat \rD_{X,{w-\sigma'}}({\rm e}^{\Gamma_{(2)}+\Lambda_{(3)}})$ is equal to:
\begin{equation}\label{bb}
	\frac{1}{3}{\rm e}^{\frac{Q(\Gamma)}{2}-Q(\Lambda)}
	(\cosh(\sqrt{3}E\cdot \Gamma) -\cos(\sqrt{3}E\cdot \Lambda) + \sqrt{3}\sin(\sqrt{3}E\cdot \Lambda)).
\end{equation}
Using Proposition \ref{-3-identities} and comparing these two identities, we can find the undetermined constant $c$ in Proposition \ref{-3-identities}:

\begin{prop}
	The constant $c$ is equal to $-3$.
\end{prop}

Now we are ready to complete the computation of the invariants of $E(2)$:
\begin{theorem} \label{K3-inv}
	Suppose $w$ is a 2-cycle in $E(2)$. Then $E(2)$ has $w$-simple type and the $\U(3)$-series of $E(2)$ is given by the following formula:
	\begin{equation} \label{E(2)-inv}
		\widehat \rD_{E(2),w}({\rm e}^{\Gamma_{(2)}+\Lambda_{(3)}})={\rm e}^{\frac{Q(\Gamma)}{2}-Q(\Lambda)}
	\end{equation}
\end{theorem}

\begin{proof}
	In the light of Lemma \ref{orbits-line-bundle-E(2)} and Porposition \ref{K3-inv-w-i},
	it suffices to consider only the empty 2-cycle $w_0$. Let $\sigma'$ be the embedded (-3)-sphere in
	$E(2)\# \cp$ constructed above.
	Consider the 2-cycle $w':=\sigma$ of $E(2)\# \cp$ and the element $z:=(\sigma_{(2)}-2E_{(2)})^2z'$ of
	$\A(E(2)\#\cp)^{\otimes 2}$ where $z'\in \A(\langle\sigma\rangle^{\perp})^{\otimes 2}\bigcap \A(E(2))^{\otimes 2}$.
	By Proposition \ref{blowup-inital}, the following invariant of $E(2)\# \cp$ is equal to $4{\rm D}_{E(2),w_0}(z')$:
	\begin{equation}\label{E(2)-cp-inv}
		{\rm D}_{E(2)\# \cp,w'-\sigma'}((\sigma_{(2)}-2E_{(2)})^2z').
	\end{equation}
	Moreover, the first identity of Proposition \ref{-3-identities} can be used to show \eqref{E(2)-cp-inv} is
	equal to ${\rm D}_{E(2)\# \cp,w'}(z'')$ for an appropriate choice of $z'' \in \A(E(2)\#\cp)^{\otimes 2}$. Replacing $w'$ with $w'':=\sigma+\tau_1+g_1$ shows that:
	\begin{equation*}
		4{\rm D}_{E(2),\tau_1+g_1}(z')={\rm D}_{E(2)\# \cp,w''}(z'').
	\end{equation*}
	Since $w'\cdot w' \equiv w''\cdot w''$ mod 3, the left hand side of the above identity is equal to ${\rm D}_{E(2)\# \cp,w'}(z'')$ by the blowup formula.
	Therefore, we can deduce that:
	\begin{equation} \label{0=w3}
		{\rm D}_{E(2),w_0}(z')={\rm D}_{E(2),\tau_1+g_1}(z')
	\end{equation}
	for $z'\in \A(\langle\sigma\rangle^{\perp})^{\otimes 2}$. As a consequence of Lemma \ref{SO-inv-pol}, Identity \eqref{0=w3} holds for any choice of $z'$.
	In particular, $E(2)$ has simple type with respect to $w_0$ and \eqref{E(2)-inv} holds for this choice of $w$.
\end{proof}

\begin{prop} \label{non-nuc-inv}
	Suppose $\Gamma,\Lambda \in H_2(B(2,3,6n-1)) \subset H_2(E(n))$ and $\Gamma',\Lambda' \in H_2(G(n))\subset H_2(E(n))$. Then:
	\begin{equation} \label{E(n)-inv-wkl}
		\widehat {\rm D}_{E(n),w_{k,l}}(\e^{(\Gamma+\Gamma')_{(2)}+(\Lambda+\Lambda')_{(3)}})=
		{\rm e}^{\frac{Q(\Gamma)}{2}-Q(\Lambda)}\widehat {\rm D}_{E(n),w_{k,3}}(\e^{(\Gamma')_{(2)}+(\Lambda')_{(3)}})
	\end{equation}
\end{prop}
\begin{proof}
	The group of orthogonal transformations $\SO(H_2(B(2,3,6n-1)),Q)$, regarded as a subgroup of $\SO(H_2(E(n)))$,
	acts as identity on the series $\widehat {\rm D}_{E(n),w}(\e^{\Gamma_{(2)}+\Lambda_{(3)}})$. This fact can be combined with the argument
	of Proposition \ref{K3-inv-w-i} to verify \eqref{E(n)-inv-wkl} for $1\leq l \leq 3$. To finish the proof, it suffices to show that $\widehat {\rm D}_{E(n),w_{k,0}}(\e^{\Gamma_{(2)}+\Lambda_{(3)}})$ is equal to
	$\widehat {\rm D}_{E(n),w_{k,3}}(\e^{\Gamma_{(2)}+\Lambda_{(3)}})$. This also can be achieved with the method of the proof of
	 Theorem \ref{K3-inv}.
\end{proof}
By Proposition \ref{wk,l-ind}, we already know that $E(n)$ has $w_{k,l}$-simple type for $1\leq l \leq 3$. The above proof can be used to show that $E(n)$ has $w_{k,0}$-simple type, too.

\subsection{Invariants of $E(3)$} \label{E(3)-invariants}
In this section, we study the $\U(3)$-polynomial invariants of $E(3)$ up to a constant and a sign ambiguity. The following is Theorem \ref{En} form the introduction:
\begin{theorem} \label{E(3)-inv}
	The 4-manifold $E(3)$ has simple type. There are also real numbers $\hbar_1$ and $\hbar_2$ such that for any 2-cycle $w$ in $E(3)$ and
	$\Gamma,\Lambda \in H_2(E(3))$, the series
	$\widehat \rD_{E(3),w}({\e}^{\Gamma_{(2)}+\Lambda_{(3)}})$ is equal to :
	\begin{equation*}
		{\rm e}^{\frac{Q(\Gamma)}{2}-Q(\Lambda)}(\hbar_1 \cosh(\sqrt{3}f\cdot \Gamma)-
		2\hbar_2\cos(-\frac{2\pi}{3}w\cdot f+\sqrt 3 f \cdot \Lambda)).
	\end{equation*}	
	Furthermore, $\hbar_1+\hbar_2=\pm1$ for an appropriate choice of the sign.
\end{theorem}

In \cite{DX:in-prep}, it is shown that the constants $\hbar_1=\frac{2}{3}$ and $\hbar_2=\frac{1}{3}$. However, we do not need the exact values of these constant in this paper. Later, we only use the fact that $\hbar_1$ and $\hbar_2$ are non-zero. To abbreviate our notation, define:
\begin{equation} \label{G}
	G(\Gamma,\Lambda,j):=\hbar_1 \cosh(\sqrt{3}f\cdot \Gamma)-2\hbar_2\cos(-\frac{2\pi}{3}j+\sqrt 3 f \cdot \Lambda).
\end{equation}	
\begin{prop} \label{E(3)-inv-w.fneq0}
	Suppose $w$ is a 2-cycle in $E(3)$ with $w\cdot f \nequiv 0$ mod 3. Then:
	\begin{equation} \label{E(3)-simple-type}
		\rD_{E(3),w}((\frac{a_2}{3})^jz)=\rD_{E(3),w-jf}(z) \hspace{1cm} \rD_{E(3),w}(a_3z)=0
	\end{equation}	
	for $z\in \A(E(3))^{\otimes 2}$. In particular, $E(3)$ has $w$-simple type.
	Furthermore, there is a power series $g \in\Q[\![t_2,t_3]\!]$ such that for $\Gamma,\,\Lambda \in H_2(E(3))$:
	\begin{equation*}
		\widehat \rD_{E(3),w}(\e^{\Gamma_{(2)}+\Lambda_{(3)}})={\rm e}^{\frac{Q(\Gamma)}{2}-Q(\Lambda)}g(\Gamma\cdot f,\Lambda\cdot f)
	\end{equation*}
	when $w\cdot f\equiv 1$ mod 3, and
	\begin{equation*}
		\widehat \rD_{E(3),w}(\e^{\Gamma_{(2)}+\Lambda_{(3)}})={\rm e}^{\frac{Q(\Gamma)}{2}-Q(\Lambda)}g(-\Gamma\cdot f,-\Lambda\cdot f)
	\end{equation*}
	when $w\cdot f\equiv 2$ mod 3. Furthermore, $g$ is even with respect to the variable $t_2$ and its constant term is equal to $\pm1$.
\end{prop}

\begin{proof}
	The 4-manifold $E(3)$ is given as the fiber sum $E(2)\#_f E(1)$. In this proof, let $\sigma_n$
	denote a section of the elliptic fibration of $E(n)$, which is a sphere of self-intersection $-n$.
	We can assume $\sigma_3=\sigma_2\#\sigma_1$. Firstly, consider the case $w\cdot f \equiv 1$ mod 3.
	Arguing as in Lemma Lemma \ref{orbits-line-bundle-E(n)}, we can assume that $w=w_1\# \sigma_1$
	where $w_1$ is a 2-cycle in $E(2)$ with $w_1 \cdot f =1$.
	Suppose $\Gamma_0$ and $\Lambda_0$ are two elements of $H_2(E(2))$ such that $\Gamma_0\cdot f =\Lambda_0 \cdot f=1$.
	Then Proposition \ref{fiber-sum-universal-formula} for $X_1=E(2)$, $X_2=E(1)$ and
	$X_3=X_4=S^2 \times f$ implies that:
	\begin{equation*}
		p(t_2,t_3)\sum_{0\leq j \leq 2}\rD_{E(2)\#_{f} E(1), w+jf}(\e^{(\Gamma_0\#\sigma_1)_{(2)}+(\Lambda_0\#\sigma_1)_{(3)}})=\hspace{2cm}
	\end{equation*}
	\begin{equation*}	
		\hspace{2cm}=q(t_2,t_3)\sum_{0\leq j \leq 2}\rD_{E(2)\#_{f} S^2\times f, w_2+jf}(\e^{(\Gamma_0)_{(2)}+(\Lambda_0)_{(3)}})
	\end{equation*}
	where:
	\begin{equation*}
		p(t_2,t_3)=\sum_{0\leq j \leq 2}\rD_{S^2 \times f, S^2\times \{{\rm pt}\} +jf}(\e^{\Delta_{(2)}+\Delta_{(3)}})
	\end{equation*}
	and
	\begin{equation*}
		q(t_2,t_3)=\sum_{0\leq j \leq 2}\rD_{E(1), w_1+jf}(\e^{(\sigma_1)_{(2)}+(\sigma_1)_{(3)}}).
	\end{equation*}	
	Note that $b^+(S^2 \times f)=b^+(E(1))=1$ and we use the invariants with respect to the metrics that have long necks along $f$ in the above identities.
	The power series $p(t_2,t_3)$ is invertible, because $p(0,0)= 1$.
	Therefore, we can conclude there is a power series $g(t_2,t_3)$ such that:
	\begin{equation} \label{E(3)-e-inv}
		\sum_{0\leq j \leq 2}\rD_{E(3), w+jf}(\e^{\Gamma_{(2)}+\Lambda_{(3)}})={\rm e}^{\frac{Q(\Gamma)}{2}-Q(\Lambda)}g(t_2,t_3)
	\end{equation}	
	with $\Gamma=\Gamma_0\#\sigma_1$, $\Lambda=\Lambda_0\#\sigma_1$ for $\Gamma_0$ and $\Lambda_0$ given as above.
	Then Lemma \ref{SO-k-inv} implies that \eqref{E(3)-e-inv} holds for arbitrary $\Gamma$ and $\Lambda$ with
	$\Gamma \cdot f=\Lambda\cdot f=1$. By Proposition \ref{a3=0,a2/3=shift}, we can use a similar argument as above to show:
	\begin{equation*}
		\sum_{0\leq j \leq 2}\rD_{E(3), w+jf}(P(a_2,a_3)\e^{\Gamma_{(2)}+\Lambda_{(3)}})=P(3,0){\rm e}^{\frac{Q(\Gamma)}{2}-Q(\Lambda)}g(t_2,t_3)
	\end{equation*}	
	In particular, this shows that:
	\begin{equation*}
		\rD_{E(3), w}((\frac{a_2}{3})^j\e^{s\Gamma_{(2)}+t\Lambda_{(3)}})=
		\rD_{E(3), w-jf}(\e^{\Gamma_{(2)}+\Lambda_{(3)}}),\hspace{.2cm}\rD_{E(3), w}(a_3\e^{s\Gamma_{(2)}+t\Lambda_{(3)}})=0
	\end{equation*}		
	The power series $g$ is even with respect to $t_2$, because $\rD(\e^{\Gamma_{(2)}+\Lambda_{(3)}})$ is even with respect to $t_2$. A similar application of
	Proposition \ref{fiber-sum-universal-formula} for $X_1=X_2=E(1)$ and $X_3=X_4=S^2 \times f$ shows that:
	\begin{equation*}
		\e^{-t_2^2+2t_3^2}p(t_2,t_3)=q(t_2,t_3)^2.
	\end{equation*}
	The constant term of the above equality and the identity $p(0,0)=1$ shows that the constant term of $g$ is equal to $\pm 1$.
	This fact completes the proof for the case that $w\cdot f \equiv1$ mod 3.
	Using a diffeomorphism of $E(3)$ which maps $f$ to $-f$, we can also treat the
	case that $w\cdot f \equiv2$ mod 3.
\end{proof}

In order to determine the power series $g$, let $\sigma$ and $\sigma'$ be two disjoint sections of the elliptic fibration of $E(3)$. Let also $w$ be chosen such that $w\cdot f = 1$ and $w \cdot \sigma=2$. Then:
\begin{equation*}\label{inv-E(3)--3spheres}
	\widehat \rD_{E(3),w}(\e^{(s\sigma+s'\sigma')_{(2)}+(r\sigma+r'\sigma')_{(3)}})={\rm e}^{-3t_2^2\frac{(s^2+s'^)}{2}+3t_3^2(r^2+r'^2)}g((s+s')t_2,(r+r')t_3)
\end{equation*}
By taking derivative with respect to $s$ and $r$, we can conclude that:
\begin{equation*}\label{inv-E(3)--3spheres}
	\widehat \rD_{E(3),w}(\sigma_{(2)}^i \sigma_{(3)}^j z)=h\frac{d^i}{ds^i}\frac{d^j}{dr^j}|_{s=t=0}\, {\rm e}^{-3\frac{t_2^2s^2}{2}+3t_3^2r^2}g((s+s')t_2,(r+r')t_3)
\end{equation*}
where $z=\e^{s'\sigma_{(2)}'+r'\sigma_{(3)}'}$ and $h={\rm e}^{-3\frac{t_2^2s'^2}{2}+3t_3^2r'^2}$. By applying these identities to the second formula in Proposition \ref{-3-identities}, we can conclude:
\begin{equation}\label{mixed-PDE-1}
	-\frac{1}{2} g_3+\frac{1}{2}g_{22}-\frac{1}{2}g=g\circ \tau
\end{equation}
where $g_{22}$ means the second derivative of the power series $g(t_2,t_3)$ with respect to $t_2$, and so on. Moreover, $\tau$ maps $(t_2,t_3)$ to $(-t_2,-t_3)$. We can use \eqref{mixed-PDE-1} to derive the following identity:
\begin{equation}\label{mixed-PDE-1-p}
	\frac{1}{2} (g\circ \tau)_3+\frac{1}{2}(g\circ \tau)_{22}-\frac{1}{2}g\circ \tau=g
\end{equation}
Replacing $g\circ \tau$ in \eqref{mixed-PDE-1-p} with the left hand side of \eqref{mixed-PDE-1} gives rise to the following PDE for $g$:
\begin{equation} \label{PDE-4-1}
	g_{2222}-g_{33}-2g_{22}-3g=0
\end{equation}

Next, let $w'$ be a 2-cycle with $w'\cdot f=1$ and $w'\cdot \sigma=0$ and consider $$\widehat \rD_{E(3),w'}(\e^{(s\sigma+s'\sigma')_{(2)}+(r\sigma+r'\sigma')_{(3)}})$$
instead.
With the same argument, the last part of Proposition \ref{-3-identities} implies that:
\begin{equation}\label{PDE-4-2}
	g_{2222}-6g_{22}+3g_{33}+9g=0 \hspace{1cm}g_{2223}-6g_{23}=0
\end{equation}
We can combine \eqref{PDE-4-1} and the first equation in \eqref{PDE-4-2} to come up with the following simpler PDE:
\begin{equation}\label{PDE-2}
	g_{33}-g_{22}+3g=0
\end{equation}
The second PDE in \eqref{PDE-4-2} and the fact that $g$ is even with respect to the variable $t_2$ imply that $g_{st}=p(t)\cosh(\sqrt 6 s)$. The equations \eqref{PDE-4-1} and \eqref{PDE-2} can be used to write two linear ordinary differential equations for $p$. It is straightforward to check that the only solution of these ODEs is $p(t)=0$. Therefore, the power series $g$ has the form $q_1(t)+q_2(s)$. Equation \eqref{PDE-2} can be used to find differential equations for $q_1$ and $q_2$. By solving these ODEs and using the fact that $g$ is even with respect to $t_2$, we can conclude:
\begin{equation} \label{g-gen-form}
	g(t_2,t_3)=a\cosh(\sqrt 3 t_2)+b\cos(\sqrt 3 t_3)+c\sin(\sqrt 3 t_3)
\end{equation}
If $g(0,0)=1$, then the initial value and \eqref{mixed-PDE-1} imply the following constraints on $a$, $b$ and $c$ which can be used to prove Theorem \ref{E(3)-inv} in the case $w\cdot f \nequiv 0$ mod 3:
\begin{equation*}
	a+b=1\hspace{1cm}a-\frac{1}{2}b-\frac{\sqrt{3}}{2}c=1.
\end{equation*}
A similar argument can be used in the case that $g(0,0)=-1$.

Next, let $w\cdot f \equiv 0$ mod 3. Using Lemma \ref{orbits-line-bundle-E(n)} and Proposition \ref{wk,l-ind}, it suffices to consider the case that $w=w_{k,1}$ for $k=0$ or $1$. The following proposition computes the invariants of $E(3)$ for this choice of 2-cycle. In this proof, we use the basis for the homology of $H_2(E(3),\Z)$ which is introduced in Subsection \ref{E(n)-str}:
\begin{prop}
	For any $(\Gamma,\Lambda)\in H_2(E(3)) \oplus H_2(E(3))$:
	\begin{equation} \label{E(3)-w.f=0}
		\widehat \rD_{E(3),w_{k,l}}({\e}^{\Gamma_{(2)}+\Lambda_{(3)}})=
		{\e}^{\frac{Q(\Gamma)}{2}-Q(\Lambda)}G(\Gamma,\Lambda, 0)
	\end{equation}
\end{prop}
\begin{proof}
	Let $\sigma'$ be a section of the elliptic fibration of $E(3)$ which is disjoint from $\sigma$. Then the homology class of $\sigma'$ is equal to $\sigma+3f+u$
	where $u \in H_2(B(2,3,17),\Z)$ and $u \cdot u=-6$. Arguing as in Lemma \ref{orbits-line-bundle-E(n)}, there is a diffeomorphism $\Phi$ of $E(3)$ that fixes $H_2(G(3))$
	and maps $u$ to a linear combination of $g_1$ and $\tau_1$. Furthermore, $\Phi_*(u)\equiv 2g_1-\tau_1$ mod 3.
	In particular, $\alpha:=\Phi(\sigma')$ is a (-3)-sphere with $\alpha \cdot w_{k,1} \equiv 1$ mod 3.
	Suppose $W$ is the subspace of the elements of $H_2(E(3))$ whose intersection numbers with
	 $\alpha$ are equal to $0$.
	Using Proposition \ref{-3-identities}, the series
	$\widehat \rD_{E(3),w_{k,l}}({\e}^{\Gamma_{(2)}+\Lambda_{(3)}})$, for any $(\Gamma,\Lambda)\in W \oplus W$, is equal to:
	\begin{equation} \label{RHS}
		-\frac{1}{3}\widehat \rD_{E(3),w_{k,1}+\alpha}((-\frac{3}{2} \alpha_{(3)}-\frac{3}{2} \alpha_{(2)}^2-a_2){\e}^{\Gamma_{(2)}+\Lambda_{(3)}})
	\end{equation}
	Since $(w_{k,1}+\alpha)\cdot f=1$, we can evaluate the expression \eqref{RHS} using our current knowledge of the invariants of $E(3)$.
	It is straightforward to check that the resulting series is equal to \eqref{E(3)-w.f=0}.
	
	The homology classes $f$, $k:=\sigma+\frac{3}{2}f$ and $\tau_2$ satisfy
	the assumption of Lemma \ref{SO-k-inv} for $V=H_2(E(3))$. Suppose $W_0$ and $W_1$
	are given as in Lemma \ref{SO-k-inv}, and $U$ is the subset of $W_0 \oplus W_1$
	consisting of the pairs that satisfy \eqref{E(3)-w.f=0}. Then $U$ is a Zariski
	closed subset of $W_0 \oplus W_1$. In order to complete the proof,
	we shall show that $U$ contains a Euclidean open set in $W_0\oplus W_1$ and hence
	$U=W_0\oplus W_1$. Let $(\Gamma,\Lambda)\in W_0\oplus W_1$ are given as below:
	\begin{equation*}
		\Gamma:=af+bk \hspace{1cm} \Lambda=a'f+b'k+c\tau_2.
	\end{equation*}
	Consider the homology classes $u_1:=\frac{\tau_1+\tau_3}{\sqrt 2}$ and $u_2:=\frac{\tau_1-\tau_3}{\sqrt 2}$ which have non-zero intersection with $\alpha$.
	There is an element $A_{t,\theta}\in \SO(H_2(E(n));f)$ such that:
	\begin{align} \label{imageunderA-tt}
		A_{t,\theta}(\Gamma):=&af+b(k+t^2f+tu_1)\nonumber\\
		A_{t,\theta}(\Lambda):=&a'f+b'(k+t^2f+tu_1)+c'(\cos(\theta) \tau_2+\sin(\theta) u_2)\nonumber
	\end{align}
	If $a$ and $a'$ are close enough to each other and $b$, $b'$ and $c'$ are close enough to $1$, then $t$ and $\theta$ can be chosen such that
	$A_{t,\theta}(\Gamma), A_{t,\theta}(\Lambda) \in W$.
	Therefore, $U$ contains an open subset of $W_0 \oplus W_1$.
\end{proof}

\subsection{Invariants of $E(n)$} \label{E(n)-invariants}

In this section, we compute the invariants of the elliptic surface $E(n)$. We start with the simpler case  that $w \cdot f \nequiv 0$ mod 3:
\begin{prop} \label{w.fneq0}
	Suppose $w$ is a 2-cycle in $E(n)$ with $w\cdot f \nequiv 0$ mod 3. Then $E(n)$ has $w$-simple type, and for $\Gamma, \Lambda \in H_2(E(n))$:
	\begin{equation*}
		\widehat \rD_{E(n),w}({\e}^{\Gamma_{(2)}+\Lambda_{(3)}})={\e}^{\frac{Q(\Gamma)}{2}-Q(\Lambda)}	G(\Gamma,\Lambda, w\cdot f)^{n-2}.
	\end{equation*}
\end{prop}

\begin{proof}
	The proof of this proposition is similar to that of Proposition \ref{E(3)-inv-w.fneq0}. Applying Proposition \ref{fiber-sum-universal-formula} for $X_1=E(n-2)$, $X_2=E(2)$ and
	$X_3=X_4=E(1)$ gives us enough relations to verify the proposition by induction.
\end{proof}

In the blown up elliptic surface $E(n)\#\cp$, there is an embedded sphere with self-intersection $-(n+1)$, given by tubing a section of the elliptic fibration and the exceptional sphere in $\cp$. Therefore, there is a copy of the Gompf nucleus $G(n+1)$ in $G(n)\#\cp$. The homology class $E+f$ can be realized by an embedded surface $\overline E$ in $M_n:=G(n)\#\cp\backslash G(n+1)$. The surface $\overline E$ determines a generator of $H_2(M_n)$. Similarly, the nucleus $G(n+2)$ can be embedded in $G(n)\#2\cp$. Therefore, there are copies of $G(4)$ in the 4-manifolds $E(4)$, $E(3)\#\cp$ and $E(2)\#2\cp$. Let $Z_0\subset E(4)$, $Z_1\subset E(3)\#\cp$ and $Z_2\subset E(2)\#2\cp$ be  the complements of $G(4)$ in these manifolds. Then the boundary of $Z_i$ is diffeomorphic to $\Sigma(2,3,23)$. It is clear from the inductive construction of $E(n)$ that there is an embedding of $Z_i$ in $E(n-i)\#i\cp$ for $n \geq 4$. In fact, if $W(n)$ is the fiber sum $E(n-4)\#_f G(4)$, then $E(n-i)\#i\cp$ is diffeomorphic to $Z_i\#_{\Sigma(2,3,23)}W(n)$.

The 4-manifold $Z_i$ gives rise to elements of $\I_*(\Sigma(2,3,23))$ as it is explained in Proposition \ref{gluing-23}. Suppose $V_0\subseteq\I_*(\Sigma(2,3,23))$ is the vector space generated by the element $\rD_{Z_0,v_0}(1)$ where $v_0$ is a 2-cycle in $Z_0$ with $v_0^2\equiv 0 \mod 3$. Similarly, define $V_1$ to be the vector space generated by the three elements $\rD_{Z_1,w}(1)$ where $w$ is one of the following elements which satisfy $w^2\equiv 1$ mod 3:
\begin{equation*}
	v_1:=\overline E+\tau_1-g_1 \hspace{1cm}v_2:=-\overline E+\tau_1-g_1\hspace{1cm}v_3:=\tau_1
\end{equation*}
Finally, let $V_2$ be the subspace of $\I_*(\Sigma(2,3,23))$ which is generated by the elements of the form $\rD_{Z_2,w}(1)$ where $w^2\equiv 2 \mod 3$.

\begin{prop} \label{subspace-HF4}
	The space $V_i$ is a subspace of $\I_4(\Sigma(2,3,23))$. Furthermore, the dimension of $V_i$ is at least $\frac{(i+2)(i+1)}{2}$.
\end{prop}
\begin{proof}
	The first part of the proposition is an immediate consequence of Proposition \ref{gluing-23}.
	In order to show that $\dim(V_0)=1$, let $\rD_{Z_0,v_0}(1)=0$. By Proposition \ref{gluing-23}, $\rD_{E(4),v_0+w_0}(z)$ vanishes for a 2-cycle
	$w_0$ in $G(4)$ and $z\in \A(G(4))^{\otimes 2}$. If
	$w_0$ is chosen such that $w_0 \cdot f\nequiv 0$ mod 3, then Proposition \ref{w.fneq0} asserts that there is $z$ such that $\rD_{E(4),v_0+w_0}(z)\neq 0$ which is a contradiction.
	
	Next, we consider the case that $i=1$. By Proposition \ref{gluing-23}, a linear relation among the vectors $\rD_{Z_1,v_l}(1)$, for $1\leq l \leq 3$, implies
	that there are constant numbers $c_l$ such that:
	\begin{equation} \label{linear-relation}
		\sum_{1\leq l \leq 3} c_l \widehat{\rD}_{E(3)\#\cp,v_l+w_0}(\e^{\sigma_{(2)}+\sigma_{(3)}})=0.
	\end{equation}
	Here $w_0$ is a 2-cycle in $G(4)$ and $\sigma$ is the embedded ($-4$)-sphere in $G(4)$. We can use the blow up formula and the results of the previous subsection to evaluate the above power series and to
	conclude that:
	\begin{equation*}
		\sum_{1\leq l \leq 3} c_l (\cosh(\sqrt{3}t_2)+
		\zeta^{-(v_l+w_0)\cdot E}\e^{\bi \sqrt 3 t_3}+\zeta^{(v_l+w_0)\cdot E}\e^{-\bi \sqrt 3 t_3})=0
	\end{equation*}
	Since $v_l\cdot E\equiv -l$ mod 3, the constants $c_l$ are equal to zero.
	Therefore, $\dim (V_1)=3$.  Similar proofs can be applied to the case that $i=2$.
\end{proof}

\begin{theorem} \label{E(n)-inv}
	Suppose $w$ is a 2-cycle in $E(n)$ and $\Gamma,\Lambda \in H_2(E(n))$. Then:
	\begin{equation}\label{invariants-of-E(n)}
		\widehat \rD_{E(n),w}({\rm e}^{\Gamma_{(2)}+\Lambda_{(3)}})=
		{\rm e}^{\frac{Q(\Gamma)}{2}-Q(\Lambda)}G(\Gamma,\Lambda,w\cdot f)^{n-2}
	\end{equation}
\end{theorem}
\begin{proof}
	Since the dimension of $\I_4(\Sigma(2,3,23))$ is equal to 9,
	Proposition \ref{subspace-HF4} implies that there are vectors $\phi_i\in V_i$ such that:
	\begin{equation} \label{relation}
		\phi_0+\phi_1+\phi_2=0
	\end{equation}	
	We use this identity to verify \eqref{invariants-of-E(n)} inductively for the remaining invariants of $E(n)$.
	Firstly, consider the case that $\phi_0\neq 0$.
	Therefore, we can assume that $\phi_0=\rD_{Z_0,v_0}(1)$.
	There are also 2-cycles $w_1$, $\dots$, $w_k$ in $Z_2$ and constant numbers
	$c_1$, $c_2$, $c_3$, $c_1'$, $\dots$, $c_k'$ such that:
	\begin{equation}\label{phi1-phi2}
		\phi_1=\sum_{l=1}^3 c_l\rD_{Z_1,v_l}(1) \hspace{1cm}\phi_2=\sum_{j=1}^k c_j'\rD_{Z_2,w_j}(1)
	\end{equation}
	Suppose $\Gamma,\Lambda \in H_2(W(n))$ and $w$ is a 2-cycle in $W(n)$.
	Then Proposition \ref{gluing-23} implies that the $\U(3)$-series
	$\widehat{\rD}_{E(n),v_0+w}(\e^{\Gamma_{(2)}+\Lambda_{(3)}})$ is equal to:
	\begin{equation*}
		-\sum_l c_l \widehat{\rD}_{E(n-1)\#\cp,v_l+w}(\e^{\Gamma_{(2)}+\Lambda_{(3)}})-
		\sum_j c_j' \widehat{\rD}_{E(n-2)\#2\cp,w_j+w}(\e^{\Gamma_{(2)}+\Lambda_{(3)}})=\hspace{1cm}
	\end{equation*}
	\begin{equation} \label{ind-E(n)-inv}
		\hspace{5cm}= {\rm e}^{\frac{Q(\Gamma)}{2}-
		Q(\Lambda)}G(\Gamma,\Lambda,w\cdot f)^{n-4}	A(\Gamma,\Lambda)
	\end{equation}
	where the term $A(\Gamma,\Lambda)$ is a linear combination of the following six expressions:
	\begin{equation} \label{6-terms}
		\cosh(\sqrt{3}f\cdot \Gamma)^2\hspace{.5cm}
		\cosh(\sqrt{3}f\cdot \Gamma)\zeta^{\mp w\cdot f}\e^{\pm \bi \sqrt 3 f \cdot \Lambda}
		\hspace{.5cm}\zeta^{\mp 2 w\cdot f}\e^{\pm 2\bi \sqrt 3 f \cdot \Lambda}\hspace{.5cm} 1
	\end{equation}		
	and the coefficients of this this linear combination do not depend on $w$.
	To derive \eqref{ind-E(n)-inv} we use the fact that $\Gamma \cdot E=-\Gamma \cdot f$ and
	$\Lambda \cdot E=-\Lambda \cdot f$. In the case $w \cdot f\nequiv 0$ mod 3,
	$A(\Gamma,\Lambda)$ is equal to $G(\Gamma,\Lambda, w\cdot f)^2$
	using Proposition \ref{w.fneq0}. This identity holds also in the case that $w\cdot f \equiv 0$ mod 3,
	because $A(\Gamma,\Lambda)$ is a linear combination of the terms in \eqref{6-terms}
	with coefficients which are independent of $w$. Therefore, for a general 2-cycle $w$ in $W(n)$
	and $\Gamma,\Lambda \in H_2(W(n))$:
	\begin{equation*}
		\widehat{\rD}_{E(n),v_0+w}(\e^{\Gamma_2+\Lambda_3})=
		{\rm e}^{\frac{Q(\Gamma)}{2}-Q(\Lambda)}G(\Gamma,\Lambda,w\cdot f)^{n-2}
	\end{equation*}
	Then Proposition \ref{non-nuc-inv} implies that $\widehat{\rD}_{E(n),w}(\e^{\Gamma_2+\Lambda_3})$
	is equal to ${\rm e}^{\frac{Q(\Gamma)}{2}-Q(\Lambda)}G(\Gamma,\Lambda,w)^{n-2}$
	for any 2-cycle $w$ in $E(n)$ and $\Gamma,\Lambda \in H_2(E(n))$.
	
	Next, assume that $\phi_0=0$. We assume that the non-zero vectors
	$\phi_1$ and $\phi_2$ are given as in \eqref{phi1-phi2}.
	Fix an arbitrary 2-cycle $w$ in $W(n+1)$ and homology classes
	$\Gamma, \Lambda \in H_2(W(n+1))$. Another application of Proposition \ref{gluing-23} shows that:
	\begin{equation} \label{lin-comb-1}
		\sum _{1\leq l \leq 3} c_l \widehat \rD_{E(n)\#\cp,v_l+w}(\e^{\Gamma_{(2)}+\Lambda_{(3)}})
	\end{equation}
	is equal to:
	\[
	  \sum_j c_j' \widehat{\rD}_{E(n-1)\#2\cp,w_j+w}(\e^{\Gamma_{(2)}+\Lambda_{(3)}}).
	\]
	By our inductive calculation of the invariants of $E(n)$, the latter expression is equal to:
	\[
	  {\rm e}^{\frac{Q(\Gamma)}{2}-Q(\Lambda)}
	  G(\Gamma,\Lambda,w\cdot f)^{n-3}
	  B(\Gamma,\Lambda).
	\]
	Here $B(\Gamma,\Lambda)$ is a linear combination of the following six expressions:
	\[
	  \cosh(\sqrt{3}f\cdot \Gamma)^2\hspace{.7cm}\cosh(\sqrt{3}f\cdot \Gamma
	  )\zeta^{\mp w\cdot f}\e^{\pm \bi \sqrt 3 f \cdot \Lambda}\hspace{.7cm}
	  \zeta^{\mp 2 w\cdot f}\e^{\pm 2\bi \sqrt 3 f \cdot \Lambda}\hspace{.7cm} 1
	\]
	As in the previous case, the coefficients of the above linear combination is determined by $c_j'$
	and $w_j$. In particular, they do not depend on $w$. Therefore, we can determine this
	coefficient by considering the case that $w\cdot f \nequiv 0$ mod 3 for which we already
	computed the invariants. Therefore, $B(\Gamma,\Lambda)$ is equal to:
	\[
	  G(\Gamma,\Lambda,w\cdot f)
	  \sum_{1\leq l \leq 3} \frac{c_l}{3}(\cosh(\sqrt{3}E\cdot \Gamma)+
	  \zeta^{-w\cdot E}\e^{\bi \sqrt 3 E \cdot \Lambda}+\zeta^{w\cdot E}\e^{-\bi \sqrt 3 E \cdot \Lambda})
	\]
	For an arbitrary 2-cycle $w\subset E(n)\#\cp$ and $\Gamma,\Lambda\in E(n)\#\cp$,
	let $\widehat {\rm P}_{w}(\Gamma,\Lambda)$ be the power series given by subtracting
	$\widehat \rD_{w}(\e^{\Gamma_{(2)}+\Lambda_{(3)}})$ from the following power series:
	\[
	  {\rm e}^{\frac{Q(\Gamma)}{2}-Q(\Lambda)}
	  G(\Gamma,\Lambda,w\cdot f)^{n-2}
	  \frac{1}{3}(\cosh(\sqrt{3}E\cdot \Gamma)+\zeta^{w\cdot E}\e^{\bi \sqrt 3 E \cdot \Lambda}+
	  \zeta^{-w\cdot E}\e^{-\bi \sqrt 3 E \cdot \Lambda}).
	\]
	Then we can rephrase our conclusion in the form of the following identity:
	\begin{equation} \label{lin-comb}
		\sum _{1\leq l \leq 3} c_l \widehat {\rm P}_{v_l+w}(\Gamma,\Lambda)=0
	\end{equation}
	where $w$ is a 2-cycle in $W(n+1)$ and $\Gamma, \Lambda \in H_2(W(n+1))$.
	Suppose $p_w^{i,j}$ is the polynomial on $H_2(E(n)\#\cp)\oplus H_2(E(n)\#\cp)$ of bi-degree
	$(i,j)$, determined by the coefficient of $t_2^it_3^j$ in $\widehat {\rm P}_{w}$.
	Then $p_w^{i,j}$ can be evaluated at:
	\begin{equation*}
		(\Gamma_1,\dots,\Gamma_i;\Lambda_1,\dots,\Lambda_j)
	\end{equation*}
	for $\Gamma_k ,\,\Lambda_k \in H_2(E(n)\#\cp)$. By induction on $i+j$, we shall show that
	$p_w^{i,j}$ vanishes for all possible choices of $w$. By considering the constant terms of
	Equation \eqref{lin-comb} for empty $w$, we have:
	\begin{equation*}
		c_1p_{v_1}^{0,0}+c_2p_{v_2}^{0,0}+c_3p_{v_3}^{0,0}=0
	\end{equation*}
	The blowup formula asserts that $p_{v_1}^{0,0}=p_{v_2}^{0,0}=0$. Therefore, if $c_3\neq 0$,
	then $p_{v_3}^{0,0}=0$. Thus Proposition \ref{non-nuc-inv} and the blowup formula show that
	$p_{w}^{0,0}=0$ for all 2-cycles $w$ in $E(n)\#\cp$. If $c_3=0$, then by \eqref{lin-comb}:
	\begin{equation*}
		c_1p_{v_1}^{2,0}(\sigma+E,\sigma+E)+c_2p_{v_2}^{2,0}(\sigma+E,\sigma+E)=0
	\end{equation*}
	and
	\begin{equation*}
		c_1p_{v_1}^{0,1}(\sigma+E)+c_2p_{v_2}^{0,1}(\sigma+E)=0.
	\end{equation*}
	The blowup formula asserts that:
	\[
	  c_1p_{v_1-E}^{0,0}+c_2p_{v_2+E}^{0,0}=0 \hspace{1cm}c_1p_{v_1-E}^{0,0}-c_2p_{v_2+E}^{0,0}=0
	\]
	Consequently, at least one of the numbers $p_{v_1-E}^{0,0}$ and $p_{v_2+E}^{0,0}$ is zero
	and we can derive the same conclusion as in the previous case.
	
	Now assume that the polynomial $p_w^{i,j}$ vanishes for $i+j\leq k$ and any 2-cycle in $E(n)\#\cp$.
	Fix $(i,j)$, $\widetilde \Gamma_1$, $\dots$, $\widetilde \Gamma_i$, $\widetilde \Lambda_1$,
	$\dots$ and $\widetilde \Lambda_j$ such that $i+j=k+1$ and $\widetilde  \Gamma_l$,
	$\widetilde \Lambda_{l'}$ are either equal to $\sigma+E$ or $f$. Then apply \eqref{lin-comb}
	to conclude that:
	\begin{equation} \label{lin-comb-2}
		\sum _{1\leq l \leq 3} c_l p_{v_l}^{i,j}(\widetilde \Gamma_1,\dots,\widetilde \Gamma_i;\widetilde 		\Lambda_1,\dots,\widetilde \Lambda_j)=0
	\end{equation}
	Blowup formula and the induction hypothesis imply that every term of the form $\sigma+E$
	can be replaced with $\sigma$. Thus if $c_3\neq 0$, then:
	\begin{equation} \label{vanishing}
		p_{v_1}^{i,j}(\Gamma_1,\dots,\Gamma_i;\Lambda_1,\dots,\Lambda_j)=0
	\end{equation}
	for $\Gamma_k,\Lambda_k \in H_2(G(n))$. Therefore, Proposition \ref{non-nuc-inv} and
	the blowup formula allows us to complete the verification of the induction step. If $c_3=0$,
	we can use the analogue of \eqref{lin-comb-2} for:
	\[
	  (\sigma+E,\sigma+E,\widetilde \Gamma_1,\dots,\widetilde \Gamma_i;
	  \widetilde \Lambda_1,\dots,\widetilde \Lambda_j)\hspace{1cm}
	  (\widetilde \Gamma_1,\dots,\widetilde \Gamma_i;\sigma+E,
	  \widetilde \Lambda_1,\dots,\widetilde \Lambda_j)
	\]
	and argue as in the basis of the induction. This completes the proof of the theorem.
\end{proof}

\subsection{Gluing 4-manifolds along Surfaces of Self-intersection Zero}\label{gluing-self-intersection-0}

In this subsection, we use the calculation of $\U(3)$-polynomial invariants for elliptic surfaces to study invariants of another family of closed 4-manifolds:

\begin{definition} \label{permissible-type}
	Suppose $X$ is a smooth 4-manifold, $\Sigma$ is an oriented surface of genus $g\geq 1$ embedded in $X$, $w$
	is a 2-cycle in $X$, and $\mathcal H$ is a subspace of $H_2(X)$. Then $(X,\Sigma,w)$ is {\it permissible
	with respect to the subspace $\mathcal H$}, if the following properties hold:
	\begin{itemize}
		\item[(i)] $b^1(X)=0$, $b^+(X)>1$;
		\item [(ii)]$\Sigma\cdot \Sigma=0$;
		\item[(iii)] $w \cdot \Sigma \nequiv 0$ mod 3;
		\item[(iv)] let $z \in \A(\mathcal H)^{\otimes 2}$ and $u$ be the 2-cycle  $w+l\Sigma$ for $l =0$, $1$ or $2$. Then:
				\begin{equation}\label{simple-type-u}
					\rD_{X,u}((\frac{a_2}{3})^3z)=\rD_{X,u}(z) \hspace{1cm} \rD_{X,u}(a_3z)=0.
				\end{equation}	
				Moreover, there are cohomology classes $K_i \in H^2(X,\Z)$ such that $K_i$ is an integral lift of $w_2(TX)$,
				$|K_i \cdot \Sigma|\leq 2g-2$, and for $\Gamma$, $\Lambda \in \mathcal H$, the power series
				${\widehat{\rD}_{X,u}}(\e^{\Gamma_{(2)}+\Lambda_{(3)}})$ is equal to:
			\begin{equation}\label{simpletype}
				\e^{\frac{Q(\Gamma)}{2}-Q(\Lambda)} \sum_{i,j} c_{ij}\zeta^{-l\Sigma \cdot (\frac{K_i-K_j}{2})}
				\e^{\frac{{\sqrt 3}}{2}(K_i+K_j)\cdot \Gamma+\frac{{\sqrt 3}}{2}\bi(K_i-K_j)\cdot \Lambda}
			\end{equation}			
	\end{itemize}
	The cohomology class $K_i$ is called a {\it basic class} of the triple $(X,\Sigma,w)$ and $c_{i,j}$ is called the coefficient associated to the
	pair $(K_i,K_j)$. In the case that $\mathcal H=H_2(X)$, we say $(X,\Sigma,w)$ is permissible.
\end{definition}
 \begin{example} \label{E(2)-simple-triple}
	The results of Subsection \ref{E(2)} shows that the triple $(E(2), f,w)$ forms a permissible triple
	where $f$ is a fiber in an elliptic fibration of $E(2)$ and $w$ is a 2-cycle that $w\cdot f \nequiv 0$ mod 3.
	In this case, the only basic class in \eqref{simpletype} is the zero cohomology class.
	More generally, we can embed a surface of genus $g$ in $E(2)$ whose self-intersection is equal to $2g-2$.
	For example, we can construct such an embedded surface by considering the union of $g$ fibers and a section of the fibration, and then smoothing out
	the intersection points. Let $\Sigma$ be the proper transform of this surface after blowing up $E(2)$
	at $2g-2$ points on the surface. Let also $w$ be a 2-cycle in $E(2)\# (2g-2)\cp$ such that $w\cdot \Sigma \nequiv 0$ mod 3. Then the blow up formula implies that
	$(E(2)\# (2g-2)\cp,\, \Sigma,\, w)$ is a permissible triple. If $E_1$, $\dots$, $E_{2g-2}$ are the exceptional classes,
	then a basic class of this triple has the form $\pm E_1 \pm  \dots \pm E_{2g-2}$.
\end{example}

\begin{example}	
	One can further generalize the previous example by considering a surface $\Sigma$ with self-intersection 0, embedded in the 4-manifold $E(n)\#k \cp$.
	Let $w$ be a 2-cycle in $X$ such that $w \cdot \Sigma \nequiv 0$ mod 3. Then the
	blowup formula and Theorem \ref{E(n)-inv} can be utilized to verify most requirements of Definition \ref{permissible-type} for permissibility of $(X,\Sigma,w)$.
	The only missing part is to verify the inequality
	$|K \cdot \Sigma|\leq 2g-2$ for basic classes $K$ of $X$. To check this inequality, note that our basic classes for $(X,\Sigma,w)$
	are the same as $\U(2)$-basic classes for $X$ \cite{KM:Rec,FS:str-thm,Li:elliptic-surf,KM:str-thm,fintushel1996blowup}. Therefore, the desired inequality is
	a consequence of the Adjunction inequality in \cite{KM:str-thm}. In fact, we expect that any tripe $(X,w,\Sigma)$, satisfying properties (i), (ii) and (iii), automatically meets the requirements in (iv), as long
	as $X$ has simple type in the sense of \cite{KM:str-thm}. However, pursuing this direction is beyond the scope of this paper.
 \end{example}

\begin{definition}\label{Cg-Ng}
	For $g\geq 1$, the set of all integer pairs $(a,b)$ which satisfy
	the following two properties is denoted by $\mathcal C_g$:
	\begin{itemize}
		\item[(i)] $a$ and $b$ have the same parity;
		\item[(ii)] $|a|+|b|\leq 2g-2$.
	\end{itemize}
	We will write $N_g$ for the number of the elements of $\mathcal C_g$ which is equal to $2g(g-1)+1$.
\end{definition}

The following proposition can be verified using the permissible triples provided by Example \ref{E(2)-simple-triple}:
\begin{prop} \label{a-b-basic-class}
	For any $(a,b)\in \mathcal C_g$, there are a permissible triple $(X,w,\Sigma)$ and basic classes
	$K_i$ and $K_j$ of the triple such that $\Sigma$ has genus $g$ and:
	\begin{equation}
		a=\frac{(K_i+K_j)}{2}\cdot \Sigma\hspace{1cm} b=\frac{(K_i-K_j)}{2}\cdot \Sigma.
	\end{equation}	
\end{prop}

For a permissible triple $(X,\Sigma,w)$, define:
\begin{equation*}
	{\rD_{X,w,\Sigma}}(\e^{\Gamma_{(2)}+\Lambda_{(3)}}):=\sum_{0\leq l \leq 2}\rD_{X,w+l\Sigma}(\e^{\Gamma_{(2)}+\Lambda_{(3)}})
\end{equation*}
More generally, ${\rD_{X,w,\Sigma}}(z)$ is defined to be the sum ${\rD_{X,w+l\Sigma}}(z)$ for different values of $0 \leq l\leq 2$. Dimension fomrula shows that if all terms in $z$ have a fixed degree, then only one of the 2-cycles $w+l \Sigma$ is involved in the definition of ${\rD_{X,w,\Sigma}}(z)$.
 \begin{lemma} \label{permissible-inv}
	Suppose $(X,w,\Sigma)$ is a permissible triple as above and $d_w\in \Z/3\Z$ is defined to be $b^++1-w\cdot w$. Then ${\rD_{X,w,\Sigma}}((\frac{a_2}{3})^m\e^{\Gamma_{(2)}+\Lambda_{(3)}})$ is equal to:
	\begin{equation*}
	\displaystyle {\sum c_{ij} \zeta^{l_{i,j}(d_w-m)} \e^{\zeta^{2l_{i,j}}(\frac{Q(\Gamma)}{2}+
		\frac{{\sqrt 3}}{2}\bi(K_i-K_j)\cdot \Lambda)+\zeta^{l_{i,j}}(-Q(\Lambda)+\frac{{\sqrt 3}}{2}(K_i+K_j)\cdot \Gamma)}}
	\end{equation*}
	where the inner sum is over all pairs of basic classes $(K_i,K_j)$.
	Moreover, $l_{i,j}\in \Z/3\Z$ is equal to $(w\cdot \Sigma) (\frac{K_i-K_{j}}{2})\cdot \Sigma$.
 \end{lemma}

  \begin{proof}
	For a 4-manifold $W$ with simple type, the power series $\rD_{X,w}$ can be recovered from ${\widehat \rD}_{X,w}$ in the following way:
	\begin{equation*}
		\rD_{X,w}((\frac{a_2}{3})^{m}\e^{\Gamma_{(2)}+\Lambda_{(3)}})=
		\frac{1}{3}\sum_{0 \leq k \leq 2} \zeta^{k(d_w-m)}{\widehat \rD}_{X,w}(\e^{\zeta^{k}\Gamma_{(2)}+\zeta^{2k}\Lambda_{(3)}})
	\end{equation*}
	Therefore:
	\begin{equation*}
		\sum_{0 \leq l \leq 2} \rD_{X,w+l\Sigma}((\frac{a_2}{3})^{m}\e^{\Gamma_{(2)}+\Lambda_{(3)}})=
		\frac{1}{3}\sum_{0 \leq k,l \leq 2} \zeta^{k(d_{w+l\Sigma}-m)}{\widehat \rD}_{X,w+l\Sigma}(\e^{\zeta^{k}\Gamma_{(2)}+\zeta^{2k}\Lambda_{(3)}})
	\end{equation*}
	Then, we can use the permissibility of $(X,w,\Sigma)$ to rewrite the right hand side in terms of basic classes. A straightforward simplification gives the desired result.
\end{proof}

Suppose $(X, w,\Sigma)$ and $(X',w',\Sigma)$ are two permissible triples such that the embedded surfaces of genus $g$ are identified with each other, and this identification is lifted to the normal bundles. Suppose also $w$ and $w'$ intersect $\Sigma$ in the same number of points with the same signs. As it is explained in Subsection \ref{FFH}, we can form the triple $(X\#_\Sigma X', w\#w',\Sigma)$. There is also a subspace $\mathcal K \subseteq H_2(X)\oplus H_2(X')$ such that there is a map $\#:\mathcal K \to H_2(X\#_\Sigma X')$. The main goal of this section is to compute $\widehat \rD_{X\#_\Sigma X', w\#w'}(\e^{\Gamma_{(2)}+\Lambda_{(3)}})$, for $\Gamma,\Lambda \in \im(\#)$, in terms of the invariants of the pairs $(X, w)$ and $(X', w')$.

The basic idea to achieve this goal is to use the gluing property in \eqref{FF-gluing-thm}. Therefore, the $R_N$-module $\bIgdN$, introduced in Subsection \ref{FFH}, for $N=3$ and $d=w\cdot \Sigma$, plays a key role in  computing the invariants of $(X\#_\Sigma X', w\#w')$. In fact, we can replace ${\mathbb I}_{g,d}^3$ with a smaller module. Before giving the definition of this smaller module, we introduce some conventions. Form now on, we drop $3$ from our notation and denote this Fukaya-Floer homology module with $\bIgd$. Moreover, for a permissible triple $(X,w,\Sigma)$ the intersection $w\cdot \Sigma$ is denoted by $d$, unless otherwise is stated.

Let $\tbIgd\subset \bIgd$ be the $\C[\![t_2,t_3]\!]$-module generated by the following relative elements in $\bIgd$:
\begin{equation} \label{relative-element}
	\rD_{X^\circ,w^\circ ,\Sigma}(\e^{\Gamma^\circ_{(2)}+\Lambda^\circ_{(3)}}):=\sum_{l\in \Z/3\Z}\rD_{X^\circ,w^\circ+l\Sigma}(\e^{\Gamma^\circ_{(2)}+\Lambda^\circ_{(3)}})
\end{equation}
where $(X,w,\Sigma)$ is a permissible triple, $\Gamma$, $\Lambda \in H_2(X,\Z)$ with $\Gamma \cdot \Sigma=\Lambda \cdot \Sigma=1$. By Identity \ref{FF-gluing-thm}, the pairing of this element with
$\rD_{\Delta_g, \delta_g}(z\e^{ D_{(2)} +  D_{(3)}})$ is equal to the following element of $\C[\![t_2,t_3]\!]$:
\begin{equation}\label{pairing-rel}
	\rD_{X,w ,\Sigma}(z\e^{\Gamma_{(2)}+\Lambda_{(3)}})
\end{equation}
Suppose $\C[\![t_2,t_3]\!][x,y,z]$ is the ring of polynomials of three variables with coefficients in $\C[\![t_2,t_3]\!]$ and $z=P(a_2,\,\Sigma_{(2)},\,\Sigma_{(3)})$ for $P \in \C[\![t_2,t_3]\!][x,y,z]$. Then Lemma \ref{permissible-inv} shows that the pairing of \eqref{relative-element} and $\rD_{\Delta_g, \delta_g}(z\e^{ D_{(2)} +  D_{(3)}})$ is equal to:
\begin{equation*}
	\displaystyle {\sum_{(a,b)\in \mathcal C_g} \zeta^{d b d_w} P(u(a,b))\sum c_{i,j}\e^{\zeta^{2d b}
	(\frac{Q(\Gamma)}{2}+\frac{{\sqrt 3}}{2}\bi(K_i-K_j)\cdot \Lambda)+
    	\zeta^{d b}(-Q(\Lambda)+\frac{{\sqrt 3}}{2}(K_i+K_j)\cdot \Gamma)}}
\end{equation*}		
where the inner sum is over all pairs of basic classes $(K_i, K_j)$ such that:
\begin{equation} \label{a-b-basic-class-id}
	\frac{(K_i+K_{j})}{2}\cdot \Sigma=a \hspace{1cm}
	\frac{(K_i-K_{j})}{2}\cdot \Sigma=b.
\end{equation}
and:
\begin{equation} \label{point-a-b}
	u(a,b):=(3\zeta^{2d b},\,\zeta^{db}\sqrt 3 a+\zeta^{2db}t_2,\,
	\zeta^{2db}\sqrt 3 \bi b-2\zeta^{db}t_3).
\end{equation}

For $\lambda=(\alpha,\beta) \in \mathcal C_g$, fix a polynomial $P_{\lambda}\in \C[\![t_2,t_3]\!][x,y,z]$ such that:
\begin{equation} \label{P-lambda}
	P_{\lambda}(u(\alpha,\beta))=1\hspace{3cm} P_{\lambda}(u(a,b))=0 \hspace{.5cm}\text{ for $(a,b)\neq (\alpha,\beta)$}
\end{equation}
Define a map $\Phi: \tbIgd \to \C[\![t_1,t_2]\!]^{N_g}$ in the following way:
\begin{equation*}
	\Phi(\eta):=\left \{\langle \eta,\, \rD_{\Delta_g,\delta_g}(P_{\lambda}(a_2,\,\Sigma_{(2)},\Sigma_{(3)})\e^{D_{(2)} + D_{(3)}})\rangle \right\}_{\lambda \in \mathcal C_g}
\end{equation*}
By \eqref{FF-gluing-thm}, the homomorphism $\Phi$ maps the relative element in \eqref{relative-element} to an element of $\C^{N_g}[\![t_2,t_3]\!]$, whose component corresponding to $\lambda=(a,b)$, denoted by $c_{X,w,\Sigma}^\lambda(\Gamma,\Lambda)$, is equal to:
\begin{equation}\label{relative-component}
	\zeta^{d bd_w} \sum c_{i,j} \e^{\zeta^{2d b}
		(\frac{Q(\Gamma)}{2}+\frac{{\sqrt 3}}{2}\bi(K_i-K_j)\cdot \Lambda)+
    		\zeta^{d b}(-Q(\Lambda)+\frac{{\sqrt 3}}{2}(K_i+K_j)\cdot \Gamma)}
\end{equation}
where the sum is over the pairs of basic classes $(K_i,K_j)$ that satisfy \eqref{a-b-basic-class-id}.
Let $\C(\!(t_1,t_2)\!)$ denote the field of fractions of $\C[\![t_2,t_3]\!]$. Then $\Phi$ induces a map $\overline \Phi:\tbIgd\otimes_{\C[\![t_1,t_2]\!]} \C(\!(t_1,t_2)\!) \to \C(\!(t_1,t_2)\!)^{N_g}$.

\begin{prop}
	The map $\Phi$ is injective. Moreover, $\overline \Phi$ is an isomorphism of vector spaces.
\end{prop}

\begin{proof}
	Suppose $\mathcal I$ is the ideal in $\A_g^3\otimes_\C\C[\![t_2,t_3]\!]$ that is generated by $a_3$, $\gamma_{(2)}$, $\gamma_{(3)}$ and
	the elements $P(a_2,\,\Sigma_{(2)},\,\Sigma_{(3)})$ where $P\in \C[\![t_2,t_3]\!][x,y,z]$ is a polynomial evaluating to zero at the points in
	\eqref{point-a-b}. Then the pairing of \eqref{relative-element} and $\rD_{\Delta_g, \delta_g}(z\e^{ D_{(2)} +  D_{(3)}})$ vanishes when
	$z\in \mathcal I$.
	Any element of $\tbIgd$ is also invariant with respect to the action of:
	\begin{equation} \label{tilde-epsilon-g}
	  \widetilde \epsilon:=\IIN([0,1] \times Y_g,[0,1] \times \gamma_{g,d}+\Sigma,
	  \e^{([0,1] \times \gamma)_{(2)}+\dots+([0,1] \times \gamma)_{(N)}})
	\end{equation}
	Recall that $Y_g=\Sigma\times S^1$ and $\gamma$ in \eqref{tilde-epsilon-g}
	denotes an $S^1$-fiber of $Y_g$.
	Any $z\in \A_g^3\otimes_\C\C[\![t_2,t_3]\!]$ can be written as a sum of an element of $\mathcal I$ and a $\C[\![t_2,t_3]\!]$-linear combination of
	the polynomials $\{P_{\lambda}\}_{\lambda\in \mathcal C_g}$. Therefore, injectivity of $\Phi$ is a consequence of Proposition \ref{bIgdN-prop}.
	
	For a given $\lambda_0 \in \mathcal C_g$, Proposition \ref{a-b-basic-class} gives
	a permissible triple $(X,w,\Sigma)$ such that the component of the relative element \eqref{relative-element}
	corresponding to $\lambda_0$ is non-zero. Furthermore, we can change the relative class as in \eqref{relative-element}
	by replacing $\Gamma$ with $\Gamma+s\Sigma$ and $\Lambda$ with $\Lambda+s\Sigma$.
	The component of this relative element corresponding to $\lambda=(a,b)$ picks the factor
	$\e^{s(\zeta^{2b}{\sqrt 3}at_2+\zeta^{b}\sqrt 3 b t_3+\zeta^{b}t_2^2-2\zeta^{2b}t_3^2)}$.
	Therefore, by taking $\C[\![t_1,t_2]\!]$--linear combinations of such expressions
	for different values of $s$, we can produce elements of $\tbIgd$ such that the component
	corresponding to $\lambda_0$ is the only non-zero element. This verifies the second part of the proposition.	
\end{proof}

Consider the restriction of the paring $\langle,\rangle$ on $\bIgd$. Above proposition implies that this pairing induces a pairing on $\C(\!(t_1,t_2)\!)^{N_g}$ using the map $\overline \Phi$:
\begin{equation}
	\langle\cdot ,\cdot \rangle: \C(\!(t_1,t_2)\!)^{N_g} \times \C(\!(t_1,t_2)\!)^{N_g} \to \C(\!(t_1,t_2)\!)
\end{equation}
Suppose this pairing is given by $\{\zeta^{2bd(g-1)}h_{\lambda,\lambda'}^{g,d}\}_{\lambda,\lambda' \in \mathcal C_g}$ with respect to the standard basis of $\C(\!(t_1,t_2)\!)^{N_g}$ where $b$ is the second coordinate of $\lambda$. The constant $\zeta^{2bd(g-1)}$ does not play an important role here. It will be used to obtain slightly simpler form for our gluing formulas in Proposition \ref{connected-sum-permissible}.
\begin{prop}\label{restriction-pairing-matrix}
	The element $h_{\lambda,\lambda'}^{g,d}(t_2,t_3)\in \C(\!(t_1,t_2)\!)$ is non-zero only if $\lambda=\lambda'$.
	Furthermore, if $\lambda=(a,b)$ and $|\lambda|_1:=|a|+|b|<2g-2$,
	then $h_{\lambda,\lambda}^{g,d}$ is zero.
\end{prop}
\begin{proof}
	Suppose $(X,w,\Sigma)$ and $(X',w',\Sigma)$ are two permissible triples such that $w$ and $w'$ intersect $\Sigma$ in the
	same set of points with the same signs. Suppose the homology classes $\Gamma,\Lambda \in H_2(X)$ and
	$\Gamma',\Lambda' \in H_2(X')$ are chosen such that their intersection with $\Sigma$ is equal to $1$.
	Then Identity \eqref{FF-gluing-thm} asserts that:
	\begin{equation*}
		\rD_{X\#_\Sigma X', w\#w',\Sigma}(\e^{(\Gamma\#\Gamma')_{(2)}+(\Lambda\#\Lambda')_{(3)}})=\hspace{4cm}
	\end{equation*}
	\begin{equation}\label{gluing-formula}		
		\hspace{3cm}\sum_{\lambda,\lambda' \in \mathcal C_g} \zeta^{2bd(g-1)}h_{\lambda,\lambda'}^{g,d}
		c_{X,w,\Sigma}^\lambda(\Gamma,\Lambda)c_{X',w',\Sigma}^{\lambda'}(\Gamma',\Lambda')
	\end{equation}
	Replacing $\Gamma$, $\Gamma'$, $\Lambda$ and $\Lambda'$ with $\Gamma+r\Sigma$, $\Gamma'-r\Sigma$, $\Lambda+s\Sigma$
	and $\Lambda'-s\Sigma$ does not change the left hand side of the above identity. On the right hand side, the term corresponding to
	$\lambda$ and $\lambda'$ changes by a factor of the form $\e^{r\cdot f(\lambda,\lambda')+s\cdot g(\lambda,\lambda')}	$.
	Here $f(\lambda,\lambda')$ and $g(\lambda,\lambda')$, which can be computed explicitly, are zero if and only if $\lambda =\lambda'$.
	Therefore, $h_{\lambda,\lambda'}^{g,d}$ has to be non-zero when $\lambda\neq \lambda'$.
	
	Let $(X, w, \Sigma)$ be the permissible triple of Example \ref{E(2)-simple-triple} where $\Sigma$ has genus $g-1$.
	Taking the connected sum of $\Sigma$ and a nomologically trivial
	torus, embedded in a 4-ball, produces a permissible triple $(X, w, \Sigma')$ such that $\Sigma'$ has genus $g$. Then $X\#_\Sigma X$ can be
	decomposed as the connected sum of $S^2\times S^2$ and another 4-manifold with $b^+>0$.
	Theorem \ref{gluing-S3} asserts that, for
	$\Gamma,\Lambda,\Gamma'$ and $\Lambda'  \in H_2(X)$, the following invariant vanishes:
	\begin{equation*}
		\rD_{X\#_{\Sigma'} X, w\#w,\Sigma}(\e^{(\Gamma\#\Gamma')_{(2)}+(\Lambda\#\Lambda')_{(3)}})
	\end{equation*}	
	If $|\lambda|_1<2g-2$, then we can find $\Gamma$ and $\Lambda$ such that $c_{X,w,\Sigma'}^{\lambda}$ is non-zero.
	Consequently, $h_{\lambda,\lambda}^{g,d}$ is zero for this choice of $\lambda$.
\end{proof}

In the light of the above proposition, let $h_{a,b}^{g,d}$ be $h_{\lambda,\lambda}^{g,d}$ for $\lambda=(a,b)\in \mathcal C_g\backslash \mathcal C_{g-1}$. These are the only non-zero terms among the coefficients of the pairing.
\begin{prop}\label{connected-sum-permissible}
	Suppose $(X, w,\Sigma)$ and $(X',w',\Sigma)$ are two permissible triples such that $w$ and $w'$ intersect $\Sigma$ in the
	same set of points with the same signs. Then:
	\begin{equation} \label{simple-type-connected-sum-permissible}
		{\rD_{X\#_\Sigma X', w\#w'}}((\frac{a_2}{3})^3z)={\rD_{X\#_\Sigma X', w\#w'}}(z)\hspace{.8cm}{\rD_{X\#_\Sigma X', w\#w'}}(a_3z)=0
	\end{equation}
	for $z\in \Sym^*(H_0(X\#_\Sigma X')\oplus \im(\#))^{\otimes 2}$.	Suppose also the intersection number of homology classes $\Gamma,\Lambda \in H_2(X)$ and
	$\Gamma',\Lambda' \in H_2(X')$ with $\Sigma$ is $1$.
	Then ${\widehat{\rD}_{X\#_\Sigma X', w\#w'}}(\e^{(\Gamma\#\Gamma')_{(2)}+(\Lambda\#\Lambda')_{(3)}})$ is equal to:
    	\begin{equation*}
		\e^{\frac{Q(\Gamma\#\Gamma')}{2}-Q(\Gamma\#\Gamma')}
		\sum_{(a,b)\in \mathcal C_g\backslash \mathcal C_{g-1}} h_{a,b}^{g,d}(t_2,t_3)\sum
		c_{ij}d_{i'j'}\e^{(\frac{{\sqrt 3}}{2}(M_{i,i'}+M_{j,j'}) \cdot \Gamma\#\Gamma'+
		\frac{{\sqrt 3}}{2}\bi(M_{i,i'}-M_{j,j'})\cdot \Lambda\# \Lambda'))}
	\end{equation*}	
	For each $(a,b)$, the second sum is over the pairs of basic classes $(K_i, K_j)$ and $(L_{i'},L_{j'})$ such that:
	\begin{equation} \label{basic classes pairs}
		\frac{(K_i+K_j)}{2}\cdot \Sigma=\frac{(L_{i'}+L_{j'})}{2}\cdot \Sigma=a\hspace{1cm}
		\frac{(K_i-K_j)}{2}\cdot \Sigma=\frac{(L_{i'}-L_{j'})}{2}\cdot \Sigma=b
	\end{equation}
	and  $M_{i,i'}$, $M_{j,j'}$ are respectively equal to $K_i\#L_{i'}$, $K_j\#L_{j'}$.
\end{prop}

 \begin{proof}
 	The series ${{\rD}_{X\#_\Sigma X', w\#w',\Sigma}}(\e^{(\Gamma\#\Gamma')_{(2)}+(\Lambda\#\Lambda')_{(3)}})$
	can be computed in terms of the cohomology classes
	$M_{i,i'}$ by plugging \eqref{relative-component} into \eqref{gluing-formula} and
	applying Proposition \ref{restriction-pairing-matrix}. Then we argue as in Lemma \ref{permissible-inv} to obtain the desired formula for
	the $\U(N)$-series ${\widehat{\rD}_{X\#_\Sigma X', w\#w'}}(\e^{(\Gamma\#\Gamma')_{(2)}+(\Lambda\#\Lambda')_{(3)}})$.
	We can follow a similar strategy to compute
	${\widehat{\rD}_{X\#_\Sigma X', w\#w'}}(a_2^ia_3^j\e^{(\Gamma\#\Gamma')_{(2)}+(\Lambda\#\Lambda')_{(3)}})$ in terms of the classes
	$M_{i,i'}$. The resulting formulas prove the identities in \eqref{simple-type-connected-sum-permissible}.
\end{proof}

The goal of the remaining part of this section is to determine the power series $h_{a,b}^{g,d}(t_2,t_3)$, up to two constants. Firstly, one can obtain a constraint on this power series by changing the orientation of $\Sigma$:
\begin{equation} \label{different-det-bdle}
	h_{a,b}^{g,2}(t_2,t_3)=h_{-a,-b}^{g,1}(-t_2,-t_3)
\end{equation}
We shall obtain more constraints by looking at some explicit 4-manifolds. In the case $g=1$, in fact we can determine $h_{0,0}^{1,d}(t_2,t_3)$ completely using our calculation of the invariants of $E(n)$:
 \begin{cor}
	For $g=1$, the only non-zero term among the pairing coefficients is given by:
	\begin{equation*}
		h^{1,d}_{0,0}(t_2,t_3)=(\hbar_1 \cosh(\sqrt{3}t_2)-2\hbar_2\cos(-\frac{2\pi}{3}d+\sqrt 3 t_3))^2.
	\end{equation*}
	In particular, if $(X,w,\Sigma)$ and $(X',w',\Sigma)$ are permissible triples such that the genus of $\Sigma$ is equal to $1$, and $w$, $w'$
	intersect $\Sigma$ in the same set of points with the same signs, then:
	\begin{equation} \label{gluing-torus}
		{\widehat{\rD}_{X\#_\Sigma X', w\#w'}}(\e^{(\Gamma\#\Gamma')_{(2)}+(\Lambda\#\Lambda')_{(3)}})=h_{0,0}^{1,d}
		{\widehat{\rD}_{X, w}}(\e^{\Gamma_{(2)}+\Lambda_{(3)}}) {\widehat{\rD}_{X', w'}}(\e^{\Gamma'_{(2)}+\Lambda'_{(3)}}).
	\end{equation}	
 \end{cor}

\begin{remark}
	Identity \eqref{gluing-torus} is a consequence of Proposition \ref{fiber-sum-universal-formula}, and it holds even if we only require (i), (ii) and (iii) of
	Definition \ref{permissible-type} for the triples $(X,w,\Sigma)$ and $(X',w',\Sigma)$.
\end{remark}

The 4-manifold $B(1):=E(1) \# \cp$ has two elliptic fibrations with fibers $f$ and $f'$ such that $f\cdot f'=1$. There are also disjoint embedded spheres $E$ and $E'$ in $B(1)$ such that $E$ (respectively, $E'$) intersects $f$ (respectively, $f'$) positively in one point and is disjoint from $f'$ (respectively, $f$). The fiber sum of $n$ copies of $B(1)$ along $f'$ produces a 4-manifold $B(n)$ which is diffeomorphic to $E(n) \# n\cp$ (Figure \ref{B(n)}). The exceptional spheres of $B(n)$ are denoted by $E^1$, $\dots$, $E^n$ where $E^i$ is given by the exceptional sphere $E$ in the $i^{\rm th}$ summand of $B(n).$\footnote{For the special case of exceptional spheres in $B(n)$, we deviate from our previous notation and put $i$ as a superscript.} Furthermore, one can glue copies of $f$ to produce a surface $f_n$ of genus $n$ with self-intersection zero. The homology class of $f_n$ is equal to $[S]-[E^1]-\dots -[E^n]$ for an appropriate surface $S$ of genus $n$ in $E(n)$. Similarly, one can glue copies of $E'$ to produce $\sigma_n$, a sphere of self-intersection $-n$, that is disjoint form $f_n$ and the exceptional spheres. The intersection number of $\sigma_n$ and $f'$ is equal to $1$. The following proposition shows that $(B(n),f_n,w)$ is permissible if $w \cdot f \nequiv 0$ mod 3:

\begin{figure}
	\centering
		\includegraphics[width=.8\textwidth]{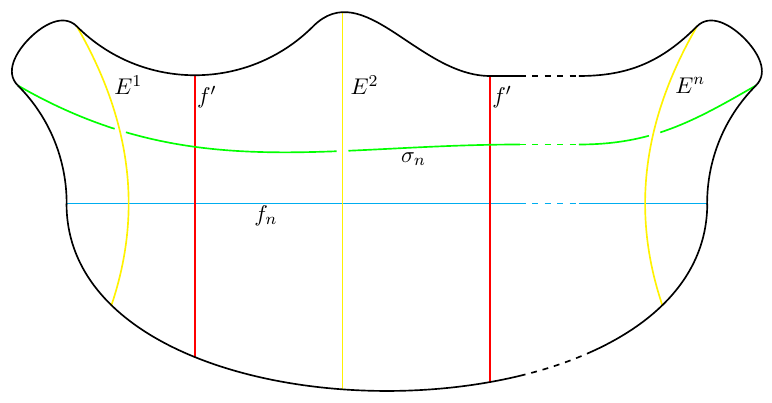}
		\caption{A schematic picture of $B(n)$: each embedded curve represents an embedded surface and each intersection point of these curves represents an intersection point. Note that $E^i$ and $\sigma_n$
		do not intersect. }
		\label{B(n)}
\end{figure}

 \begin{prop}\label{B(n)-invariants}	
	Suppose $w$ is a 2-cycle in $B(n)$ such that $w \cdot f_n \nequiv 0$ mod 3. For $n\geq 2$, the triple $(B(n),w, f_n)$ is permissible and
	${\widehat{\rD}_{B(n), w}}(\e^{\Gamma_{(2)}+\Lambda_{(3)}})$, for $\Gamma$, $\Lambda \in H_2(B(n))$, is given by:
	\begin{equation*}
		\e^{\frac{Q(\Gamma)}{2}-Q(\Lambda)}G'(\Gamma,\Lambda,w\cdot f')^{n-2}
		\prod_{i}\frac{1}{3}(\cosh(\sqrt{3}E^i\cdot \Gamma)+2\cos(-\frac{2\pi}{3}w\cdot E^i+ \sqrt 3 E^i \cdot \Lambda))
	\end{equation*}	
	Here $G'$ is given by \eqref{G} after replacing $f$ with $f'$. In particular, a basic class of $B(n)$ has the following form:
	\begin{equation*}
		(n-2k)f'\pm E^1\pm \dots \pm E^n
	\end{equation*}
	for $1\leq k \leq n-1$.
	 \end{prop}
 \begin{proof}
	This follows from the Blow-up formula and Theorem \ref{E(n)-inv}.
 \end{proof}

From the construction, it is clear that there is a diffeomorphism (see Figures \ref{B(n)B(n)} and \ref{b(2)dotsB(2)}):
 \begin{equation} \label{Phin}
	\Phi_n: B(n) \#_{f_n} B(n) \to B(2) \#_{f_2} \dots \#_{f_2} B(2)
 \end{equation}

\begin{figure}[t!]
	\centering
		\includegraphics[height=.3\textheight, width=.8\textwidth]{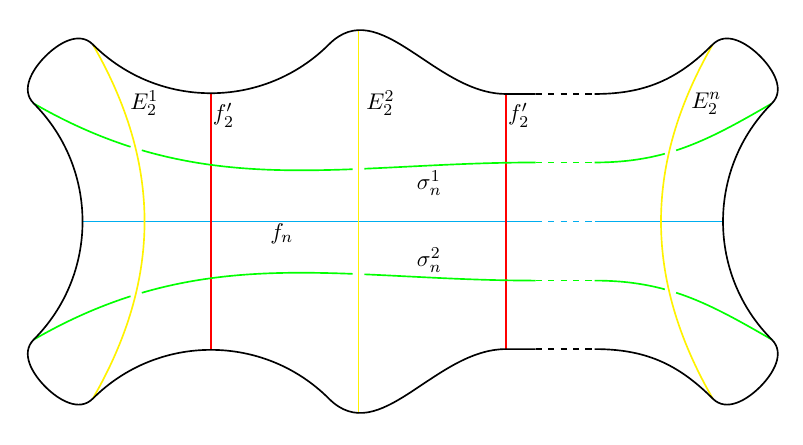}
		\caption{A schematic picture of $B(n)\#_{f_n}B(n)$}
		\label{B(n)B(n)}
\end{figure}
\begin{figure}[h!]
	\centering
		\includegraphics[height=.3\textheight, width=.8\textwidth]{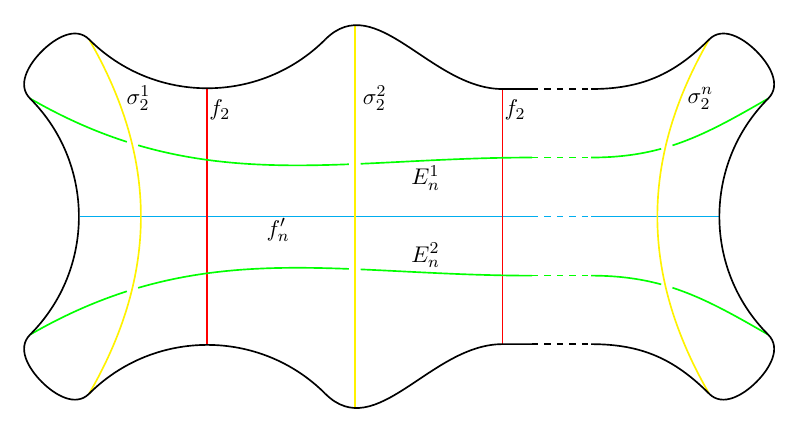}
		\caption{A schematic picture of $B(2) \#_{f_2} \dots \#_{f_2} B(2) $}
		\label{b(2)dotsB(2)}
\end{figure}

This diffeomorphism maps $f_n$ and $f_2':=f'\#f'$ in $B(n) \#_{f_n} B(n)$ to $f_n':=f' \# \dots \# f'$ and $f_2$ in $B(2) \#_{f_2} \dots \#_{f_2} B(2) $.  The sphere $\sigma_n\subset B(n)$ determines two spheres of self-intersection $-n$ in $B(n) \#_{f_n} B(n)$ which are denoted by $\sigma_n^1$ and $\sigma_n^2$. The diffeomorphism $\Phi_n$ maps $\sigma_n^i$ to $E^i_n:=E^i \# \dots \#E^i$. Therefore, the following elements of $\C[\![r_2,s_2,r_3,s_3,t_2,t_3]\!]$ are equal to each other:
 \begin{equation*}
	{\widehat{\rD}_{B(n) \#_{f_n} B(n), kf_2'+lf_n}}(\e^{(r_2f_n+s_2f_2')_{(2)}+(r_3f_n+s_3f_2')_{(3)}})=\hspace{2cm}
\end{equation*}	
\begin{equation}\label{symmetry}
	\hspace{2cm}={\widehat{\rD}_{B(2) \#_{f_2} \dots \#_{f_2} B(2) , kf_2+lf_n'}}(\e^{(r_2f_n'+s_2f_2)_{(2)}+(r_3f_n'+s_3f_2)_{(3)}})
 \end{equation}
\begin{prop} \label{genus-2-coeff}
	Suppose $(a, b) \in \mathcal C_2\backslash \{(0,0)\}$. Then:
	\begin{equation} \label{habge}
		h_{a,b}^{2,d}(t_2,t_3)=\sum_{(\gamma,\eta)\in \mathcal C_2\backslash \{(0,0)\}}h_{a,b,\gamma,\eta}^{2,d} \e^{\sqrt {3}\gamma t_2+
		\sqrt {3}i\eta t_3}
	\end{equation}	
	where $h_{a,b,\gamma,\eta}^{2,d}$ is a constant number.
\end{prop}
\begin{proof}
	The triple $(B(2),df'+2df_2,f_2)$ is permissible, and
	Proposition \ref{connected-sum-permissible} can be utilized to compute the following element of $\C[\![r_2,r_3,s_2,s_3,t_2,t_3]\!]$:
	\begin{equation*}
		{\widehat{\rD}_{B(2) \#_{f_2} B(2), df_2'+df_2}}(\e^{(s_2 f_2'+r_2f_2)_{(2)}+(s_3f_2'+r_3f_2)_{(3)}}).
	\end{equation*}	
	We can evaluate this series at $t_2=t_3=1$ to obtain a well-defined element of $\C[\![r_2,r_3,s_2,s_3]\!]$. This power series is equal to:
	\begin{equation}\label{C-2-inv}
	\e^{r_2s_2-2r_3s_3} \sum_{(a,b)\in \mathcal C_2\backslash \{(0,0)\} } h_{a,b}^{2,d}(s_2,\,s_3)\e^{\sqrt {3}ar_2+\sqrt {3}\bi br_3}
		(\sum_{(K_i,K_j)} c_{ij})^2
	 \end{equation}
	where the inner sum is over all pairs of basis classes $(K_i,K_j)$ of $(B(2),df'+2df_2,f_2)$ such that:
	\begin{equation*}
		\frac{(K_i+K_j)}{2}\cdot f_2=a\hspace{1cm} \frac{(K_i-K_j)}{2}\cdot f_2=b.
	\end{equation*}
	For each choice of $(a,b)$, the inner sum is non-zero.
	The identity \eqref{symmetry} shows that the expression \eqref{C-2-inv} is invariant with respect to the symmetry of $\C[\![r_2,r_3,s_2,s_3]\!]$
	that switches $r_2$ with $s_2$ and $r_3$ with $s_3$. This can be used to show that $h_{(a,b)}^{2,d}(t_2,\,t_3)$ has the form in \eqref{habge}.
\end{proof}

\begin{figure}[h]
	\centering
		\includegraphics[height=.3\textheight, width=.4\textwidth]{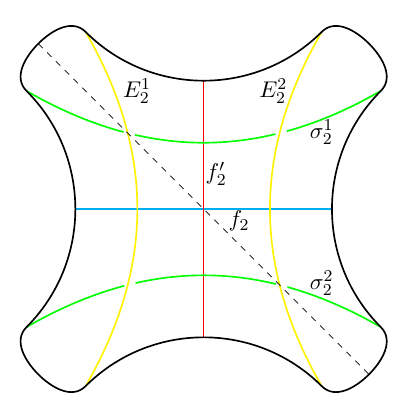}
		\caption{A schematic picture of $B(2)\#_{f_2}B(2)$: reflection with respect to the dashed line represents
		the diffeomorphism $\Phi_2$ of this 4-manifold }
		\label{B(2)B(2)}
\end{figure}

\begin{prop} \label{vanishing-constant-coeff-2}
	The constant numbers $h_{a,b,\gamma,\eta}^{2,d}$ are zero except possibly the following ones:
	\begin{equation*}
		h_{2,0,2,0}^{2,d}\hspace{1cm}h_{0,2,0,2}^{2,d}\hspace{1cm}h_{0,-2,0,-2}^{2,d}\hspace{1cm}h_{-2,0,-2,0}^{2,d}.
	\end{equation*}
	Furthermore, there are real numbers $\hbar_3$ and $\hbar_4$ such that:
	\begin{equation}
		h_{2,0,2,0}^{2,d}=h_{-2,0,-2,0}^{2,d}=\hbar_3\hspace{1cm}h_{0,2,0,2}^{2,d}=\zeta^{d}\hbar_4
		\hspace{1cm}h_{0,-2,0,-2}^{2,d}= \zeta^{-d}\hbar_4.
	\end{equation}
\end{prop}

\begin{proof}
	Firstly, for the purpose of brevity, let:
	\begin{equation*}
		{\rm N}_{X,w}(\Gamma,\Lambda):={\widehat{\rD}_{X,w}}(\e^{\Gamma_{(2)}+\Lambda_{(3)}})\e^{\frac{-Q(\Gamma)}{2}+Q(\Lambda)}.
	\end{equation*}
	Proposition \ref{genus-2-coeff} implies that ${\rm N}_{B(2) \#_{f_2} B(2), df_2'+d'f_2}(\Gamma,\Lambda)$, for $\Gamma, \Lambda \in \im(\#)$, has the following form:
	 \begin{equation*}
		\sum M_{a,b,\gamma,\eta}^{d,d'}(i,j,i',j')\e^{(\frac{{\sqrt 3}}{2}((K_i\#K_{i'}+K_j\#K_{j'}+2\gamma f_2) \cdot \Gamma)+
		\frac{{\sqrt 3}}{2}\bi((K_i\#K_{i'}-K_j\#K_{j'}+2\eta f_2)\cdot \Lambda))}
	\end{equation*}	
	where the sum is over the pairs $(a,b), (\gamma,\eta)\in\mathcal C_2 \backslash \{(0,0)\}$
	and the basic classes $K_i$, $K_j$, $K_{i'}$ and $K_{j'}$ of the permissible triple $(B(2),df'+2d'f_2,f_2)$ such that:
	\begin{equation*}
		\frac{(K_i+K_j)}{2}\cdot f_2=\frac{(K_{i'}+K_{j'})}{2}\cdot f_2=a\hspace{1cm} \frac{(K_i-K_j)}{2}\cdot f_2=\frac{(K_{i'}-K_{j'})}{2}\cdot f_2=b.
	\end{equation*}
	In the above expression, the constant number $M_{a,b,\gamma,\eta}^{d,d'}(i,j,i',j')$ is equal to $h_{a,b,\gamma,\eta}^{2,d}c_{ij}c_{i'j'}$
	where $c_{ij}$ and $c_{i'j'}$ are the coefficients associated to the pairs of basic classes $(K_i,K_j)$ and $(K_{i'},K_{j'})$ for the permissible triple
	$(B(2),df'+2d'f_2,f_2)$.
	We need the following elementary lemma:
	\begin{lemma}\label{sum-exponential}
		Suppose $V$ is a vector space and $\{f_i\}_{1\leq i\leq N}$ is a finite set of distinct complex valued
		linear functionals on $V$. Then the functions $\{\e^{f_i}\}_{1\leq i\leq N}$ are linearly independent over $\C$.
	\end{lemma}
	This Lemma is an immediate consequence of the existence of a line $l \subseteq V$ such that the restrictions $f_i|_l$ are distinct.
	We apply this lemma in the case that $V$ is the following subspace of $H_2(B(2)\#_{f_2}B(2))^{\oplus 2}$:
	\begin{equation*}
		\(\im(\#)\cap (\Phi_2)_*(\im(\#))\)\oplus\(\im(\#)\cap (\Phi_2)_*(\im(\#))\)
	\end{equation*}	
	Since $\Phi$ maps $df_2'+d'f_2$ to $df_2+d'f_2'$, we have:
	\begin{equation*}
		{\rm N}_{B(2)\#_{f_2}B(2),df_2'+d'f_2}(\Gamma,\Lambda)={\rm N}_{B(2)\#_{f_2}B(2),df_2+d'f_2'}((\Phi_2)_*(\Gamma),(\Phi_2)_*(\Lambda))
	\end{equation*}
	for $(\Gamma,\Lambda) \in V$. This identity implies that:
	 \begin{equation} \label{exponential-identity}
		\sum M_{a,b,\gamma,\eta}^{d,d'}(i,j,i',j')\e^{f_{a,b,\gamma,\eta}^{i,j,i',j'}(\Gamma,\Lambda)}
		-M_{a,b,\gamma,\eta}^{d',d}(i,j,i',j')\e^{g_{a,b,\gamma,\eta}^{i,j,i',j'}(\Gamma,\Lambda)}=0.
	\end{equation}	
	Here $f_{a,b,\gamma,\eta}^{i,j,i',j'}$ is defined by the following pair of cohomology classes:
	\begin{equation*}
		(\frac{{\sqrt 3}}{2}(K_i\#K_{i'}+K_j\#K_{j'}+2\gamma f_2),\,
		\frac{{\sqrt 3}}{2}\bi(K_i\#K_{i'}-K_j\#K_{j'}+2\eta f_2))	
	\end{equation*}
	and $g_{a,b,\gamma,\eta}^{i,j,i',j'}=\Phi_2^*(f_{a,b,\gamma,\eta}^{i,j,i',j'})$.
	In the case of $f_{a,b,\gamma,\eta}^{i,j,i',j'}$, this pair is equal to:
	\begin{equation*}
		(\frac{{\sqrt 3}}{2}(L_m^{\delta}+L_{m'}^{\delta'}), \frac{{\sqrt 3}}{2}\bi(L_{m}^{\delta}-L_{m'}^{\delta'}))
	\end{equation*}
	for an appropriate choice of $m$ and $\delta$, where $L_m^{\delta}$ is defined in the following table:
	
\begin{table}[H]
	\begin{centering}
		\begin{tabular}{|c|c|c|c|c|c|c|c|}
		\hline
		&$f_2$&$f_2'$&$E_2^1$&$E_2^2$&$\sigma_2^1$&$\sigma_2^2$\\
		\hline
		$L_1^\delta:=E_2^1+E_2^2+\delta f_2$&$2$&$\delta$&$\delta-2$&$\delta-2$&$0$&$0$\\
		\hline
		$L_2^\delta:=-E_2^1-E_2^2+\delta f_2$&$-2$&$\delta$&$\delta+2$&$\delta+2$&$0$&$0$\\
		\hline
		$L_3^\delta:=-E_2^1+E_2^2+\delta f_2$&$0$&$\delta$&$\delta+2$&$\delta-2$&$0$&$0$\\
		\hline
		$L_4^\delta:=E_2^1-E_2^2+\delta f_2$&$0$&$\delta$&$\delta-2$&$\delta+2$&$0$&$0$\\
		\hline
		$L_5^\delta:=\zeta_1-\zeta_2+\delta f_2$&$0$&$\delta$&$\delta$&$\delta$&$0$&$0$\\
		\hline
		$L_6^\delta:=\zeta_2-\zeta_1+\delta f_2$&$0$&$\delta$&$\delta$&$\delta$&$0$&$0$\\
		\hline
		$J_1^\delta:=\sigma_2^1+\sigma_2^2+\delta f_2'$&$\delta$&$2$&$0$&$0$&$\delta-2$&$\delta-2$\\
		\hline
		$J_2^\delta:=-\sigma_2^1-\sigma_2^2+\delta f_2'$&$\delta$&$-2$&$0$&$0$&$\delta+2$&$\delta+2$\\
		\hline
		$J_3^\delta:=-\sigma_2^1+\sigma_2^2+\delta f_2'$&$\delta$&$0$&$0$&$0$&$\delta+2$&$\delta-2$\\
		\hline
		$J_4^\delta:=\sigma_2^1-\sigma_2^2+\delta f_2'$&$\delta$&$0$&$0$&$0$&$\delta-2$&$\delta+2$\\
		\hline
		$J_5^\delta:=\xi_1-\xi_2+\delta f_2'$&$\delta$&$0$&$0$&$0$&$\delta$&$\delta$\\
		\hline
		$J_6^\delta:=\xi_2-\xi_1+\delta f_2'$&$\delta$&$0$&$0$&$0$&$\delta$&$\delta$\\
		\hline
		\end{tabular}
		\caption{Pairing of the cohomology classes $L_m^\delta$ and $J_m^\delta$ with some elements of $\(\im(\#)\cap (\Phi_2)_*(\im(\#))\)$}\label{pairing-table}
	\end{centering}
\end{table}

	There are similar formulas for $g_{a,b,\gamma,\eta}^{i,j,i',j'}$, where $L_m^{\delta}$ and $L_{m'}^{\delta'}$ are replaced with
	$J_m^{\delta}$ and $J_{m'}^{\delta'}$.
	In the above table, the cohomology classes $\zeta_1$ and $\zeta_2$ are equal to $E^1\#E^2$ and $E^2\#E^1$. Moreover, $\xi_i:=\Phi_2^*(\zeta_i)$. The table contains evaluations of
	$L_{m}^{\delta}$ and $J_{m}^\delta$ at some elements of $\im(\#)\cap (\Phi_2)_*(\im(\#))$.
	The evaluations of this table shows that the only possible identities among the following functionals on $V$:
	\begin{equation*}
		\hspace{2cm}L_{i}^{\delta}\hspace{1cm} J_{i}^{\delta} \hspace{2cm} 1\leq i \leq 4.
	\end{equation*}
	are the identities of the elements in the pairs $(L_{1}^{2},J_{1}^{2})$ and $(L_{2}^{-2},J_{2}^{-2})$.
	Therefore, Lemma \ref{sum-exponential} can be used to prove the first part of the proposition.
	The second part, is
	a consequence of the remaining information in \eqref{exponential-identity}, Identity \eqref{different-det-bdle} and rationality of polynomial invariants.
\end{proof}

\begin{prop} \label{non-van-h-3-4}
	The constant number $\hbar_3$ is non-zero.
\end{prop}

\begin{proof}
	Recall that the 4-manifold $X(m,n)$ is a branched double cover of $W(m,n)$ which is the blow up of ${\bf CP}^1\times {\bf CP}^1$ at $4mn$ points. Suppose $\pi:X(m,n) \to W(m,n)$ is the covering map.
	Since $\pi$ does not contract any curve, the pullback of any
	ample divisor is still ample by Nakai-Moishezon Criterion and projection formula \cite{Har:Alg-Geo}. In particular, the following divisor is ample:
	\begin{equation*}
		\pi^\ast(\{p\}\times {\bf CP}^1+{\bf CP}^1 \times \{q\} -\sum_i E_i)= f_{n-1}+f_{m-1}-\sum_{i=1}^{4mn} \pi^\ast(E_i)
	\end{equation*}
	where $\{E_i\}_{1\leq i \leq 4mn}$ is the set of exceptional classes.	

	Next we focus on the 4-manifold $X(3,4)$. Suppose $w_1=f_3$ and $w_2=f_3+f_2$. Thus there is a holomorphic line bundle $L_i$ on $X(3,4)$ that $c_1(L_i)$ is represented by $w_i$ and $w_i$
	can be decomposed as $w_i^1\#w_i^2$ with respect to the decomposition of $X(3,4)$ as $E(3)\#_{f_2}E(3)$.
	Let $S$ denote the ample class  $f_2+f_3- \sum_{i=1}^{48}\pi^*(E_i)$.
	According to Theorem \ref{nonvan}, the coefficient of $t_2^k$ in the series $\widehat {\rD}_{X(3,4),w}(\e^{S_{(2)}})\in \C[\![t_2]\!]$ is positive
	for large values of $k$. Since $X(3,4)=E(3)\#_{f_2}E(3)$ and the homology class $S$ lies in the image of $\#$, we can use Proposition \ref{vanishing-constant-coeff-2} to show that:
	\begin{equation*}
		\widehat {\rD}_{X(3,4),w_i}(\e^{S_{(2)}})={\rm e}^{\frac{Q(S)}{2}}[\frac{1}{2}\hbar_1^2 \hbar_3\cosh(\sqrt{3}(f_3+2f_2)\cdot S)+
		2\hbar_2^2\hbar_4\cos(-\frac{2\pi}{3}w_i\cdot (f_3+2f_2))]
	\end{equation*}	
	Therefore:
	\begin{equation*}
		2\widehat {\rD}_{X(3,4),w_1}(\e^{S_{(2)}})+\widehat {\rD}_{X(3,4),w_2}(\e^{S_{(2)}})=\frac{3}{2}\hbar_1^2 \hbar_3{\rm e}^{\frac{Q(S)}{2}}\cosh(\sqrt{3}(f_3+2f_2)\cdot S).
	\end{equation*}			
	This implies that $\hbar_3$ is a positive number (and $\hbar_1$ is non-zero).
\end{proof}

\begin{prop}
	The power series  $h^{g,d}_{a,b}(t_2,t_3)$ is zero, unless:
	\begin{equation*}
		(a,\,b) \in \{(\pm(2g-2),\,0),(0,\,\pm(2g-2))\}.	
	\end{equation*}
	Furthermore, we have:
	\begin{equation*}
		h^{g,d}_{\pm (2g-2),0}(t_2,t_3)=\hbar_3^{g-1}(\frac{2}{\hbar_1})^{2g-4}\e^{\pm2\sqrt 3 t_2} \hspace{1cm}
		h^{g,d}_{0,\pm(2g-2)}(t_2,t_3)=\hbar_4^{g-1}\hbar_2^{4-2g}\zeta^{\pm \frac{2\pi}{3}d}\e^{\pm2\sqrt 3 \bi t_3}.
	\end{equation*}
\end{prop}	
\begin{proof}
	This theorem can be proved by exploiting the diffeomorphism in \eqref{Phin} for $n=g$.
	Let the 2-cycle $w$ in $B(g) \#_{f_g} B(g)$ be equal to $df_2'+df_g$. Using Propositions \ref{connected-sum-permissible} and \ref{B(n)-invariants},
	we can show that ${\rm N}_{B(g) \#_{f_g} B(g),w}(\Gamma,\Lambda)$, for $\Gamma,\Lambda\in \mathcal H:=\im(\#)$, is equal to:
	 \begin{equation}\label{E-B(2g)B(2g)-1}
		\sum_{(a,b)\in \mathcal C_g\backslash \mathcal C_{g-1}} h_{a,b}^{g,d}(t_2,t_3)\sum
		c_{ij}c_{i'j'}\e^{(\frac{{\sqrt 3}}{2}(M_{i,i'}+M_{j,j'}) \cdot \Gamma+
		\frac{{\sqrt 3}}{2}\bi(M_{i,i'}-M_{j,j'})\cdot \Lambda)}
	\end{equation}	
	where the second sum is over the pairs of basic classes $(K_i, K_j)$ and $(K_{i'},K_{j'})$ of the permissible triple $(B(g),df'+2df_g,f_g)$ such that:
	\begin{equation} \label{basic-classes-f-g}
		\frac{(K_i+K_j)}{2}\cdot f_g=\frac{(K_{i'}+K_{j'})}{2}\cdot f_g=a\hspace{1cm}
		\frac{(K_i-K_j)}{2}\cdot f_g=\frac{(K_{i'}-K_{j'})}{2}\cdot f_g=b
	\end{equation} 	
	and $M_{i,i'}=K_i\#K_{i'}$ and $M_{j,j'}=K_j\#K_{j'}$.
	
	Recall that the 4-manifolds $B(g) \#_{f_g} B(g)$ and $B(2) \#_{f_2} \dots \#_{f_2} B(2)$ are diffeomorphic to each other using the diffeomorphism
	$\Phi_g$. Therefore, ${\rm N}_{B(g) \#_{f_g} B(g),w}(\Gamma,\Lambda)$ can be also computed by regarding $B(g) \#_{f_g} B(g)$ as the fiber sum
	of $g$ copies of $B(2)$ along surfaces of genus $2$. In particular, Propositions \ref{genus-2-coeff}
	and \ref{vanishing-constant-coeff-2} allow us to obtain the following explicit form for ${\rm N}_{B(g) \#_{f_g} B(g),w}(\Gamma,\Lambda)$:
	\begin{equation}\label{E-B(2g)B(2g)-2}
		\hbar_3^{g-1}(\frac{1}{36})^g\e^{\sqrt 3 M_g \cdot \Gamma}+\hbar_3^{g-1}(\frac{1}{36})^g\e^{-\sqrt 3 M_g \cdot \Gamma}
		+\hbar_4^{g-1}(\frac{\zeta^{d}}{9})^g\e^{\sqrt 3\bi M_g \cdot \Lambda}+
		\hbar_4^{g-1}(\frac{\zeta^{-d}}{9})^g\e^{-\sqrt 3\bi M_g \cdot \Lambda}
	\end{equation}
	where $M_g=\sigma_g^1+\sigma_g^2+2(g-1)f_2'$. However, this approach works for the homology classes $\Gamma, \Lambda$ $ \in \mathcal H'$ where $\mathcal H'$ is the
	image of the iterated applications of $\#$ using the decomposition of $B(g) \#_{f_g} B(g)$ as the fiber sum of $g$ copies of
	$B(2)$. Therefore, \eqref{E-B(2g)B(2g)-1} and \eqref{E-B(2g)B(2g)-2} are equal to each other for
	$\Gamma,\Lambda \in \mathcal H \cap \mathcal H'$. Fix $\Gamma,\Lambda \in \mathcal H \cap \mathcal H'$ and let:
	\begin{equation*}
		l:=\{\Gamma+sf_g \mid s \in \C\}\hspace{1cm}l':=\{\Lambda+sf_g \mid s \in \C\}
	\end{equation*}
	Applying Lemma \ref{sum-exponential} to the subspace $l\oplus l'$ of $(\mathcal H\cap \mathcal H')^{\oplus 2}$ shows that:
	\begin{equation*}
    		h^{g,d}_{\pm (2g-2),0}(t_2,t_3)(\frac{1}{36})^g(\frac{\hbar_1}{2})^{2g-4}\e^{\pm \sqrt 3 M_g' \cdot \Gamma}=
		\hbar_3^{g-1}(\frac{1}{36})^g\e^{\pm \sqrt 3 M_g \cdot \Gamma}
	\end{equation*}
	\begin{equation*}
		h^{g,d}_{0,\pm (2g-2)}(t_2,t_3)(\frac{1}{9})^g \hbar_2^{2g-4} \zeta^{\mp(2g-2)d} \e^{\pm \sqrt 3\bi M_g' \cdot \Lambda}=
		\hbar_4^{g-1}(\frac{\zeta^{\pm d}}{9})^g\e^{\pm \sqrt 3\bi M_g \cdot \Lambda}
	\end{equation*}
	where $M_g':=(g-2)f_2'+{E_2^1}+\dots+{E_2^g}$. Since $\hbar_1$ and $\hbar_2$ are non-zero \cite{DX:in-prep}, the above identity proves the second part of the proposition.
	If $(a,b)\notin\{(\pm(2g-2),0),(0,\pm(2g-2))\}$,
	then the same argument shows that:
	\begin{equation*}
		h_{a,b}^{g,d}(t_2,t_3)\sum
		c_{ij}d_{i'j'}\e^{(\frac{{\sqrt 3}}{2}(M_{i,i'}+M_{j,j'}) \cdot \Gamma+
		\frac{{\sqrt 3}}{2}\bi(M_{i,i'}-M_{j,j'})\cdot\Lambda)}=0
	\end{equation*}
	for $\Gamma,\Lambda \in \mathcal H\cap \mathcal H'$. This sum is over the pairs that satisfy \eqref{basic-classes-f-g}.
	Another application of Lemma \ref{sum-exponential}
	for the following homology classes in $\mathcal H\cap \mathcal H'$:
	\begin{equation*}
		\Gamma=s({E_2^1}\pm \dots\pm {E_2^g})\hspace{1cm}\Lambda=s({E_2^1}\pm \dots\pm{E_2^g})
	\end{equation*}	
	 shows that $h_{a,b}^{g,d}(t_2,t_3)$ has to be zero.
\end{proof}

The following theorem summarizes our results in this subsection:

\begin{theorem} \label{D-inv-con-sum}
	Suppose $(X, w,\Sigma)$ and $(X',w',\Sigma)$ are two permissible triples and the genus of $\Sigma$ is at least 2.
	Then the triple $(X\#_\Sigma X', w\#w',\Sigma)$
	is permissible with respect to the image of the map $\#:\mathcal K \to H_2(X\#_\Sigma X')$. The basic classes for $(X\#_\Sigma X', w\#w',\Sigma)$ are:
	\begin{equation} \label{Miip}
		M_{i,i'}^{\gamma}=K_i\#L_{i'}+2\gamma\Sigma
	\end{equation}
	where $K_i$, $L_{i'}$ are basic classes of $(X, w,\Sigma)$, $(X',w',\Sigma)$, $K_i\cdot \Sigma=L_{i'}\cdot \Sigma$
	and:
	\begin{equation*}
		(K_i\cdot \Sigma,\gamma) \in \{(2g-2,1),(-(2g-2),-1)\}.
	\end{equation*}
	For a pair of basic classes $M_{i,i'}^{\gamma}=K_i\#L_{i'}+2\gamma\Sigma$ and $M_{j,j'}^{\eta}=K_j\#L_{j'}+2\eta\Sigma$, let:
	\begin{equation*}
		\frac{(K_i+K_j)}{2}\cdot \Sigma=\frac{(L_{i'}+L_{j'})}{2}\cdot \Sigma=a\hspace{1cm}
		\frac{(K_i-K_j)}{2}\cdot \Sigma=\frac{(L_{i'}-L_{j'})}{2}\cdot \Sigma=b	
	\end{equation*}
	Then the coefficient associated to this pair is equal to $c_{i,j}d_{i',j'}h_{a,b,\gamma,\eta}^{g,d}$, where $c_{i,j}$ is the coefficient associated to
	$(K_i,K_j)$ for the triple $(X, w,\Sigma)$, $d_{i',j'}$ is the coefficient associated to $(L_{i'},L_{j'})$ for the triple $(X', w',\Sigma)$, and
	$h_{a,b,\gamma,\eta}^{g,d}$ is non-zero in the following cases:
	\begin{equation*}
		h^{g,d}_{(2g-2),0,1,0}=h^{g,d}_{-(2g-2),0,-1,0}=\hbar_3^{g-1}(\frac{2}{\hbar_1})^{2g-4}
	\end{equation*}
	\begin{equation*}
		h^{g,d}_{0,(2g-2),0,1}=\frac{\hbar_4^{g-1}}{\hbar_2^{2g-4}}\zeta^{d}\hspace{1cm}
		h^{g,d}_{0,-(2g-2),0,-1}=\frac{\hbar_4^{g-1}}{\hbar_2^{2g-4}}\zeta^{-d}.
	\end{equation*}
\end{theorem}

\section{Sutured Floer Homology}
\subsection{Eigenvectors} \label{ev}
For arbitrary $N$, we introduce a set of generators of the algebra $\VgdN$ in Corollary \ref{VgdN-generators}. In particular, in the special case that $N=3$ and $d\equiv 1$ or $2$ mod 3, we have the following generators of $\mathbb V_{g,d}^3$ (which will be denoted by $\Vgd$ from now on):
\begin{equation} \label{3-generators}
	\epsilon={\rm D}_{\Delta_g,\delta_{g,d}+\Sigma}(1) \hspace{.5 cm} \aleph_r={\rm D}_{\Delta_g,\delta_{g,d}}(a_r)\hspace{.5cm}
	o_r^j={\rm D}_{\Delta_g,\delta_{g,d}}(l^j_{(r)})\hspace{.5cm}\rho_{r}={\rm D}_{\Delta_g,\delta_{g,d}}(\Sigma_{(r)})
\end{equation}
where $r=2,3$ and $\{l^j\}_{1\leq j\leq 2g}$ is a set of generators for $H_1(\Sigma,\Z)$. If $y$ is one of the above elements, then the product $m(\cdot, y)$ defines an operator on $\Vgd$ which is still denoted by $y$. Recall that there is also a pairing on $\Vgd$ which is denoted by $\langle \,,\,\rangle$.

The operator $\epsilon$ is equal to $\rI_*(Y_g \times [0,1],\gamma_{g,d} \times [0,1]+ \Sigma,1)$ and the remaining operators can be described as:
\begin{equation} \label{op}
	\rI_*(Y_g \times [0,1],\gamma_{g,d} \times [0,1],q)
\end{equation}
where $q=a_r$, $l^j_{(r)}$ or $\Sigma_{(r)}$. This alternative description allows us to extend the definition of these operators to arbitrary admissible pairs. Suppose $(Y,\gamma)$ is a $3$-admissible pair, and $\Sigma$ is an embedded surface of genus $g$ in $Y$. We also assume that an integral basis $\{l^j\}_{1\leq j\leq 2g}$ for $H_1(\Sigma)$ is fixed. By replacing $(Y_g,\gamma_{g,d})$ with $(Y,\gamma)$, we can define analogues of the operators $\epsilon$, $\aleph_r$, $o_r^j$ and $\rho_{r}$ on $\rI_*(Y,\gamma)$. These operators on $\rI_*(Y,\gamma)$ are denoted by $\epsilon(\Sigma)$, $\aleph_r$, $o_r^j(\Sigma)$ and $\rho_{r}(\Sigma)$. In the case that the choice of $\Sigma$ is clear from our discussion, we drop $\Sigma$ from our notation for $\epsilon(\Sigma)$, $o_r^j(\Sigma)$ and $\rho_{r}(\Sigma)$. 

\begin{definition}
	An element $v\in \mathbb{V}_{g,d}$ is called an {\it exhaustive eigenvector} if it is a simultaneous eigenvector of the action of the
	operators in \eqref{3-generators}. An {\it exhaustive eigenspace} is the set of all
	exhaustive eigenvectors which have the same eigenvalues with respect to these operators. An exhaustive eigenvector $v$ is called non-degenerate if
	the pairing $\langle v,v\rangle \neq 0$.
\end{definition}

\begin{remark}
	Since $(o_r^j)^2=0$, the only eigenvalue of this operator is zero.
\end{remark}

Suppose $(X,w,\Sigma)$ is a permissible triple and $\Sigma$ is a surface of genus $g$ and $w \cdot \Sigma=d$. For a pair of basic classes $(K_i,K_j)$ of this triple, let $c_{i,j}$ denote the associated coefficient. Suppose also for a fixed $\lambda=(\alpha,\beta)\in \mathcal C_g$:
\begin{equation} \label{condition-basic-classes}
	\sum_{(K_i,K_j)} c_{i,j} \neq 0
\end{equation}
where the sum is over all pairs of basic classes $(K_i,K_j)$ that satisfy the following equality for $(a,b)=(\alpha,\beta)$:
\begin{equation} \label{basic-classes-con-a-b}
	\frac{(K_i+K_{j})}{2}\cdot \Sigma=a \hspace{1cm}
	\frac{(K_i-K_{j})}{2}\cdot \Sigma=b.
\end{equation}
Recall that $P_\lambda$ is the polynomial that satisfies \eqref{P-lambda}. Suppose $Q_{\lambda}\in \C[x,y,z]$ is defined as the evaluation of $P_{(a,b)}$ at $t_2=t_3=0$ and consider the following element of $\Vgd$:
\begin{equation*}
	\vab:=\rD_{X^\circ,w^\circ ,\Sigma}(Q_{\lambda}(a_2,\Sigma_{(2)},\Sigma_{(3)})).
\end{equation*}
\begin{prop}\label{eigen}
	The element $\vab\in \Vgd$ is a non-zero exhaustive eigenvector.
	The eigenvalues of $\vab$ with respect to the actions of $\epsilon$, $\aleph_2$, $\aleph_3$, $\rho_2$ and $\rho_3$ are respectively equal to
	$1$, $3\zeta^{2d \beta}$, $0$, $ \zeta^{d\beta }\sqrt 3 \alpha$ and $\zeta^{2d \beta}\sqrt 3 \bi \beta$. Furthermore, if $(\alpha,\beta)=(\pm(2g-2),0)$, then the eigenvector $\vab$ is non-degenerate.
\end{prop}
\begin{proof}
	Lemma \ref{permissible-inv} can be used to show that for an arbitrary polynomial $P\in \C[x,y,z]$:
	\begin{equation} \label{equality-a-S}
		\rD_{X,w,\Sigma}(P(a_2,\Sigma_{(2)},\Sigma_{(3)}))=
		\sum_{(a,b)\in \mathcal C_g} \zeta^{d b d_w} P(3\zeta^{2d b}, \zeta^{bd}\sqrt 3 a,\zeta^{2bd}\sqrt 3 \bi b)\sum_{(K_i,K_j)} c_{i,j}
	\end{equation}	
	where the inner sum is over all pairs of basic classes $(K_i, K_j)$ that satisfy \eqref{basic-classes-con-a-b}
	and $c_{i,j}$ is the coefficient associated to the pair $(K_i,K_j)$.
	Functorial properties of Floer homology imply that the pairing of $\vab$ and $\rD_{\Delta_g, \delta_{g,d}}(1)$ is equal to
	$\rD_{X,w,\Sigma}(Q_\lambda(a_2,\Sigma_{(2)},\Sigma_{(3)}))$ which is non-zero. Therefore, $\vab$ is a non-zero vector. Using the
	non-degeneracy of the pairing on $\Vgd$, the claim that $\vab$ is an exhaustive eigenvector can be translated to claims about the
	$\U(3)$-invariants of $(X,w)$. In particular, \eqref{equality-a-S} shows that $\vab$ is an eigenvector of $\epsilon$, $\aleph_2$, $\rho_2$ and
	$\rho_3$. The vector $\vab$ is in the kernel of the operators $\aleph_{3}$ and $o_r^j$ because $X$ has $w$-simple type and $b_1(X)=0$.
	The pairing $\langle \vab,\vab \rangle$ can be also computed by Theorem \ref{D-inv-con-sum}. Using Proposition \ref{non-van-h-3-4},
	this number is non-zero for $(\alpha,\beta)=(\pm(2g-2),0)$.
\end{proof}
Example \ref{E(2)-simple-triple} gives a permissible triple such that the condition in \eqref{condition-basic-classes} is satisfied for any $\lambda \in \mathcal C_g$. Therefore, for each $\lambda$ there is an exhaustive eigenvector in $\Vgd$. The condition in \eqref{condition-basic-classes} is not very essential in constructing such an eigenvector. This condition is used to show that $\vab$ is a non-zero element of $\Vgd$. It is possible to replace $\vab$ with the following element:
\begin{equation*}
	\vab':=\rD_{X^\circ,w^\circ ,\Sigma}(Q_{\lambda}(a_2,\Sigma_{(2)},\Sigma_{(3)})z)
\end{equation*}
where $z \in \A(X^{\circ})^{\otimes 2}$. If the triple $(X,w,\Sigma)$ has at least one pair of basic classes $(K_i,K_j)$ which satisfy \eqref{basic-classes-con-a-b} for $(a,b)=(\alpha,\beta)$, then $z$ can be chosen such that the above element of $\Vgd$ is non-zero.

\begin{prop}\label{exh}
	An exhaustive eigenspace is 1-dimensional.
\end{prop}
\begin{proof}
	Suppose $V\subset \Vgd$ is an exhaustive eigenspace, and $s_1$, $s_2$, $s_3$, $s_4$ and $s_5$ are respectively the corresponding eigenvalues of
	$\epsilon$, $\aleph_2$, $\aleph_3$, $\rho_2$ and $\rho_3$. Suppose also
	$\mathfrak I\subset \Vgd$ is the ideal generated by the elements of the following set:
	\begin{equation}\label{definition-G}
		G:=\{\epsilon-s_1,\aleph_2-s_2, \aleph_3-s_3,\rho_2-s_4,\rho_3-s_5\} \cup \{o_r^j\mid 1\leq j\leq 2g,\, 1\leq r \leq2\}
	\end{equation}
	Then an element of $\mathfrak I$ is the sum of the elements of the form $m(x,y)$ with $x\in G$ and $y \in \Vgd$. For any $v \in V$:
	\begin{equation*}
		\langle v,m(x,y)\rangle=\langle m(v,x),y\rangle=0
	\end{equation*}
	Therefore, $V$ is orthogonal to $\mathfrak I$. Since $\Vgd/\mathfrak I$ is a 1-dimensional vector space and the pairing is non-degenerate,
	the dimension of the vector space $V$ is at most 1.
\end{proof}

\begin{lemma} \label{HFgS}
	Suppose $v\in \Vgd$ is a non-degenerate exhaustive eigenvector, and
	$s_1$, $s_2$, $s_3$, $s_4$ and $s_5$ are respectively the corresponding eigenvalues of $\epsilon$, $\aleph_2$, $\aleph_3$, $\rho_2$ and $\rho_3$.
	Then the following space:
	\begin{equation*}
		H:=\gker(\epsilon-s_1)\bigcap\gker(\aleph_2-s_2)\bigcap\gker(\aleph_3-s_3)\bigcap\gker(\rho_2-s_4)\bigcap \gker(\rho_3-s_5)
	\end{equation*}
	is 1-dimensional. Here $\gker(T)$, for an operator $T$, is the union of the kernel of the operators $T^k$ for all values of $k$.
\end{lemma}
\begin{proof}
	Suppose the claim does not hold and $v'\in H$  is a vector which is linearly independent of $v$. Let $G$ be defined as in \eqref{definition-G}.
	All the operators involved in the definition of $H$ have even degree with respect to the $\Z/2\Z$-grading of $\Vgd$, induced by the $\Z/12\Z$-grading.
	Therefore, $H$ can de decomposed as $H_0\oplus H_1$ with respect to the $\Z/2\Z$-grading of $\Vgd$, and we may assume that $v' \in H_i$ for
	$i=0$ or $1$.
	By Proposition \ref{exh}, there is $x_0 \in G$ such that $m(v',x_0) \neq 0$. 
	Since the restriction of the elements of $G$ on $H$ are nilpotent,
	without loss of generality, we can also assume that the product of 
	$m(v',x_0)$ and any element $x \in G$ is zero.
	Therefore, $m(v',x_0)=cv$ for a
	non-zero complex number $c$.
	This implies that:
	\begin{equation*}
		\langle m(v',x_0),m(v',x_0)\rangle=\langle cv,cv\rangle\neq0
	\end{equation*}
	On the other hand, we have:
	\begin{equation*}
		\langle m(v',x_0),m(v',x_0)\rangle=\pm \langle m(m(v',x_0),x_0),v'\rangle=0
	\end{equation*}	
	which is a contradiction.
\end{proof}

\begin{prop}\label{ev-gen-Y}(\cite[Proposition 7.2]{KM:suture})
	Suppose $(Y,\gamma)$ is a $3$-admissible pair and $\Sigma$ is an embedded surface in $Y$ of genus $g$ 
	such that $\gamma\cdot \Sigma \equiv d$ mod $3$. If $(s_1,s_2,s_3,s_4,s_5)$ is a simultaneous eigenvalue of 
	the operators $(\epsilon(\Sigma), \aleph_2,\aleph_3,\rho_2(\Sigma),\rho_3(\Sigma))$, then 
	is also a simultaneous eigenvalue of the operators $(\epsilon, \aleph_2,\aleph_3,\rho_2,\rho_3)$
	acting on $\Vgd$.
\end{prop}

By Proposition \ref{eigen}, $v_{2g-2,0}$ is a non-degenerate exhaustive eigenvector of $\Vgd$. Suppose $s^g_1$, $s^g_2$, $s^g_3$, $s^g_4$ and $s^g_5$ denote the corresponding eigenvalues of $\epsilon$, $\aleph_2$, $\aleph_3$, $\rho_2$ and $\rho_3$. Then $s^g_1=1$, $s^g_2=3$, $s^g_3=0$, $s^g_4=\sqrt 3 (2g-2)$ and $s^g_5=0$. Following \cite{KM:suture}, we can define a variation of instanton Floer homology:
\begin{definition}
	Suppose $(Y,\gamma)$ is a $3$-admissible pair and $\Sigma$ is an embedded surface in $Y$ of genus $g$ 
	such that $\gamma\cdot \Sigma \equiv d$ mod $3$. Then $\I_*(Y,\gamma | \Sigma)$ is defined as:
	\[
	  \gker(\epsilon(\Sigma)-s^g_1)\bigcap\gker(\aleph_2-s^g_2)\bigcap\gker(\aleph_3-s^g_3)\bigcap
	  \gker(\rho_2(\Sigma)-s^g_4)\bigcap \gker(\rho_3(\Sigma)-s^g_5)
	\]
	In this definition, we allow $\Sigma$ to have more than one connected component. 
	In that case, each connected component $\Sigma'$ of $\Sigma$ is required to have genus $g$ and 
	$\gamma\cdot \Sigma' \equiv d$ mod 3. 
	In the definition of $\I_*(Y,\gamma | \Sigma)$, we include the operators $\epsilon(\Sigma')$, $\rho_2(\Sigma')-s^g_4$ 
	and $\rho_3(\Sigma')-s^g_5$ for each connected component $\Sigma'$ of $\Sigma$.
\end{definition}

\begin{remark} \label{genus-1-SHI}
	In the case that $g=1$, the action of the operators $\aleph_3$, $\rho_2$, $\rho_3$ on $\mathbb V_{1,d}$ are trivial,
	 the operator $\aleph_2$ is equal to $3\epsilon^{-1}$ and $\epsilon^3=1$. Therefore,
	similar relationships hold among the operators $\epsilon(\Sigma)$, 
	$\aleph_2$, $\aleph_3$, $\rho_2(\Sigma)$, $\rho_3(\Sigma)$ acting on $\I_*(Y,\gamma)$, where $\Sigma$
	is a genus one surface in $Y$. This can be verified in a similar way as in \ref{ev-gen-Y}. 
	(See \cite[Proposition 7.2]{KM:suture}.) Therefore, $\I_*(Y,\gamma | \Sigma)$, for a genus one surface $\Sigma$, is 
	simply equal to $\ker(\aleph_2-3)$.
\end{remark}

This variant of instanton Floer homology is also functorial with respect to cobordisms. Suppose $(W,w):(Y_0,\gamma_0) \to (Y_1,\gamma_1)$ is a cobordism of admissible pairs, $z \in \A(W)^{\otimes 2}$, and $\Sigma_i$ is an embedded oriented and connected surface in $Y_i$ such that $\Sigma_i\cdot \gamma_i\equiv d$ mod 3, and $\Sigma_0$, $\Sigma_1$ induce the same homology classes of $W$. More generally, if $\Sigma_1$ is disconnected, then each connected component of $\Sigma_1$ is required to be homologous to a connected component of $\Sigma_0$ inside $W$. Properties of instanton Floer homology discussed in Subsection \ref{IN} implies that $\rI_*(W,w,z)$ maps $\I_*(Y_0,\gamma_0 | \Sigma_0)\subseteq \I_*(Y_0,\gamma_0)$ to $\I_*(Y_1,\gamma_1 | \Sigma_1)\subseteq \I_*(Y_1,\gamma_1)$. Moreover, suppose $(X,w)$ is a cobordism from an admissible pair $(Y,\gamma)$ to the empty pair and $z \in \A(X)^{\otimes 2}$. Then the restriction of the map $\rD^{X,w}(z)$ gives rise to a functional on $\rI_*(Y,\gamma|\Sigma)$ which is denoted with the same notation.

Lemma \ref{HFgS} asserts that $\rI_*(Y_g,\gamma_{g,d}|\Sigma)$ is 1-dimensional. The following pairs from Subsection \ref{HF-S1-S} define cobordisms from two copies of $(Y_g,\gamma_{g,d})$ to the empty pair:
\begin{equation*}
	(\Omega_g,\omega_{g,d}) \hspace{1cm}(\Delta_g\coprod \Delta_g,\delta_{g,d}\coprod \delta_{g,d}).
\end{equation*}
Therefore, they determine two functionals $\rD^{\Omega_g,\omega_{g,d}}(1)$ and $\rD^{\Delta_g,\delta_{g,d}}(1) \otimes \rD^{\Delta_g,\delta_{g,d}}(1)$ on the 1-dimensional vector space $\rI_*(Y_g,\gamma_{g,d}|\Sigma) \otimes \rI_*(Y_g,\gamma_{g,d}|\Sigma)$. The non-degeneracy of the exhaustive eigenvector involved in the definition of $\rI_*(Y_g,\gamma_{g,d}|\Sigma)$ implies that the former functional is non-zero. Therefore, we have the following lemma which provides the distinguishing property of working with a non-degenerate exhaustive eigenspace for us:
\begin{lemma} \label{Omega=2Delta}
	The map $\rD^{\Delta_g,\delta_{g,d}}(1) \otimes \rD^{\Delta_g,\delta_{g,d}}(1):\rI_*(Y_g,\gamma_{g,d}|\Sigma) \otimes \rI_*(Y_g,\gamma_{g,d}|\Sigma) \to \C$ is a multiple of
	$\rD^{\Omega_g,\omega_{g,d}}(1):\rI_*(Y_g,\gamma_{g,d}|\Sigma) \otimes \rI_*(Y_g,\gamma_{g,d}|\Sigma) \to \C$.
\end{lemma}

\subsection{Excision and Sutured Manifolds Invariants}\label{suture}
Suppose $R_1$ and $R_2$ are two embedded surfaces of genus $g\geq 1$ in a 3-manifold $Y$. Suppose also there is a 1-cycle $\gamma$ in $Y$ such that $\gamma\cdot R_1=\gamma \cdot R_2$. We also assume $\gamma$ intersects $R_1$ and $R_2$ transversally, and all the intersection points have the same signs. Fix a diffeomorphism $\phi:R_1 \to R_2$ such that $\phi$ maps $\gamma \cap R_1$ to $\gamma \cap R_2$. We cut $Y$ along the surfaces $R_1$, $R_2$, and then identify the four boundary components of the resulting 3-manifold using $\phi$ such that the final 3-manifold, ${\widetilde Y}$, is an
oriented closed 3-manifold with embedded surfaces $\widetilde R_1$ and $\widetilde R_2$. Our assumption on $\phi$ implies that $\gamma$ determines a 1-cycle $\widetilde{\gamma}$ in $\widetilde{Y}$. We will also write $R$ (respectively, $\widetilde R$) for the union $R_1 \cup R_2$ (respectively, $\widetilde R_1 \cup \widetilde R_2$). Now we are ready to state our excision theorem:

\begin{theorem}\label{excision}
	The following Floer homology groups are isomorphic:
	\begin{equation*}
		\I_*(Y,\gamma|R)=\I_*({\widetilde Y}, {\widetilde\gamma}|{\widetilde R})
	\end{equation*}
\end{theorem}

\begin{proof}
	This theorem is the analogue of excision theorem for $\U(2)$-instanton Floer homology.
	The $\U(2)$ version of the excision theorem is proved in \cite{F:sur-rel} for $g=1$ (see also \cite{DB:sur-rel}) and in \cite{KM:suture} for higher values of $g$. We follow the same
	strategy as in the $\U(2)$ case to prove the theorem. In particular, the isomorphism between $\I_*(Y,\gamma|R)$ and $\I_*({\widetilde Y}, {\widetilde\gamma}|{\widetilde R})$ is induced by a cobordism of pairs
	$(W,w):(Y,\gamma) \to ({\widetilde Y}, {\widetilde\gamma})$.
	Let $Y^\circ$ be the complement of a regular neighborhood of $R$ in $Y$. Then the cobordism $W$ is the result of gluing $[0,1] \times Y^\circ$ to
	$\mathcal P \times R_1$ where $\mathcal P$ is the saddle cobordism in Figure \ref{saddle}.
	\begin{figure}
	\centering
		\includegraphics[]{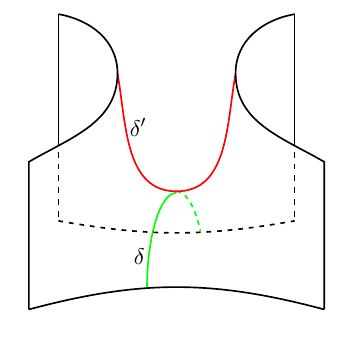}
		\caption{The saddle cobordism $\mathcal P$; the union of the two vertical boundary components on the left (respectively, right) is denoted by $\partial^v_l \mathcal P$ (respectively, $\partial^v_r \mathcal P$).}
		\label{saddle}
	\end{figure}
	The boundary of the 3-manifold $Y^\circ$ is equal to $R_1\cup \overline R_1\cup R_2\cup \overline R_2$. Then $[0,1]\times (R_1\cup \overline R_1) \subset [0,1] \times Y^\circ$ is glued to
	$\partial^v_l \mathcal P \times R_1$ by the identity map and $[0,1]\times (R_2\cup \overline R_2) \subset [0,1] \times Y^\circ$ is glued to
	$\partial^v_r \mathcal P \times R_1$ by the map $\phi$. (For the definition of $\partial^v_l \mathcal P$ and $\partial^v_r \mathcal P$ see Figure \ref{saddle}.) The surface cobordism $w:\gamma \to \widetilde \gamma$
	is also constructed in a similar way. This surface is given by gluing one copy of $\mathcal P$ for each intersection point in $\gamma \cap R_1$ to $[0,1]\times (\gamma\cap Y^{\circ})$. Reversing the cobordism
	$(W,w)$ determines another cobordism $(\overline W,\overline w): ({\widetilde Y}, {\widetilde\gamma}) \to (Y,\gamma)$. In order to prove the excision theorem, we claim that:
	\begin{equation} \label{composite-maps}
		\rI_*(\overline W, \overline w) \circ \rI_*(W,w):\rI_* (Y,\gamma|R)\to  (Y,\gamma|R), \hspace{.2cm}
		\rI_*(\overline W, \overline w) \circ \rI_*(W,w): \rI_* (\widetilde Y,\widetilde \gamma|\widetilde R)\to(\widetilde Y,\widetilde  \gamma|\widetilde  R)
	\end{equation}
	are non-zero multiples of the identity map. In the composite cobordism $\overline W \circ W:Y \to Y$, a copy of $\mathcal P\times R_1$ is glued to another copy of $\mathcal P\times R_1$ with the reverse orientation.
	In this part of $\overline W \circ W$, the union of the two copies of $\delta\times R_1$ gives rise to a copy of $Y_g:=S^1 \times R_1$. The intersection of $\overline w \circ w$ with $S^1 \times R_1$ produces a copy of
	$\gamma_{g,d}$. According to Lemma \ref{Omega=2Delta} and functoriality of $\rI_*$,
	replacing a neighborhood of $S^1 \times R_1$ with $(\Delta_g\coprod \Delta_g,\delta_{g,d}\coprod \delta_{g,d})$ does not change the map
	$\rI_*(\overline W\circ W, \overline w\circ w)$, up to multiplication by a non-zero constant number.
	But the resulting cobordism of the pair is the product cobordism $ ([0,1] \times Y,[0,1] \times \gamma)$. Therefore, the
	first map in \eqref{composite-maps} is a non-zero multiple of the identity map. Replacing $\delta$ with $\delta'$ and using the same argument proves a similar result for the second map in \eqref{composite-maps}.	
\end{proof}

The following proposition is the analogue of of Corollary 4.8 in \cite{KM:suture} and can be proved in a similar way using the excision theorem:

\begin{prop}\label{fibermfd}
	Let $Y$ be a 3-manifold, $\gamma$ be a 1-cycle in $Y$, and $R\subset Y$ be a connected surface of genus $g\geq 1$ such that
	$\gamma \cdot R \nequiv 0$ mod 3 and the intersection points of $\gamma\cap R$ are transversal and have the same signs.
	Let $\widetilde Y$ be the 3-manifold obtained by cutting $Y$ along $R$ and regluing by an orientation preserving diffeomorphism $\phi:R \to R$.
	Suppose $\phi$ maps $R\cap \gamma$ to  $R\cap \gamma$, and $\widetilde \gamma$, $\widetilde R$ are the induced 1-cycle and the embedded
	surface in $\widetilde Y$. Then
	\begin{equation*}
		\I_*(Y,\gamma|R)\cong \I_*(\widetilde{Y},\widetilde{\gamma}|\widetilde R).
	\end{equation*}
\end{prop}

Now we can define invariants for balanced sutured manifolds almost verbatim from \cite{KM:suture}. Firstly we recall the definition of balanced sutured manifolds (cf. \cite{gabai,juhasz,KM:suture}) :

\begin{definition} \label{balance-suture-manifolds}
	A {\it sutured manifold} $(M,\alpha)$ consists of an oriented 3-manifold $M$, an oriented closed 1-manifold $\alpha$ in $\partial M$
	and a decomposition of $\partial M$ as:
	\begin{equation} \label{bdry-decom}
		\partial M= R_+ \cup A(\alpha)\cup R_-.
	\end{equation}
	Each connected component of $\alpha$ is called a {\it suture} and $A(\alpha)$ is the closure of a tubular neighborhood of $\alpha$.
	The spaces $R_+$ and $R_-$ are disjoint and each of them is a union of some of
	the connected components of $\overline{\partial M\backslash A(\alpha)}$.
	In particular, each component of $\partial R_+$ and $\partial R_-$ is a parallel copy of a suture.
	Suppose $R_+$ and $R_-$ are oriented with the outward-normal-first convention.
	Similarly, each component of $\partial R_+$ (respectively, $\partial R_-$) inherits an orientation from $R_+$ (respectively, $R_-$). This orientation
	is required to agree (respectively, disagree) with the orientation of the corresponding suture.
	The sutured manifold $(M,\alpha)$ is {\it balanced} if neither $M$ nor $R_\pm$ has closed components,
	and $\chi (R_+) =\chi (R_-)$. 
\end{definition}
Note that the required conditions on $R_+$, $R_-$ and $A(\alpha)$ imply that the decomposition \eqref{bdry-decom} is unique.

\begin{example}
	Suppose $F_{g,k}$ is a surface of genus $g$ with $k\geq 1$ boundary components which is not the 2-dimensional disc. Then
	$( [-1,1] \times F_{g,k},\partial F_{g,k} \times \{0\})$ determines a balanced sutured manifold.
	The decomposition of the boundary of this product sutured manifold is given as below:
	\begin{equation}
		R_+=\{1\}\times F_{g,k} \hspace{1cm} A(\alpha)=[-1,1]\times \partial F_{g,k} \hspace{1cm}R_-=\{-1\}\times F_{g,k}.
	\end{equation}
\end{example}

\begin{example} \label{knot-suture}
	Suppose $Y$ is a 3-manifold and $K\subset Y$ is a knot.
	 Let $M(K)$ be the complement of a regular neighborhood of  $K$ in $Y$, and
	$\alpha(K)$ be the union of two oppositely oriented meridional curves on the boundary of $M(K)$.
	Then $(M(K),\alpha(K))$
	determines a balanced sutured manifold.
	The manifolds $R_+$ and $R_-$ are homeomorphic to $[0,1] \times S^1$.
\end{example}
\begin{example} \label{Seifert-suture}	
	Suppose $K$ is a null-homologous knot in $Y$ and $S$ is a Seifert surface for $K$.
	We can also associate a sutured manifold $(N(S),\alpha(S))$ to $S$.
	The three manifold $N(S)$ is defined to be the $Y\backslash \((-1,1)\times {\rm int}(S)\)$, where
	$(-1,1)\times S$ is a regular neighborhood of $S$ in $Y$. The only suture $\alpha(S)$ of this
	sutured manifold is defined to be $\{0\}\times \partial S$.
\end{example}

Let $(M,\alpha)$ be a balanced sutured manifold such that the number of sutures is greater than one and $R_+$, $R_-$ are not a union of 2-dimensional discs. We attach the product sutured manifold $([-1,1]\times F_{0,k}, \{0\}\times \partial F_{0,k})$ to $(M,\alpha)$ where $k$ is the number of sutures of $(M,\alpha)$. More precisely, we glue $M$ to $[-1,1]\times F_{0,k}$ by identifying $A(\alpha)$ with $[-1,1]\times \partial F_{0,k}$ using an orientation reversing map. The resulting manifold is oriented and has two boundary components which are given as below:
 \begin{equation*}
	\bar{R}_+=R_+\cup \{1\}\times F_{0,k} \hspace{1cm}\bar{R}_-=R_-\cup \{-1\}\times F_{0,k}
 \end{equation*}
 Since $(M,\alpha)$ is balanced, the oriented surfaces $\bar{R}_+$ and $\bar{R}_-$ have the same positive genus. We choose an orientation reversing diffeomorphism $\phi:\bar{R}_+ \to \bar{R}_-$. Identifying $\bar{R}_+$ and $\bar{R}_-$ using the map $\phi$ determines a closed 3-manifold $Y_\phi$ which is called a {\it closure} of the sutured manifold $(M,\alpha)$. The 3-manifold $Y_\phi$ only depends on $(M,\alpha)$ and the choice of the diffeomorphism $\phi$. The surface $\bar{R}_+$ also induces an oriented surface in $Y_\phi$ which is denoted by $\bar{R}$. We also fix a point $y$ on $F_{0,k}$ and require that $\phi$ maps $(1,y) \in \bar{R}_+$ to $(-1,y) \in \bar{R}_-$. Therefore, the path $[-1,1]\times \{y\}\subset [-1,1]\times F_{0,k}$ induces a 1-cycle $\gamma_\phi$ in $Y_\phi$.

\begin{definition}\label{SHI-def}
	The {\it sutured instanton homology} of the sutured manifold $(M,\alpha)$ is defined as
	\begin{equation}\label{SHI-eq}
		\SHI_* (M,\alpha):= \rI_* (Y_\phi,\gamma_\phi|\bar{R}).
	\end{equation}
	If $(M,\alpha)$ has one suture or one of $R_+$, $R_-$ is a union of discs, we replace $F_{0,k}$ with 
	$F_{1,k}$ in the definition of $(Y,\gamma,\bar{R})$ and then extend sutured instanton homology
	to these sutured manifolds using \eqref{SHI-eq}.
\end{definition}
Proposition \ref{fibermfd} implies that sutured instanton homology $\SHI_*$ is well-defined. 
\begin{remark}
	In the definition of the closure of a sutured manifold $(M,\alpha)$, we can replace $F_{0,k}$ with $F_{g,k}$ 
	for an arbitrary $g$. Then for each choice of $g$, we can define a sutured Floer homology group as above.
	Using our excision theorem and the method of \cite{KM:suture}, we can show that the rank of
	these sutured Floer homology groups is non-decreasing in $g$.
	We expect that these sutured Floer homology groups for various choices of $g$ are isomorphic to each other. 
	However, proving this seems to need
	a further study of the algebra $\mathbb V_{g,1}$. We hope to come back to this issue elsewhere.
\end{remark}

\subsection{Instanton Knot Homology}\label{KHI}

\begin{definition}
	Given a knot $K$ in a 3-manifold $Y$, let $(M(K),\alpha(K))$ be the sutured manifold of Example \ref{knot-suture}.
	The {\it $\U(3)$-knot homology} of $K$, denoted by $\KHI_*(Y,K)$, is defined to be $\SHI_* (M(K),\alpha(K))$.
\end{definition}

As it is explained in \cite[Lemma 5.2]{KM:suture}, a closure of $(M(K),\alpha(K))$ can be described as follows. Suppose $F$ is a genus $1$ surface with one boundary component, and $c, c'\subset F$ are two oriented non-separating simple closed curves which intersect in exactly one point and $c\cdot c'=1$. Let $Z(K)$ be the result of gluing the knot complement $M(K)$ to the product 3-manifold $F \times S^1$ such that the meridian of $K$ is identified with $\{{\rm point}\}\times S^1$ and the longitude of $K$ is identified with $\partial F \times \{{\rm point}\}$. Suppose also $\gamma(K)\subset Z(K)$ is the 1-cycle given by $c \times \{{\rm point}\}\subset F \times S^1$. Then $Z(K)$ is a closure of the sutured manifold $(M(K),\alpha(K))$ and $\gamma(K)$ is the corresponding 1-cycle. The embedded surface $\bar R$ is also given by the torus $T=c' \times S^1 \subset F \times S^1$. Consequently:
\begin{equation}\label{KHI-Y-K}
	\KHI_*(Y,K)= \rI_*(Z(K),\gamma(K)|T)
\end{equation}

Next, we characterize the set of the critical points of the Chern-Simons functional associated to the pair $(Z(K),\gamma(K))$:

\begin{prop} \label{rep-var}
	For a knot $K$ in a 3-manifold $Y$, the set of the critical points of the Chern-Simons functional associated to the admissible pair $(Z(K),\gamma(K))$ is a 3-sheeted covering space of
	\begin{equation} \label{mathcalR}
		\mathcal{R}=\{\rho: \pi_1(Y\backslash K)\to \SU(3)\mid \rho(\mu)=\left[
                                                             \begin{array}{ccc}
                                                               1 &  0& 0  \\
                                                                0& \zeta & 0 \\
                                                                 0 & 0 & \zeta^2 \\
                                                             \end{array}
                                                           \right] \}.
	\end{equation}
	Recall that $\mu$ is a meridian of $K$ and $\zeta=e^{2\pi i/ 3}$.
\end{prop}
\begin{proof}
	The set of the critical points of the Chern-Simons functional is given by the conjugacy classes of representations
	$\rho: \pi_1(Z(K)\backslash \gamma(K)) \to \SU(3)$ such that a meridian of the closed curve $\gamma(K)$ is mapped to $\zeta$.
	We fix a base point for $Z(K)$ on the torus  $c' \times S^1\subset F \times S^1$. Suppose $J_1=\rho(c' \times \{{\rm point}\})$ and
	$J_2=\{{\rm point}\} \times S^1$. Since $c' \times S^1$ intersects $c$ in one point, we can assume:
	\begin{equation*}
		[J_2,J_1]= \zeta\cdot \id.
	\end{equation*}
	Thus there is a unique representative for the conjugacy class of $\rho$ such that:
	\begin{equation*}
		J_1= \left[
			\begin{array}{ccc}
				0&0&1\\
				1&0&0\\
				0&1&0\\
			\end{array}  \right]
		\hspace{1cm}	
		J_2=\left[
			\begin{array}{ccc}
				1&0&0\\
				0&\zeta&0\\
				0&0&\zeta^2\\
			\end{array}
			\right]
	\end{equation*}
	Therefore, the conjugacy class of the representation $\rho|_{\pi_1(F\times S^1 \backslash \gamma(K))}$ is uniquely determined by $J_3 \in \SU(3)$
	which is equal to the image of $\rho$ for a parallel copy of $c$. Since $J_3$ has to commute with $J_2$ it is a diagonal matrix.
	Therefore, the restriction of the above representative of $\rho$ to the knot group $\pi_1(M(K))$ determines an element of $\mathcal R$. Furthermore,
	$\rho$ maps the longitude of $K$ to $[J_3, J_1]$. Now the claim can be easily verified, because the map that sends a diagonal matrix $J_3$ to
	$[J_3, J_1]$ is 3 to 1.
\end{proof}

\begin{cor}\label{nonabrep}
	Suppose $K$ is a knot in a homology sphere $Y$. If $\dim(\KHI(Y,K)) > 1$, then there exists a non-abelian representation $\rho :\pi_1(Y\backslash K) \to \SU(3)$
	such that the image of the meridian is conjugate to
	\begin{equation*}
		\left[
			\begin{array}{ccc}
				1&0&0\\
				0&\zeta&0\\
				0&0&\zeta^2\\
			\end{array}
		\right]
	\end{equation*}
\end{cor}
\begin{proof}
	Suppose there is not a representation with this property. Then the only critical points
	of the Chern-Simons functional of $(Z(K),\gamma(K))$ are the three flat connections induced by the abelian representation
	in \eqref{mathcalR}. Since these critical points are non-degenerate\footnote{This is a consequence 
	of the fact that the Alexander polynomial of a knot $K$ in an integral homology sphere does not have a root, which 
	is a third root of unity.}, $\rI_*(Z(K),\gamma(K))$
	is the homology of a chain complex which has three generators.
	The order three map $\epsilon$ has degree $4$ with respect to the $\Z/12\Z$ grading. Thus
	the three eigenspaces of this operator have the same dimensions.
	Therefore, $\ker(\epsilon-1)$ has to be at most 1-dimensional, which is a contradiction.	
\end{proof}

\begin{prop} \label{knot-suture-sefiert-suture}
Let $K$ be a null-homologous knot in $Y$ and $S$ be a Seifert surface of  genus $g\geq 1$ for $K$.
	Then $\dim(\KHI(Y,K))\geq 2\dim(\SHI_*(N(S),\alpha(S)))$
	where $(N(S),\alpha(S))$ is the sutured manifold of
	Example \ref{Seifert-suture}.
\end{prop}
\begin{proof}
	According to \eqref{KHI-Y-K}, the $\U(3)$-knot homology is equal to $\rI_*(Z(K),\gamma(K)|T)$. 
	In order to form a closure of $(N(S),\alpha(S))$ and to define $\SHI_*(N(S),\alpha(S))$, 
	we firstly glue $[-1,1]\times F_{1,1}$ along the suture $\alpha(S)$.  
	In this case, $\bar{R}^{\pm}$ are two copies of $S\cup F_{1,1}$. If we identify $\bar{R}^{\pm}$ in the obvious way, then
	the resulting space is again the 3-manifold $Z(K)$.
	Let $\gamma(S)$ and $\bar R(S)$ be the resulting 1-cycle and the surface in the closure $Z(K)$. The 1-cycle $\gamma(S)$
	is a copy of $S^1\times \{{\rm point}\} \subset S^1\times F_{1,1}$. The surface $\bar R(S)$ has genus $g+1$ is given by
	gluing the Seifert surface $S$ to $F_{1,1}$. Arguing as in \cite[Proposition 7.9]{KM:suture} and using our excision theorem, 
	we can show that:
	\[\rI_*(Z(K),\gamma(K)|T)=\rI_*(Z(K),\gamma(K)+\gamma(S)|T)\]
	and   
	\[\rI_*(Z(S),\gamma(S)|\bar R(S))=\rI_*(Z(S),\gamma(K)+\gamma(S)|\bar R(S)).\]
	Remark \ref{genus-1-SHI} implies that the homology group $\rI_*(Z(K),\gamma(K)+\gamma(S)|T)$ is equal to:
	\begin{equation}\label{KHI-reform}
	  \rI_*(Z(K),\gamma(K)+\gamma(S))\cap \ker(\aleph_2-3).
	\end{equation}
	We can further decompose the vector space in \eqref{KHI-reform} using the eigenvalues of the operators 
	$\rho_2(\bar R(S))$. In particular, vector spaces:
	\begin{equation}\label{subspace-1}
	  \rI_*(Z(K),\gamma(K)+\gamma(S))\cap \ker(\aleph_2-3)\cap\gker(\rho_2(\bar R(S))-2\sqrt{3}g)
	\end{equation}
	and 
	\begin{equation}\label{subspace-2}
	  \rI_*(Z(K),\gamma(K)+\gamma(S))\cap \ker(\aleph_2-3)\cap\gker(\rho_2(\bar R(S))+2\sqrt{3}g)
	\end{equation}	
	are two distinct summands of $\rI_*(Z(K),\gamma(K)+\gamma(S)|T)$. Since the operator $\bar R(S))$ has degree $2$ with 
	respect to the $\Z/12\Z$-grading of $\rI_*(Z(K),\gamma(K)+\gamma(S)|T)$, the vector spaces in \eqref{subspace-1} and 
	\eqref{subspace-2} have the same dimension. It is also clear from the definition that \eqref{subspace-1} contains the 
	$\rI_*(Z(S),\gamma(K)+\gamma(S)|\bar R(S))$, which verifies the claim of this proposition.
\end{proof}

\begin{conjecture} \label{non-triv-knot-non-triv-KHI}
If $Y\backslash K$ is irreducible, then $\dim(\SHI_*(N(S),\alpha(S)))\geq 1$.
\end{conjecture}

If $Y\backslash K$ is irreducible, then the sutured manifold $(N(S),\alpha(S))$ is {\it taut}. Kronheimer and Mrowka proved a non-vanishing theorem for the $\U(2)$-sutured Floer homology $\SHI_*^2$ of taut sutured manifolds \cite{KM:suture}. As it is explained in Section \ref{que-conj}, we expect that a similar non-vanishing theorem holds for our version of sutured Floer homology. 

If Conjecture \ref{non-triv-knot-non-triv-KHI} holds, then the answer to Question \ref{rank-N-rep} is positive for $N=3$ and for any non-trivial knot in an integral homology sphere $Y$. Using Prime decomposition theorem for 3-manifolds, one can decompose $Y$ as a connected sum $Y_1\#Y_2$ such that $K\subset Y_1$ and $Y_1\backslash K$ is irreducible. With the aid of Corollary \ref{nonabrep} and Proposition \ref{knot-suture-sefiert-suture}, we can conclude from Conjecture \ref{non-triv-knot-non-triv-KHI} that the answer to Question \ref{rank-N-rep} for $N=3$ and the pair $(Y_1,K)$ is positive. This clearly would imply the claim for $(Y,K)$.

\section{Gluing Theory}
\subsection{Moduli Spaces on Manifolds with Long Neck} \label{long-neck}

 Suppose $Y$ is a connected 3-manifold and $\gamma \subset Y$ is a cycle. We do not assume that $(Y,\gamma)$ is $N$-admissible. However, we assume that the critical points of the (possibly perturbed)  Chern-Simons functional of $(Y,\gamma)$ are non-degenerate. Suppose also $(X,w)$ is a pair with boundary $(Y,\gamma)$. As it is explained in Subsection \ref{cyl-mod}, we can form moduli spaces $\mathcal M_p(X,w)$ and their framed counterparts $\widetilde{\mathcal M}_p(X,w)$ by working with perturbations of the ASD equation. Uhlenbeck compactness theorem of Subsection \ref{pol-invts} has an analogue for 4-manifolds with cylindrical ends. A proof of this result for $N=2$ is given in \cite[Chapter 5]{Don:YM-Floer}, and it can be extended to the higher rank by combining the arguments of  \cite{Don:YM-Floer} and \cite{K:higher}:
\begin{theorem} \label{cyl-mod-compactness}
	Suppose $\{[A_i]\}_{i\in \N}$ is a sequence of connections in the moduli space $\widetilde {\mathcal M}_{p} (X,w)$. Then there is an element $([B],[C_1],\dots,[C_k])$ of the following space
	\begin{equation}
		\widetilde {\mathcal M}_{p_0}(X,w;\alpha_0)  \times_{\Gamma_{\alpha_0}}
		\widetilde {\mathcal M}_{p_1}(\alpha_0,\alpha_1) \times_{\Gamma_{\alpha_1}} \dots  \times_{\Gamma_{\alpha_{k-1}}}
		\widetilde {\mathcal M}_{p_k}(\alpha_{k-1},\alpha_k)
	\end{equation}
	and an element $({\bf x},{\bf y}_1,\dots,{\bf y}_k)$ of the space:
	\begin{equation}
		 (X^+)^{m_0}/S_{m_0}\times (Y \times \R)^{m_1}/S_{m_1}\times \dots \times  (Y \times \R)^{m_k}/S_{m_k}
	\end{equation}	
	for appropriate non-negative integers $m_i$ such that $\{[A_i]\}_{i\in \N}$, after passing to a sequence is {\it weakly chain convergent} to $(([B],{\bf x}),([C_1],{\bf y}_1),\dots,([C_k],{\bf y}_k))$. Furthermore, we have:
	\begin{equation} \label{energy-relation}
		\kappa(p)=\kappa(p_0)+\kappa(p_1)+\dots+\kappa(p_k)+m_0+m_1+\dots+m_k
	\end{equation}
\end{theorem}
Note that $\widetilde {\mathcal M}_{q}(\alpha,\beta)$ is the moduli space of framed connections associated to a path $q:\alpha \to \beta$ over $(Y\times \R,\gamma \times \R)$. This moduli space is equipped with an action of $\Gamma_\alpha\times \Gamma_\beta$. The weakly chain convergence of $\{[A_i]\}_{i\in \N}$ to $(([B],{\bf x}),([C_1],{\bf y}_1),\dots,([C_k],{\bf y}_k))$ means that the following holds \cite{Don:YM-Floer}: the sequence $\{[A_i]\}_{i\in \N}$, after choosing appropriate gauge representatives, is $L^p_1$-convergent to $B$ on compact sets of $X\backslash {\bf x}$. Moreover, there is a sequence of real numbers:
\begin{equation*}
	t_i^1:=0<t_i^2<\dots<t_i^k
\end{equation*}
with $\lim_i t_i^{j+1}-t_i^j=\infty$ such that the translation of $A_i|_{Y \times [0,\infty)}$ by the constant $t_i^j$ is $L^p_1$ convergent to $C_j$ on compact sets of $Y \times \R \backslash {\bf y}_j$.  Identity \eqref{energy-relation} implies that:
\begin{equation} \label{index-ineq}
	\ind(\mathcal D_{A_i}) \geq \ind(\mathcal D_{B_1})+\ind(\mathcal D_{C_1})+\dots + \ind(\mathcal D_{C_k})
\end{equation}
and the two sides of the inequality are also equal to each other mod $4N$. In \eqref{index-ineq}, equality holds if and only if the integers $m_j$ are all zero, and in this case, $L^p_1$ convergence on compact subspaces can be improved to $C^\infty$ convergence on compact subspaces.
\begin{remark} \label{compact-cylinder-cob}
	Theorem \ref{cyl-mod-compactness} can be extended to the case that $W$ has more than one boundary components in an obvious way.
	In this case, we need to fix a chain of the elements of moduli spaces for each boundary component.
\end{remark}

There is another compactness theorem we need to review, in which we stretch a 4-manifold along a neck and consider the associated moduli spaces. Suppose $(X_1,w_1)$ and $(X_2,w_2)$ are pairs whose boundaries are $(Y,\gamma)$ and $(\overline Y,\overline \gamma)$, respectively. Then we can glue these pairs to form $(X,w)$. We also fix Riemannian metrics on $X_i$ which are product metrics in neighborhoods of their boundaries associated to a fixed metric on $Y$. Let $X^T$ be the Riemannian manifold, diffeomorphic to $X$, which has an isometric copy of $Y \times (-T,T)$ and $X^T\backslash Y \times (-T,T)$ is isometric to the disjoint union of $X_1$ and $X_2$. The proof of the following theorem is similar to that of Theorem \ref{cyl-mod-compactness}.

\begin{theorem} \label{stretching-neck-compactness}
	Suppose $A_i$ is a connection on the moduli space $\mathcal M_{\kappa} (X^{T_i},w)$ such that
	$\lim _{i \to \infty}T_i=\infty$. Then there is an element $([B_1],[C_1],\dots,[C_k],[B_2])$ of
	\begin{equation}\label{limit-mod-space}
		\widetilde {\mathcal M}_{p}(X_1,w_1;\alpha_0)  \times_{\Gamma_{\alpha_0}}
		\widetilde {\mathcal M}_{p_1}(\alpha_0,\alpha_1) \times_{\Gamma_{\alpha_1}} \dots  \times_{\Gamma_{\alpha_{k-1}}}
		\widetilde {\mathcal M}_{p_k}(\alpha_{k-1},\alpha_k) \times_{\Gamma_{\alpha_k}}
		\widetilde {\mathcal M}_{p'}(X_2,w_2;\alpha_k)
	\end{equation}
	and an element $({\bf x_1},{\bf y}_1,\dots,{\bf y}_k,{\bf x_{2}})$ of the space:
	\begin{equation}
		 (X_1^+)^{m_0}/S_{m_0}\times (Y \times \R)^{m_1}/S_{m_1}\times \dots \times (Y \times \R)^{m_k}/S_{m_k}\times (X_2^+)^{m_{k+1}}/S_{m_{k+1}}
	\end{equation}	
	such that $\{[A_i]\}_{i \in \N}$ is weakly chain convergent to $(([B_1]{\bf x_1}),([C_1],{\bf y}_1),\dots,([C_k],{\bf y}_k),([B_2],{\bf x_{2}}))$. Moreover, we have:
	\begin{equation} \label{energy-relation-2}
		\kappa=\kappa(p)+\kappa(p_1)+\dots+\kappa(p_k)+\kappa(p')+m_0+m_1+\dots+m_k+m_{k+1}
	\end{equation}
\end{theorem}

The following gluing theorem can be regarded as an inverse to Theorem \ref{stretching-neck-compactness}. There are various places in the literature that similar gluing theorems are discussed \cite{Taubes:L2-mod-space,Don:YM-Floer,KM:monopoles-3-man}. Theorem \ref{geometric-gluing} can be proved with similar strategies (see e.g. \cite[Theorem 4.17 and Section 4.7.1]{Don:YM-Floer}):
\begin{theorem} \label{geometric-gluing}
	Let $(X_i,w_i)$ be given as above. For $1\leq i \leq k$, let $p_i:\alpha_{i-1} \to \alpha_i$ be a path along $(Y \times \R, w \times \R)$, and let  $\widetilde N_i$ be a compact
	$(\Gamma_{\alpha_{i-1}}\times \Gamma_{\alpha_i})$-invariant
	subspace of $\widetilde {\mathcal M}_{p_i}(\alpha_{i-1},\alpha_i)$, which consists of regular points. Suppose also we are given two other compact spaces
	as below, which contain only regular points and are respectively invariant with respect to the action of $\Gamma_{\alpha_0}$ and $\Gamma_{\alpha_k}$:
	\begin{equation*}
		\widetilde N_0\subset \widetilde {\mathcal M}_{p}(X_1,w_1;\alpha_0)\hspace{1cm}
		\widetilde N_{k+1}\subset \widetilde {\mathcal M}_{p'}(X_2,w_2;\alpha_k)
	\end{equation*}
	Here we assume that the perturbation of the ASD equation on the ends of $X_1$ and $X_2$ are induced by a fixed
	perturbation of the Chern-Simons functional of $(Y,\gamma)$.
	Then there is a space $\widetilde U_i$ containing $\widetilde N_i$, which is an open subset of the relevant moduli space and is invariant with respect to the action of the relevant group.
	Moreover, for large enough values of $T$,
	there is a {\it gluing map}:
	\begin{equation*}
		\Phi_T:\widetilde U_0 \times_{\Gamma_{\alpha_0}} \widetilde U_1\times_{\Gamma_{\alpha_1}}\dots \times_{\Gamma_{\alpha_k}} \widetilde U_{k+1}
		\to \mathcal M_{\kappa} (X^{T},w)
	\end{equation*}
	where:
	\begin{equation} \label{energy-relation-3}
		\kappa=\kappa(p)+\kappa(p_1)+\dots+\kappa(p_k)+\kappa(p').
	\end{equation}
	The gluing map is a diffeomorphism into its image and satisfies the following properties: for any fixed element $a$ in the domain of $\Phi_T$, the sequence $\Phi_T(a)$ is
	chain convergent to $a$, as $T$ goes to infinity. Moreover, if $A_i$ is a connection on the moduli space $\mathcal M_{\kappa} (X^{T_i},w)$ such that
	$\lim _{i \to \infty}T_i=\infty$ and the sequence $\{A_i\}$ is chain convergent to an element of
	\begin{equation*}
		\widetilde N_0 \times_{\Gamma_{\alpha_0}} \widetilde N_1\times_{\Gamma_{\alpha_1}}\dots \times_{\Gamma_{\alpha_k}} \widetilde N_{k+1}
	\end{equation*}
	then $A_i$ lies in in the image of the gluing map for large enough values of $i$.
\end{theorem}

\begin{remark} \label{cut-down}
	Theorems \ref{stretching-neck-compactness} and \ref{geometric-gluing}
	are strong enough to study the cut-down
	moduli spaces on a 4-manifold with a long neck.
	To demonstrate this in the context of an example, suppose
	$(X_i,w_i)$ is as above and $b^+(X_i)\geq 1$. Let $\Sigma$ be an embedded surface in $X_2$.
	Let $\nu(\Sigma)\subset X_2$ be an open neighborhood of $\Sigma$ such that
	the inclusion of $\nu(\Sigma)$ in $X_2$ induces a surjection of fundamental groups.
	Let $\kappa$ be chosen such that $\mathcal M_\kappa(X^T,w)$ has expected dimension two.
	We make the simplifying assumption that all critical points of the Chern-Simons functional on $(Y,\gamma)$
	are irreducible and non-degenerate, and all the moduli spaces on $Y\times \R$ are regular.
	By choosing a generic metric and a small compactly supported perturbation, we can assume that the moduli
	spaces of the form $\mathcal M_p(X_1,w_1,\alpha)$ with expected dimension at most $0$ contain only irreducible and
	regular elements \cite[Lemma 24]{K:higher}.
	Similarly, we can arrange for a metric, a small compactly supported perturbation on $X_2$, and a geometric representative
	$V_2(\Sigma)\subset \mathcal B^*(\nu(\Sigma))\times \nu(\Sigma)$ such that:
	\begin{itemize}
		\item[(i)] The moduli spaces of the form $\mathcal M_p(X_2,w_2,\alpha)$ with dimension at most two
		consist of irreducible and regular solutions.
		\item[(ii)] The map $r:\mathcal M_p(X_2,w_2,\alpha)\times \Sigma\to\mathcal B^*(\nu(\Sigma))\times \nu(\Sigma)$,
		for any moduli space $\mathcal M_p(X_2,w_2,\alpha)$ of dimension at most two, is transversal to $V_2(\Sigma)$.	
		Suppose $\mathcal N_p(X_1,w_1;\alpha,\Sigma)$ denotes the cut-down moduli space.
	\end{itemize}
	The chosen holonomy perturbations on $X_1$ and $X_2$ induce a holonomy perturbation on $X^T$ for large
	values of $T$. Theorem \ref{cyl-mod-compactness} implies that for large enough values of $T$, the space
	$\mathcal M_{\kappa} (X^{T},w)\times \Sigma$
	is also cut-down transversely by $V_2(\Sigma)$ and the resulting space is
	compact. Furthermore, the elements in the cut-down space $\mathcal N_{\kappa}(X^{T},w;\Sigma)$
	are in correspondence with the elements of the following space:
	\begin{equation*}
		\bigcup_{\kappa(p_1)+\kappa(p_2)=\kappa} \mathcal M_{p_1}(X_1,w_1;\alpha,\Sigma)\times
		\mathcal N_{p_2}(X_2,w_2,\alpha).
	\end{equation*}
	In the following, we use a similar strategy to study the cut-down moduli spaces on 4-manifolds with long necks,
	without going into details.
\end{remark}

The following vanishing theorem is a standard application of the above theorems \cite{Don:inv,DK,MM:Don-inv}:

\begin{theorem} \label{gluing-S3}
	Suppose $X_1$ and $X_2$ are two 4-manifolds with $b^+(X_i)\geq 1$. Then for any 2-cycle $w$ in the connected sum $X_1\#X_2$ and
	any $z \in \A(X_1\#X_2)^{\otimes 2}$, the number $\rD_{X_1\#X_2,w}(z)$ is equal to zero.
\end{theorem}
\begin{proof}
	We can assume that $w=w_1\cup w_2$ and $z=z_1\cdot z_2$ where $w_i$ is a 2-cycle in $X_i$ and $z_i \in \A(X_i)^{\otimes 2}$.
	By replacing $X_i$, $w_i$ and $z_i$ with $X_i \# \cp$, $w_i \cup E_i$ and $z_i \cdot (E_i)_{(2)}^2$, we can also assume that $w_i$ is coprime to $N$. Here
	$E_i$ is the exceptional class in $X_i \# \cp$. We fix a Riemannian metric on $X_1\#X_2$ whose restriction to the connected sum region is
	isometric to the product metric $[-T,T]\times S^3$ where $T$ is a large constant and $S^3$ has the standard metric.
	Using the standard metric on $S^3$ allows us to ensure that all the framed moduli spaces on the cylinder $\R \times S^3$ are regular. We fix a holonomy perturbation of the ASD equation on $X_i$ such that the perturbation is supported outside of a neighborhood
	of the connected sum region, and the cut-down moduli spaces $\mathcal N_p(X_i,w_i;\Theta,z_i)$ of expected dimension at most zero are regular. Here $\Theta$ is the trivial
	connection, which is the only flat connection on $S^3$. Now we can use Theorem \ref{geometric-gluing} to conclude that
	the 0-dimensional moduli space $\mathcal N_\kappa(X_1\#X_2,w_1\cup w_2;z_1 \cdot z_2)$ is empty when $T$ is large enough.
\end{proof}

Next, we utilize the gluing and the compactness theorem to prove Proposition \ref{gluing-23} on connected sums along $\Sigma(2,3,23)$. Recall that in addition to the trivial connections, there are $44$ irreducible and $8$ $\SU(2)$-reducible connections on the trivial $\SU(3)$-bundle over $\Sigma(2,3,23)$. All these connections are non-degenerate. We choose a perturbation of the Chern-Simons functional of $\Sigma(2,3,23)$ (and the empty cycle) such that the critical points of the perturbed functional is the same as those of the Chern-Simons functional and the following assumption about the regularity of the elements of $[A]\in \mathcal M_p(\Sigma(2,3,23);\alpha,\beta)$ holds: if $A$ is irreducible, then we require that $A$ is regular and if $A$ is reducible and induced by an $\SU(2)$-connection, we require that $A$ is regular as a solution of the (perturbed) ASD equation for $\SU(2)$-connections \cite{Don:YM-Floer,KM:YAFT}.

Suppose $X_1$ is a 4-manifold with $b^+(X_1)\geq 1$ and $\partial X_1=\Sigma(2,3,23)$. Suppose also $w_1$ is a closed 2-cycle in $X_1$ which is coprime to 3. Fix a metric with cylindrical ends on $X_1$ and a small holonomy perturbation of the ASD equation which is compatible with the perturbation of the Chern-Simons functional. Let ${\mathcal M}_{p}(X_1,w_1;\alpha)$ be the moduli space of solutions to a perturbation of the ASD equation associated to the pair $(X_1,w_1)$ and the path $p$:

\begin{prop} \label{mod-X-Sigma}
	Suppose a positive integer $n_0$ is given. There is a metric and a holonomy perturbation of the
	ASD equation on $X_1$ such that the following holds: let $p$ be a path along $(X_1,w_1)$
	based at a flat connection $\alpha$ on $\Sigma(2,3,23)$ such that the
	index of the elements of ${\mathcal M}_{p}(X_1,w_1;\alpha)$ is at most $n_0$. Then the moduli
	space ${\mathcal M}_{p}(X_1,w_1;\alpha)$ consists of regular solutions
	and does not have any reducible connection.	
	Moreover, suppose $z_1\in \A(X)^{\otimes 2}$ such that:
	\[
	  d_{X_1,w_1}:=-4w_1^2-4(\chi(X_1)+\sigma(X_1))-4\equiv \deg(z_1)+4 \hspace{1cm}({\rm mod} 12).
	\]
	Suppose also the expected dimension of ${\mathcal N}_{p}(X_1,w_1;\alpha,z_1)$ is zero.
	Then there is a geometric representative for $z_1$ such that
	the cut-down moduli space ${\mathcal N}_{p}(X_1,w_1;\alpha,z_1)$ is compact.
\end{prop}
Note that ${\mathcal N}_{p}(X_1,w_1;\alpha,z_1)$ might be a linear combination of different spaces. Then compactness of this space is defined to be the compactness of all the involved spaces in the linear combination. In what follows, we gloss over this point about the nature of the spaces ${\mathcal N}_{p}(X_1,w_1;\alpha,z_1)$.
\begin{proof}
	The arguments of \cite{K:higher} can be used to show that the metric and the perturbation
	can be chosen such that if $p$ is a path as in the statement of the proposition,
	then ${\mathcal M}_{p}(X_1,w_1;\alpha)$ is regular and does not
	contain any reducible solution. We can also assume that
	$V(z_1)$ is chosen such that all moduli moduli spaces ${\mathcal N}_{p}(X_1,w_1;\alpha)$
	with expected dimension at most $n_0-\deg(z_1)$ are cut down transversely.
	Next, let the dimension of ${\mathcal M}_{p}(X_1,w_1;\alpha)$
	be equal to $\deg(z_1)$, and $\{A_i\}_i$ be a sequence of connections
	in the cut-down moduli space ${\mathcal N}_{p}(X_1,w_1;\alpha,z_1)$.
	By Theorem \ref{cyl-mod-compactness}, this sequence
	converges to an element $(([B],{\bf x}),([C_1],{\bf x}),\dots,([C_k],{\bf y}_k))$ where
	$[B] \in \widetilde {\mathcal N}_{p_0}(X_1,w_1;\alpha_0,z_1)$,
	$[C_i] \in  \widetilde {\mathcal M}_{p_i}(\alpha_{i-1},\alpha_i)$, $\alpha_k=\alpha$, and:
	\begin{equation} \label{index-cylindrical}
		\deg(z_1)\geq \ind(\mathcal D_{B})+\ind(\mathcal D_{C_1})+\dots + \ind(\mathcal D_{C_k}).
	\end{equation}
	and the equality holds if and only if the multi-sets ${\bf x}$, ${\bf y}_1$, $\dots$, ${\bf y}_k$ are empty.
	
	We firstly show that $\ind(\mathcal D_{B})$ in
	\eqref{index-cylindrical} is bounded above.
	If $C_i$ is irreducible, then $\ind(\mathcal D_{C_i})$ is positive.
	In the case that $C_i$ is reducible, we cannot guarantee that
	$\ind(\mathcal D_{C_i})$ is positive. However,
	the index of $\mathcal D_{C_i}$ as an $\SU(2)$-connection is positive.
	Using Table \ref{inv-red-23}, it is straightforward to check that for a reducible connection
	$C_i \in \widetilde {\mathcal M}_{p_i}(\R \times \Sigma(2,3,23);\alpha_{i-1},\alpha_i)$
	with non-positive $\ind(\mathcal D_{C_i})$, the flat connection $\alpha_{i}$ has to be
	$\beta_8$, and $\alpha_{i-1}$ is equal to either
	the trivial connection $\Theta$ or the $\SU(2)$-connection $\beta_2$. In the first case,
	$\ind(\mathcal D_{C_i})$ is equal to $-3$ and in the latter case
	$\ind(\mathcal D_{C_i})$ is equal to $-2$. Suppose:
	\[
	  \{i_1,\dots,i_l\} \subseteq \{0,\dots,k-1\}
	\]
	is the set of indices such that $\alpha_{i_j}=\beta_8$. We have:
	\begin{equation}\label{index-eq-2}
		\ind(\mathcal D_{C_{i_{j}+1}})+\ind(\mathcal D_{C_{i_{j}+2}})+
		\dots + \ind(\mathcal D_{C_{i_{j+1}}}) \equiv 0 \mod 12
	\end{equation}
	for $1\leq j \leq k-1$. In \eqref{index-eq-2}, the last term is at least $-3$ and the other terms are
	positive. This shows that the sum in \eqref{index-eq-1} is non-negative. We can use this to conclude that
	the sum $\ind(\mathcal D_{C_1})+\dots + \ind(\mathcal D_{C_k})$ in
	\eqref{index-cylindrical} is at least $-3$. That is to say, $\ind(\mathcal D_{B})$ is not greater than
	$\deg(z_1)+3$. By increasing the value of $n_0$ if necessary, we can assume that
	${\mathcal N}_{p_0}(X_1,w_1;\alpha_0,z_1)$, which contains
	$B$, is cut down transversely. In particular, $\ind(\mathcal D_{B})\geq \deg(z_1)$.
	
	If $l\geq 1$, then the index formula \eqref{ind-DA} shows:
	\begin{equation}\label{index-eq-1}
		\ind(\mathcal D_{B})+\ind(\mathcal D_{C_1})+\dots + \ind(\mathcal D_{C_{i_1}}) \equiv
		d_{X_1,w_1}-3\equiv \deg(z_1)+1 \mod 12
	\end{equation}
	In the above expression, the first term on the left hand side is not less than $\deg(z_1)$,
	the last term is not less than $-3$, and the remaining terms are positive. Therefore,
	the sum on the left hand side is not less than $\deg(z_1)+1$. This also shows that
	the right hand side of \eqref{index-cylindrical} is at least $\deg(z_1)+1$, which is a contradiction
	and as a result $l=0$.
	Therefore, in \eqref{index-cylindrical}, $\ind(\mathcal D_{C_i})$ is always positive.
	This also implies that $k=0$ and ${\bf x}$ is empty.
	Consequently, ${\mathcal N}_{p}(X_1,w_1;\alpha,z_1)$ is compact.
\end{proof}

Form the moduli spaces ${\mathcal N}_{p}(X_1,w_1;\alpha,z_1)$ with the perturbations from Proposition \ref{mod-X-Sigma}, and define the following element of $\rI_4(\Sigma(2,3,23))$:
\begin{equation} \label{relative-Sigma}
	\rD_{X_1,w_1}(z_1):=\sum \#{\mathcal N}_{p}(X_1,w_1;\alpha,z_1) \cdot \alpha
\end{equation}
where the sum is over all irreducible connections $\alpha$ in $\Sigma(2,3,23)$ and the paths $p$ such that:
\begin{equation} \label{cut-down-space}
	{\mathcal N}_{p}(X_1,w_1;\alpha,z_1)
\end{equation}
is 0-dimensional. Here we follow the standard conventions to orient \eqref{cut-down-space} \cite{KM:YAFT,K:higher}. Since $N=3$, we do not need a homology orientation of $(X_1,w_1)$ to fix a sign for $\rD_{X_1,w_1}(z_1)$.

Next, let $X_2$ be a 4-manifold with $b^+(X_2)\geq 1$ and $\partial X_2=\overline {\Sigma(2,3,23)}$. Let also $w_2$ be a closed 2-cycle in $X_2$ which is coprime to $3$. Fix a metric with cylindrical ends on $X_2$ and a small holonomy perturbation of the ASD equation which is compatible with the perturbation of the Chern-Simons functional:

\begin{prop} \label{mod-X2-Sigma}
	Suppose a positive integer $n_0$ is given. There is a metric and a holonomy perturbation of the
	ASD equation on $X_2$ such that the following holds: let $p$ be a path along $(X_2,w_2)$
	based at a flat connection $\alpha$ on $\Sigma(2,3,23)$ such that the
	index of the elements of ${\mathcal M}_{p}(X_2,w_2;\alpha)$ is at most $n_0$. Then the moduli
	space ${\mathcal M}_{p}(X_2,w_2;\alpha)$ consists of regular solutions
	and does not have any reducible connection.	
	Moreover, Suppose $z_2\in \A(X)^{\otimes 2}$ is chosen such that $\deg(z_2)$ is divisible by $4$.
	Suppose also the expected dimension of ${\mathcal N}_{p}(X_2,w_2;\alpha,z_2)$ is zero.
	Then there is a geometric representative for $z_2$ such that
	the cut-down moduli space ${\mathcal N}_{p}(X_2,w_2;\alpha,z_2)$ is compact.
\end{prop}

The proof of this proposition is analogous to that of Proposition \ref{mod-X-Sigma}, and we leave it to the reader. The cut-down moduli spaces in Proposition \ref{mod-X2-Sigma} can be used to define the following functional on $\rI_*(\Sigma(2,3,23))$:
\begin{equation} \label{relative-Sigma}
	\rD^{X_2,w_2}(z_2)(\alpha):= \#{\mathcal N}_{p}(X_2,w_2;\alpha,z_2)
\end{equation}
where $\alpha$ is an irreducible connection on $\Sigma(2,3,23)$ and the path $p$ is chosen such that ${\mathcal N}_{p}(X_2,w_2;\alpha,z_2)$ is 0-dimensional. The index formula shows that this space is 0-dimensional only if:
\begin{equation} \label{deg-alpha}
	\deg(\alpha)\equiv 4w_2^2+4(\chi(X_2)+\sigma(X_2))-4+\deg(z_2) \hspace{1cm} ({\rm mod} 12).
\end{equation}
Therefore, $\rD^{X_2,w_2}(z_2)$ is non-zero only on the elements of $\rI_*(\Sigma(2,3,23))$ that satisfy \eqref{deg-alpha}.

\begin{prop}
	For $i=1,2$, suppose $(X_i,w_i)$ and $z_i$ are as in Propositions \ref{mod-X-Sigma} and \ref{mod-X2-Sigma}. Then:
	\begin{equation} \label{gluing-identity}
		\rD_{X_1\circ X_2,w_1\cup w_2}(z_1\cdot z_2)=\rD^{X_2,w_2}(z_2) \circ \rD_{X_1,w_1}(z_1)
	\end{equation}
\end{prop}

\begin{proof}
	Use Propositions \ref{mod-X-Sigma} and \ref{mod-X2-Sigma} to fix  metrics on $X_1$ and $X_2$.
	Let also  $X_1\circ X_2$ be equipped with a Riemannian metric,
	which is compatible with the metrics on $X_1$, $X_2$
	and has a long neck along $\Sigma(2,3,23)$.
	By slightly modifying the holonomy perturbations provided by Propositions \ref{mod-X-Sigma}
	and \ref{mod-X2-Sigma}, we can assume that the perturbation of $X_i$ on the complement of a
	compact subset of $X_i^+$ is induced by the chosen perturbation of the Chern-Simons functional
	of $\Sigma(2,3,23)$ and the claims of these propositions about the moduli spaces
	of the form ${\mathcal N}_{p}(X_1,w_1;\alpha,z_1)$ and ${\mathcal N}_{p}(X_2,w_2;\alpha,z_2)$
	still hold.
	The perturbations of the ASD equations on $X_1$ and $X_2$ induce a perturbation of the
	ASD equation on $X_1\circ X_2$. By applying
	Theorem \ref{geometric-gluing} and the similar arguments as in the proof of
	Proposition \ref{mod-X-Sigma}, we can conclude that,
	for a long enough neck along $\Sigma(2,3,23)$, we have the following diffeomorphism of
	0-dimensional moduli spaces:
	\[
	  \mathcal N_\kappa (X_1\circ X_2,w_1\cup w_2,z_1\cdot z_2)\cong
	  \bigcup_\alpha \mathcal N_{p_1}(X_1,w_1;\alpha,z_1) \times \mathcal N_{p_2}(X_2,w_2;\alpha,z_2).
	\]
	Here $\alpha$ is an irreducible connection on $\Sigma(2,3,23)$, and $\kappa$, $p_1$ and $p_2$ are chosen such that all the above cut down moduli
	spaces are 0-dimensional. Standard arguments show that
	the above diffeomorphism is compatible with respect to the orientation of the involved moduli spaces. This diffeomorphism imply the claim in \eqref{gluing-identity}.
\end{proof}

This proposition essentially proves Theorem \ref{gluing-23} from Subsection \ref{sigma-23}. We only need to extend the above proposition to the case that $w_1$ and $w_2$ are not necessarily coprime to 3 and $\deg(z_2)$ is not divisible by $4$. The assumption on $w_i$ can be removed by the blowing up trick. The dimension of the moduli space of rank 3 instantons on the closed 4-manifold $X_1 \circ X_2$ is always divisible by 4. Therefore, if we define $\rD^{X_2,w_2}(z_2)=0$ in the case that $\deg (z_2) \nequiv 0$ mod 4, then the above theorem still holds.

\subsection{Gluing Theory for Negative Embedded Spheres} \label{neg-spheres-gluing}
In this subsection, we give a proof of Proposition \ref{-3-identities} based on the techniques which are discussed in Subsection \ref{long-neck}. Suppose $X$ is a smooth 4-manifold with $b^+(X)\geq 2$. Suppose also $\sigma$ is an embedded sphere in $X$ with $\sigma\cdot \sigma =-3$. A tubular neighborhood of $\sigma$, denoted by $Z$, is a disc bundle over $\sigma$ with Euler class $-3$ and the boundary $Y=L(3,-1)$. Thus $X$ can be split as $Z \circ X_1$ where $X_1$ is the closure of the complement of $Z$. We will also write $w_0$ for a fiber of the disc bundle $Z$. Fix the orientation on $w_0$ such that $w_0$ intersects $\sigma$ negatively in one point.

Fix a Riemannian metric with a cylindrical end on $Z$. Let also $L$ be the complex line bundle over $Z$ associated to the 2-cycle $w_0$. Since $b^+(Z)=b^1(Z)=0$, there is a unique ASD connection on $L$ with finite energy. This connection, denoted by $B$, is asymptotic to a flat connection $\chi$ which maps the generator of $\pi_1(Y)$ to $\zeta=\e^{2\pi\bi/3 }$. Consider the ASD $\U(N)$-connection $A:=B^{k_1} \oplus \dots \oplus B^{k_N}$, asymptotic to the flat connection $\alpha:=\chi^{k_1}\oplus\dots \oplus \chi^{k_N}$. The first Chern class of the underlying $\U(N)$-connection is equal to $(k_1+\dots+k_N){\rm P.D.}[w_0]$. The topological energy of $A$ is given by the following formula:
\begin{equation*}\label{energy}
	\kappa(A)=\frac{1}{12N}\sum_{1\leq i,j\leq N}(k_i-k_j)^2.
\end{equation*}
Therefore, we can use \eqref{ind-DA} to compute the index of $\mathcal D_A$. In particular, if $|k_i-k_j|\le 3$ for all $i,j$, then \eqref{ind-DA} shows that \cite{Luc:Thesis}:
\begin{equation*}\label{dimension}
	\ind(\mathcal D_A)= 1-N^2+\sum_{1\leq i,j\leq N}|k_i-k_j|	  	
\end{equation*}

The following table consists of various choices of $\U(3)$-connections. For each connection $A$, the first Chern class of the underlying bundle
of $A$ is equal to $k{\rm P.D.}[w_0]$ where $k$ is also given in the table:

\begin{center}
\renewcommand{\arraystretch}{1.2}
    \begin{tabular}{ |c|c|c|c|c|c| }
    	\hline
    	$A$ & $\alpha$ & $k$ &  $\ind(\mathcal D_A) $ & $\dim (\widetilde {\mathcal M}(A))$ & $\Gamma _{\alpha}$ \\
    	\hline
    	  $B^{2} \oplus 1 \oplus 1$ & $\chi^2 \oplus 1 \oplus 1 $& $2$ &$ 0$ & $4 $& ${\rm S}(\U(1)\times \U(2))$   \\
    	\hline
    	$B \oplus B\oplus 1$ & $\chi\oplus \chi \oplus 1$ & $ 2$ & $-4 $& $0$ & ${\rm S}(\U(2)\times \U(1))$   \\
    	\hline
    	$B^{-1} \oplus B^2\oplus B$ & $\chi^2\oplus \chi^2 \oplus \chi$ & $ 2$ & $4 $& $8$ & ${\rm S}(\U(2)\times \U(1))$   \\
    	 \hline
    	$B\oplus B^{-1} \oplus B^{-1}$ & $\chi\oplus \chi^2 \oplus \chi^2$ & $-1$ & $0 $& $4$ & ${\rm S}(\U(1)\times \U(2))$   \\
	\hline
    	$B^{-1}\oplus 1 \oplus 1$ & $\chi^2\oplus 1 \oplus 1$ & $ -1$ & $-4 $& $0$ & ${\rm S}(\U(1)\times \U(2))$   \\
    	\hline
    	$B^{-2} \oplus B\oplus 1$ & $\chi\oplus \chi \oplus 1$ & $ -1$ & $4 $& $8$ & ${\rm S}(\U(2)\times \U(1))$   \\
    	\hline
    \end{tabular}
\end{center}
We will write  ${\mathcal M}(A)$ (respectively, $\widetilde{\mathcal M}(A)$) for the moduli space (respectively, the framed moduli space) of connections on $Z$ corresponding to the path which is represented by $A$. From now on, we assume that an anti-self-dual metric with positive scalar curvature is fixed on $Z$ \cite{LeBr:exp-metric}.\footnote{We use such metrics to ensure regularity of moduli spaces of ASD connections with respect to these metrics. The construction of \cite{LeBr:exp-metric} provides asymptotically cylindrical ASD metrics rather than cylindrical metrics. However, that is enough for our purposes because we can pick a cylindrical approximation to the these metrics and then argue as in \cite[Corollary 5.13]{CDX:polygons}.}

\begin{prop} \label{minimal}
	For each $A$ in this table, $A$ has the minimal topological energy among all ASD connections with the same limiting flat connection as $A$ and the same $c_1$.
\end{prop}
\begin{proof}
	Let $A'$ be an ASD connection with a smaller topological energy and the same limiting flat connection and $c_1$ as $A$. The difference
	$\dim(\widetilde{\mathcal M}(A))-\dim(\widetilde{\mathcal M}(A'))$ is at least $12$.
	On the other hand, any such connection is regular because of the choice of the metric. Therefore, the dimension of $\widetilde{\mathcal M}(A')$ is at least $\dim(\Gamma_{\alpha})-\dim(\Gamma_{A'})$.
	This can be used to rule out the existence of $A'$.
\end{proof}

We study the moduli spaces ${\mathcal M}(A)$ for various choices of $A$.  The same argument as in Proposition \ref{minimal} shows that if $A=B \oplus B\oplus 1$, then the moduli space ${\mathcal M}(A)$ contains only the completely reducible connection $A$. Next, we turn to the case that $A=B^{2} \oplus 1 \oplus 1$. Analogous argument as in Proposition \ref{minimal} shows that there are three types of connections in this space:
\begin{itemize}
	\item finitely many irreducible connections;
	\item the complete reducible connection $A=B^{2} \oplus 1 \oplus 1$;
	\item reducible connections of the form $R\oplus 1$ where $R$ is an irreducible $\U(2)$-connection in the 1-dimensional
	space $\mathcal M(B^{2}\oplus 1)$.
\end{itemize}

\begin{prop}\label{nearreducible2}
	The unique component of $ \mathcal M(B^{2}\oplus 1)$ containing $B^{2}\oplus 1$ is a half-line $[0,\infty)$.
	All the other components are either circles or copies of $\R$ consisting of only irreducible connections.
	The corresponding components of $\widetilde {\mathcal M}(B^{2}\oplus 1\oplus 1)$ are $\C^2$, $S^3\times S^1$ and $S^3\times\R$.
\end{prop}
 \begin{proof}
 	The proof of the first part is straightforward. For the second part, note that the orbit of the connection $R\oplus1$
	in the framed moduli space, for an irreducible connection $R$, is
	$\Gamma_\alpha / \Gamma_{R\oplus 1}={\rm S}(\U(1)\times \U(2))/ {\rm S}(\U(1)_2\times \U(1))=S^3$, where
	$\U(1)_2\times \U(1)$ denotes $3 \times 3$ diagonal matrices where the first
	two diagonal entires are equal to each other.
 \end{proof}

Now we are ready to prove the first part of Proposition \ref{-3-identities}:
\begin{prop}\label{universal2}
	Suppose $X$ and $\sigma$ are as above and $z \in \mathbb{A}(\langle\sigma\rangle^{\perp})^{\otimes 2}$. Suppose also $w$ is a 2-cycle in $X$ such that $w \cdot \sigma \equiv 1$ mod 3.
	Then there is a constant $c$ such that:
	\begin{equation*}
		\rD^3_{X,w}((-\frac{3}{2}\sigma_{(3)}-\frac{3}{2}\sigma_{(2)}^2 - a_2)z)=c \rD^3_{X,w-\sigma}(z).
	\end{equation*}
\end{prop}
\begin{proof}
	For the simplicity of the exposition, we assume that $z=1$. A similar proof works in the more general case.
	Equip $X$ with a Riemannian metric that has a neck of length $T$ along the lens space $Y$,
	and denote the resulting Riemannian manifold by $X^T$.
	We firstly study the 4-dimensional moduli spaces of the form $\mathcal M_{\kappa_0}(X^T,w)$ for large values of $T$.
	Let $\alpha_1$ and $\alpha_2$ denote the flat connections $\chi^{2}\oplus 1 \oplus 1$ and $\chi\oplus \chi\oplus 1$.
	We write $G_i$ for the stabilizer group $\Gamma_{\alpha_i}$.
	We also write $\mathcal R_1$ for $\widetilde {\mathcal M}(B^{2}\oplus 1\oplus 1)$ and
	$\mathcal R_2$ for $\widetilde {\mathcal M}(B\oplus B\oplus 1)$. Clearly $\mathcal R_i$ is a $G_i$-manifold.
	By Theorem \ref{geometric-gluing} and the above description of the low-dimensional moduli spaces on $Z$,
	the moduli space $\mathcal M_{\kappa_0}(X^T,w)$
	can be covered with two open sets $\mathcal{U}$ and $\mathcal{V}$ of the following form:
	\begin{equation*}
		\mathcal{U}=\widetilde{\mathcal M}_{p_1}(X_1,w\cap X_1; \alpha_1)
		\times_{G_1}\mathcal R_1
		 \hspace{.5cm}\mathcal{V}=\widetilde{\mathcal M}_{p_2}(X_1,w\cap X_1;
		 \alpha_2)\times_{G_2} \mathcal R_2
	\end{equation*}
	where framed moduli spaces $\widetilde{\mathcal M}_{p_1}(X_1,w\cap X_1; \alpha_1)$ and
	$\widetilde{\mathcal M}_{p_2}(X_1,w\cap X_1; \alpha_2)$ are respectively of dimensions $4$ and $8$.
	By choosing a generic metric on $X_1$ and a small holonomy perturbation of the
	ASD equation supported in $X_1$,
	we can assume that the action of $\Gamma_{\alpha_i}$ on
	$\widetilde{\mathcal M}_{p_i}(X_1,w\cap X_1; \alpha_i)$ is free.
	In particular, there are $G_i$ equivariant maps:
	\begin{equation}\label{f_i}
		\widetilde f_i:\widetilde{\mathcal M}_{p_i}(X_1,w\cap X_1; \alpha_i) \to EG_i
	\end{equation}
	where
	$EG_i$ is the universal principal $G_i$-bundle over the classifying space $BG_i$. Every connected component of
	$\widetilde{\mathcal M}_{p_1}(X_1,w\cap X_1; \alpha_1)$ can be identified with $G_1$ and we can assume that
	the map $\widetilde f_1$ on different connected components of $\widetilde{\mathcal M}_{p_1}(X_1,w\cap X_1; \alpha_1)$ are equal
	to each other.
		
	Gluing theory also gives a description of the intersection $\mathcal U \cap \mathcal V$ for large values of $T$. Let
	$\mathcal N$ be the $G_1\times G_2$-manifold $\widetilde{\mathcal M}_{q}(\alpha_1,\alpha_2)$
	where the path $q$ is chosen such that the index of any element in
	$\widetilde{\mathcal M}_{q}(\alpha_1,\alpha_2)$ is equal to $0$.
	This space consists of
	ASD connections of the form $R\oplus 1$ on $\R\times L(3,1)$ which are asymptotic to
	$\chi^2\oplus1$ and $\chi\oplus\chi$ on the two ends. After dividing by the action of translations, there are
	finitely many choices for the connection $R$. Therefore, this manifold, as a $(G_1\times G_2)$-space,
	is the union of finitely many spaces of the following form:
	\begin{equation}\label{overlap-comp}
	  \R \times \frac{G_1\times G_2}{{\rm S}(\U(1)_2\times \U(1))}
	\end{equation}
	In particular, the action of $G_i$ on this space is free. The $G_1$-space $ \mathcal N\times_{G_2}
	\mathcal R_2$ can be identified with the end of $\mathcal R_1$.
	The $G_2$-space $\widetilde{\mathcal M}_{p_1}(X_1,w\cap X_1; \alpha_1) \times_{G_1} \mathcal N$
	can be identified with the end of $\widetilde{\mathcal M}_{p_2}(X_1,w\cap X_1; \alpha_2)$.
	Using these identifications, the overlap $\mathcal U \cap \mathcal V$ is identified with the following space:
	\begin{equation} \label{overlap}
	  \widetilde{\mathcal M}_{p_1}(X_1,w\cap X_1; \alpha_1) \times_{G_1} \mathcal N\times_{G_2}
	  \mathcal R_2
	\end{equation}
	In summary, $\mathcal M_{\kappa_0}(X^T,w)$  is the pushout of the following below:
	\begin{equation}\label{pushout}
	\begin{tikzcd}
	 \widetilde{\mathcal M}_{p_1}(X_1,w\cap X_1; \alpha_1)
	  \times_{G_1} \mathcal N\times_{G_2} \mathcal R_2\arrow[r, hook]  \arrow[d, hook]
	  \arrow[dr, phantom, "\mathlarger{\mathlarger{\mathlarger{\mathlarger{\mathlarger{\ulcorner}}}}}"]
	&   \widetilde{\mathcal M}_{p_1}(X_1,w\cap X_1; \alpha_1) \times_{G_1}\mathcal R_1 \arrow[d]\\
	  \widetilde{\mathcal M}_{p_2}(X_1,w\cap X_1; \alpha_2) \times_{G_2}\mathcal R_2 \arrow[r]& \mathcal M_{\kappa_0}(X^T,w)
	\end{tikzcd}
	\end{equation}	
	We can form a universal version of the above diagram as follows:
	\begin{equation}\label{universal-pushout}
	\begin{tikzcd}
	 EG_1\times_{G_1} \mathcal N\times_{G_2}\mathcal R_2 \arrow[r, hook, "h_2"]  \arrow[d,"h_1"]
	  \arrow[dr, phantom, "\mathlarger{\mathlarger{\mathlarger{\mathlarger{\mathlarger{\ulcorner}}}}}"]
	&EG_1\times_{G_1}\mathcal R_1 \arrow[d]\\
	  EG_2 \times_{G_2}\mathcal R_2 \arrow[r]& M
	\end{tikzcd}
	\end{equation}	
	The free action of $G_2$ on $EG_1\times_{G_1} \mathcal N$ gives a $G_2$-equivariant map
	$f:EG_1\times_{G_1} \mathcal N\to EG_2$. The vertical map $h_1$ is induced by $f$.
	The horizontal map $h_2$ is induced by the inclusion of $\mathcal N\times_{G_2}\mathcal R_2$ into
	$\mathcal R_1$.
	The space $M$ is then the pushout of $h_1$ and $h_2$.
	 The maps $\widetilde f_1$ and $\widetilde f_2$ determine the maps:
	\[
	  f_i:\widetilde{\mathcal M}_{p_i}(X_1,w\cap X_1; \alpha_i) \times_{G_i}\mathcal R_i \to
	  EG_i \times_{G_i}\mathcal R_i
	\]
	and
	\[
	  g: \widetilde{\mathcal M}_{p_1}(X_1,w\cap X_1; \alpha_1)  \times_{G_1} \mathcal N\times_{G_2}\mathcal R_2\to
	  EG_1\times_{G_1} \mathcal N\times_{G_2}\mathcal R_2.
	\]
	 The maps $\widetilde f_1$ and $\widetilde f_2$ can be chosen such that $f_1$, $f_2$ and $g$ give rise to a map
	from Diagram \eqref{pushout} to Diagram \eqref{universal-pushout}.
	We will write $F$ for the induced map from $\mathcal M_{\kappa_0}(X^T,w)$ to $M$.
	
	There is a similar description of the universal bundle over the space $\mathcal M_{\kappa_0}(X^T,w) \times \sigma$.
	Suppose $q_1$ is the path over $Z$ determined by the connection $B^2\oplus 1\oplus 1$.
	We can form an equivariant universal bundle $\widetilde {\mathbb P}_1$ over
	$\widetilde{\mathcal B}_{q_1}(Z,w\cap Z;\alpha)\times \sigma$ similar to the the universal bundles in Subsection \ref{pol-invts}.
	There is an action of $ G_1$ on $\widetilde {\mathbb P}_1$ which lifts the obvious action of this group on
	$\widetilde{\mathcal B}_{q_1}(Z,w\cap Z;\alpha)\times \sigma$. This bundle induces a
	${\rm PU}(3)$-bundle over $\mathcal R_1\times \sigma$ with an action of $G_1$ which we still denote by
	$\widetilde {\mathbb P}_1$. Similarly, we can construct a
	${\rm PU}(3)$-bundle $\widetilde {\mathbb P}_2$ over $\mathcal R_2\times \sigma$ with an action of $G_2$.
	Since $\mathcal R_2$ consists of only the class of the connection $B\oplus B \oplus 1$,
	the bundle $\widetilde {\mathbb P}_2$ has the following form \cite[Proposition 46]{Luc:Thesis}:
	\begin{equation} \label{vanishing-2}
		\widetilde {\mathbb P}_2=F\boxtimes L\oplus P.
	\end{equation}
	where $F$ and $P$ are the standard $\U(2)$- and $\U(1)$-representations of $G_2$. To be more precise, \eqref{vanishing-2}
	gives a lift of $\widetilde {\mathbb P}_2$ to a $\U(3)$-bundle with an action of $G_2$.
	The restriction of $\widetilde {\mathbb P}_1$, as a ${\rm PU}(3)$-bundle with an action of $G_1$, to the subset
	$\mathcal N\times_{G_2}\mathcal R_2\times \sigma$ of $\mathcal R_1\times \sigma$
	can be identified with $\mathcal N\times_{G_2}\widetilde {\mathbb P}_2$.
	Therefore, $\widetilde {\mathbb P}_1$ and $\widetilde {\mathbb P}_2$ gives rise to the following diagram of
	${\rm PU}(3)$-bundles:
	\begin{equation}\label{universal-pushout-bundles}
	\begin{tikzcd}
	 {\mathbb P}_3:=EG_1\times_{G_1} \mathcal N\times_{G_2}\widetilde {\mathbb P}_2 \arrow[r, "\overline h_2"]  \arrow[d,"\overline h_1"]
	  \arrow[dr, phantom, "\mathlarger{\mathlarger{\mathlarger{\mathlarger{\mathlarger{\ulcorner}}}}}"]
	&EG_1\times_{G_1}\widetilde {\mathbb P}_1 \arrow[d]\\
	  EG_2\times_{G_2}\widetilde {\mathbb P}_2 \arrow[r]& \overline {\mathbb P}
	\end{tikzcd}
	\end{equation}	
	The ${\rm PU}(3)$-bundle $\overline {\mathbb P}$ over $M\times \sigma$ is the pushout of the above diagram.
	Pullback of $\overline {\mathbb P}$ with respect
	to the map $(F,id):\mathcal M_{\kappa_0}(X^T,w)\times \sigma \to M\times \sigma$ is equal to the universal bundle.
	
	Suppose $\varphi\in H^{4}(M,\Q)$ is defined to be:
	\[\frac{3}{2}c_3 (\overline {\mathbb P})/\sigma -\frac{3}{2}(c_2(\overline {\mathbb P})/\sigma)^2 -
	c_2(\overline {\mathbb P})/x.\]
	The description of $\widetilde {\mathbb P}_2$ in \eqref{vanishing-2} can be employed to show that the restriction of $\varphi$
	to the subspace $EG_2\times_{G_2}\mathcal R_2$ vanishes. Thus we can use the vanishing of cohomology classes of
	$EG_1\times_{G_1} \mathcal N\times_{G_2}\mathcal R_2$ in odd degrees and the Mayer-Vietoris exact sequence for
	the pushout diagram in \eqref{universal-pushout} to conclude that there is a unique choice of a relative cohomology class
	\[
	  \psi \in H^{4}(EG_1\times_{G_1}\mathcal R_1,EG_1\times_{G_1}\mathcal N\times_{G_2} \mathcal R_2)
	\]	
	such that $\varphi$ is equal to the image of $\psi$ in $H^{4}(M,\Q)$. This discussion shows that the pairing of
	$ F^*(\varphi)$ and the fundamental class of the 4-dimensional moduli space $\mathcal M_{\kappa_0}(X^T,w)$ is
	equal to the pairing of $f_1^*(\psi)$ and the relative fundamental class of
	$\widetilde{\mathcal M}_{p_1}(X_1,w\cap X_1; \alpha_1) \times_{G_1}\mathcal R_1$.
	The description of the moduli spaces and the map $\widetilde f_1$ shows that there is a constant $c$ such that the latter pairing is equal to:
	\[
	  c\cdot \# {\mathcal M}_{p_1}(X_1,w\cap X_1; \alpha_1)
	\]
	Since $\widetilde {\mathcal M}(B^{-1}\oplus1\oplus 1)$ consists of a single point, we have:
	\begin{equation} \label{0-dim-mod-space-long-neck}
		{\mathcal M}_{p_1}(X_1,w\cap X_1; \alpha_1)=
		\widetilde{\mathcal M}_{p_1}(X_1,w\cap X_1; \alpha_1)\times_{\Gamma_{\alpha_1}}
		\widetilde {\mathcal M}(B^{-1}\oplus1\oplus 1)=\mathcal M_{\kappa_1}(X^T,w-\sigma)
	\end{equation}
	Here $\kappa_1$ is chosen such that the space $\mathcal M_{\kappa_1}(X^T,w-\sigma)$ is 0-dimensional.
	We use Theorem \ref{geometric-gluing} to conclude the second identity for large enough values of $T$.
	Identities \eqref{0-dim-mod-space-long-neck} allow us to verify the desired claim.
\end{proof}

The second part of Proposition \ref{-3-identities} can be proved by applying the first part to the $(-3)$-sphere $\sigma$ with the reverse orientation. Next, we turn to the proof of the last part of Proposition \ref{-3-identities}. As in the previous case, we need to study some low dimensional moduli spaces over $Z$. The following table consists of various choices of $\U(3)$-connections on $Z$ with vanishing $c_1$. The proof of Proposition \ref{minimal} shows that each connection $A$ in this table has the minimal energy among all ASD connections with the same limiting flat connection and vanishing $c_1$.

\begin{table}[H]
	\begin{center}
		\resizebox{\hsize}{!}{
		\renewcommand{\arraystretch}{1.3}
		\begin{tabular}{ |c|c|c|c|c| }
			\hline
			$ A$ & $\alpha$ &$\ind(\mathcal D_A)$ & $\dim (\widetilde{\mathcal M}(A))$ & $\Gamma _{\alpha}$ \\
			\hline
	 		 $1\oplus 1 \oplus 1$  & $1\oplus 1 \oplus 1 $&$ -8$ & $0 $& $\SU(3)$   \\
	 		\hline
			$B \oplus B^{-1}\oplus 1$ & $\chi\oplus \chi^{2} \oplus 1$ & $0$& $2$ & ${\rm S}(\U(1)\times \U(1)\times \U(1))$   \\
			\hline
			$B\oplus B \oplus B^{-2}$ & $\chi\oplus \chi \oplus \chi$ &  $4 $& $12$ & $\SU(3)$   \\
			\hline
			$B^{-1}\oplus B^{-1} \oplus B^{2}$  & $\chi^{2}\oplus \chi^{2} \oplus \chi^{2}$ & $4 $& $12$ & $\SU(3)$   \\
			\hline
		\end{tabular}}
	\end{center}
\end{table}	

For $A=B \oplus B^{-1}\oplus 1$, the moduli space $\mathcal M(A)$ contains three types of connections:
\begin{itemize}
	\item finitely many irreducible connections;
	\item the completely reducible connection $A=B\oplus B^{-1}\oplus 1$;
	\item reducible connections of the form $R\oplus 1$ where $R$ is an irreducible $\SU(2)$-connection in the 1-dimensional space
	$\mathcal M(B\oplus B^{-1})$.
\end{itemize}
\begin{prop}\label{nearreducible}
	The connected components of $\mathcal M(B\oplus B^{-1})$ which contains $B\oplus B^{-1}$ is a half-line $[0,\infty)$.
	All the other components are either circles or copies of $\R$ consisting of only irreducible connections. The corresponding components of $\widetilde{\mathcal M}(A)$ are
	$\C$, $S^1\times S^1$ and $S^1\times \R$, and the action of $\Gamma_\alpha=S^1$ is standard.
\end{prop}

\begin{prop}\label{universal1}
	Suppose $X$ and $\sigma$ are as above and $z \in \mathbb{A}(\langle\sigma\rangle^{\perp})^{\otimes 2}$. Suppose also $w$ is a 2-cycle in $X_1$.
	Then the following formulas hold:
	\begin{itemize}
		\item $\rD^3_{X,w}((\sigma_{(2)}^4+4a_2\sigma_{(2)}^2+3\sigma_{(3)}^2)z)=0$
		\item $\rD^3_{X,w}((\sigma_{(2)}^3 \sigma_{(3)}+3a_3\sigma_{(2)}^2+a_2 \sigma_{(2)} \sigma_{(3)}) z)=0$
	\end{itemize}
\end{prop}
\begin{proof}
	Similar to the proof of Proposition \ref{universal2}, we may assume that $z=1$.
	We need to study the $8$-dimensional moduli space for the first identity and the $10$-dimensional moduli space for the second identity.
	Suppose $\kappa_0$ is a constant number such that the expected dimension of the moduli space $\mathcal M_{\kappa_0}(X^T,w)$ is
	$8$ or $10$.
	This moduli space is compact. Moreover, Theorem \ref{geometric-gluing} and the above description of the low-dimensional
	moduli spaces on $Z$ show that $\mathcal M_{\kappa_0}(X^T,w)$ can be given as the pushout of the following diagram:
	\begin{equation}\label{pushout-2}
	\begin{tikzcd}
	 \widetilde{\mathcal M}_{p_1}(X_1,w; \alpha_1)
	  \times_{G_1} \mathcal N\times_{G_2} \mathcal R_2\arrow[r, hook]  \arrow[d, hook]
	  \arrow[dr, phantom, "\mathlarger{\mathlarger{\mathlarger{\mathlarger{\mathlarger{\ulcorner}}}}}"]
	&   \widetilde{\mathcal M}_{p_1}(X_1,w; \alpha_1) \times_{G_1}\mathcal R_1 \arrow[d]\\
	  \widetilde{\mathcal M}_{p_2}(X_1,w; \alpha_2) \times_{G_2}\mathcal R_2 \arrow[r]& \mathcal M_{\kappa_0}(X^T,w)
	\end{tikzcd}
	\end{equation}	
	where $\alpha_1=\chi\oplus \chi^{2} \oplus 1$, $\alpha_2=1\oplus 1 \oplus 1$, $G_i=\Gamma_{\alpha_i}$,
	$\mathcal R_1=\widetilde{\mathcal M}(B\oplus B^{-1} \oplus 1)$, $\mathcal R_2=\widetilde{\mathcal M}(1\oplus 1 \oplus 1)$,
	$\mathcal N=\widetilde {\mathcal M}_q(\alpha_1,\alpha_2)$ and the path $q$ is chosen such that the index of any
	element in $\mathcal N$ is equal to $0$. As in the previous case, we can form the universal version of the above diagram as below:
	\begin{equation}\label{universal-pushout-2}
	\begin{tikzcd}
	 EG_1\times_{G_1} \mathcal N\times_{G_2}\mathcal R_2 \arrow[r, hook, "h_2"]  \arrow[d,"h_1"]
	  \arrow[dr, phantom, "\mathlarger{\mathlarger{\mathlarger{\mathlarger{\mathlarger{\ulcorner}}}}}"]
	&EG_1\times_{G_1}\mathcal R_1 \arrow[d]\\
	  EG_2 \times_{G_2}\mathcal R_2 \arrow[r]& M
	\end{tikzcd}
	\end{equation}	
	with a map $F: \mathcal M_{\kappa_0}(X^T,w) \to M$.
	
	There are also ${\rm PU}(3)$-bundles $\widetilde {\mathbb P}_i$ on $\mathcal R_i$ with an action of $G_i$. These bundles give
	rise to a pushout diagram of bundles:
	\begin{equation}\label{universal-pushout-bundles-2}
	\begin{tikzcd}
	 {\mathbb P}_3:=EG_1\times_{G_1} \mathcal N\times_{G_2}\widetilde {\mathbb P}_2 \arrow[r, "\overline h_2"]  \arrow[d,"\overline h_1"]
	  \arrow[dr, phantom, "\mathlarger{\mathlarger{\mathlarger{\mathlarger{\mathlarger{\ulcorner}}}}}"]
	&EG_1\times_{G_1}\widetilde {\mathbb P}_1 \arrow[d]\\
	  EG_2\times_{G_2}\widetilde {\mathbb P}_2 \arrow[r]& \overline {\mathbb P}
	\end{tikzcd}
	\end{equation}	
	such that $(F,id)^{*}(\overline {\mathbb P})$ is equal to the universal bundle on $\mathcal M_{\kappa_0}(X^T,w)\times \sigma$.
	Let $\varphi_1  \in H^{8}(M,\Q)$ and $\varphi_2\in H^{10}(M,\Q)$ be defined as follows:
	\[\varphi_1:=(c_2(\overline {\mathbb P})/\sigma)^4+4(c_2 (\overline {\mathbb P})/x)\cdot(c_2(\overline {\mathbb P})/\sigma)^2+
	3(c_3(\overline {\mathbb P})/\sigma)^2,\]
	\[
	\varphi_2:=(c_2(\overline {\mathbb P})/\sigma)^3\cdot(c_3(\overline {\mathbb P})/\sigma)+
	3(c_3 (\overline {\mathbb P})/x)\cdot(c_2(\overline {\mathbb P})/\sigma)^2+
	(c_2(\overline {\mathbb P})/x)\cdot (c_2(\overline {\mathbb P})/\sigma)\cdot(c_3(\overline {\mathbb P})/\sigma).\]
	
	The restriction of cohomology classes $\varphi_1$ and $\varphi_2$ to the open sets $EG_i\times_{G_i}\mathcal R_i$ vanish.
	We show this claim for the connected component  $EG_1\times_{G_1}\C$ of $EG_1\times_{G_1}\mathcal R_1$. The restriction of
	the bundle $\widetilde{\mathbb P}_1$ to the $G_1$-manifold $\C$ is given by \cite[Proposition 46]{Luc:Thesis}:
	\[
	  P\boxtimes L \oplus Q\boxtimes L^{-1}\oplus R\boxtimes \underline{\C}
	\]
	where $P,Q,R$ are the $G_1$-equivariant line bundles over $\C$ associated to the three standard 1-dimensional representations
	of $G_1$. Note that $P\otimes Q\otimes R$ is the trivial bundle with the trivial action of $G_1$.
	Suppose $p, q\in H_{G_1}^*(\C)$ denote the equivariant first Chern classes of the bundles $P$ and $Q$. Then:
	\[
	  c_2(EG_1\times_{G_1}\widetilde {\mathbb P}_1)/\sigma=p-q\hspace{1cm}
	  c_3(EG_1\times_{G_1}\widetilde {\mathbb P}_1)/\sigma=q^2-p^2
	\]
	\[
	  c_2(EG_1\times_{G_1}\widetilde {\mathbb P}_1)/x=-p^2-q^2-pq\hspace{1cm}
	  c_3(EG_1\times_{G_1}\widetilde {\mathbb P}_1)/x=-pq(p+q)
	\]
	These identities can be used to show that the restriction of $\varphi_1$ and $\varphi_2$ to
	$EG_1\times_{G_1} \C$ are equal to zero. Using similar arguments,
	it is even easier to show that the restriction of these two cohomology classes to
	$EG_2\times_{G_2}\mathcal R_2$ and the remaining connected components of
	$EG_1\times_{G_1}\mathcal R_1$ are equal to zero. Since the cohomology groups of
	$EG_1\times_{G_1} \mathcal N\times_{G_2}\mathcal R_2$ in odd degrees vanish,
	the Mayer-Vietoris exact sequence for Diagram \eqref{universal-pushout-2} imply that
	$\varphi_1$ and $\varphi_2$ are both equal to zero.
	In particular, this verifies the claim of this proposition.
\end{proof}

\subsection{Gluing Theory for Fukaya-Floer Homology} \label{IIN}
One of the primary goals of this subsection is to define the Fukaya-Floer homology for an $N$-admissible pair $(Y,\gamma)$ and an $(N-1)$-tuple $L=(l_2,\dots,l_N)$ of the elements of $H_1(Y)$. As it is mentioned in Subsection \ref{FFH}, we shall construct a chain complex $(\mathfrak{C}_*^{N,j}(Y,\gamma,L),d_{N,j})$ over $R_{N,j}$, for each non-negative integer number $j$. We also define a chain map:
\[
  \hspace{3cm} F^k_j:\mathfrak{C}_*^{N,j}(Y,\gamma,L) \to \mathfrak{C}_*^{N,k}(Y,\gamma,L)   \hspace{1cm} j\geq k\geq 0
\]
such that $F^k_j$ is a homomorphism of $R_{N,j}$-modules and for a triple of integers $j\geq k\geq l\geq 0$, the map $F^l_k\circ F^k_j$ is chain homotopy equivalent to $F^l_j$. Let ${\mathbb I}^{N,j}_*(Y,\gamma,L)$ be the homology of the chain complex $(\mathfrak{C}_*^{N,j}(Y,\gamma,L),d_{N,j})$ and $f_j^k:{\mathbb I}^{N,j}_*(Y,\gamma,L) \to {\mathbb I}^{N,k}_*(Y,\gamma,L)$ be the map induced by $F_j^k$. Then the Fukaya-Floer homology group ${\mathbb I}^N_*(Y,\gamma,L)$ is the inverse limit of the inverse system $(\{{\mathbb I}^{N,j}_*(Y,\gamma,L)\}_j, \{f_j^k\}_{j\geq k})$.

For the simplicity of exposition, we assume that $l_3= \dots=l_N=0$ and $l_2$ is an integral homology class. Later we shall explain how the definition should be adapted to the arbitrary case. For each $i\in \N$, let $\eta_i$ be an oriented closed curve that represents $l_2$. Let also $\nu(\eta_i)$ be a regular neighborhood of $\eta_i$ such that the inclusion of $\nu(\eta_i)$ into $Y$ induces a surjective map at the level of fundamental groups. Furthermore, we assume that the open sets $\nu(\eta_i)$ are disjoint.

For a fixed integer $j \geq 0$, let $\CS_{\pi_j}$ be a perturbation of the Chern-Simons functional associated to $(Y,\gamma)$ such that all critical points of $\CS_{\pi_j}$ are irreducible and non-degenerate, and all moduli spaces $\mathcal M_p(\alpha,\beta)$, with dimension at most $2j+1$, consist of regular points. If $i$ is a non-negative integer number not greater than $j$ and $\dim(\mathcal M_p(\alpha,\beta))\leq 2j+1$, then we may also assume that the restriction of any element in $\mathcal M_p(\alpha,\beta)$ to the open subspace $\nu(\eta_i)\times (0,1)$ is irreducible.\footnote{This is a consequence of unique continuation property of (non-perturbed) ASD connections and Uhlenbeck compactness theorem. For more details see \cite{K:higher}.} We define:
\begin{equation} \label{FF-chain-gp}
	\mathfrak{C}_*^{N,j}(Y,\gamma,L):=\mathfrak{C}_*^{\pi_j}(Y,\gamma)\otimes R_{N,j}.
\end{equation}
Next, we need to define the differential $d_{N,j}$.

For any path $p$ between two critical points $\alpha$ and $\beta$ of $\CS_{\pi_j}$ and $S=\{i_1,\dots, i_k\}\subset \{1,\dots, j\}$, we have the diagonal $\R$-action $\{\sigma_t\}_{t\in \R}$on the following space:
\begin{equation} \label{MabS}
	{\mathcal M}_p(\alpha,\beta)\times (\eta_{i_1} \times \R) \times \dots \times (\eta_{i_k} \times \R)
\end{equation}
Here the action of $t \in \R$ maps $(x,r) \in \eta_i \times \R$ to $(x,r-t)$. If $p$ is not the constant path or $S$ is not the empty set, then this action is free and we will write $\breve {\mathcal M}_p(\alpha,\beta;S)$ for the quotient space. Otherwise, $\breve {\mathcal M}_p(\alpha,\beta;S)$ is defined to be empty. In the case that $S$ is empty, the quotient space is the space $\breve {\mathcal M}_p(\alpha,\beta)$. If $S'\subset S$, then there is an obvious projection map from $\breve {\mathcal M}_p(\alpha,\beta;S)$ to $\breve {\mathcal M}_p(\alpha,\beta;S')$, which is denoted by $\pi_{S\to S'}$.

We can form a partial compactification of $\breve {\mathcal M}_p(\alpha,\beta;S)$ by defining
\begin{equation*}
	\breve {\mathcal M}^+_p(\alpha,\beta;S):=\bigcup_{k} \bigcup_{\{(p_i,S_i)\}_{1\leq i \leq k}}\prod_{i=1}^k \breve {\mathcal M}_{p_i}(\alpha_{i-1},\alpha_i;S_i)
\end{equation*}
where $S_i \subset S$ and $p_i:\alpha_{i-1} \to \alpha_i$ is a path between two critical points of $\CS_{\pi_j}$ such that the sets $S_i$ are disjoint, their union is equal to $S$, and the composition $p_1 \circ \dots \circ p_k$ is equal to $p$. Note that the set $S_i$ may be empty. A sequence $u_l\in  \breve {\mathcal M}_p(\alpha,\beta;S)$ is chain convergent to $u_\infty=(u_\infty ^1,\cdots u_\infty^k)\in  \prod_{i=1}^k \breve {\mathcal M}_{p_i}(\alpha_{i-1},\alpha_i;S_i)$, if there is a sequence of $k$-tuple of real numbers $(t_l^1,\dots, t_l^k)$ such that:
\begin{equation*}
	t_l^1<\dots<t_l^k\hspace{2cm} \lim_{l\to \infty} t_l^{m+1}-t_l^{m}=\infty \hspace{.5cm} \forall 1\leq m\leq k-1
\end{equation*}
and
\begin{equation*}
	\sigma_{t_l^m}\circ \pi_{S \to S_m} (u_l) \xrightarrow{C^\infty_{loc}} u_{\infty}^{m}.
\end{equation*}
We use this notion of convergence to define a topology on $\breve {\mathcal M}^+_p(\alpha,\beta;S)$.
\begin{remark}
	Due to the possibility of bubbling off instantons, the space $\breve {\mathcal M}^+_p(\alpha,\beta;S)$ is not necessarily compact.
\end{remark}

\begin{example}
	Suppose $A \in \mathcal M_p(\alpha,\beta)$ where $p:\alpha \to \beta$ is a non-trivial path between two critical points of the perturbed
	Chern-Simons functional. Suppose also $x\in \eta_1$ is fixed. Then $u_i=[A,(x,i)]$ defines a sequence of elements
	of $ \breve {\mathcal M}_p(\alpha,\beta;S)$ where $S=\{1\}$. This sequence is convergent to $(u_\infty^1,u_\infty^2)$:
	\begin{equation*}
		u_\infty^1:=[A] \in \breve {\mathcal M}_p(\alpha,\beta) \hspace{1cm}u_\infty^2:=[\pi^*(\beta),(x,0)] \in \breve{\mathcal M}_q(\beta,\beta;S)
	\end{equation*}
	Here $q$ is the constant path from $\beta$ to $\beta$, and $\pi^*(\beta)$ is the pull back of the connection $\beta$ to $\R \times Y$.
\end{example}

For $S=\{i_1,\dots, i_k\}$, define the map $\Phi_S$ as follows:
\begin{equation}\label{Phis}
	\Phi_S:\breve {\mathcal M}_p(\alpha,\beta;S) \to \prod_{i\in S}\mathcal{B}^\ast( \nu(\eta_i)\times (0,1))\times \eta_i
\end{equation}
\begin{equation*}
	\Phi_s([A,(x_{i_1},t_{i_1})\dots,(x_{i_k},t_{i_k})])=((T^*_{t_{i_1}}A|_{\nu(\eta_{i_1})\times (0,1)},x_{i_1}),\dots,(T^*_{t_{i_k}}A|_{\nu(\eta_{i_k})\times (0,1)},x_{i_k})).
\end{equation*}
where $T_{t_i}:\R \times Y \to \R \times Y$ denotes the translation that maps $(x,t)$ to $(x,t+t_i)$. This map extends to $\breve {\mathcal M}^+_p(\alpha,\beta;S)$ in the obvious way and the extension is continuous. As in Subsection \ref{pol-invts}, we can form a universal ${\rm PU}(N)$-bundle $\mathbb P_i$ and an associated $\SU(N^N)$ vector bundle $\mathbb E_i$ on $\mathcal{B}^\ast( \nu(\eta_i)\times (0,1))\times \eta_i$. We need to find a subspace of $\mathcal{B}^\ast( \nu(\eta_i)\times (0,1))\times \eta_i$ which represents $c_2(\mathbb P_i)$ or equivalently $\frac{1}{N^N}c_2(\mathbb E_i)$, and use the inverse image of this  representative to cut down the moduli space $\breve {\mathcal M}_p(\alpha,\beta;S)$. We can proceed as in Subsection \ref{pol-invts}. The rank stratification of the vector bundle ${\rm Hom}(\C^{N^N-1},\mathbb E_i)$ determines a codimension four representative $V_2(\eta_i)$ for the second Chern class of the universal bundle. For all choices of $\alpha$, $\beta$, $S$ and $p$, we may assume that $\Phi_S$ in \eqref{Phis} is transversal to:
\[
  V_2(\eta_{i_1})\times \dots \times V_2(\eta_{i_k})
\]
Let $\breve {\mathcal N}_p(\alpha,\beta;S)$ be the inverse image of the above space by the map $\Phi_S$. Then the closure of $\breve {\mathcal N}_p(\alpha,\beta;S)$ in $\breve {\mathcal M}^+_p(\alpha,\beta;S)$, denoted by $ \breve {\mathcal N}^+_p(\alpha,\beta;S)$, is also a stratified space with smooth strata of the following form:
\begin{equation*}
	\prod_{i=1}^k \breve {\mathcal N}_{p_i}(\alpha_{i-1},\alpha_i;S_i).
\end{equation*}

\begin{lemma} \label{comp}
	If the dimension of $\breve {\mathcal N}^+_p(\alpha,\beta;S)$ is at most $1$, then this space is compact.
\end{lemma}
\begin{proof}
	This is a consequence of Theorem \ref{compact-cylinder-cob} and
	the standard dimension counting argument used for the definition of the polynomial invariants for closed 4-manifolds.
\end{proof}

Now we are in a position to define the differential $d_{N,j}$. Firstly consider the operator:
\begin{equation*}
	d_S:\mathfrak{C}_*^{\pi_j}(Y,\gamma)\to \mathfrak{C}_*^{\pi_j}(Y,\gamma)
\end{equation*}
\begin{equation*}
	d_S(\alpha)= \sum_{p:\alpha\to \beta}\# \breve {\mathcal N}^+_p(\alpha,\beta;S)~\cdot \beta
\end{equation*}
where the sum is over all paths $p$ from $\alpha$ to other critical points of $\CS_{\pi_j}$ such that the dimension of the moduli space $\mathcal M_p(\alpha,\beta)$ is equal to $2|S|+1$. The cut-down moduli space $\breve {\mathcal N}^+_p(\alpha,\beta;S)$ is $0$-dimensional and hence it is compact by Lemma \ref{comp}. Finally, the Fukaya-Floer differential $d_{N,j}$ is defined as:
\begin{equation*}
	d_{N,j}:= \sum_{S\subset \mathbf{N}^+}   \frac{\prod_{i\in S} t_{2,i} }{N^{N|S|}}d_S
\end{equation*}
The term $N^{N|S|}$ appears in the above expression due to the the relationship between $c_2(\mathbb P_i)$ and $c_2(\mathbb E_i)$.

\begin{prop}\label{DD}
	The map $d_{N,j}$ defines a differential, i.e., $d_{N,j}^2=0$.
\end{prop}
\begin{proof}
	For critical points $\alpha$ and $\beta$ of $\CS_{\pi_j}$ and $S\subset \N$, let $h_S(\alpha,\beta)$
	be a number that satisfies the following identity:
	\begin{equation*}
		d_{N,j}^2 \alpha= \sum_{\beta, S} \frac{m_S(\alpha,\beta)}{N^{N|S|}}(\prod_{i\in S} t_{2,i}) \beta.
	\end{equation*}	
	Fix $\alpha$, $\beta$ and $S\subseteq \{1,\dots,k\}$, and suppose $p:\alpha \to \beta$ is a path that
	$\dim(\mathcal M_p(\alpha,\beta))$ is equal to $2|S|+2$.
	Then the term $m_S(\alpha,\beta)$ is equal to:
	\begin{equation*}
		m_S(\alpha,\beta)=\sum_{\substack{p_1:\alpha \to \gamma\\ p_2:\gamma \to \beta}}\sum_{S_1,S_2}
		\# \breve {\mathcal N}^+_{p_1}(\alpha,\gamma;S_1)\cdot \# \breve {\mathcal N}^+_{p_2}(\gamma,\beta;S_2)
	\end{equation*}	
	such that $p_1 \circ p_2=p$, $S_1\cup S_2=S$, $S_1$ and $S_2$ are disjoint, and:
	\begin{equation*}
		\dim(\mathcal M_{p_1}(\alpha,\gamma))=2|S_1|+1\hspace{1cm} \dim(\mathcal M_{p_2}(\gamma,\beta))=2|S_2|+1.
	\end{equation*}
	Therefore, $m_{S}(\alpha,\beta)$ is equal to the signed count of the boundary points of the compact 1-manifold $\breve {\mathcal N}^+_p(\alpha,\beta;S)$.
	This implies that $m_{S}(\alpha,\beta)$ is equal to zero.	
\end{proof}

The above discussion can be generalized to the case of arbitrary $(N-1)$-tuple $(l_2,\dots, l_N)$ in a straightforward way. For each $l_k$, we choose a sequence of disjoint representatives $\{\eta_{k,i}\}_{i\in \N}$ and the open neighborhoods $\nu(\eta_{k,i})$. We keep assuming that $\eta_{k,i}$ is an oriented simple closed curve in $Y$, the open sets $\nu(\eta_{k,i})$ are disjoint and the inclusion of $\nu(\eta_{k,i})$ in $Y$ induces a surjective map at the level of fundamental groups. In a more general case that $l_i$ is a homology class with complex coefficients, we need to consider the straightforward generalization that $\eta_{k,i}$ is a linear combination of disjoint closed curves. Then for each $(N-1)$-tuple $\overline S=(S_2,\dots,S_N)$ of subsets of $\{1,2,\dots,j\}$, we can form the moduli space $\breve {\mathcal M}_p(\alpha,\beta; \overline S)$ and its partial compactification $\breve {\mathcal M}^+_p(\alpha,\beta; \overline S)$. As in the previous case, each stratum of the partial compactification is a product of moduli spaces of the form $\breve {\mathcal M}_p(\alpha,\beta; \overline S)$. The next step is to cut down $\breve {\mathcal M}^+_p(\alpha,\beta; \overline S)$ with divisors $V_k(\eta_{k,i})$ for $i \in S_k$. As in the case of closed 4-manifolds, we can use the auxiliary complex vector bundles of rank $N^N$ associated to the universal ${\rm PU}(N)$-bundle and construct geometric representatives for $V_k(\eta_{k,i})$. The desired representatives are linear combinations of codimension $2k$ stratified spaces with even dimensional strata. We will write $\breve {\mathcal N}^+_p(\alpha,\beta; \overline S)$ for the cut-down moduli space. A generic choice of these divisors allow us to obtain a cut-down moduli space. This space is compact when its dimension is less than or equal to 1.

The 0-dimensional cut-down moduli spaces $\breve {\mathcal N}^+_p(\alpha,\beta; \overline S)$ can be used to define an operator $d_{\overline S}$ acting on $\mathfrak{C}_*^{\pi_j}(Y,\gamma)$:
\begin{equation*}
	d_{\overline S} \alpha=\sum_{p:\alpha \to \beta} \#\breve {\mathcal N}^+_p(\alpha,\beta; \overline S)\cdot \beta.
\end{equation*}
We combine these operators to form:
\begin{equation*}
	d_{N,j}:= \sum_{\overline S=(S_2,\cdots, S_N)} (\prod_{ i\in S_k} t_{k,i} )  d_{\overline S}
\end{equation*}
As in the previous case, the 1-dimensional moduli spaces can be used to show that $d_{N,j}^2=0$.

Before giving the definition of the maps $F_j^k$, we study the functorial properties of the chain complex $(\mathfrak{C}_*^{N,j}(Y,\gamma,L),d_{N,j})$. Let $(X,w)$ be a pair whose boundary is $(Y,\gamma)$. Suppose $\Gamma^2,\cdots, \Gamma^N$ are properly embedded surfaces in $X$ such that $[\partial \Gamma^i]=l_i\in H_1(Y)$. For $z\in \mathbb{A}(X)^{N-1}$, we want to define the relative invariant:
\begin{equation*}
	{\rm D}^{N,j}_{X,w}(z\e^{\sum_i \Gamma^i_{(i)}})\in \mathbb{I}_\ast^{N,j} (Y,\gamma,L)
\end{equation*}
where $L=(l_2,\cdots,l_N)$. This is similar to the definition of the differential of the Fukaya-Floer chain complex.
For the simplicity of exposition, we assume that $\Gamma^3=\dots=\Gamma^N=0$, and $z=1$. As in the case of the definition of the differential in Fukaya-Floer homology, the more general case is just slightly different. In the manifold $X^+$ with a cylindrical end, let $\{\Sigma_i\}_{i\in \N}$ be a sequence of surfaces which are given by perturbing the surface $\Gamma^2$. We assume that these surfaces intersect generically and their intersections with $\R^{\geq 0} \times Y$ are equal to the disjoint product surfaces $\{\eta_i\times \R^{\geq 0}\}_{i \in \N}$. We choose a holonomy perturbation of the ASD equation on $X^+$, compatible with the chosen perturbation $\CS_{\pi_j}$ of the Chern-Simons functional of $Y$, such that all moduli spaces $\mathcal M_p(X,w;\beta)$ consists of regular points. Given a subset $S=\{i_1,\dots, i_k\}\subset \{1,\dots,j\}$ and a path $p$ along $(X,w)$ based at the critical point $\beta$ of $\CS_{\pi_j}$, consider the space:
\begin{equation}
	{\mathcal M}_p(X,w;\beta,S):={\mathcal M}_p(X,w;\beta)\times \Sigma_{i_1}  \times \dots \times \Sigma_{i_k}.
\end{equation}
For $S' \subset S$, the obvious projection map from ${\mathcal M}_p(X,w;\beta,S)$ to ${\mathcal M}_p(X,w;\beta,S')$ is denoted by $\pi_{S \to S'}$. There is also a partially defined translation map for ${\mathcal M}_p(X,w;\beta,S)$. Suppose ${\mathcal M}_p^{cyl}(X,w;\beta,S)$ denotes the following subset of ${\mathcal M}_p(X,w;\beta,S)$:
\begin{equation*}
	{\mathcal M}_p(X,w;\beta)\times (\eta_{j_1}\times \R^{\geq 0})  \times \dots \times (\eta_{j_k}\times \R^{\geq 0})
\end{equation*}
For $u=([A],(x_1,t_1),\dots,(x_k,t_k))\in{\mathcal M}_p^{cyl}(X,w;\beta,S)$ define $\sigma_t(u)$ to be the following element:
\begin{equation*}
	([T^*_{t}(A|_{Y \times \R^{\geq 0}})],(x_1,t_1-t),\dots,(x_k,t_k-t)).
\end{equation*}
Note that the connection $T^*_{t}(A|_{Y \times \R^{\geq 0}})$ is partially defined on the cylinder $Y \times \R$.

As in the cylinder case, we form a partial compactification of ${\mathcal M}_p(X,w;\beta,S)$ given by:
\begin{equation*}
	{\mathcal M}_p^+(X,w;\beta,S):=\bigcup_{k} \bigcup_{\{(p_i,S_i)\}_{1\leq i \leq k}}{\mathcal M}_{p_1}(X,w,\alpha_1,S_1) \times \prod_{i=2}^k \breve {\mathcal M}_{p_i}(\alpha_{i-1},\alpha_i;S_i)
\end{equation*}
where $\alpha_1$, $\dots$, $\alpha_k$ are critical points of $\CS_{\pi_j}$, $\alpha_k=\beta$, $p_1$ is a path along $(X,w)$, $p_i:\alpha_{i-1}\to \alpha_i$ is a path along the cylinder, $p=p_1\circ \dots \circ p_k$, the sets $S_i$ are disjoint and their union is equal to $S$. As before, $S_i$ may be empty. In a little more detail, a sequence $u_l\in {\mathcal M}_p(X,w;\beta,S)$ is chain convergent to $u_\infty=(u_\infty ^1,\cdots u_\infty ^k)\in {\mathcal M}_{p_1}(X,w,\alpha_1,S_1) \times \prod_{i=2}^k \breve {\mathcal M}_{p_i}(\alpha_{i-1},\alpha_i;S_i)$, if there is a sequence of $(k-1)$-tuple of real numbers $(t_l^2,\dots, t_l^k)$ such that:
\begin{equation*}
	t_l^1:=0<t_l^2<\dots<t_l^k\hspace{2cm} \lim_{l\to \infty} t_l^{m+1}-t_l^{m}=\infty \hspace{.5cm} \forall\,\, 1\leq m\leq k-1
\end{equation*}
and
\begin{equation*}
	\pi_{S \to S_1} (u_l) \xrightarrow{C^\infty_{loc}} u_{\infty}^{1} \hspace{1cm}\sigma_{t_l^i}\circ \pi_{S \to S_i} (u_l)
	\xrightarrow{C^\infty_{loc}} u_{\infty}^{i}\hspace{3mm} i\geq 2.
\end{equation*}
Here part of the assumption is that $\sigma_{t_l^i}\circ \pi_{S \to S_i} (u_l)$ is well-defined for $i\geq 2$. That is to say, $\pi_{S \to S_i} (u_l) \in {\mathcal M}_p^{cyl}(X,w;\beta,S_i)$.

We momentarily assume that $S=\{1\}$. Suppose $\nu(\Sigma_1)$ is an open neighborhood of $\Sigma_1$ such that the inclusion map of $\nu(\Sigma_1)$ induces a surjective map. Suppose also $\mathcal B^{**}(\nu(\Sigma_1))$ is the set of connections on $\nu(\Sigma_1)$ whose restrictions to the sets of the from $\nu(\eta_1) \times (t-1,t)$ are irreducible. We can assume that the perturbation $\pi_j$ is small enough such that the restriction map $r_1:{\mathcal M}_p(X,w;\beta,S)\to \mathcal{B}^{**}(\nu(\Sigma_1))\times \Sigma_1$ is well-defined.\footnote{This is again a consequence of unique continuation and Uhlenbeck compactness theorem \cite{K:higher}.} We can form a universal ${\rm PU}(N)$-bundle $\mathbb P_1$ and the associated $\SU(N^N)$-bundle $\mathbb E_1$ on $\mathcal{B}^{**}(\nu(\Sigma_1))\times \Sigma_1$. Our goal is to define a geometric representative for $c_2(\mathbb P_1)$ or equivalently $\frac{1}{N^N}c_2(\mathbb E_1)$, which is compatible with our choice of the geometric representative in the case of cylinders. Note that ${\mathcal M}_p(X,w;\beta,S)$ is the union of the following two sets:
\begin{equation*}
	B_1={\mathcal M}_p(X,w;\beta)\times (\eta_{1}\times [1,\infty))\hspace{1cm}B_2={\mathcal M}_p(X,w;\beta)\times (\Sigma_1\backslash (\eta_{1}\times (2,\infty)))
\end{equation*}
The map $r_1|_{B_1}$ can be composed with the following map:
\begin{equation*}
	F:\mathcal{B}^{**}(\nu(\Sigma_1))\times (\eta_{1}\times [1,\infty)) \to  \mathcal{B}^*( \nu(\eta_1)\times (0,1))\times \eta_1
\end{equation*}
\begin{equation*}
	F([A],(x,t))=([A|_{\nu(\Sigma)\times (t-1,t)}],x)
\end{equation*}
Therefore, we can choose the geometric representative $V_2(\Sigma_1)\subset  \mathcal{B}^{**}(\nu(\Sigma_1))\times \Sigma_1 $ for $c_2(\mathbb E_1)$ such that:
\begin{equation*}
	V_2(\Sigma_1)\cap \(\mathcal{B}^{**}(\nu(\Sigma_1))\times (\eta_{1}\times [1,\infty))\)=F^{-1}(V_2(\eta_1)).
\end{equation*}
Arguing as in \cite{K:higher}, we can also assume that $V_2(\Sigma_1)$ is transversal to the map $r_1$ for all choices of the path $p$. In this process, firstly we slightly modify $V_2(\eta)$ to make the map $r_1|_{B_1}$ transversal. Then we extend $F^{-1}(V_2(\eta_1))$ to $\mathcal{B}^{**}(\nu(\Sigma_1))\times \Sigma_1 $ such that $r_1|_{B_2}$ is also transversal. We will write ${\mathcal N}_p(X,w;\beta,S)$ for the cut-down moduli space $r_1^{-1}(V_2(\Sigma_1))$. A similar construction can be used to define ${\mathcal N}_p(X,w;\beta,S)$ in the case that $S$ has more than one element. The closure of ${\mathcal N}_p(X,w;\beta,S)$ in ${\mathcal M}_p^+(X,w;\beta,S)$, denoted by ${\mathcal N}^+_p(X,w;\beta,S)$, is also a stratified space with smooth strata of the following form:
\begin{equation*}
	{\mathcal N}_{p_1}(X,w,\alpha_1,S_1) \times \prod_{i=2}^k \breve {\mathcal N}_{p_i}(\alpha_{i-1},\alpha_i;S_i).
\end{equation*}
As in the cylinder case, the cut-down moduli space ${\mathcal N}^+_p(X,w;\beta,S)$ is compact when its dimension is at most one. The relative invariant ${\rm D}^N_{X,w}(\e^{\Gamma_{(2)}})$ is defined using 0-dimensional moduli spaces as below:
\begin{equation} \label{rel-elt}
	{\rm D}^{N,j}_{X,w}(\e^{\Gamma^2_{(2)}}):=
	\sum_{S\subset \{1,\dots,j\}} \frac{\prod_{i\in S} t_{2,i}}{N^{N|S|}} \# {\mathcal N}^+_p(X,w;\beta,S)\cdot \beta
	\in \mathfrak{C}_*^{N,j}(Y,\gamma,L)
\end{equation}
Following the proof of Proposition \ref{DD}, we can show that \eqref{rel-elt} determines a cycle in $\mathfrak{C}_*^{N,j}(Y,\gamma,L)$. We will denote the corresponding element in $\mathbb{I}^{N,j}_*(Y,\gamma, L)$ with the same notation.

Definition of the relative element in \eqref{rel-elt} can be extended to similar situations. For example, suppose $(Y_0,\gamma_0)$ and $(Y_1,\gamma_1)$ are two $N$-admissible pairs and $L_i=(l^2_i,\dots,l^N_i)$ is an $(N-1)$-tuple of the elements in $H_1(Y_i)$. Suppose also $(W,w,z)$ is a morphism from $(Y_0,\gamma_0)$ to $(Y_1,\gamma_1)$, and $\Gamma^j$ is a properly embedded surface in $W$ such that $[\partial_i \Gamma^j]=l_i^j$. Then there is a chain map:
\begin{equation*}
	\mathfrak C^{N,j}(W,w,z\e^{\Gamma^2_{(2)}+\dots+\Gamma^N_{(N)}}):
	{\mathfrak{C}}_*^{N,j}(Y_0,\gamma_0,L_0)  \to  {\mathfrak{C}}_*^{N,j}(Y_1,\gamma_1,L_1)
\end{equation*}
which induces a map at the level of homology:
\begin{equation*}
	\mathbb I^{N,j}_*(W,w,z\e^{\Gamma^2_{(2)}+\dots+\Gamma^N_{(N)}}):
	\mathbb I^{N,j}_*(Y_0,\gamma_0,L_0) \to \mathbb I^{N,j}_*(Y_1,\gamma_1,L_1).
\end{equation*}
Alternatively, if $(Y, \gamma,L)$ is as above, $(X,w)$ is a 4-manifold whose boundary is equal to $(\overline Y,\overline \gamma)$, and $\Gamma^j$ is an embedded surface such that $[\partial \Gamma^j]=-l_j$, then we have an $R_{N,j}$-linear map:
\begin{equation*}
	\rD_{N,j}^{X,w}(z\e^{\Gamma^2_{(2)}+\dots+\Gamma^N_{(N)}}):\mathbb I^{N,j}_*(Y,\gamma,L) \to R_{N,j}.
\end{equation*}

Now we are ready to define the maps $F_j^k$ for a triple $(Y,\gamma,L)$. For a pair of non-negative integers $j$ and $k$ with $j\geq k$, we have chosen perturbations $\CS_{\pi_j}$ and $\CS_{\pi_k}$ of the Chern-Simons functional of the pair $(Y,\gamma)$. Since $j\geq k$, the functional $\CS_{\pi_j}$ satisfies all the properties that we required for $\CS_{\pi_k}$. In particular, the chain complex $(\mathfrak C_*^{N,j}(Y,\gamma,L)\otimes R_{N,k},d_{N,j})$ gives an alternative chain complex to define $\mathbb I^{N,k}_*(Y,\gamma,L)$. The functoriality mentioned in the previous paragraph implies that there is a chain map:
\[
  \mathfrak C^{N,j}([0,1]\times Y,[0,1]\times \gamma,\e^{\Gamma^2_{(2)}+\dots+\Gamma^N_{(N)}}):
  {\mathfrak{C}}_*^{\pi_j}(Y,\gamma)\otimes R_{N,k}  \to {\mathfrak{C}}_*^{\pi_k}(Y,\gamma)\otimes R_{N,k}
\]
where $\Gamma^i$ is the surface $[0,1] \times \eta_{i,1}$. The map $F_j^k$ is the composition of the above map and the obvious map from ${\mathfrak{C}}_*^{\pi_j}(Y,\gamma)\otimes R_{N,j}$ to ${\mathfrak{C}}_*^{\pi_j}(Y,\gamma)\otimes R_{N,k} $. A standard argument shows that $F^l_k\circ F^k_j$ is chain homotopy equivalent to $F^l_j$ for any triple $j\geq k \geq l\geq 0$. Therefore, the maps $f_j^k$, induced by the maps $F_{j}^k$, has the required properties for an inverse system. This completes the definition of Fukaya-Floer homology $\mathbb I_*^{N}(Y,\gamma,L)$ for the triple. It is also standard to show that this $R_N$-module is independent of the choices that were made. The homomorphism $f_j^k$ maps ${\rm D}^{N,j}_{X,w}(z\e^{\sum_i \Gamma^i_{(i)}})\in \mathbb I_*^{N,j}(Y,\gamma,L)$ to ${\rm D}^{N,k}_{X,w}(z\e^{\sum_i \Gamma^i_{(i)}})\in \mathbb I_*^{N,k}(Y,\gamma,L)$. Thus we have an induced element of $\mathbb I_*^{N}(Y,\gamma,L)$ which we denote by ${\rm D}^{N}_{X,w}(z\e^{\sum_i \Gamma^i_{(i)}})$. Similarly, we can use the construction of the previous paragraph to define $\mathbb I^{N}_*(W,w,z\e^{\Gamma^2_{(2)}+\dots+\Gamma^N_{(N)}})$ and $\rD_{N}^{X,w}(z\e^{\Gamma^2_{(2)}+\dots+\Gamma^N_{(N)}})$.

Now we are ready to give a proof of \eqref{FF-gluing-thm}. For the convenience of the reader, we restate the claim as the following proposition:
\begin{prop} \label{FF-gluing-thm-2}
	Let $(X_1,w_1)$ be a pair whose boundary is equal to an admissible pair $(Y,\gamma)$.
	Let $(X_2,w_2)$ be another pair whose boundary is equal to
	$(\overline Y,\overline\gamma)$. Let $z_1\in \A(X_1)^{\otimes (N-1)}$ and $z_2\in \A(X_2)^{\otimes (N-1)}$.
	Let $\Gamma^j$ be a properly embedded surface in $X_1$ with boundary $l_j$. Let $\Lambda^j$ be a properly embedded
	surface in $X_2$ whose boundary is equal to $l_j$ with the reverse orientation. Then we can form the closed 4-manifold $X_2\circ X_1$ and the 2-cycle
	 $w_2\circ w_1$. The embedded surfaces $\Gamma^j$ and $\Lambda^j$ can be glued to each other along their boundary to form a closed surface
	 $\Gamma^j\#\Lambda^j$. Then the following invariant of the closed 4-manifold $X_2\circ X_1$:
	 \begin{equation} \label{inv-total-mfld}
			{\rm D}^N_{X_2\circ X_1,w_2\circ w_1}(z_1\cdot z_2\cdot \e^{(\Gamma^2\#\Lambda^2)_{(2)}+\dots+(\Gamma^N\#\Lambda^N)_{(N)}})
			\in \C[\![t_2,\dots,t_N]\!]\subset R_N
	 \end{equation}
	is equal to:
	 \begin{equation*}
			{\rm D}_N^{X_2,w_2}(z_2\cdot \e^{\Lambda^2_{(2)}+\dots+\Lambda^N_{(N)}}) \circ
			{\rm D}^N_{X_1,w_1}(z_1\cdot \e^{\Gamma^2_{(2)}+\dots+\Gamma^N_{(N)}}).	
	 \end{equation*}
\end{prop}

\begin{proof}
	We make simplifying assumptions as before; assume $z_1=z_2=1$ and $\Gamma^3$,
	$\cdots$, $\Gamma^N$,
	$\Lambda^3$, $\cdots$, $\Lambda^N$ are empty.
	Choose two series of properly embedded surfaces $\{\Sigma_i\}_{i \in \N}\subset X_1$ and
	$\{T_i\}_{i \in \N}\subset X_2$ such that $\Sigma_i$
	(respectively, $T_i$) is given by perturbing $\Gamma^2$ (respectively, $\Lambda^2$).
	We also assume that $\partial \Sigma_i$
	(respectively, $\partial T_i$) is equal to $\eta_i$ (respectively, $\eta_i$ with the reverse orientation),
	the curves $\eta_i$ are disjoint and
	the embedded surfaces $\Sigma_i$ and $T_i$ intersect generically. We fix metrics on
	$X_1$ and $X_2$ which are product metrics
	in a neighborhood of their boundaries corresponding to a fixed metric on $Y$.
	For each non-negative integer $j$, we prove that the image of the element in \eqref{inv-total-mfld} in $R_{N,j}$
	is equal to:
	\[
	  {\rm D}_{N,j}^{X_2,w_2}(\e^{\Lambda^2_{(2)}}) \circ
	  {\rm D}^{N,j}_{X_1,w_1}(\e^{\Gamma^2_{(2)}}).	
	\]
	To achieve this goal, we proceed as before to define open sets $\nu(\eta_i)\subset Y$, $\nu(\Sigma_i)\subset X_1^+$,
	$\nu(T_i)\subset X_2^+$ and the geometric representatives:
	\begin{equation*}
		V_2(\eta_i)\subset \mathcal{B}^{*}(\nu(\eta_i)\times (0,1))\times \eta_i\hspace{1cm}
		V_2(\Sigma_i)\subset \mathcal{B}^{**}(\nu(\Sigma_i))\times \Sigma_i\hspace{1cm}
		V_2(T_i)\subset \mathcal{B}^{**}(\nu(T_i))\times T_i\
	\end{equation*}
	suitable for the definition of the differential $d_{N,j}$ of the Fukaya-Floer chain complex and the relative elements
	${\rm D}^{N,j}_{X_1,w_1}(\e^{\Gamma^2_{(2)}})$
	and ${\rm D}_{N,j}^{X_2,w_2}(\e^{\Lambda^2_{(2)}})$.
	
	Suppose $X^T$ is the metric on $X_2 \circ X_1$ induced by the metrics on $X_1$ and $X_2$ with a neck of length $T$ along $Y$. Suppose $w\subset X^T$
	 is the 2-cycle induced by $w_1$ and $w_2$, and $R_i^T\subset X^T$ is the embedded surface induced by the surfaces $\Sigma_i$ and $T_i$. The 4-manifold $X$ can be decomposed into three pieces:
	 \begin{equation*}
	X_1 \hspace{1cm} X_2\hspace{1cm}
		Y \times [-\frac{T}{2},\frac{T}{2}].
	 \end{equation*}
	We assume that $X_1$ and $X_2$ are disjoint and they intersect $Y \times [-\frac{T}{2},\frac{T}{2}]$ in $Y \times [-\frac{T}{2},-\frac{T}{2}+2]$ and $Y \times [\frac{T}{2}-2,\frac{T}{2}]$, respectively.
	 The Riemann surface $R_i^T$ can be decomposed into union of three sets:
	 \begin{equation*}
		\Sigma_i^c \subset X_1 \hspace{1cm}T_i^c \subset X_2\hspace{1cm}
		\eta_i \times [-\frac{T}{2}+1,\frac{T}{2}-1]\subset Y \times [-\frac{T}{2}+1,\frac{T}{2}-1].
	 \end{equation*}
	 Suppose $\mathcal B^{**}_\kappa(X^T,w)$ is the subset of $\mathcal B_\kappa^*(X,w)$ which consists of connections whose restriction to any set of the following form is irreducible:
	 \begin{equation*}
		\hspace{3cm}\nu(\eta_i) \times (t-1,t)\hspace{1cm}-T+1<t<T
	 \end{equation*}
	 Suppose $\kappa$ is chosen such that the dimension of the moduli space $\mathcal M_\kappa(X^T,w)$ is at most
	 $2j$. Then unique continuation and gluing theory show that for large enough values of $T$, the moduli space
	 $\mathcal M_\kappa(X^T,w)$ is a subset of
	 $\mathcal B^{**}_\kappa(X^T,w)$.
	 Similar to the case of 4-manifolds with cylindrical ends, $V_2(\eta_i)$, $V_2(\Sigma_i)$ and $V_2(T_i)$
	 can be used to define a geometric representative  $V(R_i^T)\subset \mathcal B^{**}_\kappa(X^T,w)\times \nu(R_i^T)$.
	 Another application of gluing theory shows that for large enough values of
	 $T$, these divisors determine a transversal cut of
	 $\mathcal M_\kappa(X^T,w)\times R_{i_1}^T\times \dots \times R_{i_k}^T$
	 for any set $S=\{i_1,\dots,i_k\}\subset \{1,\dots,j\}$.
	 In the case that the cut down moduli space is zero dimensional, it can be identified with the following
	 set for large values of $T$:
	 \begin{equation*}
	 	\bigcup_{p_1,p_2,S_1,S_2} 	{\mathcal N}_{p_1}(X_1,w_1,\alpha,S_1) \times {\mathcal N}_{p_2}(X_2,w_2,\alpha,S_2)
	 \end{equation*}
	 where $p_i$ is a loop along $(X_i,w_i)$ based at the connection $\alpha$ such that $\kappa=\kappa(p_1)+\kappa(p_2)$. Moreover, $S_1$ and $S_2$ are disjoint sets with $S=S_1\cup S_2$.
	 This geometric results for different choices of $S$ can be translated to the following algebraic identity:
	 \begin{align}
		{\rm D}_{N,j}^{X_2,w_2}(\e^{\Lambda^2_{(2)}}) \circ {\rm D}^{N,j}_{X_1,w_1}(\e^{\Gamma^2_{(2)}})
		&=\sum_{S \subset \{1,\dots,j\}} (\prod_{ i\in S} t_{2,i})\rD_{X^T,w}^N(\prod_{ i\in S} (R_i^T)_{(2)})\nonumber\\
		&=\sum_{S \subset \{1,\dots,j\}} (\prod_{ i\in S} t_{2,i})
		\rD_{X^T,w}^N((\Gamma^2\#\Lambda^2)_{(2)}^{|S|})	\label{comp}
	 \end{align}
	 In the second equality, we used the fact that $R_j^T$ represents the homology class of $\Gamma^2\#\Lambda^2$.
	 The term in \eqref{comp} is equal to the image of  $\rD_{X^T,w}^N(\e^{(\Gamma^2\#\Lambda^2)_{(2)}})$
	 in $R_{N,j}$.
\end{proof}

\section{Questions and Conjectures}\label{que-conj}
In this section, we propose some questions and conjectures for future directions. This section is divided to two parts: the first subsection is concerned with the polynomial invariants of 4-manifolds. In the second part, we discuss some conjectures related to the algebra $\VgdN$.
\subsection{Structure of Polynomial Invariants and 4-manifolds with Simple type}
In Subsection \ref{blowup-sub}, the simple type property of 4-manifolds is defined using $\U(3)$-polynomial invariants. As we pointed out earlier, the definition is motivated by Kronheimer and Mrowka's simple type property, defined by $\U(2)$-polynomial invariants \cite{KM:str-thm}. There is another version of simple type property defined by Seiberg-Witten invariants. It is unknown if there is an example of a smooth 4-manifolds with $b^+\geq 2$ which do not have Kronheimer-Mrowka simple type or Seiberg-Witten simple type. It is also shown in \cite{FL:wit-conj} that many 4-manifolds with Seiberg-Witten simple type has Kronheimer-Mrowka simple type. Therefore, it is natural to ask whether there is any relationship among $\U(3)$-simple type, Kronheimer-Mrowka simple type and Seiberg-Witten simple type. A more challenging question would be to investigate whether there is a 4-manifold with $b^+\geq 2$ which does not have $\U(3)$-simple type. A more approachable question is the following:
\begin{question}
	What is the analogue of the simple type condition with respect to $\U(N)$-polynomial invariants?
\end{question}
\noindent
As in the $\U(2)$ and the $\U(3)$ case, the simple condition has to be formulated in terms of point classes. In the light of Proposition \ref{relations-T^3}, it is plausible that one of the required conditions is:
\begin{equation*}
	\rD_{X,w}^N(a_2^N z)=N^N\rD_{X,w}^N(z).
\end{equation*}
For $N=2,3$, the blowup formula for $\U(N)$ simple type manifolds have simpler form \cite{fintushel1996blowup,Luc:Thesis}. One might hope that the same holds for higher values of $N$, and follows this direction to gain more insights into the correct definition for the simple type condition. The physics literature \cite{MM:higher,EGM:blowup} suggests that the blowup formula for an arbitrary $N$ is related to function theory on a hyper-elliptic curve with coefficients in $\Q[a_2.\dots,a_N]$. Evaluation of $(a_2,\dots,a_N)$ determines hyper-elliptic curve on complex numbers, and the simple type condition is related to the evaluations that produce a fully degenerate curve.

The relationship between $\U(2)$-polynomial invariants and Seiberg-Witten invariants goes beyond the simple type conditions. In \cite{KM:str-thm}, Kronheimer and Mrowka prove that $\U(2)$-polynomial invariants are completely determined by a finite set of cohomology classes (known as Kronheimer-Mrowka basic classes) and a set of rational numbers, one for each basic class. In \cite{Wit:SW}, Witten argues that basic classes and corresponding rational numbers can be determined in terms of Seiberg-Witten invariants. To be more detailed, recall that for each $\text{spin}^{\text{c}}$ structure $\mathfrak s$ on a 4-manifold $X$ (satisfying appropriate conditions such as $b^+(X)\geq 2$) there is a Seiberg-Witten invariant $SW_X(\mathfrak s)$. Then Witten's conjecture states that any basic class of $X$ is equal to $c_1(S^+_{\mathfrak s})$ where $\mathfrak s$ is a $\text{spin}^{\text{c}}$ structure with non-zero $SW_X(\mathfrak s)$ and $S^+_{\mathfrak s}$ is the half-spin bundle associated to $\mathfrak s$. Witten's conjecture is generalized to higher values of $N$ in \cite{MM:higher}. The calculations of $\U(3)$-polynomial invariants in this paper agree with the Moore-Mari\~no Conjecture in \cite{MM:higher} and can be exploited to fix the undetermined constants in this conjecture. In particular, a modified version of the Moore-Mari\~no Conjecture states that:

\begin{conjecture}
	Let $X$ be a four-manifold which has $\U(3)$-simple type. Let $\{K_i\}$ be the set of Kronheimer-Mrowka basic classes.
	Then the $\U(3)$-series of $X$ has the following form:
	\begin{equation*}
		\widehat{\rD}_{X,w}(\e^{\Gamma_{(2)}+\Lambda_{(3)}})=\e^{\frac{Q(\Gamma)}{2}-Q(\Lambda)} \sum_{i,j} c_{ij}\zeta^{-w \cdot (\frac{K_i-K_j}{2})}
		\e^{\frac{{\sqrt 3}}{2}(K_i+K_j)\cdot \Gamma+\frac{{\sqrt 3}}{2}\bi(K_i-K_j)\cdot \Lambda}
	\end{equation*}		
	where $c_{i,j}$ is given as:
	\begin{equation*}
		2^{\chi+\frac{3}{2}\sigma+\frac{1}{2}K_i\cdot K_j}3^{2+\frac{7}{4}\chi+\frac{11}{4}\sigma} SW_X(\mathfrak s_i)SW_X(\mathfrak s_j)
	\end{equation*}		
	Here $\mathfrak s_i$ is chosen such that the associated basic class is equal to $K_i$.
\end{conjecture}

\subsection{The Algebra $\VgdN$}
In Subsection \ref{ev}, a list of simultaneous eigenvectors for the operators $\epsilon$, $\aleph_2$, $\aleph_3$, $\rho_{(2)}$ and $\rho_{(3)}$, acting on $ \mathbb V_{g,d}^3$, is constructed. We also showed that there are at least two non-degenerate simultaneous eigenvectors, which form the essential ingredient to prove the excision theorem in Subsection \ref{suture}. However, we do not know whether our approach produces all simultaneous eigenvectors:
\begin{conjecture}\label{ev-conj}
	Suppose $ V_{g,d}^{N}\subset \mathbb V_{g,d}^N$ is the set of vectors which are invariant with respect to the action of $\epsilon$.
	Then for $N=3$, any simultaneous eigenvector of the
	operators acting on $V_{g,d}^{N}$, that are induced by $\aleph_2$, $\aleph_3$, $\rho_{(2)}$ and $\rho_{(3)}$, have the form
	$3\zeta^{2d \beta}$, $0$, $ \zeta^{d\beta }\sqrt 3 \alpha$ and $\zeta^{2d \beta}\sqrt 3 \bi \beta$ with $(\alpha,\beta)\in \mathcal C_g$.
\end{conjecture}
\noindent
For $N=2$, there are three operators $\epsilon$, $\aleph_2$ and $\rho_{(2)}$. In this case, Mu\~noz obtains a complete understanding of the action of $\epsilon$, $\aleph_2$ and $\rho_{(2)}$ in \cite{Mun:ring}. In particular, his results show that the simultaneous eigenvectors of $\aleph_2$ and $\rho_{(2)}$, acting on $V_{g,d}^2$, have the following form:
\begin{equation*}
	\hspace{4cm}((-1)^{r}2,\pm 2r\bi^{r+1})\hspace{1cm} 0\leq r\leq g-1
\end{equation*}
All of these eigenvalues can be produced using the method of Proposition \ref{eigen}. Therefore, the analogue of Conjecture \ref{ev-conj} holds for $N=2$. Mu\~noz's method of understanding the action of  $\epsilon$, $\aleph_2$ and $\rho_{(2)}$ is based on the characterization of the ring structure of the cohomology ring $\mathcal N_{2,d}(\Sigma_g)$ in \cite{Zag:stable-ring,KN:stable-ring,ST:stable-ring,Bar:stable-ring}, which is not available for higher values of $N$.

If Conjecture \ref{ev-conj} holds, then we can use the method of \cite{KM:suture} and show that $\SHI_*(M,\alpha)$ is non-zero for a {\it taut balanced sutured manifold} \cite{juhasz,KM:suture}. This non-vanishing result can be used to show that Conjecture \ref{non-triv-knot-non-triv-KHI} holds. We can aslo use this to show that $\KHI_*(K)$ detects the genus of $K$. Thus the answer to Question \ref{rank-N-rep} for a non-trivial knot and $N=3$ is positive. In fact, in order to make this series of conclusions, we need the following weaker version of Conjecture \ref{ev-conj}:
\begin{conjecture}
If $(x,y)$ is a pair of simultaneous eigenvalues for $(\rho_{(2)},\rho_{(3)})$ in $\mathbb{V}_{g,d}$, then $|x|+|y|\le \sqrt{3}(2g-2)$.
\end{conjecture}

There is also a symplectic analogue of the algebra $\VgdN$.  The manifold $\Nnd(\Sigma_g)$ is K\"ahler and the associated Gromov-Witten invariants can be used to define the {\it Quantum Cohomology} ring $QH^*(\Nnd)$ \cite{RT:QH,MS:QH}. The underlying vector space of $QH^*(\Nnd)$ is $H^*(\Nnd)$ and the ring structure is also a deformation of the cup product. Therefore, it has similar structure to $V_{g,d}^N=\ker(\epsilon-1)$, and it is natural to make the following conjecture:
\begin{conjecture}
	The ring $V_{g,d}^N$ is isomorphic to $QH^*(\Nnd)$.
\end{conjecture}
\noindent
This conjecture for $N=2$ is proved by Mu\~noz \cite{Mu:QHNnd} using the characterization of the cohomology ring $H^*(\Nnd)$ in \cite{Zag:stable-ring,KN:stable-ring,ST:stable-ring,Bar:stable-ring}. The $N=2$ special case of this conjecture was also proved using an adiabatic limit argument in \cite{Sal:prod-AF}.

\appendix
\section{Invariants of Flat Connections on $\Sigma(2,3,23)$}
There are 44 irreducible flat $\SU(3)$-connections on $\Sigma(2,3,23)$ \cite{Bod:SU(3)-Bris}. These flat connections are determined by their holonomies along the loop $x_3$ in the standard presentation of the fundamental group of $\Sigma(2,3,23)$ (see \eqref{fund-gp}). For each flat connection, the conjugacy class of this holonomy is determined by its eigenvalues which have the form $\e^{2\pi \bi k/23}$, $\e^{2\pi \bi l/23}$ and $\e^{2\pi \bi m/23}$. The possible values of $\{k,l,m\}$ are given in Table~\ref{table-irr-hol}. The complex conjugation diffeomorphism of $\Sigma(2,3,23)$ maps a flat connection with the associated triple $\{k,l,m\}$ to the flat connection with the associated triple $\{23-k,23-l,23-m\}$. If this pair gives the same flat connections, we denote this connection with $\alpha_j$ for an appropriate choice of the integer $j$. Otherwise, the resulting connections are denoted by $\alpha_{j}^{1}$ and $\alpha_{j}^{2}$

\begin{table}[H]
	\begin{centering}
	\begin{tabular}{|c|c|c|c|c|c|c|c|}
		\hline
		$\alpha_1$&$\{0,4,19\}$&$\alpha_2$&$\{0,5,18\}$&$\alpha_3$&$\{0,6,17\}$&$\alpha_4$&$\{0,7,16\}$\\
		\hline		
		$\alpha_5$&$\{0,8,15\}$&$\alpha_6$&$\{0,9,14\}$&$\alpha_7$&$\{0,10,13\}$&$\alpha_8$&$\{0,11,12\}$\\ 		
		\hline		
		$\alpha_9^{1}$&$\{1,4,18\}$&$\alpha_{9}^{2}$&$\{5,19,22\}$&$\alpha_{10}^{1}$&$\{1,5,17\}$&$\alpha_{10}^{2}$&$\{6,18,22\}$\\ 	
		\hline
		$\alpha_{11}^{1}$&$\{1,6,16\}$&$\alpha_{11 }^{2}$&$\{7,17,22\}$&$\alpha_{12}^{1}$&$\{1,7,15\}$&$\alpha_{12}^{2}$&$\{8,16,22\}$\\
		\hline		
		$\alpha_{13}^{1}$&$\{1,8,14\}$&$\alpha_{13}^{2}$&$\{9,15,22\}$&$\alpha_{14}^{1}$&$\{1,9,13\}$&$\alpha_{14}^{2}$&$\{10,14,22\}$\\	
		\hline		
		$\alpha_{15}^{1}$&$\{1,10,12\}$&$\alpha_{15}^{2}$&$\{11,13,22\}$&$\alpha_{16}^{1}$&$\{2,4,17\}$&$\alpha_{16}^{2}$&$\{6,19,21\}$\\ 		
		\hline		
		$\alpha_{17}^{1}$&$\{2,5,16\}$&$\alpha_{17}^{2}$&$\{7,18,21\}$&$\alpha_{18}^{1}$&$\{2,6,15\}$&$\alpha_{18}^{2}$&$\{8,17,21\}$\\ 		
		\hline		
		$\alpha_{19}^{1}$&$\{2,7,14\}$&$\alpha_{19}^{2}$&$\{9,16,21\}$&$\alpha_{20}^{1}$&$\{2,8,13\}$&$\alpha_{20}^{2}$&$\{10,15,21\}$\\
		\hline		
		$\alpha_{21}^{1}$&$\{2,9,12\}$&$\alpha_{21}^{2}$&$\{11,14,21\}$&$\alpha_{22}^{1}$&$\{3,4,16\}$&$\alpha_{22}^{2}$&$\{7,19,20\}$\\ 		
		\hline		
		$\alpha_{23}^{1}$&$\{3,5,15\}$&$\alpha_{23}^{2}$&$\{8,18,20\}$&$\alpha_{24}^{1}$&$\{3,6,14\}$&$\alpha_{24}^{2}$&$\{9,17,20\}$\\ 		
		\hline
		$\alpha_{25}^{1}$&$\{3,7,13\}$&$\alpha_{25}^{2}$&$\{10,16,20\}$&$\alpha_{26}^{1}$&$\{3,8,12\}$&$\alpha_{26}^{2}$&$\{11,15,20\}$\\
		\hline																		
	\end{tabular}
	\caption{Holonomies of irreducible flat $\SU(3)$-connections on $\Sigma(2,3,23)$ along $x_3$}\label{table-irr-hol}
	\end{centering}
\end{table}

The gauge theoretical invariants of these flat connections are given in the following tables:

\begin{table}[H]
		\resizebox{\hsize}{!}{
	\begin{tabular}{|c|c|c|c|c|c|c|c|c|c|c|c|c|c|}
		\hline
		&&&&&&&&&&&&&\\
		$\alpha$&$\alpha_1$&$\alpha_2$&$\alpha_3$&$\alpha_4$&$\alpha_5$&$\alpha_6$&$\alpha_7$&$\alpha_8$&$\alpha_9^{i}$&$\alpha_{10}^i$&$\alpha_{11}^i$&$\alpha_{12}^i$&$\alpha_{13}^i$\\
		&&&&&&&&&&&&&\\
		\hline
		&&&&&&&&&&&&&\\
		$CS(\alpha)$&$\frac{1}{138}$&$\frac{49}{138}$&$\frac{31}{138}$&$\frac{85}{138}$&$\frac{73}{138}$&$\frac{133}{138}$&$\frac{127}{138}$&$\frac{55}{138}$&$\frac{43}{138}$&$\frac{127}{138}$&$\frac{7}{138}$&$\frac{97}{138}$&$\frac{121}{138}$\\
		&&&&&&&&&&&&&\\
		\hline
		&&&&&&&&&&&&&\\
		$\rho_{ad_{\alpha}}$&$-\frac{364}{23}$&$-\frac{540}{23}$&$-\frac{520}{23}$&$-\frac{488}{23}$&$-\frac{444}{23}$&$-\frac{572}{23}$&$-\frac{504}{23}$&$-\frac{424}{23}$&$-\frac{472}{23}$&$-\frac{412}{23}$&$-\frac{524}{23}$&$-\frac{532}{23}$&$-\frac{528}{23}$\\		
		&&&&&&&&&&&&&\\
		\hline
		&&&&&&&&&&&&&\\
		$\deg$&$4$&$0$&$10$&$2$&$0$&$8$&$6$&$10$&$10$&$4$&$8$&$4$&$6$\\
		&&&&&&&&&&&&&\\
		\hline
	\end{tabular}}
	\caption{Gauge theoretical invariants of irreducible flat $\SU(3)$-connections on $\Sigma(2,3,23)$ (first part)}\label{inv-irr-23-1}
\end{table}
\begin{table}[H]
	\resizebox{\hsize}{!}{
	\begin{tabular}{|c|c|c|c|c|c|c|c|c|c|c|c|c|c|}
		\hline
		&&&&&&&&&&&&&\\
		$\alpha$&$\alpha_{14}^i$&$\alpha_{15}^i$&$\alpha_{16}^i$&$\alpha_{17}^i$&$\alpha_{18}^i$&$\alpha_{19}^i$&$\alpha_{20}^i$&$\alpha_{21}^i$&$\alpha_{22}^i$&$\alpha_{23}^i$&$\alpha_{24}^i$&$\alpha_{25}^i$&$\alpha_{26}^i$\\
		&&&&&&&&&&&&&\\
		\hline
		&&&&&&&&&&&&&\\
		$CS(\alpha)$&$\frac{79}{138}$&$\frac{109}{138}$&$\frac{19}{138}$&$\frac{1}{138}$&$\frac{55}{138}$&$\frac{43}{138}$&$\frac{103}{138}$&$\frac{97}{138}$&$\frac{67}{138}$&$\frac{85}{138}$&$\frac{37}{138}$&$\frac{61}{138}$&$\frac{19}{138}$\\
		&&&&&&&&&&&&&\\
		\hline
		&&&&&&&&&&&&&\\
		$\rho_{ad_{\alpha}}$&$-\frac{420}{23}$&$-\frac{484}{23}$&$-\frac{476}{23}$&$-\frac{456}{23}$&$-\frac{516}{23}$&$-\frac{472}{23}$&$-\frac{508}{23}$&$-\frac{440}{23}$&$-\frac{468}{23}$&$-\frac{488}{23}$&$-\frac{404}{23}$&$-\frac{492}{23}$&$-\frac{476}{23}$\\		
		&&&&&&&&&&&&&\\
		\hline
		&&&&&&&&&&&&&\\
		$\deg$&$0$&$4$&$8$&$6$&$0$&$10$&$4$&$2$&$0$&$2$&$8$&$0$&$8$\\
		&&&&&&&&&&&&&\\
		\hline
	\end{tabular}}
	\caption{Gauge theoretical invariants of irreducible flat $\SU(3)$-connections on $\Sigma(2,3,23)$ (second part)}\label{inv-irr-23-2}
\end{table}

There are 8 non-trivial flat $\SU(2)$-connections on $\Sigma(2,3,23)$. As in the irreducible case, these connections are determined by the conjugacy class of their holonomies along $x_3$. For each $2\leq k \leq 9$, there is a unique flat $\SU(2)$-connection on $\Sigma(2,3,23)$ where the eigenvalues of holonomy along $x_3$ are equal to by $\e^{2\pi \bi k/23}$ and $\e^{-2\pi \bi k/23}$. We will write $\beta_k$ for this connection. The gauge theoretical invariants of these connections are given in Table~\ref{inv-red-23}. In this table, $\widetilde \alpha$ denotes the reducible $\SU(3)$-connection associated to an $\SU(2)$-connection $\alpha$.

\begin{table}[H]
	\begin{centering}
		\begin{tabular}{|c|c|c|c|c|c|c|c|c|}
			\hline
			&&&&&&&&\\
			$\alpha$&$\beta_{2}$&$\beta_{3}$&$\beta_{4}$&$\beta_{5}$&$\beta_{6}$&$\beta_{7}$&$\beta_{8}$&$\beta_{9}$\\
			&&&&&&&&\\
			\hline
			&&&&&&&&\\
			$CS(\alpha)$&$\frac{1}{552}$&$\frac{169}{552}$&$\frac{73}{552}$&$\frac{265}{552}$&$\frac{193}{552}$&$\frac{409}{552}$&$\frac{361}{552}$&$\frac{49}{552}$\\
			&&&&&&&&\\
			\hline
			&&&&&&&&\\
			$\rho_{ad_{\alpha}}$&$-\frac{343}{69}$&$-\frac{559}{69}$&$-\frac{475}{69}$&$-\frac{643}{69}$&$-\frac{511}{69}$&$-\frac{631}{69}$&$-\frac{451}{69}$&$-\frac{523}{69}$\\
			&&&&&&&&\\
			\hline
			&&&&&&&&\\
			$\rho_{ad_{\widetilde \alpha}}$&$-\frac{206}{23}$&$-\frac{406}{23}$&$-\frac{410}{23}$&$-\frac{402}{23}$&$-\frac{382}{23}$&$-\frac{534}{23}$&$-\frac{490}{23}$&$-\frac{434}{23}$\\
			&&&&&&&&\\
			\hline
			&&&&&&&&\\
			$\deg(\alpha)$&$1$&$5$&$3$&$7$&$5$&$1$&$7$&$3$\\
			&&&&&&&&\\
			\hline
			&&&&&&&&\\
			$\deg(\widetilde \alpha)$&$1$&$9$&$7$&$11$&$9$&$5$&$3$&$7$\\
			&&&&&&&&\\
			\hline
		\end{tabular}
		\caption{Gauge theoretical invariants of reducible flat $\SU(3)$-connections on $\Sigma(2,3,23)$}\label{inv-red-23}
	\end{centering}
\end{table}


\bibliography{references}
\bibliographystyle{hplain}
\Addresses

\end{document}